\documentclass[12pt]{amsart}
\usepackage{amssymb}
\usepackage[mathcal]{eucal}
\usepackage{longtable,array}
\usepackage{enumitem}
\usepackage{hyperref}
\usepackage{tikz-cd}
\usepackage[nobysame]{amsrefs}
\usepackage[total={6.25truein,9truein},centering]{geometry}

\allowdisplaybreaks

\makeatletter
\let\c@equation\c@subsubsection

\makeatother

\newtheorem{theorem}[subsubsection]{Theorem}
\newtheorem{lemma}[subsubsection]{Lemma}
\newtheorem{proposition}[subsubsection]{Proposition}
\newtheorem{corollary}[subsubsection]{Corollary}

\newtheorem{theoremintro}{Theorem}

\newtheorem{corollaryintro}[theoremintro]{Corollary}

\theoremstyle{definition}

\theoremstyle{remark}
\newtheorem{example}[subsubsection]{Example}
\newtheorem{remark}[subsubsection]{Remark}
\newtheorem{definition}[subsubsection]{Definition}

\newenvironment{alphenumerate}{\begin{enumerate}[label=\textup{(\alph*)}]}{\end{enumerate}}

\newcommand{\AAA}{\mathbb{A}}
\newcommand{\FF}{\mathbb{F}}
\newcommand{\ZZ}{\mathbb{Z}}

\newcommand{\RR}{\mathbb{R}}
\newcommand{\CC}{\mathbb{C}}

\newcommand{\TT}{\mathbb{T}}

\newcommand{\DD}{\mathbb{D}}
\newcommand{\EE}{\mathbb{E}}
\newcommand{\LL}{\mathbb{L}}
\newcommand{\MM}{\mathbb{M}}
\newcommand{\OO}{\mathbb{O}}
\newcommand{\SSS}{\mathbb{S}}
\newcommand{\HH}{\mathbb{H}}
\newcommand{\KK}{\mathbb{K}}

\newcommand{\bA}{\mathbf{A}}
\newcommand{\bb}{\mathbf{b}}
\newcommand{\bC}{\mathbf{C}}

\newcommand{\bE}{\mathbf{E}}

\newcommand{\bh}{\mathbf{h}}
\newcommand{\bi}{\mathbf{i}}
\newcommand{\bj}{\mathbf{j}}
\newcommand{\bk}{\mathbf{k}}

\newcommand{\bm}{\mathbf{m}}
\newcommand{\bn}{\mathbf{n}}

\newcommand{\bP}{\mathbf{P}}
\newcommand{\bQ}{\mathbf{Q}}

\newcommand{\bs}{\mathbf{s}}

\newcommand{\bu}{\mathbf{u}}
\newcommand{\bv}{\mathbf{v}}
\newcommand{\bx}{\mathbf{x}}

\newcommand{\by}{\mathbf{y}}

\newcommand{\bz}{\mathbf{z}}

\newcommand{\cA}{\mathcal{A}}

\newcommand{\cD}{\mathcal{D}}
\newcommand{\cE}{\mathcal{E}}

\newcommand{\cI}{\mathcal{I}}
\newcommand{\cL}{\mathcal{L}}
\newcommand{\cM}{\mathcal{M}}
\newcommand{\cN}{\mathcal{N}}

\newcommand{\cS}{\mathcal{S}}

\newcommand{\rB}{\mathrm{B}}

\newcommand{\rD}{\mathrm{D}}

\newcommand{\rG}{\mathrm{G}}
\newcommand{\rH}{\mathrm{H}}

\newcommand{\rI}{\mathrm{I}}

\newcommand{\rU}{\mathrm{U}}
\newcommand{\rX}{\mathrm{X}}
\newcommand{\rY}{\mathrm{Y}}

\newcommand{\sA}{\mathsf{A}}
\newcommand{\sC}{\mathsf{C}}

\newcommand{\bsalpha}{\boldsymbol{\alpha}}
\newcommand{\bsbeta}{\boldsymbol{\beta}}
\newcommand{\bsgamma}{\boldsymbol{\gamma}}

\newcommand{\bseta}{\boldsymbol{\eta}}
\newcommand{\bslambda}{\boldsymbol{\lambda}}
\newcommand{\bsmu}{\boldsymbol{\mu}}
\newcommand{\bsnu}{\boldsymbol{\nu}}
\newcommand{\bszeta}{\boldsymbol{\zeta}}

\DeclareMathOperator{\Aut}{Aut}

\DeclareMathOperator{\diag}{diag}

\DeclareMathOperator{\Exp}{Exp}
\DeclareMathOperator{\ev}{ev}
\DeclareMathOperator{\id}{id}
\DeclareMathOperator{\Log}{Log}

\DeclareMathOperator{\GL}{GL}
\DeclareMathOperator{\Lie}{Lie}
\DeclareMathOperator{\Mat}{Mat}
\DeclareMathOperator{\Char}{Char}

\DeclareMathOperator{\Span}{Span}

\DeclareMathOperator{\Gal}{Gal}
\DeclareMathOperator{\Hom}{Hom}
\DeclareMathOperator{\Tr}{Tr}
\DeclareMathOperator{\rank}{rank}
\DeclareMathOperator{\Res}{Res}
\DeclareMathOperator{\sgn}{sgn}
\DeclareMathOperator{\ord}{ord}

\DeclareMathOperator{\im}{Im}

\newcommand{\oF}{\mkern2.5mu\overline{\mkern-2.5mu F}}

\newcommand{\oK}{\mkern2.5mu\overline{\mkern-2.5mu K}}

\newcommand{\oM}{\mkern2.5mu\overline{\mkern-2.5mu M}}
\newcommand{\oalpha}{\overline{\alpha}}
\newcommand{\oGamma}{\overline{\Gamma}}
\newcommand{\ochi}{\overline{\chi}}
\newcommand{\oeta}{\overline{\eta}}
\newcommand{\okappa}{\overline{\kappa}}
\newcommand{\ophi}{\mkern2.5mu\overline{\mkern-2.5mu \phi}}
\newcommand{\opsi}{\mkern2.5mu\overline{\mkern-2.5mu \psi}}
\newcommand{\otheta}{\mkern2.5mu\overline{\mkern-2.5mu \theta}}
\newcommand{\oTheta}{\overline{\Theta}}
\newcommand{\oUpsilon}{\overline{\Upsilon}}

\newcommand{\perf}{\mathrm{perf}}

\newcommand{\Reg}{\mathrm{Reg}}
\newcommand{\sep}{\mathrm{sep}}
\newcommand{\St}{\mathrm{St}}

\newcommand{\nr}{\mathrm{nr}}

\newcommand{\tpi}{\widetilde{\pi}}

\newcommand{\tbA}{\widetilde{\bA}}
\newcommand{\tbP}{\widetilde{\bP}}
\newcommand{\tAA}{\widetilde{\AAA}}
\newcommand{\tKK}{\widetilde{\KK}}

\newcommand{\LLhat}{\widehat{\LL}}
\newcommand{\C}{\CC_{\infty}}

\newcommand{\oDD}{\overline{\DD}}
\newcommand{\oEE}{\overline{\EE}}
\newcommand{\soEE}{\mkern1mu\overline{\mkern-1mu \EE}}

\newcommand{\sobE}{\mkern1mu\overline{\mkern-1mu \bE}}
\newcommand{\oFF}{\overline{\FF}}

\newcommand{\tauid}{{\tau=\mathrm{id}}}

\newcommand{\iso}{\stackrel{\sim}{\longrightarrow}}
\newcommand{\mayeq}{\stackrel{?}{=}}
\newcommand{\power}[2]{{#1 [\![ #2 ]\!]}}
\newcommand{\laurent}[2]{{#1 (\!( #2 )\!)}}
\newcommand{\FZ}[2]{{#1 \langle #2 \rangle }}

\newcommand{\norm}[1]{\lvert #1 \rvert}
\newcommand{\dnorm}[1]{\lVert #1 \rVert}

\newcommand{\inorm}[1]{{\lvert #1 \rvert}_{\infty}}

\newcommand{\Aord}[2]{{[ #1 ]}_{#2}}
\newcommand{\bigAord}[2]{{\left[ #1 \right]}_{#2}}
\newcommand{\pd}{\partial}

\newcommand{\tr}{{\mathsf{T}}}
\newcommand{\apairing}[2]{\langle #1 \mathbin{,} #2 \rangle}

\newcommand{\bpairing}[2]{[ #1 \mathbin{,} #2 ]}
\newcommand{\assign}{\mathrel{\vcenter{\baselineskip0.5ex \lineskiplimit0pt
                     \hbox{\scriptsize.}\hbox{\scriptsize.}}}%
                     =}
\newcommand{\rassign}{=%
                     \mathrel{\vcenter{\baselineskip0.5ex \lineskiplimit0pt
                     \hbox{\scriptsize.}\hbox{\scriptsize.}}}%
                     }

\begin{document}

\title{Convolutions of Goss and Pellarin $L$-series}

\author{Wei-Cheng Huang}
\address{Department of Mathematics, University of Rochester, Rochester, NY 14627, U.S.A.}
\email{weicheng.huang1111@gmail.com}

\author{Matthew A. Papanikolas}
\address{Department of Mathematics, Texas A{\&}M University, College Station, TX 77843, U.S.A.}
\email{papanikolas@tamu.edu}

\subjclass{Primary 11M38; Secondary 11G09, 11M32}

\date{November 24, 2026}

\begin{abstract}
We establish special value results of convolutions of Goss and Pellarin $L$-series attached to Drinfeld modules that take values in Tate algebras. Applying the class module formula of Demeslay to certain rigid analytic twists of one Drinfeld module by another, we extend the special value formula for the Pellarin $L$-function associated to the Carlitz module and the Anderson-Thakur function to Drinfeld modules of arbitrary rank and their rigid analytic trivializations. By way of the theory of Schur polynomials these identities take the form of specializations of convolutions of Rankin-Selberg type. These convolution $L$-series are also identified with covolumes of Stark units.
\end{abstract}

\keywords{Goss $L$-series, Pellarin $L$-series, Drinfeld modules, Anderson $t$-modules, Tate algebras, Poonen pairings, class module formulas, Schur polynomials}

\maketitle

\tableofcontents

\section{Introduction} \label{S:Intro}

Let $\FF_q$ be a field with $q=p^m$ elements for $p$ a prime. For a variable $\theta$ we let $A \assign \FF_q[\theta]$ be a polynomial ring in~$\theta$ over~$\FF_q$, and let $K \assign \FF_q(\theta)$ be its fraction field. We take $K_{\infty} \assign \laurent{\FF_q}{\theta^{-1}}$ for the completion of $K$ at $\infty$, and let $\C$ be the completion of an algebraic closure of $K_{\infty}$. We normalize the $\infty$-adic norm $\inorm{\,\cdot\,}$ on $\C$ so that $\inorm{\theta} = q$, and letting $\deg \assign -\ord_{\infty} = \log_q \inorm{\,\cdot\,}$, we see that $\deg a = \deg_{\theta} a$ for any $a \in A$. Finally, we let $A_+$ denote the monic elements of $A$.

For a variable $z$ independent from $\theta$, we let $\TT_z \subseteq \power{\C}{z}$ denote the Tate algebra of power series that converge on the closed unit disk of~$\C$, and we let $\TT_z(K_\infty) \assign \TT_z \cap \power{K_{\infty}}{z}$. We let $\LLhat_z$ (resp.\ $\KK_{\infty}$) denote the completion of the fraction field of $\TT_z$ (resp.\ $\TT_z(K_{\infty})$), and we note that $\TT_z(K_{\infty}) = \laurent{\FF_q[z]}{\theta^{-1}}$ and $\KK_{\infty} = \laurent{\FF_q(z)}{\theta^{-1}}$. The Gauss norm $\dnorm{\,\cdot\,}$ on $\TT_z$ extends uniquely to $\LLhat_z$. We let $\AAA \assign \FF_q(z)[\theta]$.

\subsection{Motivation}
In groundbreaking work Pellarin~\cite{Pellarin12} introduced a new class of $L$-functions that take values in Tate algebras. In particular, he defined for $s \in \ZZ_+$,
\begin{equation}
L(\AAA,s) = \sum_{a \in A_+} \frac{a(z)}{a^s} \quad \in \TT_z(K_{\infty}),
\end{equation}
which for fixed $s$ is in fact an entire function of $z$. By ordering the sum appropriately by degree, $L(\AAA,s)$ can be extended to all $s\in \ZZ$. This $L$-function possesses intriguing special value formulas, such as
\begin{equation} \label{E:Pellarinvalueintro}
L(\AAA,1) = -\frac{\tpi}{(z-\theta)\omega_z},
\end{equation}
where $\tpi \in \C$ is the Carlitz period (see Example~\ref{Ex:Carlitz2}) and
\begin{equation} \label{E:omegazdef}
\omega_z \assign (-\theta)^{1/(q-1)} \prod_{i=0}^{\infty} \biggl( 1 - \frac{z}{\theta^{q^i}} \biggr)^{-1} \quad \in \TT_z^{\times}
\end{equation}
is the Anderson-Thakur function defined in~\cite{AndThak90}. Subsequently, Angl\`es, Pellarin, and Tavares Ribeiro~\cite{APT16} considered classes of Drinfeld modules over $\LLhat_z$ defined by conjugating the Carlitz module by~$\omega_z$ or similar functions. In particular,
if we let $\bA \assign \FF_q(z)[t]$ be the polynomial ring in a new variable~$t$, they defined the $\FF_q(z)$-algebra homomorphism $\bC:\bA \to \AAA[\tau]$ so that
\begin{equation}
\bC_t = \omega_z^{-1} \cdot \sC_t \cdot \omega_z = \theta + (z-\theta)\tau.
\end{equation}
Then $\bC$ is a twist of the Carlitz module~$\sC$ (see Example~\ref{Ex:Carlitz1}) by $\omega_z$, taking values in the twisted polynomial ring $\AAA[\tau]$ in the $q$-th power Frobenius operator~$\tau$ on $\C$ that is extended $\FF_q(z)$-linearly to $\LLhat_z$ (see \S\ref{SSS:twistedpolys}).

About the same time Taelman~\cites{Taelman09, Taelman10, Taelman12} developed a theory of special values of Goss $L$-functions attached to Drinfeld modules defined over finite extensions of~$K$. Although at first unrelated to Pellarin's $L$-functions, Angl\`{e}s, Pellarin, and Tavares Ribeiro~\cite{APT16} used Demeslay's extension~\cites{DemeslayPhD, Demeslay22} of Taelman's class module formula (see Theorem~\ref{T:Demeslay}) to establish a connection. Demeslay's formula here is an identity in $\KK_{\infty}$,
\begin{equation} \label{E:DemeslayC}
\prod_f \frac{\Aord{\FF_f(z)}{\AAA}}{\Aord{\bC(\FF_f(z))}{\AAA}} = \Reg_{\bC} \cdot \Aord{\rH(\bC)}{\AAA}.
\end{equation}
On the left-hand side the product is taken over all irreducible $f \in A_+$, for which we write $\FF_f \assign A/fA$, and for a finitely generated torsion $\bA$-module $M$, $\Aord{M}{\AAA}$ denotes the \emph{$\AAA$-order} of $M$, i.e., the monic generator of the Fitting ideal of $M$ in $\bA$, coerced into $\AAA$. Moreover,
\[
\prod_f \frac{\Aord{\FF_f(z)}{\AAA}}{\Aord{\bC(\FF_f(z))}{\AAA}} = \prod_f \frac{f}{f-f(z)}
= L(\AAA,1).
\]
The right-hand side of~\eqref{E:DemeslayC} includes the regulator of $\bC$ and the $\AAA$-order of the class module $\rH(\bC)$. In this case $\Reg_{\bC}$ can be identified with the right-hand side of~\eqref{E:Pellarinvalueintro} and the class module is trivial (see also Example~\ref{Ex:Carlitz3} and \S\ref{SS:Pellarin}). Thus Demeslay's identity bridges Taelman special $L$-value identities and the special value formulas of Pellarin.

Using this story about the Carlitz module as a guide, the goal of the present paper has been to address three questions that arise naturally in the context of Drinfeld modules of arbitrary rank defined over global function fields, especially defined over $A$ itself.
\begin{itemize}
\item What is a reasonable definition of a Pellarin $L$-series attached to a Drinfeld module over~$A$?
\item To what extent can the twist $\bC$ of the Carlitz module by~$\omega_z$ be generalized in the context of Drinfeld modules of higher rank?
\item What does Demeslay's class module formula reveal about special values of their $L$-functions?
\end{itemize}
Answers to these questions in the case that a Drinfeld module is conjugated by $\omega_z$ have been obtained by Angl\`es and Tavares Ribeiro~\cite{AT17} and Gezmi\c{s}~\cite{Gezmis19} (see Example~\ref{Ex:LphixC}). Their work indicates that rather than this being a circumstance tied to a single Drinfeld module, that one could consider interactions between two Drinfeld modules.

Thus to answer these questions in full, for two Drinfeld modules $\phi$ and $\psi$ defined over~$A$, we define a $t$-module $\EE(\phi \times \psi)$ that is the twist of $\phi$ by the rigid analytic trivialization of~$\psi$. Then its associated $L$-function includes a Rankin-Selberg type convolution of a Goss $L$-series and a Pellarin $L$-series (see Theorem~\ref{T:LEEintro}) and can be evaluated using Demeslay's identity (see Theorem~\ref{T:EEAordintro} and Corollaries~\ref{C:Lmunurxrintro} and~\ref{C:Lmunurxlintro}). We now summarize these results.

\subsection{Rigid analytic twists}
Let $\sA = \FF_q[t]$, and let $\phi$, $\psi : \sA \to A[\tau]$ be Drinfeld modules defined over~$A$ by
\begin{equation} \label{E:phipsidefintro}
\phi_t = \theta + \kappa_1 \tau + \cdots + \kappa_r \tau^r, \quad
\psi_t = \theta + \eta_1 \tau + \cdots + \eta_{\ell} \tau^{\ell}, \quad \kappa_r, \eta_\ell \in \FF_q^{\times}.
\end{equation}
Thus $\phi$ has rank $r$ and $\psi$ has rank $\ell$, and moreover because their leading coefficients are in $\FF_q^{\times}$, both $\phi$ and $\psi$ have everywhere good reduction. Using the theory of Anderson generating functions, we define a rigid analytic trivialization $\Upsilon_{\psi,z} \in \GL_{\ell}(\TT_z)$ for $\psi$ (see \eqref{E:Upsilondef}) such that if
\begin{equation} \label{E:Thetadefintro}
\Theta_{\psi,z} = \begin{pmatrix}
0 & \cdots & 0 & (z-\theta)/\eta_{\ell} \\
1 & \cdots & 0 & -\eta_1/\eta_{\ell} \\
\vdots & \ddots & \vdots & \vdots \\
0 & \cdots & 1 & -\eta_{\ell-1}/\eta_{\ell}
\end{pmatrix} \quad \in \Mat_{\ell}(\AAA),
\end{equation}
then
\[
\Upsilon_{\psi,z}^{(1)} = \Upsilon_{\psi,z} \Theta_{\psi,z},
\]
where $\Upsilon_{\psi,z}^{(1)}$ denotes the Frobenius twists of the entries of $\Upsilon_{\psi,z}$ (see \S\ref{SSS:Frobops}). We then define $\EE = \EE(\phi \times \psi) : \bA \to \Mat_{\ell}(\AAA[\tau])$ to be the Anderson $t$-module determined by
\begin{align} \label{E:EEtintro}
\EE_t &= \Upsilon_{\psi,z}^{-1} \cdot \phi_t^{\oplus \ell} \cdot \Upsilon_{\psi,z}\\
&= \theta \rI_{\ell} + \kappa_1 \Theta_{\psi,z} \tau + \kappa_2 \Theta_{\psi,z} \Theta_{\psi,z}^{(1)} \tau^2 + \cdots + \kappa_r \Theta_{\psi,z} \Theta_{\psi,z}^{(1)} \cdots \Theta_{\psi,z}^{(r-1)} \tau^{r}, \notag
\end{align}
where $\rI_{\ell}$ is the $\ell \times \ell$ identity matrix.

Thus $\EE(\phi \times \psi)$ is the conjugation of the direct sum $\phi^{\oplus \ell}$ by the rigid analytic trivialization $\Upsilon_{\psi,z}$. It induces an $\bA$-module structure on $\LLhat_z^{\ell}$, and as such it is a $t$-module over~$\LLhat_z$. It was shown in~\cite{GezmisP19} that the induced exponential function $\Exp_{\phi} : \LLhat_z \to \LLhat_z$ over $\LLhat_z$ is surjective, and it follows that $\Exp_{\EE} : \LLhat_z^{\ell} \to \LLhat_z^{\ell}$ is also surjective. Moreover, if we let $\Lambda_{\phi} = \ker \Exp_{\phi}$ and $\Lambda_{\EE} = \ker \Exp_{\EE}$, then
\[
\Lambda_{\EE} = \Upsilon_{\psi,z}^{-1}\, \Lambda_{\phi}^{\oplus \ell},
\]
and $\Lambda_{\EE}$ is a free $\AAA$-module of rank $r\ell$. Furthermore, for $\nu \in \sA$, $\nu \neq 0$, we find that the $\nu$-torsion submodule $\EE[\nu] = \{ \bx \in \LLhat_z^{\ell} \mid \EE_{\nu}(\bx) = 0 \}$ yields an isomorphism of $\bA$-modules, $\EE[\nu] \cong (\bA/\nu \bA)^{r\ell}$. To that end if $\lambda \in \sA_+$ is irreducible, then we have a Tate module
\[
T_{\lambda}(\EE) = \varprojlim \EE[\lambda^m] \cong \bA_{\lambda}^{r \ell},
\]
where $\bA_{\lambda}$ is the $\lambda$-adic completion of $\bA$. See \S\ref{SS:EEprops} for more details.

Demeslay's class module formula (Theorem~\ref{T:Demeslay}) applies generally to Anderson $t$-modules over Tate algebras, and in particular to $\EE$. In this case it states
\begin{equation} \label{E:DemeslayEEintro}
\prod_f \frac{\Aord{\Lie(\soEE)(\FF_f(z))}{\AAA}}{\Aord{\soEE(\FF_f(z))}{\AAA}} = \Reg_{\EE} \cdot \bigAord{\rH(\EE)}{\AAA} \quad \in \KK_{\infty},
\end{equation}
where as in \eqref{E:DemeslayC} the product is taken over all irreducible $f \in A_+$. Here $\oEE$ denotes the reduction of $\EE$ modulo $f$, $\Reg_{\EE}$ is the regulator of $\EE$, and $\rH(\EE)$ is its class module. Our initial task is to determine the factors in this product for each~$f$, and then we associate it to the special value of an $L$-function.

\subsection{Characteristic polynomials of Frobenius}
For our Drinfeld module $\phi$ in \eqref{E:phipsidefintro}, if we fix $f \in A_+$ irreducible of degree $d$ and let $\lambda \in \sA_+$ be irreducible with $\lambda(\theta) \neq f$, then by work of Gekeler, Hsia, Takahashi, and Yu~\cites{Gekeler91, HsiaYu00, Takahashi82}, the characteristic polynomial $P_{\phi,f}(X) = \Char(\tau^d,T_{\lambda}(\ophi),X) = X^r + c_{r-1} X^{r-1} + \cdots + c_0 \in A[X]$ of $\tau^d$ acting on $T_{\lambda}(\ophi)$ satisfies $c_0 = (-1)^r \ochi_{\phi}(f) f$, where $\chi_{\phi}(a) \assign ((-1)^{r+1} \kappa_r)^{\deg a}$ and $\ochi_{\phi} = \chi_{\phi}^{-1}$, and moreover,
\begin{equation} \label{E:ophiorderintro}
\bigAord{\ophi(\FF_f)}{A} =  (-1)^r \chi_{\phi}(f) \cdot P_{\phi,f}(1).
\end{equation}
See \S\ref{SS:munu} for more details.
Our first result is to establish an identity corresponding to~\eqref{E:ophiorderintro} for $\Aord{\soEE(\FF_f(z))}{\AAA}$. Letting $P_{\psi,f}^{\vee}(X) = \Char(\tau^d,T_{\lambda}(\opsi)^{\vee},X) \in K[X]$, where $T_{\lambda}(\opsi)^{\vee}$ is the dual of $T_{\lambda}(\opsi)$, we let $P_{\psi,f(z)}^{\vee}(X) \in \FF_q(z)[X]$ denote $P_{\psi,f}^{\vee}(X)|_{\theta=z}$. We then define
\begin{equation} \label{E:bPfintro}
\bP_f(X) \assign \bigl( P_{\phi,f} \otimes P_{\psi,f(z)}^{\vee} \bigr)(X) \quad \in \AAA[X],
\end{equation}
where if $P(X) = (X-\alpha_1) \cdots (X-\alpha_r)$ and $Q(X) = (X-\beta_1) \cdots (X-\beta_{\ell})$, then $(P\otimes Q)(X) = \prod_{i,j}(X- \alpha_i\beta_j)$. The reason to study $\bP_f(X)$ is that if $\alpha_f \in \Gal(K^{\sep}/K)$ is a Frobenius element for~$f$, then (see Proposition~\ref{P:EEcharpolys})
\begin{equation} \label{E:EEcharpolysintro}
\bP_f(X) = \Char(\alpha_f,T_{\lambda}(\EE(\phi\times \psi)),X).
\end{equation}
This is a key ingredient for the following result (stated later as Theorem~\ref{T:EEAord}).

\begin{theoremintro} \label{T:EEAordintro}
Let $f \in \sA_+$ be irreducible. Then
\[
\bigAord{\soEE(\FF_f(z))}{\AAA} = (-1)^{r\ell} \chi_{\phi}(f)^{\ell}\, \ochi_{\psi}(f)^r \cdot \bP_f(1).
\]
\end{theoremintro}

Unfortunately the techniques of \cites{APT16, AT17, Demeslay22, Gezmis19} used to evaluate such $\AAA$-orders in the case where one twists by the Carlitz module (i.e., $\psi = \sC$) relied on the fact that~$\sC$ is rank~$1$. In order to account for $\psi$ having higher rank, we develop general results about Anderson $t$-modules in finite characteristic that we apply to $\EE$ and prove Theorem~\ref{T:EEAordintro}.

The main line of our argument is to adapt constructions of Poonen~\cite{Poonen96} of Galois equivariant and non-degenerate pairings on torsion submodules of Drinfeld modules and their adjoints over finite fields to higher dimensions and to more general fields. This takes up the bulk of \S\ref{S:Anderson}. What we produce is a bit more general than what we require, though we anticipate it will be useful for future work (e.g., see Corollary~\ref{C:charpoly1}).

For $n \geqslant 0$, let $Z = \{ z_1, \dots, z_n\}$ denote a set of variables (if $n=0$, then $Z= \emptyset$), and let $\FZ{\FF_q}{Z}$ denote either $\FF_q(Z)$ or $\laurent{\FF_q}{Z}$. We let $\cA = \FZ{\FF_q}{Z}[t]$. For fixed $f \in A_+$ irreducible of degree $d$, we have a structure map $\iota:\cA \to \FZ{\FF_f}{Z}$ that extends the map $\sA \to \FF_f$ sending $t$ to a root of~$f$. We consider an Anderson $t$-module $\cE : \cA \to \Mat_{\ell}(\FZ{\FF_f}{Z}[\tau])$ over $\FZ{\FF_f}{Z}$ together with its adjoint $\cE^* : \cA \to \Mat_{\ell}(\FZ{\FF_f}{Z}[\sigma])$ (see \S\ref{SS:cAtmodules}), and for $\nu \in \cA$, following Poonen we define an $\FZ{\FF_q}{Z}$-bilinear pairing (see Proposition~\ref{P:apairing})
\[
\apairing{\cdot}{\cdot}_{\nu} : \cE[\nu] \times \cE^{*}[\nu] \to \FZ{\FF_q}{Z}
\]
that is $\Gal(\oFF_f/\FF_f)$-equivariant. If the torsion modules $\cE[\nu]$ and $\cE^*[\nu]$ have maximal dimension over $\FZ{\FF_q}{Z}$, then the pairing is non-degenerate.

For $\lambda \in \sA_+$ irreducible with $\lambda(\theta) \neq f$, we obtain a pairing (see Theorem~\ref{T:Tatepairing})
\[
\bpairing{\cdot}{\cdot}_{\lambda} : T_{\lambda}(\cE) \times T_{\lambda}(\cE^*) \to \cA_{\lambda}
\]
that is $\cA_{\lambda}$-bilinear and $\Gal(\oFF_f/\FF_f)$-equivariant. Moreover, under certain conditions (see Definition~\ref{D:Tatemodules}), which are satisfied in the case $\cE =\EE(\phi \times \psi)$, the pairing $\bpairing{\cdot}{\cdot}_{\lambda}$ is non-degenerate. The main result in this part of the paper (Theorem~\ref{T:charpoly1}) is the following.

\begin{theoremintro} \label{T:charpoly1intro}
Let $\cE : \cA \to \Mat_{\ell}(\FZ{\FF_f}{Z}[\tau])$ be a $t$-module defined over $\FZ{\FF_f}{Z}$. Let $\lambda \in \sA_+$ be irreducible with $\lambda(\theta) \neq f$ such that Definitions~\ref{D:Tatemodules}(a)--(c) are satisfied. Then
\[
\bigl[ \cE(\FZ{\FF_f}{Z}) \bigr]_{\cA} = \gamma \cdot \Char(\tau^d,T_{\lambda}(\cE),1),
\]
where $\gamma \in \FZ{\FF_q}{Z}^{\times}$ uniquely forces the right-hand expression to be monic in~$t$.
\end{theoremintro}

When $f \neq \theta$, Theorem~\ref{T:EEAordintro} follows from Theorem~\ref{T:charpoly1intro} by using~\eqref{E:EEcharpolysintro}. It is here that we need to work over $\laurent{\oFF_f}{z}$ as well as $\oFF_f(z)$, since the entries of $\Upsilon_{\psi,z}$ reduce to elements of $\laurent{\oFF_f}{z}$ modulo primes different from $\theta$. For $f=\theta$ we employ a separate direct argument.

\subsection{Convolution \texorpdfstring{$L$}{L}-functions and special values}
For the Drinfeld module $\phi$ from \eqref{E:phipsidefintro}, Goss defined two $L$-functions, for $s \in \ZZ$,
\[
L(\phi^{\vee},s) = \prod_f Q_{\phi,f}^{\vee} \bigl( f^{-s} \bigr)^{-1}, \quad 
L(\phi,s) = \prod_f Q_{\phi,f}\bigl( f^{-s} \bigr)^{-1},
\]
where $Q_{\phi,f}(X)$ and $Q_{\phi,f}^{\vee}(X)$ are the reciprocal polynomials of $P_{\phi,f}(X)$ and $P_{\phi,f}^{\vee}(X)$ respectively. These give rise to multiplicative functions $\mu_{\phi}$, $\nu_{\phi}: A_+ \to A$ such that
\[
\sum_{m=1}^{\infty} \mu_{\phi}(f^m) X^m = Q_f^{\vee}(fX)^{-1}, \quad
\sum_{m=1}^{\infty} \nu_{\phi}(f^m) X^m = Q_f(X)^{-1}.
\]
We similarly define $\mu_{\psi}$ and $\nu_{\psi}$. See \S\ref{SS:munu} for details. In a similar fashion we define
\[
L(\EE^{\vee},s) = L(\EE(\phi \times \psi)^{\vee},s) \assign \prod_f \bQ_f^{\vee} \bigl( f^{-s} \bigr)^{-1}, \quad s \geqslant 0,
\]
where $\bQ_f^{\vee}(X)$ is the reciprocal polynomial of $\bP_{f}^{\vee}(X) = (P_{\phi,f}^{\vee} \otimes P_{\psi,f(z)})(X) \in K[z][X]$. This product for $L(\EE^{\vee},s)$ converges in $\TT_z(K_{\infty})$ for integers $s \geqslant 0$. Moreover, it follows from Theorem~\ref{T:EEAordintro} that (see Proposition~\ref{P:LEE0})
\begin{equation} \label{E:LEE0intro}
L(\EE^{\vee},0) = \prod_f \frac{\bigAord{\FF_f(z)^{\ell}}{\AAA}}{\bigAord{\soEE(\FF_f(z))}{\AAA}},
\end{equation}
and thus Demeslay's identity~\eqref{E:DemeslayEEintro} implies that $L(\EE^{\vee},0) = \Reg_{\EE} \cdot \bigAord{\rH(\EE)}{\AAA}$.

On the other hand, we can express $L(\EE^{\vee},s)$ in an alternative form as the convolution of $L$-series for $\phi$ and $\psi$, following the situation for Maass forms on $\GL_n$ (see \cites{Bump89, Goldfeld}). For fixed $f \in A_+$ irreducible, using Cauchy's identity (see Theorem~\ref{T:Cauchy}) we can write $\bQ_f^{\vee}(X)^{-1}$ in terms of Schur polynomials evaluated at the roots of $P_{\phi,f}^{\vee}(X)$ and $P_{\psi,f(z)}(X)$. In particular if $\alpha_1, \dots, \alpha_r \in \oK$ are the roots of $P_{\phi,f}^{\vee}(X)$, then for $k_1, \dots, k_{r-1} \geqslant 0$, we define
\begin{align}
\bsmu_{\phi,\theta}\bigl( f^{k_1}, \dots, f^{k_{r-1}} \bigr) &\assign S_{k_1, \dots, k_{r-1}} (\alpha_1, \dots, \alpha_r) \cdot f^{k_1 + \cdots + k_{r-1}} \\
\bsnu_{\phi,\theta} \bigl( f^{k_1}, \dots, f^{k_{r-1}} \bigr) &\assign S_{k_1, \dots, k_{r-1}} \bigl( \alpha_1^{-1}, \dots, \alpha_{r}^{-1} \bigr),
\end{align}
where $S_{k_1, \dots, k_{r-1}}(x_1, \dots, x_r)$ is the Schur polynomial defined in~\eqref{E:Sk}. We can extend $\bsmu_{\phi,\theta}$ and $\bsnu_{\phi,\theta}$ to functions on $(A_+)^{r-1}$ multiplicatively, and then we find that for $a_1, \dots, a_{r-1} \in A_+$, we have $\bsmu_{\phi,\theta}(a_1, \dots, a_{r-1}) \in A$ and $\bsnu_{\phi,\theta}(a_1, \dots, a_{r-1}) \in A$, and
\[
\bsmu_{\phi,\theta}(a_1, \dots, a_{r-1}) = \chi_{\phi}(a_1 \cdots a_{r-1}) \cdot \bsnu_{\phi,\theta}(a_{r-1}, \dots, a_1).
\]
Furthermore, for any $a \in A_+$,
\[
\bsmu_{\phi,\theta}(a,1, \dots, 1) = \mu_{\phi}(a), \quad \bsnu_{\phi,\theta}(a,1,\dots, 1) = \nu_{\phi}(a).
\]
The functions $\bsmu_{\phi,\theta}$, $\bsnu_{\phi,\theta} : (A_+)^{r-1} \to A$ satisfy various relations induced by relations on Schur polynomials. See \S\ref{SS:bsmubsnu} for details.

Returning to the situation of Drinfeld modules $\phi$ and $\psi$, we set $\bsnu_{\psi,z}(a_1, \dots, a_r) = \bsnu_{\psi,\theta}(a_1, \dots, a_r)|_{\theta=z} \in \FF_q[z]$. When $r$, $\ell \geqslant 2$, we define an $L$-function $L(\bsmu_{\phi,\theta} \times \bsnu_{\psi,z},s)$ as follows. If $r = \ell$, then
\begin{equation}
L(\bsmu_{\phi,\theta} \times \bsnu_{\psi,z},s) \assign \sum_{a_1, \dots, a_{r-1} \in A_+} \frac{\bsmu_{\phi,\theta}(a_1, \dots, a_{r-1}) \bsnu_{\psi,z}(a_1, \dots, a_{r-1})}{a_1 \cdots a_{r-1} (a_1 a_2^2 \cdots a_{r-1}^{r-1})^s}.
\end{equation}
If $r< \ell$, then
\begin{equation}
L(\bsmu_{\phi,\theta} \times \bsnu_{\psi,z},s) \assign
\sum_{a_1, \ldots, a_r \in A_+}
\frac{\chi_{\phi}(a_r)\bsmu_{\phi,\theta}(a_1, \dots, a_{r-1}) \bsnu_{\psi,z}(a_1, \dots, a_r,1, \ldots, 1)}{a_1 \cdots a_r (a_1 a_2^2 \cdots a_r^r)^s},
\end{equation}
and if $r > \ell$, then
\begin{multline}
L(\bsmu_{\phi,\theta} \times \bsnu_{\psi,z},s)  \assign \\
\sum_{a_1, \ldots, a_{\ell} \in A_+}
\frac{\ochi_{\psi}(a_{\ell}) a_{\ell}(z) \bsmu_{\phi,\theta}(a_1, \dots, a_{\ell},1, \dots, 1) \bsnu_{\psi,z}(a_1, \dots, a_{\ell-1})}{a_1 \cdots a_{\ell} (a_1 a_2^2 \cdots a_{\ell}^{\ell})^s}.
\end{multline}
The different versions are governed by the different applications of Cauchy's identity (Theorem~\ref{T:Cauchy}, Corollary~\ref{C:Cauchynl}) that are needed. See \S\ref{SS:Lconvrxr}--\ref{SS:Lconvrxl} for more details.

As $\bsmu_{\phi,\theta}(a_1, \dots, a_{r-1}) \in A$ and $\bsnu_{\psi,z}(a_1, \dots, a_{r-1}) \in \FF_q[z]$, we interpret $L(\bsmu_{\phi,\theta} \times \bsnu_{\psi,z},s)$ as a convolution of Goss and Pellarin $L$-series. Now $L(\bsmu_{\phi,\theta} \times \bsnu_{\psi,z},s)$ is related to $L(\EE(\phi\times \psi)^{\vee},s)$ by the following result (stated later as Theorems~\ref{T:LEErxr} and~\ref{T:LEErxl}), where we let $L(\AAA,\chi_{\phi}\ochi_{\psi},s) = \sum_{a \in A_+} \chi_{\phi}(a)\ochi_{\psi}(a)a(z) \cdot a^{-s}$ be a twist of $L(\AAA,s)$.

\begin{theoremintro} \label{T:LEEintro}
Let $\phi$, $\psi : \sA \to A[\tau]$ be Drinfeld modules of ranks $r$ and $\ell$ respectively with everywhere good reduction as defined in~\eqref{E:phipsidefintro}. Assume that $r$, $\ell \geqslant 2$, and let $s \geqslant 0$.
\begin{alphenumerate}
\item If $r= \ell$, then
\[
L( \EE(\phi \times \psi)^{\vee},s) = L(\AAA,\chi_{\phi}\ochi_{\psi},rs+1) \cdot L(\bsmu_{\phi,\theta} \times \bsnu_{\psi,z},s).
\]
\item If $r \neq \ell$, then
\[
L(\EE(\phi\times \psi)^{\vee},s) = L(\bsmu_{\phi,\theta} \times \bsnu_{\psi,z},s).
\]
\end{alphenumerate}
In both cases, $L(\EE(\phi\times \psi)^{\vee},s) \in \TT_z(K_{\infty})^{\times}$.
\end{theoremintro}

Taking $s=0$ in Theorem~\ref{T:LEEintro} provides special value identities for $L(\bsmu_{\phi,\theta} \times \bsnu_{\psi,z},0)$.
If $r=\ell$ and $\kappa_r = \eta_r$, then $L(\AAA,\chi_{\phi}\ochi_{\psi},s) = L(\AAA,s)$, and~\eqref{E:Pellarinvalueintro} and~\eqref{E:DemeslayEEintro} imply the following corollary (stated as Corollary~\ref{C:Lmunurxr}). Under the condition that $\dnorm{\Upsilon_{\psi,z}}$ is less than the radius of convergence of $\Log_{\phi}(z)$, part (b) emerges because we can evaluate $\Reg_{\EE}$ exactly and $\rH(\EE)$ is trivial. For the case where $\kappa_r \neq \eta_r$, see Corollary~\ref{C:Lmunurxr}(c).

\begin{corollaryintro} \label{C:Lmunurxrintro}
Let $\phi$, $\psi : \sA \to A[\tau]$ be Drinfeld modules both of rank $r \geqslant 2$ with everywhere good reduction, as defined in~\eqref{E:phipsidefintro}, and assume $\kappa_r = \eta_r$.
\begin{alphenumerate}
\item Then
\begin{align*}
L(\bsmu_{\phi,\theta} \times \bsnu_{\psi,z},0)
&= \sum_{a_1 \in A_+} \cdots \sum_{a_{r-1} \in A_+} \frac{\bsmu_{\phi,\theta}(a_1, \dots, a_{r-1}) \bsnu_{\psi,z}(a_1, \dots, a_{r-1})}{a_1 \cdots a_{r-1}} \\[10pt]
&= (\theta-z) \cdot \frac{\omega_z}{\tpi} \cdot \Reg_{\EE} \cdot \Aord{\rH(\EE)}{\AAA}.
\end{align*}
\item If $\dnorm{\Upsilon_{\psi,z}} < R_{\phi}$, where $R_{\phi}$ is the radius of convergence of $\Log_{\phi}(z)$ in~\eqref{E:Rphi}, then
\[
L(\bsmu_{\phi,\theta} \times \bsnu_{\psi,z},0)
=(\theta - z) \cdot \frac{\omega_z}{\tpi} \cdot \det \Bigl( \Upsilon_{\psi,z}^{-1} \Log_{\phi} \bigl(\Upsilon_{\psi,z} \bigr) \Bigr).
\]
\end{alphenumerate}
\end{corollaryintro}

In the case that $r \neq \ell$, we obtain the following (stated later as Corollay~\ref{C:Lmunurxl}).

\begin{corollaryintro} \label{C:Lmunurxlintro}
Let $\phi$, $\psi : \sA \to A[\tau]$ be Drinfeld modules of ranks $r$ and $\ell$ respectively with everywhere good reduction, as defined in~\eqref{E:phipsidefintro}. Assume that $r$, $\ell \geqslant 2$ and that $r \neq \ell$.
\begin{alphenumerate}
\item If $r < \ell$, then
\begin{align*}
L(\bsmu_{\phi,\theta} \times \bsnu_{\psi,z},0)
&= \sum_{a_1, \ldots, a_{r} \in A_+} \frac{\chi_{\phi}(a_r) \bsmu_{\phi,\theta}(a_1, \dots, a_{r-1}) \bsnu_{\psi,z}(a_1, \dots, a_{r},1, \dots, 1)}{a_1 \cdots a_r} \\[10pt]
&= \Reg_{\EE} \cdot \Aord{\rH(\EE)}{\AAA}.
\end{align*}
\item If $r > \ell$, then
\begin{align*}
L(\bsmu_{\phi,\theta} \times \bsnu_{\psi,z},0)
&= \sum_{a_1, \ldots, a_{\ell} \in A_+} \frac{\ochi_{\psi}(a_\ell) a_{\ell}(z) \bsmu_{\phi,\theta}(a_1, \dots, a_{\ell},1, \ldots, 1) \bsnu_{\psi,z}(a_1, \dots, a_{\ell-1})}{a_1 \cdots a_{\ell}} \\[10pt]
&= \Reg_{\EE} \cdot \Aord{\rH(\EE)}{\AAA}.
\end{align*}
\item If $\dnorm{\Upsilon_{\psi,z}} < R_{\phi}$, where $R_{\phi}$ is the radius of convergence of $\Log_{\phi}(z)$ in~\eqref{E:Rphi}, then
\[
L(\bsmu_{\phi,\theta} \times \bsnu_{\psi,z},0)
= \det \Bigl( \Upsilon_{\psi,z}^{-1} \Log_{\phi} \bigl(\Upsilon_{\psi,z} \bigr) \Bigr).
\]
\end{alphenumerate}
\end{corollaryintro}

Similar results hold when $r=1$ or $\ell =1$. The cases where $\phi$ or $\psi$ are the Carlitz module are worked out in Examples~\ref{Ex:Carlitz3}, \ref{Ex:LphixC}, and~\ref{Ex:LCxphi}, and we observe that Corollaries~\ref{C:Lmunurxrintro} and~\ref{C:Lmunurxlintro} degenerate to these cases. Example~\ref{Ex:Carlitz3} is Pellarin's original case in~\cite{Pellarin12}, and Example~\ref{Ex:LphixC} was worked out by Angl\`es, Gezmi\c{s}, and Tavares Ribeiro~\cites{AT17, Gezmis19}.

\begin{remark} \label{R:badred}
Throughout we have restricted our attention to cases of everywhere good reduction. This was partly to provide a more simplified and unified treatment of what was to the authors quite complicated to parcel out among various constituent identities. However, \eqref{E:Thetadefintro} and~\eqref{E:EEtintro} indicate that there are delicate issues to resolve in determining appropriate Euler factors at bad primes (especially primes dividing~$\eta_\ell$).
\end{remark}

\begin{remark}
In many works on Pellarin $L$-series (e.g., \cites{ANT17a, ANT17b, AnglesPellarin14, APT16, AT17, Demeslay22, Gezmis19, PellarinPerkins16}), one considers $L$-functions in several $z$-variables, say $z_1, \dots z_n$. The present techniques should apply, but it would entail several tensor products of polynomials as in~\eqref{E:bPfintro} and applications of Cauchy's identity (Theorem~\ref{T:Cauchy}). In the same vein, one could attempt to extend these results to convolutions of $L$-functions for Drinfeld modules defined over a finite extension of $K$, such as in \cites{ANT17a, ANT17b, AT17, ANT20, ANT22}. However, these potential extensions fell beyond the scope of the present paper.

One might wonder if the present framework can be adapted to convolutions of Goss $L$-functions with values in~$\C$, say for tensor products of Drinfeld modules defined over~$A$, and whether their special values can be related to class module formulas for $t$-modules from~\cites{Fang15,ANT20,ANT22}. This is the subject of additional work of the first author~\cite{Huang26}.
\end{remark}

\begin{remark} \label{R:convergence}
As pointed out by one of the referees, there are a number of interesting questions regarding convergence of the $L$-functions we consider. We have exclusively defined our $L$-functions with $s \in \ZZ$, but going back to Goss's original $L$-series constructions for Drinfeld modules \cites{Goss79, Goss83, Goss92, Goss, Goss13}, one can expand the domain of $s$ for his $L$-functions to a non-archimedean symmetric space~$\SSS_{\infty} = \CC_{\infty}^{\times} \times \ZZ_p$, in which $\ZZ$ embeds as $(\theta^j,j)$ for $j \in \ZZ$ (see~\cite{Goss}*{\S 8.1}). Many Goss $L$-series are then seen to be entire on $\SSS_{\infty}$ (e.g., see~\citelist{\cite{Goss79} \cite{Goss}*{Ch.~8} \cite{Goss92} \cite{FGHP22}*{\S 1.3}}).

Furthermore, in the context of Pellarin $L$-series taking values in Tate algebras, Pellarin proved in~\cite{Pellarin12}*{Rem.~7} that, for fixed~$s \in \SSS_{\infty}$, $L(\AAA,s) \in \TT^{\times}$ is a rigid analytic entire function on~$\CC_{\infty}$. This has been extended to various other $\TT$-valued $L$-series (e.g., see~\citelist{\cite{AnglesPellarin14}*{Prop.~6} \cite{APT16}*{\S 4.2} \cite{Gezmis19}*{Thm.~1.4} \cite{Goss13}*{Thm.~1}}). It would be interesting to determine these finer convergence properties for $L(\EE(\phi \times \psi)^{\vee},s)$.

Another point is that the present convolution $L$-series run parallel to multiple Dirichlet series over $\CC$ (e.g., see~\cites{DGH03, FBGH06}), which also extend to multivariable entire or meromorphic functions on $\CC^n$. Analytic information about a multiple Dirichlet series can translate into information on their coefficients, for example leading to results on moments of $L$-series. It is a compelling question to investigate how such identities may hold for the $L$-series here.
\end{remark}

As suggested by one the referees, to form explicit descriptions of $L(\EE(\phi \times \psi)^{\vee},0)$ in terms of logarithms of special points, we investigate the module $\rU_{\St}(\EE/\AAA)$ of Stark units for $\EE$, which is an $\bA$-submodule of the unit module $\rU(\EE/\AAA) \subseteq \Lie(\EE)(\TT_{z}(K_{\infty}))$ (see \S\ref{SS:Stark}). We prove the following result (stated later as Theorem~\ref{T:StarkLvalue}) which is a version of a theorem of Angl\`es and Tavares Ribeiro~\cite{AT17}*{Thm.~1}. Their result was extended to Anderson $t$-modules over finite extensions of $K$ in joint work with Ngo Dac and Pellarin~\cites{ANT20, ANT22, APT18}, whereas Stark units for $\omega_z$ twists of the Carlitz module were studied in \cites{APT16, AT17}.

\begin{theoremintro}[{cf.\ Angl\`es, Tavares Ribeiro~\cite{AT17}*{Thm.~1}}] \label{T:StarkLvalueintro}
For $\EE = \EE(\phi \times \psi) : \bA \to \Mat_{\ell}(A[z][\tau])$, the following hold.
\begin{alphenumerate}
\item $\rU(\EE/\AAA)/\rU_{\St}(\EE/\AAA)$ is a finitely generated torsion $\bA$-module, and
\[
\biggl[ \frac{\rU(\EE/\AAA)}{\rU_{\St}(\EE/\AAA)} \biggr]_{\AAA}
= \bigl[ \rH(\EE) \bigr]_{\AAA}.
\]
\item Moreover,
\[
L(\EE(\phi \times \psi)^{\vee},0) = \bigl[ \Lie(\EE)(\AAA) : \rU_{\St}(\EE/\AAA) \bigr]_{\AAA} = \cL(\DD/\tAA)|_{\bszeta=1},
\]
where the middle term is the covolume of $\rU_{\St}(\EE/\AAA)$.
\end{alphenumerate}
\end{theoremintro}

In this result, $\cL(\DD/\tAA)$ is an $L$-value associated to a deformation $\DD$ of $\EE$ by an additional variable $\bszeta$ over the ring $\tAA = \FF_q(z,\bszeta)[\theta]$, similar to deformations studied in~\cite{AT17}. It is defined as a product of local factors of $\tAA$-orders as in~\eqref{E:DemeslayEEintro} and~\eqref{E:LEE0intro}. See \S\ref{SS:deformations} for details. We prove in Proposition~\ref{P:deformationorder} and Corollary~\ref{C:deformationconv} that $\cL(\DD/\tAA)$ has an explicit description in terms of convolution $L$-values in the Tate algebra $\TT_{z,\bszeta}(K_{\infty})^{\times}$. For example, when $r=\ell$ and $\kappa_r=\eta_r$,
\[
\cL(\DD/\tAA) = \cL(\tAA) \cdot \cL(\bsmu_{\phi,\theta} \times \bsnu_{\psi,z}),
\]
where
\[
\cL(\tAA) \assign \sum_{a \in A_+} \frac{a(z)}{a} \bszeta^{r \deg a}
\]
and
\[
\cL(\bsmu_{\phi,\theta} \times \bsnu_{\psi,z}) \assign \sum_{a_1, \ldots, a_{r-1} \in A_+} \frac{\bsmu_{\phi,\theta}(a_1, \dots, a_{r-1}) \bsnu_{\psi,z}(a_1, \dots, a_{r-1})}{a_1 \cdots a_{r-1}} \bszeta^{\delta(a_1, \dots, a_{r-1})},
\]
with $\delta(a_1, \dots, a_{r-1}) = \deg a_1 + 2 \deg a_2 + \dots + (r-1) \deg a_{r-1}$. Comparing with Theorem~\ref{T:LEEintro} and Corollary~\ref{C:Lmunurxrintro}, we readily find that $L(\EE^{\vee},0) = \cL(\DD/\tAA)|_{\bszeta=1}$. However, the identity in Theorem~\ref{T:StarkLvalueintro}(b) implies that $L(\EE^{\vee},0)$ can be expressed as a determinant of Stark units, which are logarithms of special points in $\EE(A[z])$ and by Theorem~\ref{T:StarkLvalueintro}(a) analogues of circular units. In this sense, $L(\EE^{\vee},0)$ is expressible in terms of log-algebraic identities, e.g., as in \cites{And94, And96, ANT17b, ANT20, AnglesTaelman15, AT17, CEP18, GreenNgoDac23, GreenP18}. See \S\ref{SS:logalg} for a brief discussion on log-algebraicity.

\subsubsection*{Outline}
After summarizing preliminary material in~\S\ref{S:prelim}, we consider Anderson $t$-modules over Tate algebras and over ``$\tau$-perfect'' fields of finite characteristic in~\S\ref{S:Anderson}. Especially in~\S\ref{S:Anderson} we extend Poonen pairings to $t$-modules over more general fields of finite characteristic so as to find identities for $\AAA$-orders of $t$-modules in terms of characteristic polynomials of Frobenius. In~\S\ref{S:RAtwists} we introduce the rigid analytic twists $\EE(\phi\times \psi)$ associated to two Drinfeld modules over~$A$, as well as determine identities for their $\AAA$-orders over fields of finite characteristic. In~\S\ref{S:Lseries} we review the theories of Goss and Pellarin $L$-series, as well as Demeslay's class module formula. We introduce the $L$-function of $E(\phi \times \psi)$ and demonstrate how we can apply Demeslay's formula to it. As a stepping stone to the next section we consider the cases where one of $\phi$ or $\psi$ is the Carlitz module. In~\S\ref{S:convolutions} we introduce the convolution $L$-series $L(\bsmu_{\phi,\theta} \times \bsnu_{\psi,z},s)$, relate it to $L(\EE(\phi \times \psi)^{\vee},s)$ and Pellarin's $L$-series, and investigate special value identities. We analyze a few examples, including one where $\dnorm{\Upsilon_{\psi,z}} \geqslant R_{\phi}$, at the end of~\S\ref{S:convolutions}. Finally, in \S\ref{S:Stark} we work out the theory of Stark units for $\EE$ and investigate their connections with special $L$-values.

\subsubsection*{Acknowledgments}
The authors thank Y.-T.~Chen, O.~Gezmi\c{s}, and J.~Ye for a number of helpful discussions during the preparation of this manuscript. The authors especially thank M.~P.~Young for explaining the phenomenon of Schur polynomials in the case of Hecke eigenvalues of Maass forms. The authors further thank the referees for a number of helpful comments, especially for the suggestion to develop the results in \S\ref{S:Stark} on Stark units and deformation $L$-values.

\subsubsection*{AI Declaration}
The authors did not use AI tools in carrying out the research in this work, nor in the preparation of this manuscript.

\section{Preliminaries} \label{S:prelim}

\subsection{Notation} \label{SS:notation}
We will use the following notation throughout.
\begin{longtable}{p{1.25truein}@{\hspace{5pt}$=$\hspace{5pt}}p{4.5truein}}
$A$ & $\FF_q[\theta]$, polynomial ring in variable $\theta$ over $\FF_q$. \\
$A_+$ & the monic elements of $A$. \\
$K$ & $\FF_q(\theta)$, the fraction field of $A$. \\
$K_{\infty}$ & $\laurent{\FF_q}{\theta^{-1}}$, the completion of $K$ at $\infty$. \\
$\C$ & the completion of an algebraic closure $\oK_{\infty}$ of $K_{\infty}$.\\
$\inorm{\,\cdot\,}$; $\deg$ & $\infty$-adic norm on $\C$, extended to the sup norm on a finite dimensional $\C$-vector space; $\deg = -\ord_{\infty} = \log_q\inorm{\,\cdot\,}$.\\
$\FF_f$ & $A/fA$ for $f \in A_+$ irreducible. \\
$\sA$ & $\FF_q[t]$, for a variable $t$ independent from $\theta$. \\
$\TT_t$; $\LL_t$ & Tate algebra in $t$ $= \{ \sum a_i t^i \in \power{\C}{t} \mid \inorm{a_i} \to 0 \} = $ completion of $\C[t]$ with respect to Gauss norm; and $\LL_t$ its fraction field. \\
$\LLhat_t$ & completion of $\LL_t$ with respect to Gauss norm. \\
$\TT_t(K_{\infty})$; $\LL_t(K_{\infty})$ & $\TT_t \cap \power{K_{\infty}}{t} = \laurent{\FF_q[t]}{\theta^{-1}}$; and $\LL_t(K_{\infty})$ its fraction field. \\
$\LLhat_t(K_{\infty})$ & completion of $\LL_t(K_{\infty})$, or equivalently $\laurent{\FF_q(t)}{\theta^{-1}}$. \\
$\AAA$ & $\FF_q(z)[\theta]$, for a third variable $z$ independent from $\theta$, $t$. \\
$\KK$ & $\FF_q(z,\theta)$, the fraction field of $\AAA$. \\
$\KK_{\infty}$ & $\laurent{\FF_q(z)}{\theta^{-1}}$, the completion of $\KK$ at $\infty$, or equivalently $\LLhat_z(K_{\infty})$. \\
\raggedright $\TT_z$, $\TT_z(K_{\infty})$, \linebreak \raggedright $\LLhat_z$, $\LLhat_z(K_{\infty})$ & same as above, with $t$ replaced by $z$. \\
$\bA$ & $\FF_q(z)[t]$.\\
$\FZ{F}{Z}$ & $F(Z)$ or $\laurent{F}{Z}$, where $Z = \{ z_1, \dots, z_n \}$, $n \geqslant 0$.\\
$\cA$ & $\FZ{\FF_q}{Z}[t]$, for $Z = \{ z_1, \dots, z_n \}$, $n \geqslant 0$.\\
$\Mat_{k \times \ell}(R)$ & for a ring $R$, the left $R$-module of $k \times \ell$ matrices over $R$. \\
$\Mat_{k}(R)$; $R^k$ & $\Mat_{k \times k}(R)$; $\Mat_{k\times 1}(R)$.\\
$B^{\tr}$ & the transpose of a matrix $B$. \\
$\Char(B,X)$ & the monic characteristic polynomial in $X$ of a square matrix $B$.\\
$\Char(\alpha,V,X)$ & the monic characteristic polynomial of a linear map $\alpha : V \to V$.
\end{longtable}

\subsubsection{Rings of scalars, functions, and operators} \label{SSS:fields}
For a variable $t$ independent from $\theta$ we let $\sA \assign \FF_q[t]$. We let $\TT_t$ denote the standard \emph{Tate algebra}, $\TT_t \subseteq \power{\C}{t}$,
consisting of power series that converge on the closed unit disk of~$\C$, and we take
\begin{equation} \label{E:TateKinfty}
\TT_t(K_\infty) \assign \TT_t \cap \power{K_{\infty}}{t} = \laurent{\FF_q[t]}{\theta^{-1}}.
\end{equation}
We let $\dnorm{\,\cdot\,}$ denote the Gauss norm on $\TT_t$, such that $
\lVert \sum_{i=0}^{\infty} a_i t^i \rVert = \max_i \{ \inorm{a_i} \}$,
under which $\TT_t$ is a complete normed $\C$-vector space, and likewise $\TT_t(K_\infty)$ is a complete normed $K_{\infty}$-vector space. We extend the degree map on $\C$ to $\TT_t$ by taking
$\deg = \log_q \dnorm{\,\cdot\,}$. We let $\LL_t$ (respectively $\LL_t(K_\infty)$) denote the fraction field of $\TT_t$ (respectively $\TT_t(K_{\infty})$), to which we extend the Gauss norm, and we take $\LLhat_t$ (respectively $\LLhat_t(K_{\infty})$) for its completion.
We note the following properties. With respect to $\dnorm{\,\cdot\,}$,
\begin{itemize}
\item $\TT_t$ is the completion of $\C[t]$ and $\TT_t(K_{\infty})$ is the completion of $K_{\infty}[t]$,
\item $\LLhat_t$ is the completion of $\C(t)$ and $\LLhat_t(K_{\infty})$ is the completion of $K_{\infty}(t)$,
\item $\LLhat_t(K_{\infty}) = \laurent{\FF_q(t)}{\theta^{-1}}$.
\end{itemize}
We make any finite dimensional $\LLhat_t$-vector space a complete normed vector space by taking the sup norm. For more information on Tate algebras the reader is directed to~\cite{FresnelvdPut}. 

Having defined several rings so far in terms of the variable $t$, we make also identical copies, $\TT_z$, $\TT_z(K_{\infty})$, etc., with a new variable $z$. We fix also
\[
\AAA \assign \FF_q(z)[\theta], \quad \KK \assign \FF_q(z,\theta), \quad \KK_{\infty} \assign \laurent{\FF_q(z)}{\theta^{-1}} = \LLhat_z(K_{\infty}).
\]
Through abuse of notation, the Gauss norm and degree on these new rings are also denoted $\dnorm{\,\cdot\,}$ and $\deg$. When not clear from the context, we will use $\dnorm{\,\cdot\,}_t$, $\dnorm{\,\cdot\,}_z$, etc.\ to indicate the underlying variable.

We finally let $\bA \assign \FF_q(z)[t]$. The reader has undoubtedly noticed that we now have four similar polynomial rings to keep track of, $A \cong \sA$ and $\AAA \cong \bA$, but each will have its own use in due course. A fifth ring ``$\cA$'' will also appear in \S\ref{S:Anderson}.

\begin{remark}
From the point of view of the present paper, the rings in terms of $t$ are rings of functions and operators, whereas the rings in terms of $z$ are rings of scalars.  As such the rings ($A$, $\AAA$) are scalars and ($\sA$, $\bA$) are operators, in the sense that we view them as operating on a given Drinfeld module or Anderson $t$-module. We note that the roles of $t$ and $z$ here would be $\theta$ and $t_1, \dots, t_s$ in \cites{APT16, DemeslayPhD, Demeslay22}, and $z$ and $t_1, \dots, t_s$ in~\cite{GezmisP19}.
\end{remark}

\subsubsection{Frobenius operators} \label{SSS:Frobops}
We take $\tau : \C \to \C$ for the \emph{$q$-th power Frobenius automorphism}, which we extend to $\laurent{\C}{t}$ and $\laurent{\C}{z}$ (and their compositum) by requiring it to commute with multiplication by $t$ and $z$. For $g = \sum c_i t^i \in \laurent{\C}{t}$, we define the \emph{$n$-th Frobenius twist},
\[
g^{(n)} \assign \tau^n(g) = \sum c_i^{q^n} t^i, \quad \forall\, n \in \ZZ,
\]
and likewise for $g \in \laurent{\C}{z}$. Then $\tau$ induces $\FF_q(t)$-linear automorphisms of $\TT_t$, $\LL_t$, $\LLhat_t$, etc., and $\FF_q(z)$-linear automorphisms of their $z$-counterparts, and the fixed rings of $\tau$ are
\[
 \TT_t^{\tau} = \FF_q[t], \quad \LL_t^{\tau} = \FF_q(t), \quad \LLhat_t^{\tau} = \FF_q(t),
\]
(e.g., see \citelist{\cite{Demeslay22}*{Lem.~2.2} \cite{P08}*{Lem.~3.3.2}}). Furthermore, $\tau$ restricts to (non-invertible) endomorphisms of $\TT_t(K_{\infty})$, $\LL_t(K_{\infty})$, $\LLhat_t(K_{\infty})$, etc.

\subsubsection{Twisted polynomials} \label{SSS:twistedpolys}
Let $R$ be any commutative $\FF_q$-algebra, and let $\tau : R \to R$ be an injective $\FF_q$-algebra endomorphism. Let $R^{\tau}$ be the $\FF_q$-subalgebra of $R$ of elements fixed by $\tau$. For $n \in \ZZ$ for which $\tau^n$ is defined on $R$ and a matrix $B=(b_{ij})$ with entries in $R$, we let $B^{(n)}$ be defined by twisting each entry. That is, $(b_{ij})^{(n)} = (b_{ij}^{(n)})$. For $\ell \geqslant 1$ we let $\Mat_\ell(R)[\tau] = \Mat_\ell(R[\tau])$ be the ring of \emph{twisted polynomials} in $\tau$ with coefficients in $\Mat_\ell(R)$, subject to the relation $\tau B = B^{(1)} \tau$ for $B \in \Mat_\ell(R)$.
In this way, $R^{\ell}$ is a left $\Mat_{\ell}(R)[\tau]$-module, where if $\beta = B_0 + B_1 \tau + \dots + B_m \tau^m \in \Mat_{\ell}(R)[\tau]$ and $\bx \in R^{\ell}$, then
\begin{equation} \label{E:taueval}
\beta(\bx) = B_0 \bx + B_1 \bx^{(1)} + \cdots + B_m \bx^{(m)}.
\end{equation}
More generally \eqref{E:taueval} extends to $\beta \in \Mat_{k \times \ell}(R[\tau])$ and $\bx \in \Mat_{\ell \times n}(R)$, thus defining
\begin{equation} \label{E:gentaueval}
(\beta, \bx) \mapsto \beta(\bx) : \Mat_{k \times \ell}(R[\tau]) \times \Mat_{\ell \times n}(R) \to \Mat_{k \times n}(R),
\end{equation}
which is $R^{\tau}$-bilinear and $R$-linear in the first entry. Furthermore,
\begin{equation} \label{E:tauevalassoc}
(\beta_1 \beta_2)(\bx) = \beta_1(\beta_2(\bx)),
\end{equation}
for all $\beta_1$, $\beta_2$ matrices over $R[\tau]$ and $\bx$ over $R$ of the appropriate sizes.

If furthermore $\tau$ is an automorphism of $R$, then we set $\sigma \assign \tau^{-1}$ and form the twisted polynomial ring $\Mat_{\ell}(R)[\sigma]$, subject to $\sigma B = B^{(-1)} \sigma$ for $B \in \Mat_{\ell}(R)$.
Then $R^{\ell}$ is a left $\Mat_{\ell}(R)[\sigma]$-module, where for $\gamma = C_0 + C_1\sigma + \dots + C_m\sigma^m \in \Mat_{\ell}(R)[\sigma]$ and $\bx \in R^{\ell}$,
\begin{equation} \label{E:sigmaeval}
\gamma(\bx) = C_0 \bx + C_1 \bx^{(-1)} + \cdots + C_m \bx^{(-m)}.
\end{equation}
As above this evaluation operation extends to
\begin{equation} \label{E:gensigmaeval}
(\gamma, \bx) \mapsto \gamma(\bx) : \Mat_{k \times \ell}(R[\sigma]) \times \Mat_{\ell \times n}(R) \to \Mat_{k \times n}(R),
\end{equation}
which is $R^{\tau}$-bilinear and $R$-linear in the first entry. It is also associative as in
\eqref{E:tauevalassoc}.

For $\beta \in \Mat_{\ell}(R)[\tau]$ (or $\gamma \in \Mat_{\ell}(R)[\sigma]$), we write $\pd \beta$ (or $\pd \gamma$) for the constant term with respect to $\tau$ (or $\sigma$). We have natural inclusions of $\FF_q$-algebras,
\[
\Mat_{\ell}(R)[\tau] \subseteq \power{\Mat_{\ell}(R)}{\tau}, \quad \Mat_{\ell}(R)[\sigma] \subseteq \power{\Mat_{\ell}(R)}{\sigma},
\]
into \emph{twisted power series} rings, the latter when $\tau$ is an automorphism.

\subsubsection{Ore anti-involution} \label{SSS:star}
We assume that $\tau : R \to R$ is an automorphism, and recall the anti-isomorphism $* : R[\tau] \to R[\sigma]$ of $\FF_q$-algebras originally defined by Ore~\cite{Ore33a} (see also~\citelist{\cite{Goss}*{\S 1.7} \cite{NamoijamP24}*{\S 2.3} \cite{Poonen96}}),  given by
\[
\biggl( \sum_{i=0}^\ell b_i\tau^i \biggr)^* = \sum_{i=0}^{\ell} b_i^{(-i)} \sigma^i.
\]
One verifies that $(\alpha\beta)^* = \beta^*\alpha^*$ for $\alpha$, $\beta \in R[\tau]$. For $B = (\beta_{ij}) \in \Mat_{k \times \ell}(R[\tau])$, we set
\[
B^* \assign \bigl( \beta_{ij}^* \bigr)^{\tr} \in \Mat_{\ell \times k}(R[\sigma]),
\]
where $M^{\tr}$ denotes the transpose of a matrix $M$, which then satisfies
\begin{equation} \label{E:BCstar}
(BC)^* = C^*B^* \in \Mat_{m \times k}(R[\sigma]), \quad B \in \Mat_{k \times \ell}(R[\tau]),\ C \in \Mat_{\ell \times m}(R[\tau]).
\end{equation}
The inverse of $* : \Mat_{k \times \ell}(R[\tau]) \to \Mat_{\ell \times k}(R[\sigma])$ is also denoted by ``$*$.''

\subsubsection{Division algorithms} \label{SSS:divalg}
When $R$ is a field and $\tau : R \to R$ is an automorphism, $R[\tau]$ and $R[\sigma]$ possess both left and right division algorithms (see \citelist{\cite{Cohn}*{\S 7.3} \cite{Goss}*{\S 1.6} \cite{Ore33b}}). For $\alpha$, $\beta \in R[\tau]$ with $\beta \neq 0$, there exist unique $\eta$, $\eta'$, $\delta$, $\delta' \in R[\tau]$ so that
\[
\alpha = \beta \eta + \delta, \quad \alpha = \eta'\beta + \delta', \quad \deg_\tau \delta < \deg_\tau \beta,\quad \deg_\tau \delta' < \deg_\tau \beta.
\]
In this way $R[\tau]$ is both a left and right principal ideal domain. The right division algorithm does not require $\tau$ to be an automorphism, but it is required for the left division algorithm. By applying the $*$-anti-involution, similar statements hold for $R[\sigma]$.

\begin{proposition}[{cf.\ Anderson~\cite{And86}*{Props.~1.4.2, 1.4.4}}] \label{P:AndUBV}
Let $R$ be a field such that $\tau: R \to R$ is an automorphism. Let $B \in \Mat_{\ell \times k}(R[\tau])$.
\begin{alphenumerate}
\item There exist $U \in \GL_{\ell}(R[\tau])$ and $V \in \GL_{k}(R[\tau])$ so that the entries of $D = UBV$ vanish off its diagonal.
\item If $\ell \geqslant k$ and if the diagonal entries of $D$ are $\delta_1, \dots, \delta_k \in R[\tau]$, then 
\[
\dim_{R} \biggl( \frac{\Mat_{1 \times k}(R[\tau])}{\Mat_{1 \times \ell}(R[\tau]) B} \biggr) = \sum_{i=1}^k \dim_{R} \biggl( \frac{R[\tau]}{R[\tau]\delta_i} \biggr).
\]
\end{alphenumerate}
With appropriate changes, similar results hold for matrices over $R[\sigma]$.
\end{proposition}

The proof of this proposition follows exactly as the proofs of \cite{And86}*{Props.~1.4.2, 1.4.4}, relying on the left and right division algorithms for $R[\tau]$ and $R[\sigma]$, and generally follows from standard arguments for modules over skew polynomial rings~\cite{Cohn}*{\S 7.2--7.3}.

\subsubsection{Orders of finite $F[x]$-modules} \label{SSS:orders}
For $F[x]$ a polynomial ring in one variable over a field~$F$, we say that an $F[x]$-module is \emph{finite} if it is finitely generated and torsion.
Now fix a finite $F[x]$-module $M$. Then there are monic polynomials $f_1, \dots, f_\ell \in F[x]$ so that
\[
M \cong F[x]/(f_1) \oplus \dots \oplus F[x]/(f_\ell).
\]
We set $\Aord{M}{F[x]} \assign f_1 \cdots f_\ell \in F[x]$, which is a generator of the Fitting ideal of $M$, and we call $\Aord{M}{F[x]}$ the \emph{$F[x]$-order} of $M$. If $m_x : M \to M$ is left-multiplication by~$x$, then
\[
\Aord{M}{F[x]} = \Char(m_x,M,X)|_{X=x},
\]
where $\Char(m_x,M,X) \in F[X]$ is the characteristic polynomial of $m_x$ as an $F$-linear map.
For a variable $y$ independent from $x$, but $M$ still an $F[x]$-module, we will write
\[
\Aord{M}{F[y]} \assign \Aord{M}{F[x]}|_{x=y} = \Char(m_x,M,y).
\]
This will be of particular use for us when $M$ is an $\sA$-module (or $\bA$-module), where
\[
\Aord{M}{A} = \Aord{M}{\sA}|_{t=\theta} = \Char(m_t,M,\theta) \in A, \ \textup{(or $\Aord{M}{\AAA} = \Aord{M}{\bA}|_{t=\theta} = \Char(m_t,M,\theta) \in \AAA$),}
\]
coercing $A$-orders and $\AAA$-orders to be elements of our scalar fields.

\subsection{Drinfeld modules, Anderson \texorpdfstring{$t$}{t}-modules, and their adjoints} \label{SS:Drinfeld}
Given a field $F\supseteq \FF_q$ and an $\FF_q$-algebra map $\iota : \sA \to F$, we call $F$ an \emph{$\sA$-field}.  The kernel of $\iota$ is the \emph{characteristic} of $F$, and if $\iota$ is injective then the characteristic is \emph{generic}. If $F \subseteq \C$ has generic characteristic, then we always assume that $\iota(t) = \theta$. Otherwise, $\iota(t) \rassign \otheta \in F$.

\subsubsection{Drinfeld modules and Anderson $t$-modules}
A \emph{Drinfeld module} over $F$ is defined by an $\FF_q$-algebra homomorphism $\phi : \sA \to F[\tau]$ such that
\begin{equation} \label{E:Drindef}
\phi_t = \otheta + \kappa_1 \tau + \dots + \kappa_r \tau^r, \quad \kappa_r \neq 0.
\end{equation}
We say that $\phi$ has \emph{rank $r$}. We then make $F$ into an $\sA$-module by setting
\[
t \cdot x \assign \phi_t(x) = \otheta x + \kappa_1 x^q + \cdots + \kappa_r x^{q^r}, \quad x \in F.
\]
Similarly an \emph{Anderson $t$-module of dimension~$\ell$} over $F$ is defined by an $\FF_q$-algebra homomorphism $\psi : \sA \to \Mat_{\ell}(F)[\tau]$ such that
\begin{equation} \label{E:tmoddef}
\psi_t = \pd \psi_t + E_1 \tau + \dots + E_{w} \tau^{w}, \quad E_i \in \Mat_{\ell}(F),
\end{equation}
where $\pd \psi_t - \otheta\cdot \rI_{\ell}$ is nilpotent. A Drinfeld module is then
a $t$-module of dimension~$1$. 
We write $\psi(F)$ for $F^{\ell}$ with the $\sA$-module structure given by $a \cdot \bx \assign \psi_a(\bx)$ through~\eqref{E:taueval}. Similarly, we write $\Lie(\psi)(F)$ for~$F^{\ell}$ with $F[t]$-module structure defined by $\pd \psi_a$ for $a \in \sA$. For $a \in \sA$, the $a$-torsion submodule of $\psi(\oF)$ is denoted
\[
\psi[a] \assign \{ \bx \in \smash{\oF}^{\ell} \mid \psi_a(\bx) = 0 \}.
\]

Given $t$-modules $\phi : \sA \to \Mat_{k}(F)[\tau]$, $\psi : \sA \to \Mat_{\ell}(F)[\tau]$, a morphism $\eta : \phi \to \psi$ is a matrix $\eta \in \Mat_{\ell \times k}(F[\tau])$ such that
$\eta \phi_a = \psi_a \eta$ for all $a \in \sA$.
Moreover, $\eta$ induces an $\sA$-module homomorphism $\eta : \phi(F) \to \psi(F)$, and we have a functor $\psi \mapsto \psi(F)$ from the category of $t$-modules to $\sA$-modules. We also have an induced map of $F[t]$-modules, $\pd \psi : \Lie(\phi)(F) \to \Lie(\psi)(F)$.

Anderson defined $t$-modules in~\cite{And86}, and following his language we sometimes abbreviate ``Anderson $t$-module'' by ``$t$-module.'' For more information about Drinfeld modules and $t$-modules see \cites{Goss,Thakur}. 

\subsubsection{Exponential and logarithm series} \label{SS:ExpLog}
Suppose now that $F \subseteq \C$ has generic characteristic and that $\psi$ is defined over $F$. Then there is a twisted power series $\Exp_{\psi} \in \power{\Mat_{\ell}(F)}{\tau}$, called the \emph{exponential series} of~$\psi$, such that
\[
\Exp_{\psi} = \sum_{i=0}^{\infty} B_i \tau^i, \quad B_0 = \rI_{\ell},\ B_i \in \Mat_{\ell}(F),
\]
and for all $a \in \sA$, $\Exp_{\psi} {}\cdot \pd\psi_a = \psi_a \cdot \Exp_{\psi}$.
This functional identity for $a=t$ induces a recursive relation that uniquely determines $\Exp_{\psi}$. That the coefficient matrices have entries in $F$ is due to Anderson \cite{And86}*{Prop.~2.1.4, Lem.~2.1.6}. In fact if $\psi$ is defined over $H$ with $K \subseteq H \subseteq F$, then $B_i \in \Mat_{\ell}(H)$ even if $H$ is not perfect (see \cite{NamoijamP24}*{Rem.~2.6} for further discussion). The exponential series induces an $\FF_q$-linear and entire function,
\[
\Exp_{\psi} : \C^{\ell} \to \C^{\ell}, \quad \Exp_{\psi}(\bz) = \sum_{i=0}^{\infty} B_i \bz^{(i)}, \quad \bz \assign (z_1, \dots, z_{\ell})^{\tr},
\]
called the \emph{exponential function} of $\psi$. That $\Exp_{\psi}$ converges everywhere is equivalent to
$\lim_{i \to \infty} \inorm{B_i}^{1/q^i} = 0 \Leftrightarrow \lim_{i \to \infty} \deg (B_i)/q^i = -\infty$.
We also identify the exponential function with the $\FF_q$-linear formal power series $\Exp_\psi(\bz) \in \power{\C}{\bz}^{\ell}$. The functional equation for $\Exp_\psi$ induces the identities,
\[
\Exp_\psi (\pd \psi_a \bz) = \psi_a \bigl( \Exp_\psi(\bz) \bigr), \quad \forall\, a \in \sA.
\]
The exponential function of $\psi$ is always surjective for Drinfeld modules, but it may not be surjective when $\ell \geqslant 2$. We say that $\psi$ is \emph{uniformizable} if $\Exp_{\psi} : \C^{\ell} \to \C^{\ell}$ is surjective. The kernel of $\Exp_{\psi} \subseteq \C^{\ell}$,
\[
\Lambda_\psi \assign \ker \Exp_{\psi},
\]
is a finitely generated and discrete $\pd\psi(\sA)$-submodule of $\C^{\ell}$ called the \emph{period lattice} of~$\psi$.
Thus if $\psi$ is uniformizable, then we obtain an exact sequence of $\sA$-modules,
\[
0 \to \Lambda_{\psi} \to \C^{\ell} \xrightarrow{\Exp_{\psi}} \psi(\C) \to 0.
\]
As an element of $\power{\Mat_{\ell}(F)}{\tau}$ the series $\Exp_{\psi}$ is invertible, and we let $\Log_{\psi} \assign \Exp_{\psi}^{-1} \in \power{\Mat_{\ell}(F)}{\tau}$ be the \emph{logarithm series} of~$\psi$, satisfying
\[
\Log_{\psi} = \sum_{i=0}^{\infty} C_i \tau^i, \quad C_0 = \rI_{\ell},\ C_i \in \Mat_{\ell}(F).
\]
Together with the \emph{logarithm function}, $\Log_\psi(\bz) = \sum_{i \geqslant 0} C_i \bz^{(i)} \in \power{\C}{\bz}^{\ell}$, we have
$\pd \psi_a \cdot \Log_{\psi} = \Log_{\psi} {} \cdot \psi_a$ and $\pd\psi_a ( \Log_{\psi}(\bz) ) = \Log_{\psi} (\psi_a (\bz))$, for all $a \in \sA$.
In general $\Log_\psi(\bz)$ converges only on an open polydisc in $\C^{\ell}$. For example, if $\phi : \sA \to \C[\tau]$ is a Drinfeld module as in \eqref{E:Drindef}, then $\Log_{\phi}(z)$ converges on the open disk of radius $R_\phi$, where
\begin{equation} \label{E:Rphi}
R_{\phi} = \inorm{\theta}^{-\max \{ (\deg \kappa_i - q^i)/(q^i-1)\, \mid \, 1 \leqslant i \leqslant r,\, \kappa_i \neq 0 \}}
\end{equation}
(see \citelist{\cite{EP14}*{Rem.~6.11} \cite{KhaochimP23}*{Cor.~4.5}}).

\subsubsection{Adjoints of $t$-modules}
Assume now that $F$ is a perfect $\sA$-field and that $\psi : \sA \to \Mat_{\ell}(F)[\tau]$ is an Anderson $t$-module over $F$ defined as in \eqref{E:tmoddef}. The \emph{adjoint} of $\psi$ is defined to be the $\FF_q$-algebra homomorphism $\psi^* : \sA \to \Mat_{\ell}(F)[\sigma]$ defined by
\[
\psi_a^* \assign (\psi_a)^*, \quad \forall\, a \in \sA.
\]
Since for $a$, $b \in \sA$ we have $\psi_{ab} = \psi_a\psi_b = \psi_b \psi_a$, \eqref{E:BCstar} implies that $\psi^*$ respects multiplication, which is the nontrivial part of checking that $\psi^*$ is an $\FF_q$-algebra homomorphism. From~\eqref{E:tmoddef}, we have
\[
\psi^*_t = (\psi_t)^* = (\pd \psi_t)^{\tr} + \bigl(E_1^{(-1)} \bigr)^{\tr} \sigma + \cdots + \bigl( E_w^{(-w)} \bigr)^{\tr} \sigma^{w},
\]
and so for any $\bx \in F^{\ell}$, we have $\psi^*_t(\bx) = (\pd \psi_t)^{\tr} \bx + (E_1^{(-1)})^{\tr} \bx^{(-1)} + \cdots + (E_w^{(-w)})^{\tr} \bx^{(-w)}$.
In this way the map $\psi^*$ induces an $\sA$-module structure on $F^{\ell}$, which we denote $\psi^*(F)$. Similarly we denote $\Lie(\psi^*)(F) = F^{\ell}$ with an $F[t]$-module structure induced by $\pd \psi_a^{\tr}$ for $a \in \sA$. For $a \in \sA$, the $a$-torsion submodule of $\psi^*(\oF)$ is denoted
\[
\psi^*[a] \assign \{ \bx \in \smash{\oF}^{\ell} \mid \psi^*_a(\bx) = 0\}.
\]

If $\eta : \phi \to \psi$ is a morphism of $t$-modules as above, then $\eta^* \in \Mat_{k \times \ell}(F)[\sigma]$ provides a morphism $\eta^* : \psi^* \to \phi^*$ such that $\eta^* \psi_a^* = \phi_a^* \eta^*$ for all $a \in \sA$ (and vice versa). Furthermore, $\pd\eta^* : \Lie(\psi^*)(F) \to \Lie(\phi^*)(F)$ is an $F[t]$-module homomorphism. Adjoints of Drinfeld modules were investigated extensively by Goss~\cite{Goss}*{\S 4.14} and Poonen~\cite{Poonen96}, and we will explore further properties for adjoints of Anderson modules in \S\ref{SS:pairings}.

\subsection{Tate modules and characteristic polynomials for Drinfeld modules} \label{SS:munu}
We fix a Drinfeld module $\phi : \sA \to A[\tau]$ of rank $r$ in generic characteristic, given by
\[
\phi_t = \theta + \kappa_1 \tau + \cdots + \kappa_r \tau^r, \quad \kappa_i \in A,\ \kappa_r \neq 0.
\]
Letting $f \in A_+$ be irreducible of degree~$d$, the reduction of $\phi$ modulo~$f$ is a Drinfeld module $\ophi : \sA \to \FF_f[\tau]$ of rank $r_0 \leqslant r$, where $\FF_f = A/fA$. Then $\phi$ has \emph{good reduction} modulo~$f$ if $r_0=r$ or equivalently if $f \nmid \kappa_r$.

For $\lambda \in \sA_+$ irreducible, we form the $\lambda$-adic \emph{Tate modules},
\[
T_{\lambda}(\phi) \assign \varprojlim \phi[\lambda^m], \quad T_{\lambda}(\ophi) \assign \varprojlim \ophi[\lambda^m].
\]
As an $\sA_{\lambda}$-module, $T_{\lambda}(\phi) \cong \sA_{\lambda}^{r}$, and if $\lambda(\theta) \neq f$, then likewise $T_{\lambda}(\ophi) \cong \sA_{\lambda}^{r_0}$. Fixing henceforth that $\lambda(\theta) \neq f$, we set $P_f(X) \assign \Char(\tau^d,T_{\lambda}(\ophi),X)|_{t=\theta}$ to be the characteristic polynomial of the $q^d$-th power Frobenius acting on $T_{\lambda}(\ophi)$ but, for convenience, with coefficients forced into $A$ (rather than $\sA$). Thus we have
\begin{equation} \label{E:Pfdef}
P_f(X) = X^{r_0} + c_{r_0-1} X^{r_0-1} + \cdots + c_0 \in A[X].
\end{equation}
Takahashi~\cite{Takahashi82}*{Prop.~3} showed that the coefficients in~$A$ and are independent of the choice of $\lambda$ (see also Gekeler~\cite{Gekeler91}*{Cor.~3.4}).
We note that if $\phi$ has good reduction modulo~$\lambda$ and if $\alpha_f \in \Gal(K^{\sep}/K)$ is a Frobenius element, then (e.g., see \citelist{\cite{Goss92}*{\S 3} \cite{Goss}*{\S 8.6}})
\[
\Char(\tau^d,T_{\lambda}(\ophi),X) = \Char(\alpha_f,T_{\lambda}(\phi),X) \in \sA[X].
\]

\subsubsection{Properties of $P_f(X)$} \label{SSS:Pf}
The following results are due to Gekeler~\cite{Gekeler91}*{Thm.~5.1} and Takahashi~\cite{Takahashi82}*{Lem.~2, Prop.~3}.
\begin{itemize}
\item We have $c_0 = c_f^{-1}f$ for some $c_f \in \FF_q^{\times}$.
\item The ideal $(P_f(1)) \subseteq A$ is an Euler-Poincar\'e characteristic for $\ophi(\FF_f)$.
\item The roots $\gamma_1, \dots, \gamma_{r_0}$ of $P_f(x)$ in $\oK$ satisfy $\deg_{\theta} \gamma_i = d/r_0$.
\end{itemize}
Extending these a little further, for $1 \leqslant j \leqslant r_0$, we have $\deg_{\theta} c_{r_0-j} \leqslant jd/r_0$. Additionally,
\begin{equation}
\bigl[ \ophi(\FF_f) \bigr]_{A} = c_f P_f(1)
\end{equation}
by \cite{CEP18}*{Cor.~3.2}. Here we use the convention from \S\ref{SSS:orders} that $\Aord{\ophi(\FF_f)}{A} = \Aord{\ophi(\FF_f)}{\sA}|_{t=\theta}$. Following the exposition in~\cite{CEP18}*{\S 3}, we let $P_f^{\vee}(X) \assign \Char(\tau^d,T_{\lambda}(\ophi)^{\vee},X)|_{t=\theta}$ be the characteristic polynomial in $K[X]$ of $\tau^d$ acting on the the dual space of $T_{\lambda}(\ophi)$. We let $Q_f(X) = X^{r_0} P_f(1/X)$ and $Q_f^{\vee}(X) = X^{r_0}P_f^{\vee}(1/X)$ be the reciprocal polynomials of $P_f(X)$ and $P_f^{\vee}(X)$, and consider $Q_f^{\vee}(fX) = 1 + c_f c_1 X + c_f c_2 f X^2 + \cdots + c_f c_{r_0-1} f^{r_0-2} X^{r_0-1} + c_f f^{r_0-1}X^{r_0}$. To denote the dependence on $\phi$, we write $P_{\phi,f}(X)$, $Q_{\phi,f}(X)$, etc.

By varying $f$, we use $Q_f^{\vee}(fX)$ and $Q_f(X)$ to define multiplicative functions $\mu_{\phi}$, $\nu_{\phi} : A_+ \to A$ such that on powers of a given~$f$,
\begin{equation} \label{E:munugen}
\sum_{m=1}^{\infty} \mu_{\phi}(f^m) X^m \assign \frac{1}{Q_f^{\vee}(fX)}, \quad
\sum_{m=1}^{\infty} \nu_{\phi}(f^m) X^m \assign \frac{1}{Q_f(X)}.
\end{equation}

\subsubsection{Everywhere good reduction} \label{SSS:everywhere}
Hsia and Yu~\cite{HsiaYu00} have determined precise formulas for~$c_f$ in terms of the $(q-1)$-st power residue symbol. Of particular interest presently is the case that $\phi$ has everywhere good reduction, i.e., when $\kappa_r \in \FF_q^{\times}$. In this case, Hsia and Yu~\cite{HsiaYu00}*{Thm.~3.2, Eqs.~(2) \& (8)} showed that $c_f = (-1)^{r + d(r+1)} \kappa_r^d$.
This prompts the definition of a completely multiplicative function $\chi_{\phi} : A_+ \to \FF_q^{\times}$,
\begin{equation} \label{E:chidef}
\chi_{\phi}(a) \assign \bigl((-1)^{r+1} \kappa_r \bigr)^{\deg_{\theta} a},
\end{equation}
for which we see that $c_f = (-1)^r \chi_{\phi}(f)$. Letting $\ochi_{\phi} : A_+ \to \FF_q^{\times}$ be the multiplicative inverse of $\chi_{\phi}$, we see that
\begin{align} \label{E:PfandPfvee}
P_f(X) &= X^r + c_{r-1} X^{r-1} + \cdots + c_1 X + (-1)^r \ochi_{\phi}(f) \cdot f, \\
P_f^{\vee}(X) &= X^r + \frac{(-1)^r \chi_{\phi}(f) c_1}{f}X^{r-1} + \cdots + \frac{(-1)^r \chi_{\phi}(f) c_{r-1}}{f} X + \frac{(-1)^r \chi_{\phi}(f)}{f}, \notag
\end{align}
and likewise
\begin{align} \label{E:QfandQfvee}
Q_f(X) &= 1 + c_{r-1} X + \cdots + c_1 X^{r-1} + (-1)^r \ochi_{\phi}(f) \cdot f X^r, \\
Q_f^{\vee}(fX) &= \begin{aligned}[t]
1 + (-1)^r \chi_{\phi}(f) c_1X + \cdots + (-1)^r &\chi_{\phi}(f) c_{r-1}f^{r-2} X^{r-1}\\
&{}+ (-1)^r \chi_{\phi}(f)f^{r-1} X^r.
\end{aligned}\notag
\end{align}
Moreover,
\begin{equation} \label{E:munuf}
\mu_{\phi}(f) = (-1)^{r+1}\chi_{\phi}(f) c_1, \quad \nu_{\phi}(f) = -c_{r-1}.
\end{equation}
We record the induced recursive relations (cf.~\cite{CEP18}*{Lem.~3.5}) on $\mu_{\phi}$ and $\nu_{\phi}$, where taking $m +r \geqslant 1$ and using the convention that $\mu_{\phi}(b) = \nu_{\phi}(b) = 0$ if $b \in K \setminus A_+$,
\begin{align}
\label{E:murec}
\mu_{\phi}(f^{m+r}) &= \begin{aligned}[t]
\mu_{\phi}(f) \mu_{\phi}(f^{m+r-1}) - (-1)^r \chi_{\phi}(f) \sum_{j=2}^{r-1} &c_j f^{j-1} \mu_{\phi}(f^{m+r-j})\\
&{} - (-1)^r \chi_{\phi}(f) f^{r-1} \mu_{\phi}(f^m),
\end{aligned}
\\
\label{E:nurec}
\nu_{\phi}(f^{m+r}) &=
\nu_{\phi}(f) \nu_{\phi}(f^{m+r-1}) - \sum_{j=2}^{r-1} c_{r-j} \nu_{\phi}(f^{m+r-j}) - (-1)^r \ochi_{\phi}(f) f \nu_{\phi}(f^m).
\end{align}

Later we may write ``$\mu_{\phi,\theta}(a)$'' and ``$\nu_{\phi,\theta}(a)$'' for $\mu_{\phi}(a)$ and $\nu_{\phi}(a)$ to emphasize that they take values in~$A$, and use $\mu_{\phi,z}(a) = \mu_{\theta}(a)|_{\theta=z}$ and similarly $\nu_{\phi,z}(a)$ to switch to values in~$\FF_q[z]$. We use $\mu_{\phi}$ and $\nu_{\phi}$ to define Goss $L$-functions $L(\phi^{\vee},s-1)$ and $L(\phi,s)$ in~\S\ref{SS:GossL}.

\subsection{\texorpdfstring{$t$}{t}-motives and dual \texorpdfstring{$t$}{t}-motives} \label{SS:tmotives}
We recall here basic information about $t$-motives and dual $t$-motives attached to $t$-modules, which will be extended to Anderson $t$-modules over $\tau$-perfect fields and expanded on in~\S\ref{S:Anderson}. For this section we fix a perfect $\sA$-field $F$ and $t$-module $\psi : \sA \to \Mat_{\ell}(F)[\tau]$ as in~\eqref{E:tmoddef}. Recall that $\otheta = \iota(t) \in F$.

\subsubsection{$t$-motive of $\psi$}
We let $\cM_{\psi} \assign \Mat_{1\times \ell}(F[\tau])$, and make $\cM_{\psi}$ into a left $F[t,\tau]$-module by using the inherent structure as a left $F[\tau]$-module and setting
\[
a \cdot m \assign m \psi_a, \quad m \in \cM_{\psi},\ a \in \sA.
\]
Then $\cM_{\psi}$ is called the \emph{$t$-motive} of $\psi$. We note that for any $m\in \cM_{\psi}$,
\[
(t-\otheta)^{\ell} \cdot m \in \tau \cM_{\psi},
\]
since $\pd\psi_t-\otheta \rI_{\ell}$ is nilpotent (and $F$ is perfect). If we need to emphasize the dependence on the base field~$F$, we write
\[
\cM_{\psi}(F) \assign \cM_{\psi} = \Mat_{1 \times \ell}(F[\tau]).
\]

A morphism $\eta : \phi \to \psi$ of $t$-modules over $F$ of dimensions~$k$ and $\ell$, defined as in \S\ref{SS:Drinfeld}, induces a morphism of left $F[t,\tau]$-modules $\eta^{\dagger} : \cM_{\psi} \to \cM_{\phi}$, given by $\eta^{\dagger} (m) \assign m\eta$ for $m \in \cM_{\psi}$. The functor from $t$-modules over $F$ to $t$-motives over $F$ is fully faithful, and so every left $F[t,\tau]$-module homomorphism $\cM_{\psi} \to \cM_{\phi}$ arises in this way.

By construction $\cM_{\psi}$ is free of rank~$\ell$ as a left $F[\tau]$-module, and we say $\ell$ is the \emph{dimension} of $\cM_{\psi}$. If $\cM_{\psi}$ is further free of finite rank over $F[t]$, then $\cM_{\psi}$ is said to be \emph{abelian} and $r = \rank_{F[t]} \cM_{\psi}$ is the \emph{rank} of $\cM_{\psi}$. We will say that $\psi$ is abelian or has rank~$r$ if $\cM_{\psi}$ possesses the corresponding properties. The $t$-motives in Anderson's original definition in~\cite{And86} are abelian, as will be most of the $t$-motives in this paper, but for example, see~\citelist{\cite{BP20} \cite{Goss}*{Ch.~5} \cite{HartlJuschka20} \cite{NamoijamP24}*{Ch.~2--4}} for $t$-motives in this wider context.

\subsubsection{Dual $t$-motive of $\psi$}
We let $\cN_{\psi} \assign \Mat_{1 \times \ell}(F[\sigma])$, and similar to the case of $t$-motives, we define a left $F[t,\sigma]$-module structure on $\cN_{\psi}$ by setting
\[
a \cdot n \assign n \psi_a^*, \quad n \in \cN_{\psi},\ a \in \sA.
\]
The module $\cN_{\psi}$ is the \emph{dual $t$-motive} of $\psi$. As in the case of $t$-motives, for any $n \in \cN_{\psi}$ we have $(t-\otheta)^{\ell} \cdot n \in \sigma\cN_{\psi}$. Also if we need to emphasize the dependence on~$F$, we write
\[
\cN_{\psi}(F) \assign \cN_{\psi} = \Mat_{1\times \ell}(F[\sigma]).
\]

Again for a morphism $\eta : \phi \to \psi$ of $t$-modules of dimensions $k$ and $\ell$, we obtain a morphism of left $F[t,\sigma]$-modules, $\eta^{\ddagger} : \cN_{\phi} \to \cN_{\psi}$, given by $\eta^{\ddagger}(n) \assign n \eta^*$ for $n \in \cN_{\phi}$. Also, every morphism of left $F[t,\sigma]$-modules $\cN_{\phi} \to \cN_{\psi}$ arises in this way.

The dual $t$-motive $\cN_{\psi}$ is free of rank $\ell$ as a left $F[\sigma]$-module, and $\ell$ is the \emph{dimension} of $\cN_{\psi}$. If $\cN_{\psi}$ is free of finite rank over $F[t]$, then we say $\cN_{\psi}$ is \emph{$\sA$-finite}, and we call $r = \rank_{F[t]}(\cN_{\psi})$ the \emph{rank} of $\cN_{\psi}$. It has been shown by Maurischat~\cite{Maurischat26} that for a $t$-module $\psi$, the $t$-motive $\cM_{\psi}$ is abelian if and only if the dual $t$-motive $\cN_{\psi}$ is $\sA$-finite. In this case the rank of $\cM_{\psi}$ is the same as the rank of $\cN_{\psi}$. We will say that $\psi$ is $\sA$-finite or has rank~$r$ if $\cN_{\psi}$ has those properties.
Dual $t$-motives were initially introduced in \cite{ABP04} over fields of generic characteristic. See \citelist{\cite{BP20} \cite{HartlJuschka20} \cite{Maurischat26} \cite{NamoijamP24}*{Ch.~2--4}}, for more information.

We call $\bm = (m_1, \dots, m_r)^{\tr} \in \Mat_{r\times 1}(\cM_{\psi}(F))$ a \emph{basis} of $\cM_{\psi}(F)$ if $m_1, \dots, m_r$ form an $F[t]$-basis of $\cM_{\psi}(F)$. Likewise $\bn = (n_1, \dots, n_r)^{\tr} \in \Mat_{r\times 1}(\cN_{\psi}(F))$ is a \emph{basis} of $\cN_{\psi}(F)$ if $n_1, \dots, n_r$ form an $F[t]$-basis of $\cN_{\psi}(F)$. We then define $\Gamma$, $\Phi \in \Mat_r(F[t])$ so that
\[
\tau \bm = \Gamma \bm, \quad \sigma \bn = \Phi \bn.
\]
It follows that $\det \Gamma = c(t-\theta)^\ell$, $\det \Phi = c'(t-\theta)^\ell$,
where $c$, $c' \in F^{\times}$ (e.g., see \cite{NamoijamP24}*{\S 3.2}). Then $\Gamma$ \emph{represents multiplication by $\tau$} on $\cM_{\psi}$ and $\Phi$ \emph{represents multiplication by $\sigma$} on $\cN_{\psi}$.

\begin{example} \label{Ex:Carlitz1}
\emph{Carlitz module.} The Carlitz module $\sC : \sA \to F[\tau]$ over $F$ is defined by
\[
\sC_t = \otheta + \tau,
\]
and it has dimension~$1$ and rank~$1$. Then $\bm = \{1\}$ is an $F[t]$-basis for $\cM_{\sC} = F[\tau]$, and $\bn = \{1\}$ is an $F[t]$-basis for $\cN_{\sC} = F[\sigma]$. One finds that $\tau \cdot 1 = (t-\theta)\cdot 1$ in $\cM_{\sC}$ and $\sigma \cdot 1 = (t-\theta)\cdot 1$ in $\cN_{\sC}$, so $\Gamma = \Phi = t-\theta$.
\end{example}

\begin{example} \label{Ex:Drinfeld1}
\emph{Drinfeld modules.} Let $\phi : \sA \to F[\tau]$ be a Drinfeld module over $F$ of rank~$r$ defined as in~\eqref{E:Drindef}. Then $\bm = (1,\tau, \dots, \tau^{r-1})^{\tr}$ is a basis for $\cM_{\psi}$ and $\bn = (1, \sigma, \dots, \sigma^{r-1})^{\tr}$ is a basis for $\cN_{\psi}$. Furthermore, $\tau \bm = \Gamma\bm$ and $\sigma \bn = \Phi \bn$, where
\begin{equation} \label{E:Gammadef}
\Gamma = \begin{pmatrix}
0 & 1 & \cdots & 0 \\
\vdots & \vdots & \ddots & \vdots \\
0 & 0 & \cdots & 1 \\
(t-\otheta)/\kappa_r & -\kappa_1/\kappa_r & \cdots & -\kappa_{r-1}/\kappa_r
\end{pmatrix},
\end{equation}
and $\Phi$ occurs similarly. See \citelist{\cite{CP12}*{\S 3.3--3.4} \cite{NamoijamP24}*{Ex.~3.35, Ex.~4.117} \cite{Pellarin08}*{\S 4.2}} for details.
\end{example}

\subsubsection{Rigid analytic trivializations}
We now specialize to the situation that $\psi : \sA \to \Mat_{\ell}(F)$ is an Anderson $t$-module defined over a perfect field $F$ with $K \subseteq F \subseteq \C$ of generic characteristic. We further assume that $\psi$ is abelian of rank~$r$, equivalently $\sA$-finite of rank~$r$ by~\cite{Maurischat26}. If $\Gamma$ represents multiplication by $\tau$ on $\cM_{\psi}$, we set
\[
\Theta \assign \Gamma^{\tr} \in \Mat_r(F[t]).
\]
Then we say that $\cM_{\psi}$ is \emph{rigid analytically trivial} if there exists $\Upsilon \in \GL_r(\TT_t)$ so that
\begin{equation}
\Upsilon^{(1)} = \Upsilon \Theta.
\end{equation}
By~\cite{And86}*{Thm.~4}, $\psi$ is uniformizable if and only if $\cM_{\psi}$ is rigid analytically trivial.

\begin{remark}
Definitions of rigid analytic trivializations using the dual $t$-motive $\cN_\psi$ have been investigated extensively in the context of transcendence theory (e.g., see \citelist{\cite{ABP04} \cite{Chang20} \cite{CP12} \cite{Maurischat22} \cite{NamoijamP24}*{Ch.~3--4} \cite{P08} \cite{Pellarin08}}). However, for our purposes the rigid analytic trivialization for $\cM_{\psi}$ is more convenient. Moreover, the two types of rigid analytic trivializations are related by \cite{HartlJuschka20}*{Thm.~2.5.13} (see \citelist{\cite{CP12}*{\S 3.4} \cite{NamoijamP24}*{\S 4.4--4.6}} for additional discussion).
\end{remark}

\begin{example} \label{Ex:Carlitz2}
\emph{Carlitz module.}  The Carlitz exponential $\Exp_{\sC} = \sum_{i \geqslant 0} D_i^{-1} \tau^i$ and logarithm $\Log_{\sC} = \sum_{i \geqslant 0} L_i^{-1} \tau^i$ are defined for $D_i$, $L_i \in A$ (see \citelist{\cite{Goss}*{Ch.~3} \cite{Thakur}*{Ch.~2}}). Its period lattice $\Lambda_{\sC} = A \tpi$ is generated by the \emph{Carlitz period} (see \cite{Carlitz35}*{Thm.~5.1}),
\[
\tpi = -(-\theta)^{q/(q-1)} \prod_{i=1}^{\infty} \Bigl( 1 - \theta^{1-q^i} \Bigr)^{-1} \in K_\infty\bigl( (-\theta)^{1/(q-1)} \bigr),
\]
for a fixed choice of $(-\theta)^{1/(q-1)}$. The radius of convergence of $\Log_{\sC}$ is $R_{\sC} = \norm{\tpi} = \norm{\theta}^{q/(q-1)}$. The rigid analytic trivialization for $\sC$ is the \emph{Anderson-Thakur function} \cite{AndThak90}*{\S 2.5},
\begin{equation} \label{E:omegadef}
\omega \assign (-\theta)^{1/(q-1)} \prod_{i=0}^{\infty} \biggl( 1 - \frac{t}{\theta^{q^i}} \biggr)^{-1} = \sum_{m=0}^{\infty} \exp_{\sC} \biggl( \frac{\tpi}{\theta^{m+1}} \biggr) t^m \in \TT_t^{\times}.
\end{equation}
The functional equation $\omega^{(1)} = (t-\theta)\omega$ implies that
\begin{equation}
\frac{\omega^{(i)}}{\omega} = \bigl( t-\theta^{q^{i-1}} \bigr) \cdots \bigl( t-\theta^q \bigr) (t-\theta) \in A[t], \quad i \geqslant 0,
\end{equation}
where we use the convention that the empty product is~$1$ so that the identity holds for $i=0$. Furthermore, we recover the Carlitz period by taking a residue at $t=\theta$,
\[
\Res_{t=\theta} \omega = (t-\theta)\omega |_{t=\theta} =  -\tpi.
\]
For additional properties and generalizations of $\omega$, see \cites{AnglesPellarin15,APT16,Pellarin12,Perkins14}.
\end{example}

\begin{example} \label{Ex:Drinfeld2}
\emph{Drinfeld modules.} For a Drinfeld module $\phi : \sA \to F[\tau]$ of rank~$r$ defined as in \eqref{E:Drindef} over a perfect $\sA$-field $F$ of generic characteristic with $K \subseteq F \subseteq \C$, we fix generators $\pi_1, \dots, \pi_r \in \Lambda_{\phi}$. For $i=1, \dots, r$, we define the \emph{Anderson generating function},
\[
g_i \assign \sum_{m=0}^{\infty} \Exp_{\phi} \biggl( \frac{\pi_i}{\theta^{m+1}} \biggr) t^m \in \TT_t,
\]
which satisfies $\phi_t(g_i) = \theta g_i + \kappa_1 g_i^{(1)} + \cdots + \kappa_r g_i^{(r)} = tg_i$.
The Anderson-Thakur function $\omega$ is then the Anderson generating function for $\tpi$ by~\eqref{E:omegadef}. By Example~\ref{Ex:Drinfeld1},
\begin{equation} \label{E:Thetadef}
\Theta = \Gamma^{\tr} = \begin{pmatrix}
0 & \cdots & 0 & (t-\theta)/\kappa_r \\
1 & \cdots & 0 & -\kappa_1/\kappa_r \\
\vdots & \ddots & \vdots & \vdots \\
0 & \cdots & 1 & -\kappa_{r-1}/\kappa_r
\end{pmatrix},
\end{equation}
and one finds that
\begin{equation} \label{E:Upsilondef}
\Upsilon \assign \begin{pmatrix}
g_1 & g_1^{(1)} & \cdots & g_1^{(r-1)} \\
g_2 & g_2^{(1)} & \cdots & g_2^{(r-1)} \\
\vdots & \vdots &  & \vdots \\
g_r & g_r^{(1)} & \cdots & g_r^{(r-1)}
\end{pmatrix} \in \GL_r(\TT_t)
\end{equation}
satisfies $\Upsilon^{(1)} = \Upsilon\Theta$. Thus $\Upsilon$ is a rigid analytic trivialization for $\cM_{\phi}$, originally determined by Pellarin~\cite{Pellarin08}*{\S 4.2} (see also \cite{Goss94}*{\S 2.6} for rank $2$). That $\Upsilon$ is invertible in $\Mat_r(\TT_t)$ takes some effort, but see \citelist{\cite{CP11}*{\S 2.5} \cite{GezmisP19}*{Prop.~6.2.4} \cite{Pellarin08}*{\S 4.2.3}} for more details. From the theory of Anderson generating functions, one knows that each $g_i$ has a meromorphic continuation to $\C$, with simple poles at $t = \theta$, $\theta^q$, $\theta^{q^2}, \ldots$. We find  that
\[
\Res_{t=\theta} g_i = (t-\theta)g_i|_{t=\theta} = -\pi_i, \quad
g_{i}^{(j)}(\theta) = \eta_{ij}, \quad 1 \leqslant i \leqslant r,\ 1 \leqslant j \leqslant r-1,
\]
where $\eta_{ij}$ is a quasi-period for $\phi$. See \citelist{\cite{CP11}*{\S 2.5} \cite{CP12}*{\S 3.4} \cite{NamoijamP24}*{Ex.~4.117} \cite{Pellarin08}*{\S 4.2}}. Anderson generating functions originally appeared in \cite{And86}*{\S 3.2}, and they have been studied extensively for Drinfeld modules in \cites{AndThak90, CP11, CP12, EP14, Goss94, Green22, KhaochimP23, MaurischatPerkins22, Pellarin08, Pellarin12}.
About our notation, our use of $\Upsilon$ in \eqref{E:Upsilondef} coincides with that in \cites{CP11,CP12,GezmisP19}, but would be $\Upsilon^{\tr}$ in \cite{KhaochimP23}; $(\Upsilon^{\tr})^{(-1)}$ in \cite{NamoijamP24}*{Eq.~(4.45)}; and $\widehat{\Psi}^{\tr}$ in \cite{Pellarin08}*{\S 4.2}.
\end{example}

\subsection{Schur polynomials} \label{SS:Schur}
We review properties of symmetric polynomials and especially Schur polynomials. For more details on symmetric polynomials see \citelist{\cite{Aigner}*{Ch.~8} \cite{Stanley}*{Ch.~7}}.
Letting $\bx = \{ x_1, \dots, x_n \}$ be independent variables, the \emph{elementary symmetric polynomials} $\{e_i\}_{i=0}^{n} = \{ e_{n,i}\}_{i=0}^n \subseteq \ZZ[\bx]$ are defined by
\begin{equation} \label{E:elesymm}
\sum_{i=0}^n e_i(\bx)T^i = (1+x_1 T)(1+x_2 T) \cdots (1+x_n T).
\end{equation}
We adopt the convention that $e_i = 0$ if $i < 0$ or $i > n$.
The \emph{complete homogeneous symmetric polynomials} $\{h_i\}_{i \geqslant 0} = \{ h_{n,i} \}_{i \geqslant 0} \subseteq \ZZ[x_1,\dots, x_n]$ are defined by
\begin{equation} \label{E:homsymm}
\sum_{i=0}^\infty h_i(\bx) T^i = \frac{1}{(1-x_1 T)(1-x_2 T) \cdots (1-x_n T)},
\end{equation}
and similarly if $i < 0$ then we take $h_i=0$. Then $h_i$ consists of the sum of all monomials in $x_1, \dots, x_n$ of degree~$i$. The \emph{Vandermonde determinant} is
\begin{equation}
V(\bx) = \prod_{1 \leqslant i < j \leqslant n} (x_i - x_j).
\end{equation}
When nonzero we have $\deg e_i=i$ and $\deg h_i=i$, and $\deg V = \binom{n}{2}$.

\begin{definition} \label{D:polytensor}
For polynomials $P(T) = (T-x_1) \cdots (T-x_k)$ and $Q(T) = (T - y_1) \cdots (T-y_\ell)$, we set
\[
(P \otimes Q)(T) \assign \prod_{\substack{1 \leqslant i \leqslant k \\ 1 \leqslant j \leqslant \ell}} (T- x_i y_j).
\]
Letting $B_m$ be the coefficient of $T^m$ in $(P\otimes Q)(T)$, we find that $B_m$ is symmetric in both $x_1, \dots, x_k$ and $y_1, \dots, y_{\ell}$, its total degree in $x_1, \dots, x_k$ is $k\ell - m$, and its total degree in $y_1, \dots, y_\ell$ is also $k\ell -m$. As such,
$B_m \in \ZZ[e_{k,1}(\bx), \ldots, e_{k,k\ell-m}(\bx); e_{\ell, 1}(\by), \ldots, e_{\ell,k\ell-m}(\by)]$.
The coefficients of $(P\otimes Q)(T)$ and its inverse $(P\otimes Q)(T)^{-1}$ are expressible in terms of Schur polynomials (see Theorem~\ref{T:Cauchy} and Corollary~\ref{C:Cauchynl} for $(P\otimes Q)(T)^{-1}$).
\end{definition}

\subsubsection{Schur polynomials} \label{SSS:Schur}
Let $\lambda$ denote an integer partition $\lambda_1 \geqslant \cdots \geqslant \lambda_n \geqslant 0$ of length~$n$, where $\lambda_i=0$ is allowed. We set
\begin{equation} \label{E:slambda}
s_{\lambda}(\bx) = s_{\lambda_1 \cdots \lambda_n}(\bx) \assign
V(\bx)^{-1} \cdot \det \begin{pmatrix}
x_1^{\lambda_1+n-1} & \cdots & x_n^{\lambda_n+n-1}\\
\vdots & & \vdots \\
x_1^{\lambda_{n-1}+1} & \cdots & x_n^{\lambda_{n-1}+1} \\
x_1^{\lambda_n} & \cdots & x_n^{\lambda_n}
\end{pmatrix}
\end{equation}
We have the following properties (see \citelist{\cite{Aigner}*{\S 8.3} \cite{Stanley}*{\S 7.15}}).
\begin{itemize}
\item $s_{\lambda}(\bx)$ is a symmetric polynomial in $\ZZ[x_1, \dots, x_n]$.
\item $\deg s_{\lambda}(\bx) = \lambda_1 + \dots + \lambda_n$.
\item For $0 \leqslant i \leqslant n$ we have $s_{\underbrace{\scriptstyle 1\,\,\cdots\,\,1}_{i} \underbrace{\scriptstyle 0\,\,\cdots\,\, 0}_{n-i}}(\bx) = e_i(\bx)$.
\item For $i \geqslant 0$ we have $s_{i \underbrace{\scriptstyle 0\,\,\cdots\,\,0}_{n-1}}(\bx) = h_i(\bx)$.
\end{itemize}
The polynomial $s_{\lambda}$ is called the \emph{Schur polynomial for $\lambda$}. Following the exposition of Bump and Goldfeld~\cites{Bump89,Goldfeld}, when $n \geqslant 2$ (which we now assume), we consider the subset of Schur polynomials where $\lambda_n=0$ as follows. For integers $k_1, \dots, k_{n-1} \geqslant 0$, form
\[
\lambda : k_1 + \cdots + k_{n-1} \geqslant k_2 + \cdots + k_{n-1} \geqslant \cdots \geqslant k_{n-1} \geqslant 0 \geqslant 0.
\]
We set $S_{k_1, \dots, k_{n-1}}(\bx)$ to be the Schur polynomial $s_\lambda$, i.e.,
\begin{equation} \label{E:Sk}
S_{k_1, \dots, k_{n-1}}(\bx) \assign
V(\bx)^{-1} \cdot \det \begin{pmatrix}
x_1^{k_1+\cdots+k_{n-1}+n-1} & \cdots & x_n^{k_1+\cdots+k_{n-1}+n-1}\\
x_1^{k_2+\cdots+k_{n-1}+n-2} & \cdots & x_n^{k_2+\cdots+k_{n-1}+n-2}\\
\vdots & & \vdots \\
x_1^{k_{n-1}+1} & \cdots & x_n^{k_{n-1}+1} \\
1 & \cdots & 1
\end{pmatrix}.
\end{equation}
The degree of $S_{k_1, \dots, k_{n-1}}(\bx)$ is $k_1 + 2k_2 + \cdots + (n-1)k_{n-1}$.

\begin{lemma} \label{L:slambdaSk}
Let $\lambda : \lambda_1 \geqslant \cdots \geqslant \lambda_n \geqslant 0$ be an integer partition. Then
\[
s_{\lambda}(\bx) = (x_1 \cdots x_n)^{\lambda_{n}} \cdot S_{\lambda_1-\lambda_2,\lambda_2-\lambda_3, \ldots, \lambda_{n-1}-\lambda_n}(\bx).
\]
\end{lemma}

\begin{proof}
By pulling $x_i^{\lambda_n}$ out of the $i$-th column in \eqref{E:slambda} we see that $s_{\lambda} = (x_1 \cdots x_n)^{\lambda_n} \cdot s_{\lambda'}$, where $\lambda'$ is the partition $\lambda_1-\lambda_n \geqslant \lambda_2-\lambda_n \geqslant \cdots \geqslant \lambda_{n-1}-\lambda_n \geqslant 0\geqslant 0$. The identity then follows from~\eqref{E:Sk}.
\end{proof}

As a result, we see from the properties of $s_{\lambda}$ above that
\begin{alignat}{2}
\label{E:Sei}
S_{\underbrace{\scriptstyle 0,\,\ldots\,,0,1,0,\,\ldots\,,0}_{\textup{$i$-th place}}}(\bx) &= e_i(\bx), \qquad &&1 \leqslant i \leqslant n-1,\\
\label{E:Shi}
S_{i,\underbrace{\scriptstyle 0,\,\cdots\,, 0}_{n-2}}(\bx) &= h_i(\bx), \qquad &&i \geqslant 0.
\end{alignat}

\begin{lemma} \label{L:Skreorder}
For $k_1, \dots, k_{n-1} \geqslant 0$, we have
\[
(x_1 \cdots x_n)^{k_1 + \dots + k_{n-1}} \cdot S_{k_1, \dots, k_{n-1}} \bigl( x_1^{-1}, \dots, x_n^{-1} \bigr) = S_{k_{n-1}, \dots, k_1}(\bx).
\]
\end{lemma}

\begin{proof}
Substituting $x_1 \leftarrow x_1^{-1}, \dots, x_n \leftarrow x_n^{-1}$ into~\eqref{E:Sk} and using the multilinearity of the determinant, we have
\[
\begin{aligned}[b]
(x_1 \cdots x_n &)^{k_1 + \dots + k_{n-1}}\\
&{}\cdot S_{k_1, \dots, k_{n-1}} \bigl( x_1^{-1}, \dots, x_n^{-1} \bigr)
\end{aligned}
= \frac{\det \begin{pmatrix}
1 & \cdots & 1\\
x_1^{k_{1}+1} & \cdots & x_n^{k_{1}+1}\\
\vdots & & \vdots \\
x_1^{k_1+\cdots+k_{n-2}+n-2} & \cdots & x_n^{k_1+\cdots+k_{n-2}+n-2} \\
x_1^{k_1+\cdots+k_{n-1}+n-1} & \cdots & x_n^{k_1+\cdots+k_{n-1}+n-1}
\end{pmatrix}}{\displaystyle (x_1 \cdots x_n)^{n-1} \prod_{1 \leqslant i<j \leqslant n}\bigl(x_i^{-1} - x_j^{-1} \bigr)}.
\]
By rearranging the rows, we check that the numerator on the right is $(-1)^{\binom{n}{2}}$ times the determinant in~\eqref{E:Sk} with $(k_1, \dots, k_{n-1}) \leftarrow (k_{n-1}, \dots, k_1)$. Likewise the denominator is $(-1)^{\binom{n}{2}}$ times the $V(\bx)$ term of~\eqref{E:Sk}.
\end{proof}

\subsubsection{Cauchy's identities}
In order to work with $(P\otimes Q)(T)^{-1}$, which arises as part of the Euler product for one of the main $L$-series $L(\EE(\phi\times \psi)^{\vee},s)$ we consider in this paper (see \S\ref{S:Lseries}--\ref{S:convolutions}), we use the following identities expressed in terms of Schur polynomials.

\begin{theorem}[{Cauchy's Identity, see \citelist{\cite{Aigner}*{Cor.~8.16} \cite{Bump89}*{\S 2.2} \cite{Stanley}*{Thm.~7.12.1}}}] \label{T:Cauchy}
For variables $\bx = \{x_1, \dots, x_n\}$ and $\by = \{y_1, \dots, y_n\}$, let $X=x_1 \cdots x_n$ and $Y=y_1 \cdots y_n$. Then as power series in $\power{\ZZ[\bx,\by]}{T}$,
\[
\prod_{1 \leqslant i,j \leqslant n} (1 - x_i y_j T)^{-1}
= (1 - XY T^n)^{-1} \underset{k=(k_1, \dots, k_{n-1})}{\sum_{k_1=0}^{\infty} \cdots \sum_{k_{n-1}=0}^{\infty}} S_k(\bx) S_k(\by) T^{k_1 + 2k_2 + \cdots + (n-1)k_{n-1}}.
\]
\end{theorem}

If instead we have $\bx = \{x_1, \dots, x_n\}$ and $\by = \{ y_1, \dots, y_{\ell} \}$ with $n < \ell$, then Cauchy's identity reduces to the following result by setting $x_{n+1} = \cdots = x_{\ell} = 0$ and simplifying.

\begin{corollary}[{Bump \cite{Bump89}*{\S 2.2}}] \label{C:Cauchynl}
For variables $\bx = \{x_1, \dots, x_n\}$ and $\by = \{y_1, \dots, y_\ell\}$ with $n < \ell$, let $X=x_1 \cdots x_n$. Then as power series in $\power{\ZZ[\bx,\by]}{T}$,
\[
\prod_{\substack{1 \leqslant i \leqslant n \\ 1 \leqslant j \leqslant \ell}} (1-x_i y_j T)^{-1}
=\underset{\substack{k=(k_1, \dots, k_{n-1}) \\ k'=(k_1, \dots, k_{n},0 \dots, 0)}}{\sum_{k_1=0}^{\infty} \cdots \sum_{k_{n}=0}^{\infty}} S_{k}(\bx) S_{k'}(\by) X^{k_n} T^{k_1 +2k_2 + \cdots + n k_n}.
\]
\end{corollary}

\subsubsection{Pieri's rules}
In general products of Schur polynomials can be expressed as linear combinations of Schur polynomials using the Littlewood-Richardson rule~\cite{Stanley}*{\S A1.3}. Pieri's rule and its dual are special cases, which we state here in terms of the polynomials $S_{k_1, \dots, k_{n-1}}(\bx)$. For $k$, $k_1, \dots, k_{n-1} \geqslant 0$, \emph{Pieri's rule} \cite{Stanley}*{Thm.~7.15.7} is
\begin{multline} \label{E:Pieri}
h_{k}(\bx) \cdot S_{k_1, \dots, k_{n-1}}(\bx) \\
= \sum_{\substack{m_0 + \cdots + m_{n-1}=k \\ m_1 \leqslant k_1,\, \ldots,\, m_{n-1} \leqslant k_{n-1}}} S_{k_1 + m_0-m_1,\, k_2 + m_1-m_2,\,\ldots,\,k_{n-1}+m_{n-2}-m_{n-1}}(\bx) X^{m_{n-1}},
\end{multline}
where $X=x_1\cdots x_n$. For the \emph{dual Pieri rule}~\cite{Stanley}*{p.~340}, we define for $k_1, \ldots, k_{n-1} \geqslant 0$,
\[
\cI_{k_1, \dots, k_{n-1}} \assign
\Bigl\{ (m_0, \dots, m_{n-1}) \in \{ 0, 1\}^n \Bigm| k_j = 0 \Rightarrow \bigl( m_j=1 \Rightarrow m_{j-1}=1 \bigr) \Bigr\}.
\]
Then for $0 \leqslant k \leqslant n-1$,
\begin{multline} \label{E:dualPieri}
e_k(\bx) \cdot S_{k_1, \dots, k_{n-1}}(\bx) \\
= \sum_{\substack{m_0 + \cdots + m_{n-1} = k \\ (m_0, \ldots, m_{n-1})\, \in\, \cI_{k_1, \dots, k_{n-1}}}}
S_{k_1+m_0-m_1,\, k_2 + m_1-m_2,\,\ldots,\,k_{n-1}+m_{n-2}-m_{n-1}}(\bx) X^{m_{n-1}}.
\end{multline}

\subsubsection{Jacobi-Trudi identity} \label{SS:JT}
We can also express Schur polynomials in terms of complete homogeneous symmetric polynomials by the \emph{Jacobi-Trudi identity}~\citelist{\cite{Aigner}*{Thm.~8.7} \cite{Stanley}*{Thm.~7.16.1}}.
We state it here for $S_{k_1, \dots, k_{n-1}}(\bx)$. For $k_1, \dots, k_{n-1} \geqslant 0$, we have
\begin{equation} \label{E:JT}
S_{k_1, \dots, k_{n-1}}(\bx) = \det \Bigl( h_{k_i + \cdots + k_{n-1} - i+j}(\bx) \Bigr)_{i,j=1}^{n-1},
\end{equation}
where $(i,j)$ is the index of the $i$-th row and $j$-th column.

\section{Anderson \texorpdfstring{$t$}{t}-modules} \label{S:Anderson}

The theory of Anderson $t$-modules over Tate algebras was initiated by Angl\`{e}s, Pellarin, and Tavares Ribeiro \cites{AnglesPellarin15, APT16} and continued in several articles including \cites{APT18, ANT22, AT17, Gezmis19, Gezmis20, GezmisP19, GezmisPellarin22, Tavares21}. They also are the main objects of study of Demeslay~\cites{DemeslayPhD,Demeslay22} for his class module formula, which we recount in~\S\ref{SS:Demeslayformula}. We will follow Demeslay's definition of Anderson $t$-modules in~\cite{Demeslay22}, but in order to discuss reductions modulo irreducible elements of $A$ we will require a slightly more general construction over ``$\cA$-fields.''

\subsection{Anderson \texorpdfstring{$t$}{t}-modules over \texorpdfstring{$\cA$}{A}-fields} \label{SS:cAtmodules}
Throughout \S\ref{S:Anderson} we will consider a more general situation, which will help streamline our arguments and may be of benefit for future work. We let $Z \assign \{ z_1, \dots, z_n \}$, for variables $z_1, \dots, z_n$ and for $n \geqslant 0$ (so $n=0 \Leftrightarrow Z = \emptyset$). Then for a field $F$ we let
\[
\FZ{F}{Z} \assign F(Z) \quad \textup{or}\quad \laurent{F}{Z},
\]
so $\FZ{F}{Z}$ denotes either a field of rational functions or formal Laurent series.
We thus define fields $\FZ{\FF_q}{Z} \subseteq \FZ{\oFF_q}{Z}$, and note that $\FZ{\FF_q}{Z} \cap \oFF_q = \FF_q$. We let
\[
\cA \assign \FZ{\FF_q}{Z}[t].
\]
For the purposes of this paper, we will use primarily $Z = \emptyset$ or $Z=\{z\}$, which has prompted us to combine these cases into a single exposition, but the cases of $Z = \{ z_1, \dots, z_n \}$ for $n \geqslant 2$ may be of future interest.

\begin{definition}
Given an $\sA$-field $F$ and an $\FZ{\FF_q}{Z}$-algebra $H \supseteq F$, we extend $\iota : \sA \to F$ to $\iota: \cA \to H$ by composing the natural maps
\[
\iota : \cA = \FZ{\FF_q}{Z} \otimes_{\FF_q} \sA \xrightarrow{\id \otimes \iota} \FZ{\FF_q}{Z} \otimes_{\FF_q} F \xrightarrow{a \otimes b\, \mapsto\, ab}  H.
\]
We assume that we have an extension $\tau : H \to H$ that is $\FZ{\FF_q}{Z}$-linear. We further assume that $H^\tau = \FZ{\FF_q}{Z}$, where $H^{\tau}$ denotes the elements of $H$ fixed by $\tau$. If $\tau : H \to H$ is an automorphism, we say $H$ is \emph{$\tau$-perfect}, and if $H$ is a field, we call it an \emph{$\cA$-field}. Primarily we will be interested in $\tau$-perfect $\cA$-fields. The \emph{characteristic} of $H$ is $\ker \iota \subseteq \cA$, and as in the case of $\sA$-fields, the characteristic is generic if $\ker \iota = (0)$ and finite otherwise.

One advantage to this framework is that we can discuss objects over $\sA$-fields (where $\FZ{\FF_q}{Z} = \FF_q$) and $\bA$-fields (where $\FZ{\FF_q}{Z} = \FF_q(z)$) simultaneously. We note that $\tau$-perfect $\sA$-fields are the same as perfect $\sA$-fields.
\end{definition}

Let $H$ be an $\cA$-field. An \emph{Anderson $t$-module over $H$} (or simply $t$-module over $H$) is an $\FZ{\FF_q}{Z}$-algebra homomorphism,
\[
\cE : \cA \to \Mat_{\ell}(H[\tau])
\]
such that if
\begin{equation} \label{E:cEtdef}
\cE_t = \pd\cE_t + E_1 \tau + \cdots + E_w \tau^w, \quad E_i \in \Mat_{\ell}(H),
\end{equation}
then $\pd\cE_t - \otheta\rI_{\ell}$ is nilpotent. As in \S\ref{SS:notation}, $\cE$ defines an $\cA$-module structure on $\cE(H)=H^{\ell}$ through the operation of~$\tau$, and correspondingly we have $\Lie(\cE)(H) = H^{\ell}$ with $H[t]$-module structure induced by $\pd \cE_a$ for $a \in \cA$. If $\ell=1$, then $\cE$ is a \emph{Drinfeld module over~$H$}. If $\cE_a \in \Mat_{\ell}(F[\tau])$ for all $a \in \cA$ (equivalently simply $\cE_t \in \Mat_{\ell}(F[\tau])$), then $\cE$ is a \emph{constant} Anderson $t$-module with respect to $H/F$.

If $H$ is a $\tau$-perfect $\cA$-field, the \emph{adjoint} of $\cE$ is defined by the $\FZ{\FF_q}{Z}$-algebra map,
\[
\cE^* : \cA \to \Mat_{\ell}(H[\sigma]),
\]
where
\begin{equation} \label{E:EEstartdef}
\cE^*_t = (\cE_t)^* = (\pd\cE_t)^{\tr} + \bigl( E_1^{(-1)} \bigr)^{\tr} \sigma + \cdots + \bigl( E_{w}^{(-w)} \bigr)^{\tr} \sigma^w.
\end{equation}
Similarly for adjoints of $t$-modules over perfect $\sA$-fields, $\cE^*$ defines an $\cA$-module structure on $\cE^*(H) = H^{\ell}$.
As in \S\ref{SS:Drinfeld}, if $\cD : \cA \to \Mat_{k}(H[\tau])$ is another $t$-module over $H$, then $\eta \in \Mat_{\ell \times k}(H[\tau])$ is a morphism $\eta : \cD \to \cE$ of $t$-modules over $H$ if
$\eta \cdot \cD_a = \cE_a \cdot \eta$, for all $a \in \cA$.
Correspondingly, $\eta^* : \cE^* \to \cD^*$ is a morphism of their adjoints.

\subsection{\texorpdfstring{$t$}{t}-motives and dual \texorpdfstring{$t$}{t}-motives for Anderson \texorpdfstring{$t$}{t}-modules}
We extend the notions of $t$-motives and dual $t$-motives from \S\ref{SS:tmotives} to this setting. We start with a $\tau$-perfect $\cA$-field $H$, and let $\cE : \cA \to \Mat_{\ell}(H[\tau])$ be a $t$-module over $H$ defined as in~\eqref{E:cEtdef}.

The \emph{$t$-motive} of $\cE$ is defined to be $\cM_{\cE} = \Mat_{1 \times \ell}(H[\tau])$, which as in \S\ref{SS:tmotives} is given the structure of a left $H[t,\tau]$-module by setting $a \cdot m = m \cE_a$ for any $m \in \cM_{\cE}$, $a \in \cA$. If needed we will write $\cM_{\cE}(H) = \cM_{\cE}$ to emphasize the dependence on~$H$. We say $\cE$ and $\cM_{\cE}$ are \emph{abelian} and have \emph{rank $r$} if $\cM_{\cE}$ is free and rank~$r$ as an $H[t]$-module. 

The \emph{dual $t$-motive} of $\cE$ is defined to be $\cN_{\cE} = \Mat_{1 \times \ell}(H[\sigma])$, which as in \S\ref{SS:tmotives}, is given the structure of a left $H[t,\sigma]$-module by setting $a \cdot n = n \cE_a^*$ for any $n \in \cN_{\cE}$, $a \in \cA$. Dual $t$-motives over Tate algebras were previously studied by Demeslay~\cite{DemeslayPhD}*{\S 1.2}. If needed we will write $\cN_{\cE}(H) = \cN_{\cE}$ to emphasize the dependence on~$H$. We say $\cE$ and $\cN_{\cE}$ are \emph{$\cA$-finite} and have \emph{rank $r$} if $\cN_{\cE}$ is free and rank~$r$ as an $H[t]$-module. It is not yet known if abelian and $\cA$-finite are equivalent for $t$-modules over general $H$, though it would be interesting to investigate how the work of Maurischat~\cite{Maurischat26} applies here.

The dual $t$-motive of a traditional $t$-module over a perfect $\sA$-field can be used to recover the $t$-module (e.g., see \citelist{\cite{BP20}*{\S 1.5} \cite{HartlJuschka20}*{Prop.~2.5.8} \cite{NamoijamP24}*{\S 3.1}}), and one has a similar construction for Drinfeld modules over Tate algebras~\cite{GezmisP19}*{Lem.~4.2.2}. The connection between $t$-modules and dual $t$-motives over $\tau$-perfect $\cA$-fields is similar, and furthermore, we can relate the adjoint of a $t$-module with its $t$-motive.

Define maps $\gamma_0$, $\gamma_1 : \Mat_{1\times \ell}(H[\tau]) \to H^{\ell}$ and $\delta_0$, $\delta_1 : \Mat_{1\times \ell}(H[\sigma]) \to H^{\ell}$ by
\[
\gamma_0(m) \assign \pd m^{\tr} = c_0^{\tr}, \quad \gamma_1(m) \assign \sum_{i=0}^{i_0} \Bigl( c_{i}^{(-i)} \Bigr)^{\tr}, \quad \forall\,m = \sum_{i=0}^{i_0} c_i \tau^i \in \cM_{\cE}(H),
\]
and
\[
\delta_0(n) \assign \pd n^{\tr} = d_0^{\tr}, \quad \delta_1(n) \assign \sum_{i=0}^{i_0} \Bigl(d_{i}^{(i)} \Bigr)^{\tr}, \quad \forall\,n = \sum_{i=0}^{i_0} d_i \sigma^i \in \cN_{\cE}(H).
\]
Both $\gamma_0$, $\delta_0$ are $H$-linear, while $\gamma_1$, $\delta_1$ are $\FZ{\FF_q}{Z}$-linear (by assumption $H^{\tau} = H^{\sigma} = \FZ{\FF_q}{Z}$).

\begin{lemma}
Let $H$ be a $\tau$-perfect $\cA$-field, and let $\cE : \cA\to \Mat_{\ell}(H[\tau])$ be a $t$-module over~$H$. As defined above the following hold.
\begin{alphenumerate}
\item Each of the maps $\gamma_0$, $\gamma_1 : \cM_{\cE}(H) \to H^{\ell}$ and $\delta_0$, $\delta_1 : \cN_{\cE}(H) \to H^{\ell}$ are surjective.
\item $\ker \gamma_0 = \tau \cM_{\cE}(H)$ and $\ker \delta_0 = \sigma \cN_{\cE}(H)$.
\item $\ker \gamma_1 = (\tau - 1)\cM_{\cE}(H)$ and $\ker \delta_1 = (\sigma - 1)\cN_{\cE}(H)$.
\end{alphenumerate}
\end{lemma}

\begin{proof}
Part (a) is clear, as are the statements about $\ker \gamma_0$ and $\ker \delta_0$ (though note that we use that $\tau$ and $\sigma$ are automorphisms of $H$). The arguments to determine $\ker \gamma_1$ and $\ker \delta_1$ are similar to the situation of traditional $t$-modules, but for completeness we include the argument for $\ker \gamma_1$. Showing $\ker \gamma_1 \supseteq (\tau - 1)\cM_{\cE}(H)$ is straightforward, and to show the opposite containment, we let $m = \sum_{i=0}^{i_0} c_i \tau^i \in \ker \gamma_1$. From the left division algorithm of~\S\ref{SSS:divalg} we have $m'$, $s \in \cM_{\cE}(H)$ so that $m = (\tau-1) m' + s$.
Since $\ker \gamma_1 \supseteq (\tau - 1)\cM_{\cE}(H)$ we see that $\gamma_1(s) = 0$. As $\gamma_1$ is injective on $\Mat_{1\times \ell}(H)$, we conclude $s=0$.
\end{proof}

If $\cD : \cA \to \Mat_{k}(H[\tau])$ is a $t$-module, and $\eta \in \Mat_{\ell \times k}(H[\tau])$ represents a morphism $\eta : \cD \to \cE$, then the following diagrams of $H$-vector spaces have exact rows and commute:
\begin{equation}
\begin{tikzcd}[column sep=large]
0 \arrow{r} & \cM_{\cE}(H) \arrow{r}{\tau(\cdot)} \arrow{d}{(\cdot) \eta} & \cM_{\cE}(H) \arrow{r}{\gamma_0} \arrow{d}{(\cdot)\eta} & H^{\ell} \arrow{r} \arrow{d}{\pd\eta^{\tr}(\cdot)} & 0 \\
0 \arrow{r} &\cM_{\cD}(H) \arrow{r}{\tau(\cdot)} & \cM_{\cD}(H) \arrow{r}{\gamma_0} & H^{k} \arrow{r} & 0,
\end{tikzcd}
\end{equation}
and
\begin{equation}
\begin{tikzcd}[column sep=large]
0 \arrow{r} & \cN_{\cD}(H) \arrow{r}{\sigma(\cdot)} \arrow{d}{(\cdot)\eta^*} & \cN_{\cD}(H) \arrow{r}{\delta_0} \arrow{d}{(\cdot)\eta^*} & H^{k} \arrow{r} \arrow{d}{\pd\eta(\cdot)} & 0 \\
0 \arrow{r} &\cN_{\cE}(H) \arrow{r}{\sigma(\cdot)} & \cN_{\cE}(H) \arrow{r}{\delta_0} & H^{\ell} \arrow{r} & 0.
\end{tikzcd}
\end{equation}
The first two vertical columns are simply $\eta^{\dagger}: \cM_{\cE} \to \cM_{\cD}$ and $\eta^{\ddagger}: \cN_{\cD} \to \cN_{\cE}$. Furthermore, we have the following commutative diagrams of $\FZ{\FF_q}{Z}$-vector spaces with exact rows:
\begin{equation} \label{E:gamma1diag}
\begin{tikzcd}[column sep=large]
0 \arrow{r} & \cM_{\cE}(H) \arrow{r}{(\tau-1)(\cdot)} \arrow{d}{(\cdot) \eta} & \cM_{\cE}(H) \arrow{r}{\gamma_1} \arrow{d}{(\cdot)\eta} & \cE^*(H) \arrow{r} \arrow{d}{\eta^{*}(\cdot)} & 0 \\
0 \arrow{r} &\cM_{\cD}(H) \arrow{r}{(\tau-1)(\cdot)} & \cM_{\cD}(H) \arrow{r}{\gamma_1} & \cD^*(H) \arrow{r} & 0,
\end{tikzcd}
\end{equation}
and
\begin{equation} \label{E:delta1diag}
\begin{tikzcd}[column sep=large]
0 \arrow{r} & \cN_{\cD}(H) \arrow{r}{(\sigma-1)(\cdot)} \arrow{d}{(\cdot)\eta^*} & \cN_{\cD}(H) \arrow{r}{\delta_1} \arrow{d}{(\cdot)\eta^*} & \cD(H) \arrow{r} \arrow{d}{\eta(\cdot)} & 0 \\
0 \arrow{r} &\cN_{\cE}(H) \arrow{r}{(\sigma-1)(\cdot)} & \cN_{\cE}(H) \arrow{r}{\delta_1} & \cE(H) \arrow{r} & 0.
\end{tikzcd}
\end{equation}
We summarize these findings in the following proposition.

\begin{proposition}
Let $H$ be a $\tau$-perfect $\cA$-field, and let $\cE : \cA \to \Mat_{\ell}(H[\tau])$ be an Anderson $t$-module over $H$.
\begin{alphenumerate}
\item We have isomorphisms of $H[t]$-modules,
\[
\frac{\cM_{\cE}(H)}{\tau \cM_{\cE}(H)} \cong \Lie(\cE^{*})(H), \quad
\frac{\cN_{\cE}(H)}{\sigma \cN_{\cE}(H)} \cong \Lie(\cE)(H).
\]
\item We have isomorphisms of $\cA$-modules,
\[
\frac{\cM_{\cE}(H)}{(\tau-1) \cM_{\cE}(H)} \cong \cE^{*}(H), \quad
\frac{\cN_{\cE}(H)}{(\sigma-1) \cN_{\cE}(H)} \cong \cE(H).
\]
\item The isomorphisms in \textup{(a)} and \textup{(b)} are functorial in $\cE$.
\end{alphenumerate}
\end{proposition}

\subsubsection{Abelian and $\cA$-finite $t$-modules} \label{SSS:abcAfin}
Assume $\cE$ is abelian and $\cA$-finite and that both $\cM_{\cE}$ and $\cN_{\cE}$ have rank~$r$ as $H[t]$-modules. We fix $H[t]$-bases $\bm \in \Mat_{r \times 1}(\cM_{\cE}(H))$ and $\bn \in \Mat_{r\times 1}(\cN_{\cE}(H))$, and define $\Gamma$, $\Phi \in \Mat_r(H[t])$ so that
$\tau \bm = \Gamma \bm$ and $\sigma \bn = \Phi \bn$.
As in the case of \S\ref{SS:tmotives}, the fact that $(\pd \cE_t - \otheta \rI_\ell)^{\ell} = 0$ implies that $\cM_{\cE}/\tau \cM_{\cE}$ and $\cN_{\cE}/\sigma \cN_{\cE}$ are finite $H[t]$-modules annihilated by powers of $t-\otheta$. As such we find that $\det \Gamma = c(t-\otheta)^{\ell}$ and $\det \Phi = c'(t-\otheta)^{\ell}$ for some $c$, $c' \in H^{\times}$.
Thus $\Aord{\cM_{\cE}/\tau\cM_{\cE}}{H[t]} = \Aord{\cN_{\cE}/\sigma \cN_{\cE}}{H[t]} = (t-\otheta)^{\ell}$.

\subsection{Kernels of morphisms} \label{SS:kernels}
Let $H$ be a $\tau$-perfect $\cA$-field. For $\eta \in \Mat_{\ell \times k}(H[\tau])$, we consider $\FZ{\FF_q}{Z}$-vector spaces,
\[
(\ker \eta)(H) = \{ \bx \in H^k \mid \eta(\bx) = 0 \}, \quad
(\ker \eta^*)(H) = \{  \bx \in H^{\ell} \mid \eta^*(\bx) = 0 \}.
\]
We are particularly interested in the case when $\cD : \cA \to \Mat_k(H[\tau])$ and $\cE : \cA \to \Mat_{\ell}(H[\tau])$ are Anderson $t$-modules, and $\eta : \cD \to \cE$ is a morphism.  For $a \in \cA$, we set
\[
\cE[a](H) \assign (\ker \cE_a)(H), \quad
\cE^*[a](H) \assign (\ker \cE^*_a)(H).
\]

\begin{lemma} \label{L:kerneldim}
Let $\eta \in \Mat_{\ell}(H[\tau])$ be given by $\eta = N_j \tau^j + N_{j+1} \tau^{j+1} + \cdots + N_m \tau^m$, for $N_i \in \Mat_{\ell}(H)$
such that $\det N_m \neq 0$. Then $(\ker \eta)(H)$ and $(\ker \eta^*)(H)$ are $\FZ{\FF_q}{Z}$-vector spaces of dimension at most $(m-j)\ell$.
\end{lemma}

\begin{proof}
Since $\tau$ is an automorphism, it suffices to prove the case $j=0$.
The argument is essentially the same as \cite{APT16}*{Lem.~5.7} (cf.~\cite{vdPS}*{\S 1.2}), but we sketch the main points. Let
\[
M \assign \begin{pmatrix}
0 & \rI_{\ell} & \cdots & 0 \\
\vdots & \vdots & \ddots & \vdots \\
0 & 0 & \cdots & \rI_{\ell} \\
-N_m^{-1}N_0 & -N_m^{-1}N_1 & \cdots & -N_m^{-1}N_{m-1}
\end{pmatrix} \in \Mat_{\ell m}(H).
\]
For any $v_1, \dots, v_s \in H^{\ell m}$ such that $\tau(v_i) = v_i^{(1)} = Mv_i$ for each $i$, the following holds: if $v_1, \dots, v_s$ are $\FZ{\FF_q}{Z}$-linearly independent ($H^{\tau} = \FZ{\FF_q}{Z}$), then they are are $H$-linearly independent~(cf.\ \citelist{\cite{And86}*{pf.~of Thm.~2} \cite{vdPS}*{pf.~of Lem.~1.7}}). The map $\bx \mapsto (\bx, \bx^{(1)}, \dots, \bx^{(m-1)})^{\tr}$ is an $\FZ{\FF_q}{Z}$-linear isomorphism from $(\ker \eta)(H)$ to $\{ v \in H^{\ell m} \mid v^{(1)} = Mv \}$, which provides the desired conclusion. The argument for $(\ker \eta^*)(H)$ is similar.
\end{proof}

Let $\cE : \cA \to \Mat_\ell(H[\tau])$ be an Anderson $t$-module over $H$, and let $\eta \in \Mat_{\ell}(H[\tau])$ represent an endomorphism of $\cE$. Applying the snake lemma to \eqref{E:gamma1diag} and \eqref{E:delta1diag}, we obtain exact sequences of $\FZ{\FF_q}{Z}$-modules,
\begin{equation} \label{E:snake1}
0 \to (\ker \eta^*)(H) \to \frac{\cM_{\cE}(H)}{\cM_{\cE}(H)\eta} \xrightarrow{(\tau - 1)(\cdot)} \frac{\cM_{\cE}(H)}{\cM_{\cE}(H)\eta}
\end{equation}
and
\begin{equation} \label{E:snake2}
0 \to (\ker \eta)(H) \to \frac{\cN_{\cE}(H)}{\cN_{\cE}(H)\eta^*} \xrightarrow{(\sigma - 1)(\cdot)} \frac{\cN_{\cE}(H)}{\cN_{\cE}(H)\eta^*}.
\end{equation}

\begin{proposition} \label{P:kerdim}
Let $\cE : \cA \to \Mat_{\ell}(H[\tau])$ be an Anderson $t$-module over $H$, and let $\eta : \cE \to \cE$ be an endomorphism given by $\eta \in \Mat_{\ell}(H[\tau])$ with $\pd \eta \in \GL_{\ell}(H)$. Let $D = \diag(d_1, \dots, d_\ell) \in \Mat_{\ell}(H[\tau])$ be a diagonal matrix associated to $\eta$ as in Proposition~\ref{P:AndUBV}. Then none of $d_1, \dots, d_{\ell}$ is $0$, and
\[
\dim_{\FZ{\FF_q}{Z}} (\ker \eta)(H) \leqslant \sum_{i=1}^{\ell} \deg_{\tau} d_i, \quad
\dim_{\FZ{\FF_q}{Z}} (\ker \eta^*)(H) \leqslant \sum_{i=1}^{\ell} \deg_{\tau} d_i.
\]
\end{proposition}

\begin{proof}
Let $U$, $V \in \GL_{\ell}(H[\tau])$ be chosen so that $D = U\eta V$. Then $\pd D = \pd U\cdot \pd \eta \cdot \pd V$. Since $\pd \eta$ is invertible, so must be $\pd D$, which implies that none of $d_1, \dots, d_{\ell}$ is zero.

By \eqref{E:snake1}, we have the isomorphism of $\FZ{\FF_q}{Z}$-vector spaces,
\begin{equation} \label{E:keretastar}
(\ker \eta^*)(H) \cong J \assign \ker \biggl( \frac{\cM_{\cE}(H)}{\cM_{\cE}(H)\eta} \xrightarrow{(\tau-1)(\cdot)} \frac{\cM_{\cE}(H)}{\cM_{\cE}(H)\eta} \biggr).
\end{equation}
Suppose that $m_1, \dots, m_k \in \cM_{\cE}(H)$ represent $\FZ{\FF_q}{Z}$-linearly independent classes in $J$. We claim that $m_1, \dots, m_k$ represent $H$-linearly independent classes in $\cM_{\cE}(H)/\cM_{\cE}(H)\eta$. Since $m_1, \dots, m_k \in J$, there are $\beta_1, \dots, \beta_k \in \cM_{\cE}(H)$ so that $(\tau - 1)m_i = \beta_i \eta$ for each~$i$. Suppose that $m_1, \dots, m_k$ are $H$-linearly dependent modulo $\cM_{E}(H)\eta$, so after reordering terms, we can choose $j \leqslant k$ minimal with $c_2, \dots, c_{j} \in H$ and $\gamma \in \cM_{\cE}(H)$ so that
\begin{equation} \label{E:tmp1}
m_1 + c_2 m_2 + \cdots + c_{j} m_{j} = \gamma \eta.
\end{equation}
Substituting for each $m_i$ we find,
\[
\tau m_1 + c_2 \tau m_2 + \cdots + c_j \tau m_j = \bigl( \gamma + \beta_1 + c_2 \beta_2 + \cdots + c_j \beta_j\bigr) \eta,
\]
and since $\pd \eta \in \GL_{\ell}(H)$, it follows that $\gamma + \beta_1 + c_2 \beta_2 + \cdots c_j \beta_j \in \tau\cM_{\cE}(H)$. Multiplying \eqref{E:tmp1} by $\tau$ and subtracting, we have
\[
\bigl(c_2^{(1)} - c_2 \bigr) \tau m_2 + \cdots + \bigl(c_j^{(1)} - c_j \bigr) \tau m_j =
\bigl( \gamma^{(1)}\tau - \gamma - \beta_1 - c_2\beta_2 - \cdots - c_j \beta_j \bigr) \eta.
\]
By the previous sentence, the left-hand factor on the right is in $\tau \cM_{\EE}(H)$, so we can cancel~$\tau$ from both sides and obtain,
\[
\bigl( c_2 - c_2^{(-1)} \bigr) m_2 + \cdots + \bigl( c_j - c_j^{(-1)} \bigr) m_j = \bigl( \gamma - \tau^{-1}(\gamma + \beta_1 + c_2\beta_2 + \cdots + c_j\beta_j) \bigr) \eta.
\]
The minimality of $j$ implies $c_2= c_2^{(-1)}, \ldots, c_j = c_j^{(-1)}$, whence $c_2, \dots, c_j \in H^{\tau} = \FZ{\FF_q}{Z}$, and thus \eqref{E:tmp1} contradicts the $\FZ{\FF_q}{Z}$-linear independence of the classes of $m_1, \dots, m_k$. The desired dimension bound for $(\ker \eta^*)(H)$ then follows from Proposition~\ref{P:AndUBV}(b).

To find the same bound for the $(\ker \eta)(H)$ case, we observe that $D^* = V^*\cdot \eta^* \cdot U^*$, and so $\deg_{\sigma} d_i^* = \deg_{\tau} d_i$ for each $i$. The rest follows exactly as in the $(\ker \eta^*)(H)$ case.
\end{proof}

\begin{definition} \label{D:full}
Let $\eta : \cE \to \cE$ be given as in Proposition~\ref{P:kerdim}. We say that $(\ker \eta)(H)$ has \emph{full dimension} if its $\FZ{\FF_q}{Z}$-dimension is equal to $\sum_{i=1}^{\ell} \deg_{\tau} d_i$. Under similar conditions we say $(\ker \eta^*)(H)$ has full dimension.
\end{definition}

\begin{corollary} \label{C:Atorsion}
Let $\cE : \cA \to \Mat_{\ell}(H[\tau])$ be an abelian and $\cA$-finite Anderson $t$-module over~$H$ with $\rank_{H[t]} \cM_{\cE}(H) = \rank_{H[t]} \cN_{\cE}(H) = r$. For $\nu \in \cA$ not divisible by the characteristic of $H$, the following hold.
\begin{alphenumerate}
\item $\cE[\nu](H)$ has full dimension if and only if $\cE[\nu](H) \cong (\cA/\nu\cA)^{r}$ as $\cA$-modules.
\item $\cE^*[\nu](H)$ has full dimension if and only if $\cE^*[\nu](H) \cong (\cA/\nu\cA)^{r}$ as $\cA$-modules.
\end{alphenumerate}
\end{corollary}

\begin{proof}
Combining \eqref{E:keretastar} with the definition of the $\cA$-module action on~$\cM_{\cE}$, we have
\[
\cE^*[\nu](H) \cong J \assign \ker \biggl( \frac{\cM_{\cE}(H)}{\nu\cM_{\cE}(H)} \xrightarrow{(\tau-1)(\cdot)} \frac{\cM_{\cE}(H)}{\nu\cM_{\cE}(H)} \biggr).
\]
As in the proof of Proposition~\ref{P:kerdim}, any $\FZ{\FF_q}{Z}$-basis of $J$ is an $H$-linearly independent subset of $\cM_{\cE}(H)/\nu\cM_{\cE}(H) = \cM_{\cE}(H)/\cM_{E}(H)\cE_{\nu}$, and is also an $H$-basis since $\cE^*[\nu]$ has full dimension. Thus the natural map
\[
H \otimes_{\FZ{\FF_q}{Z}} J \to \frac{\cM_{\cE}(H)}{\nu\cM_{\cE}(H)}
\]
is an isomorphism of $H$-vector spaces and also of $H[t]$-modules. Since $\cM_{\cE}(H)/\nu\cM_{\cE}(H) \cong (H[t]/\nu H[t])^r$, it follows that $J \cong (\cA/\nu\cA)^r$. The argument for $\cE[\nu](H)$ is similar.
\end{proof}

\subsection{Anderson \texorpdfstring{$t$}{t}-modules over Tate algebras} \label{SS:tmodTatealg}
We now consider $\LLhat_z$ as a $\tau$-perfect $\bA$-field (taking $\FZ{\FF_q}{Z} = \FF_q(z)$), where $\iota(t)=\theta$, and as noted in~\S\ref{SS:notation}, $\tau : \LLhat_z \to \LLhat_z$ is an $\FF_q(z)$-linear automorphism. We fix an Anderson $t$-module $\bE : \bA \to \Mat_{\ell}(\LLhat_z[\tau])$. Demeslay~\citelist{\cite{DemeslayPhD}*{\S 1.1.1} \cite{Demeslay22}*{Prop.~2.5} } observed that it has a unique exponential series
\[
\Exp_{\bE} = \sum_{i=0}^{\infty} B_i \tau^i, \quad B_0 = \rI_{\ell},\ B_i \in \Mat_{\ell}(\LLhat_z),
\]
such that $\Exp_{\bE} {}\cdot \pd \bE_a = \bE_a \cdot \Exp_{\bE}$ for $a \in \bA$.
Using the arguments of \cite{And86}*{Prop.~2.1.4}, Demeslay proved that $\lim_{i \to \infty} \deg(B_i)/q^i = -\infty$, and so as in \S\ref{SS:ExpLog} the exponential function $\Exp_{\bE} : \LLhat_z^{\ell} \to \LLhat_z^{\ell}$ is well-defined on all of $\LLhat_z^{\ell}$. If $\Exp_{\bE} : \LLhat_z^{\ell} \to \LLhat_z^{\ell}$ is surjective, then $\bE$ is said to be uniformizable.
Its period lattice is $\Lambda_{\bE} \assign \ker \Exp_{\bE} \subseteq \LLhat_z^{\ell}$.

The logarithm series $\Log_{\bE} \in \power{\Mat_{\ell}(\LLhat_t)}{\tau}$ is defined as the inverse of $\Exp_{\bE}$,
\[
\Log_{\bE} = \sum_{i=0}^{\infty} C_i \tau^i, \quad C_0 = \rI_{\ell},\ C_i \in \Mat_{\ell}(\LLhat_z),
\]
and satisfies $\pd \bE_{a} \cdot \Log_{\bE} = \Log_{\bE}{} \cdot \bE_a$ for $a \in \bA$.
As in the case of constant $t$-modules, the logarithm function $\Log_{\bE}(\bz)$ may converge only on an open polydisc in $\LLhat_z^{\ell}$.

If $\bE$ is defined over an $\bA$-field $M$ with $\KK \subseteq M \subseteq \LLhat_z$, then $\{ B_i \}$, $\{ C_i \} \subseteq\Mat_{\ell}(M)$, even if $\tau : M \to M$ is not an automorphism. This follows from the arguments of \citelist{\cite{And86}*{Prop.~2.1.4, Lem.~2.1.6} \cite{Goss}*{Lem.~5.9.3}} (see also \cite{NamoijamP24}*{Rem.~2.6}).

\subsubsection{Discrete subspaces of $\LLhat_z$-vector spaces}
We adopt the following description of discreteness in $\LLhat_z$-vector spaces due to Demeslay~\citelist{\cite{APT16}*{App.~A} \cite{DemeslayPhD} \cite{Demeslay22}}. Suppose that $W$ is a finite dimensional $\KK_{\infty}$-vector space with basis $e_1, \dots, e_k$, and suppose that $\Lambda \subseteq W$ is an $\FF_q(z)$-subspace. Setting $\OO_{\infty} \assign \power{\FF_q(z)}{\theta^{-1}}$ and $\MM_{\infty} \assign \theta^{-1}
\power{\FF_q(z)}{\theta^{-1}}$
to be the valuation ring and maximal ideal of $\KK_{\infty}$, the following holds.

\begin{lemma} \label{L:discrete}
With notation as above, the following are equivalent.
\begin{alphenumerate}
\item There exists $n \geqslant 1$ such that $\Lambda \cap \bigl( \bigoplus_{i=1}^k \MM_{\infty}^n \cdot e_i \bigr) = \{ 0 \}$.
\item $\Lambda \cap \bigl( \bigoplus_{i=1}^k \MM_{\infty} \cdot e_i \bigr)$ is finite dimensional as an $\FF_q(z)$-vector space.
\end{alphenumerate}
\end{lemma}

\begin{proof}
If (b) holds, we note that if $\lambda_1, \dots, \lambda_m$ is an $\FF_q(z)$-basis of $\Lambda \cap \bigl( \bigoplus_{i=1}^k \MM_{\infty} \cdot e_i \bigr)$, then their coordinates in terms of the $\KK_{\infty}$-basis $\{e_i\}$ have bounded $\infty$-adic valuation, which implies (a).
Now suppose that (a) holds and that $\{\lambda_j \}$ is an infinite $\FF_q(z)$-linearly independent subset of $\Lambda \cap \bigl( \bigoplus_{i=1}^k \MM_{\infty} \cdot e_i \bigr)$. Since $\MM_{\infty}/\MM_{\infty}^n$ is a finite dimensional $\FF_q(z)$-vector space with basis $\{\theta^{-1}, \dots, \theta^{-n+1}\}$, the infinitude of $\{\lambda_j\}$ implies that some nontrivial $\FF_q(z)$-linear combination of $\{\lambda_j \}$ is in $\bigoplus_{i=1}^k \MM_{\infty}^n \cdot e_i$. The $\FF_q(z)$-linear independence of $\{\lambda_j\}$ then implies that this linear combination is nonzero, which contradicts~(a).
\end{proof}

\begin{definition}
An $\FF_q(z)$-subspace $\Lambda \subseteq W$ is \emph{discrete} if the equivalent conditions of Lemma~\ref{L:discrete} hold. We note that this is independent of the choice of $\KK_{\infty}$-basis of $W$. Furthermore, if $V$ is instead a finite dimensional $\LLhat_z$-vector space and $\Lambda \subseteq V$ is an $\FF_q(z)$-subspace, then we say $\Lambda$ is \emph{discrete in $V$} if it is discrete in $W=\Span_{\KK_{\infty}}(\Lambda)$.
\end{definition}

\begin{lemma}[Demeslay~\citelist{\cite{DemeslayPhD}*{Lem.~2.1.2} \cite{Demeslay22}*{Lem.~2.2}}] \label{L:lattice}
Let $\Lambda$ be an $\AAA$-submodule of a $\KK_{\infty}$-vector space $W$ of dimension $k \geqslant 1$. The following are equivalent.
\begin{alphenumerate}
\item $\Lambda$ is a free $\AAA$-module of rank~$k$, and $\Span_{\KK_{\infty}}(\Lambda) = W$.
\item $\Lambda$ is a discrete $\FF_q(z)$-subspace of $W$, and every open subspace of the $\FF_q(z)$-vector space $W/\Lambda$ has finite codimension.
\end{alphenumerate}
\end{lemma}

\begin{definition} \label{D:lattice}
An $\AAA$-submodule $\Lambda$ of a $k$-dimensional $\KK_{\infty}$-vector space $W$ is called an \emph{$\AAA$-lattice} if it satisfies the equivalent conditions of Lemma~\ref{L:lattice}. More generally, if $\Lambda$ is an $\AAA$-submodule of a finite dimensional $\LLhat_z$-vector space $V$, then $\Lambda$ is an \emph{$\AAA$-lattice in $V$} if it is an $\AAA$-lattice in $W=\Span_{\KK_{\infty}}(\Lambda)$.
\end{definition}

\subsubsection{Exponential functions and period lattices} \label{SSS:Expfnpers}
Let $\bE : \bA \to \Mat_{\ell}(\LLhat_z[\tau])$ be an Anderson $t$-module. Its exponential function $\Exp_{\bE} : \LLhat_z^{\ell} \to \LLhat_z^{\ell}$ is locally an isometry~\cite{Demeslay22}*{\S 2.2} (cf.\ \citelist{\cite{GezmisP19}*{Lem.~3.3.2} \cite{HartlJuschka20}*{Lem.~2.5.4}}). Therefore, we can find $\epsilon > 0$ such that on the open polydisc, $\rD = \{ \bz \in \LLhat_z^{\ell} \mid \dnorm{\bz} < \epsilon \}$, we have
\begin{equation} \label{E:Ddef}
\bz \in \rD \quad \Rightarrow \quad \dnorm{\Exp_{\bE}(\bz)} = \dnorm{\bz}.
\end{equation}
Moreover, $\Log_{\bE}(\bz)$ converges on $\rD$. We thus have following lemma.

\begin{lemma} \label{L:latticeisdiscrete}
For an Anderson $t$-module $\bE: \bA \to \Mat_{\ell}(\LLhat_z[\tau])$, let $\Lambda_{\bE} = \ker \Exp_{\bE}$. The following hold.
\begin{alphenumerate}
\item $\Lambda_{\bE}$ is a discrete $\FF_q(z)$-subspace of $\LLhat_z^{\ell}$.
\item If $\bE$ is a Drinfeld module over $\LLhat_z$, then $\Lambda_{\bE} \subseteq \LLhat_z$ is an $\AAA$-lattice.
\end{alphenumerate}
\end{lemma}

\begin{proof}
Let $W = \Span_{\KK_{\infty}}(\Lambda_{\bE})$. Choosing $\rD$ as in \eqref{E:Ddef}, it must be that $\Lambda_{\bE} \cap \rD = \{0\}$. It follows that $\Lambda_{\bE}$ is discrete by Lemma~\ref{L:discrete}. If $\bE$ is a Drinfeld module, then $\Lambda_{\bE} \subseteq W$ is a discrete $\AAA$-submodule. Letting $\lambda_1, \dots, \lambda_k \in \Lambda_{\bE}$ be a $\KK_{\infty}$-basis of $W$ and setting $\Lambda' = \Span_{\AAA}(\lambda_1, \dots, \lambda_k)$, we see that $W/\Lambda' \cong \bigoplus_{i=1}^k \MM_{\infty} \cdot \lambda_i$. Since $\Lambda_{\bE} \subseteq W$ is discrete, Lemma~\ref{L:discrete}(b) implies that $\Lambda_{\bE} \cap  \bigl(\bigoplus_{i=1}^k \MM_{\infty} \cdot \lambda_i \bigr)$ is a finite dimensional $\FF_q(z)$-vector space, and so $\Lambda_{\bE}/\Lambda'$ itself has finite dimension over $\FF_q(z)$. Therefore, $\Lambda_{\bE}$ is finitely generated and free as an $\AAA$-module with rank~$k$. Thus $\Lambda_{\bE}$ is an $\AAA$-lattice by Lemma~\ref{L:lattice}.
\end{proof}

We should note that Lemma~\ref{L:latticeisdiscrete}(b) does not necessarily imply that when $\bE$ is a Drinfeld module of rank~$r$ over $\LLhat_z$, that its period lattice has rank~$r$ as an $\AAA$-module. Nevertheless, for constant Drinfeld modules over $\LLhat_z$ this is indeed the case.

\begin{theorem}[Gezmi\c{s}-Papanikolas \cite{GezmisP19}*{Thm.~7.1.1, Prop.~8.2.1}] \label{T:constant}
Let  $\phi : \AAA \to \C[\tau]$ be a Drinfeld module  of rank~$r$. As a Drinfeld module over $\LLhat_z$, the following hold.
\begin{alphenumerate}
\item $\Exp_{\phi} : \LLhat_z \to \LLhat_z$ is surjective.
\item $\Lambda_{\phi} = \ker \Exp_{\phi} \subseteq \LLhat_z$ is free of rank $r$ over $\AAA$.
\item If $\pi_1, \dots, \pi_r \in \C$ form an $A$-basis of $\ker (\Exp_{\phi}|_{\C})$, then
$\Lambda_{\phi} = \AAA \pi_1 + \cdots + \AAA \pi_r$.
\end{alphenumerate}
\end{theorem}

\begin{remark}
In \cite{GezmisP19}, the Drinfeld modules under consideration are Drinfeld modules over $\TT_z$, or more generally over Tate algebras of several variables. However, the arguments needed for Theorem~\ref{T:constant}, originally due to Anderson for constant $t$-modules, transfer equally well to Drinfeld modules over~$\LLhat_z$ with little modification. The primary required information is that $\LLhat_z$ is complete, that $\Exp_{\phi}: \LLhat_z \to \LLhat_z$ is locally an isometry, and that $\Lambda_{\phi} \subseteq \LLhat_z$ is discrete. As there is little change from \cite{GezmisP19} we do not include the details. See also \citelist{\cite{HartlJuschka20}*{\S 2.5} \cite{NamoijamP24}*{\S 3.4}} for more information concerning the case of $t$-modules.
\end{remark}

\subsection{\texorpdfstring{$t$}{t}-modules in finite characteristic} \label{SS:tmodsfinite}
We now consider a more specific situation. Fix $f \in A_+$ irreducible of degree~$d$, and let $\FF_f \assign A/fA$. We make $\FF_f$ into an $\sA$-field by by picking a root $\otheta \in \oFF_q$ of $f(t)=0$ and defining $\iota : \sA \to \FF_f$ by $t \mapsto \otheta$. As in \S\ref{SS:cAtmodules}, we let $Z = \{ z_1, \dots, z_n\}$ for $n \geqslant 0$, thus defining fields, $\FZ{\FF_q}{Z} \subseteq \FZ{\FF_f}{Z} \subseteq \FZ{\oFF_f}{Z}$. We note that
\begin{equation}
\Gal(\FZ{\FF_f}{Z}/\FZ{\FF_q}{Z}) \cong \Gal(\FF_f/\FF_q),
\end{equation}
and that
\begin{equation} \label{E:GaloFF}
\Gal(\oFF_f(Z)/\FF_f(Z)) \cong \Gal(\oFF_f/\FF_f), \quad
\Aut(\laurent{\oFF_f}{Z}/\laurent{\FF_f}{Z}) \hookleftarrow \Gal(\oFF_f/\FF_f).
\end{equation}
We identify these Galois groups going forward, and in particular, when we refer to the ``Galois action'' or ``Galois equivariance,'' it will be through the groups $\Gal(\FF_f/\FF_q)$ or $\Gal(\oFF_f/\FF_f)$ identified as either $\Gal(\FZ{\FF_f}{Z}/\FZ{\FF_q}{Z})$ or (a subgroup of) $\Aut(\FZ{\oFF_f}{Z}/\FZ{\FF_f}{Z})$.
As in \S\ref{SS:cAtmodules}, the inclusion $\FF_f \subseteq \FZ{\FF_f}{Z}$ extends $\iota$,
\[
\iota : \cA \to \FZ{\FF_f}{Z} \subseteq \FZ{\oFF_f}{Z},
\]
to make $\FZ{\FF_f}{Z}$ and $\FZ{\oFF_f}{Z}$ into $\tau$-perfect $\cA$-fields. Moreover, $\tau$ is the generator of $\Gal(\FZ{\FF_f}{Z}/\FZ{\FF_q}{Z})$. Also $\tau^d$ generates $\Gal(\oFF_f/\FF_f)$, and in the two cases in~\eqref{E:GaloFF} we have $\smash{\oFF_f(Z)}^{\tau^d} = \FF_f(Z)$ and $\smash{\laurent{\oFF_f}{Z}}^{\tau^d} = \laurent{\FF_f}{Z}$, so in both cases $\smash{\FZ{\oFF_f}{Z}}^{\tau^d} = \FZ{\FF_f}{Z}$.

We now fix an abelian and $\cA$-finite $t$-module $\cE : \cA \to \Mat_{\ell}(\FZ{\FF_f}{Z}[\tau])$, such that the ranks of its $t$-motive and dual $t$-motive agree, and we let $\cE^* : \cA \to \Mat_{\ell}(\FZ{\FF_f}{Z}[\sigma])$ be its adjoint. Because $\cE$ is defined over $\FZ{\FF_f}{Z}$, we see that $\tau^d\cdot \rI_{\ell} : \cE \to \cE$ and $\sigma^d\cdot \rI_{\ell}: \cE^* \to \cE^*$ are endomorphisms. Fix also $\FZ{\FF_f}{Z}[t]$-bases,
\[
\bm = (m_1, \dots, m_r)^{\tr} \in \Mat_{r\times 1}(\cM_{\cE}(\FZ{\FF_f}{Z})), \quad
\bn = (n_1, \dots, n_r)^{\tr} \in \Mat_{r\times 1}(\cN_{\cE}(\FZ{\FF_f}{Z})),
\]
as in~\S\ref{SSS:abcAfin}, as well as $\Gamma$, $\Phi \in \Mat_r(\FZ{\FF_f}{Z}[t])$, so that
$\tau \bm = \Gamma \bm$, $\sigma \bn = \Phi \bn$.
We note that for each $i$, we have $m_i \in \Mat_{1\times \ell}(\FZ{\FF_f}{Z}[\tau])$ and $n_i \in \Mat_{1 \times \ell}(\FZ{\FF_f}{Z}[\sigma])$. Consider the $\FZ{\FF_q}{Z}[t]$-linear isomorphisms
\[
\bi : \Mat_{1\times r}(\FZ{\FF_f}{Z}[t]) \to \cM_{\cE}(\FZ{\FF_f}{Z}), \quad
\bj : \Mat_{1\times r}(\FZ{\FF_f}{Z}[t]) \to \cN_{\cE}(\FZ{\FF_f}{Z}),
\]
given by
\begin{align*}
\bi(\bu) &= \bi(u_1, \dots, u_r) = \bu \cdot \bm = u_1 m_1 + \cdots + u_r m_r, \\
\bj(\bu) &= \bj(u_1, \dots, u_r) = \bu \cdot \bn = u_1 n_1 + \cdots + u_r n_r.
\end{align*}
The maps $\bi$ and $\bj$ are Anderson $t$-frames in the sense of \citelist{\cite{ChangGreenMishiba21}*{\S 2.3.2} \cite{GezmisP19}*{\S 4.4} \cite{NamoijamP24}*{\S 3.2}}. Finally, we let
\begin{equation} \label{E:GHdef}
\rG \assign \Gamma^{(d-1)} \cdots  \Gamma^{(1)} \Gamma, \quad \rH \assign \Phi^{(-d+1)} \cdots \Phi^{(-1)} \Phi,
\end{equation}
and then we represent $\tau^d$ and $\sigma^d$ in terms of our $\FZ{\FF_f}{Z}[t]$-bases.

\begin{lemma} \label{L:GH}
The following diagrams of $\FZ{\FF_f}{Z}[t]$-modules commute:
\begin{displaymath}
\begin{tikzcd}
\Mat_{1\times r}(\FZ{\oFF_f}{Z}[t]) \arrow{r}{\bi} \arrow{d}{(\cdot)\rG} & \cM_{\cE}(\FZ{\oFF_f}{Z}) \arrow{d}{(\cdot)\tau^d} & \Mat_{1\times r}(\FZ{\oFF_f}{Z}[t]) \arrow{r}{\bj} \arrow{d}{(\cdot)\rH} & \cN_{\cE}(\FZ{\oFF_f}{Z}) \arrow{d}{(\cdot)\sigma^d} \\
\Mat_{1\times r}(\FZ{\oFF_f}{Z}[t]) \arrow{r}{\bi} & \cM_{\cE}(\FZ{\oFF_f}{Z}) & \Mat_{1\times r}(\FZ{\oFF_f}{Z}[t]) \arrow{r}{\bj} & \cN_{\cE}(\FZ{\oFF_f}{Z}).
\end{tikzcd}
\end{displaymath}
\end{lemma}

\begin{proof}
We first note that since the entries of $\bm$ are in $\Mat_{1\times \ell}(\FZ{\FF_f}{Z}[\tau])$, we have $\tau^d \bm = \bm \tau^d$. Furthermore, we have $\tau^d\bm = \Gamma^{(d-1)} \cdots \Gamma^{(1)}\Gamma \bm = \rG \bm$, and so $\rG \bm = \bm \tau^d$. For $\bu \in \Mat_{1\times r}(\FZ{\FF_f}{Z}[t])$, we thus have $\bi(\bu)\tau^d = \bi(\bu\rG)$, and the first diagram commutes. The second diagram is the same.
\end{proof}

Taking $\eta = (1 - \tau^d)\cdot \rI_{\ell}$ as an endomorphism of $\cE(\FZ{\oFF_f}{Z})$, \eqref{E:snake1} and \eqref{E:snake2} together with Lemma~\ref{L:GH} imply the following result. We note that the kernels of $(1-\tau^d)\rI_{\ell}$ and $(1-\sigma^d)\rI_{\ell}$ have full dimension in the sense of Definition~\ref{D:full}.

\begin{lemma} \label{L:Ffpoints}
Let $\cE : \cA \to \Mat_{\ell}(\FZ{\FF_f}{Z}[\tau])$ be an abelian and $\cA$-finite $t$-module of rank~$r$. With $\rG$, $\rH \in \GL_r(\FZ{\FF_f}{Z}[t])$ as above, we have $\cA$-module isomorphisms,
\begin{alphenumerate}
\item $\displaystyle \cE^*(\FZ{\FF_f}{Z}) \cong  \ker \biggl( \frac{\Mat_{1\times r}(\FZ{\oFF_f}{Z}[t])}{\Mat_{1 \times r}(\FZ{\oFF_f}{Z}[t]) (\rI - \rG)} \xrightarrow{(1 - \tau)(\cdot)} \frac{\Mat_{1\times r}(\FZ{\oFF_f}{Z}[t])}{\Mat_{1 \times r}(\FZ{\oFF_f}{Z}[t]) (\rI - \rG)} \biggr)$,

\smallskip
\item $\displaystyle \cE(\FZ{\FF_f}{Z}) \cong  \ker \biggl( \frac{\Mat_{1\times r}(\FZ{\oFF_f}{Z}[t])}{\Mat_{1 \times r}(\FZ{\oFF_f}{Z}[t]) (\rI - \rH)} \xrightarrow{(1 - \sigma)(\cdot)} \frac{\Mat_{1\times r}(\FZ{\oFF_f}{Z}[t])}{\Mat_{1 \times r}(\FZ{\oFF_f}{Z}[t]) (\rI - \rH)} \biggr)$,
\end{alphenumerate}
where $\rI = \rI_r$.
\end{lemma}

\subsection{Poonen pairings} \label{SS:pairings}
We now investigate generalizations of pairings defined by Poonen~\cite{Poonen96} for Drinfeld modules over finite fields to $t$-modules over $\FZ{\FF_f}{Z}$. For additional exposition in the case of Drinfeld modules see~\citelist{\cite{Goss}*{\S 4.14} \cite{Thakur}*{\S 2.10}}.

We fix a $t$-module $\cE : \cA \to \Mat_{\ell}(\FZ{\FF_f}{Z}[\tau])$, determined by
\begin{equation}
\cE_t = \pd\cE_t + E_1 \tau + \cdots + E_w \tau^w, \quad E_i \in \Mat_{\ell}(\FZ{\FF_f}{Z}).
\end{equation}
We also fix an endomorphism $\eta : \cE \to \cE$ over $\FZ{\FF_f}{Z}$ with
\begin{equation} \label{E:etadef}
\eta  = N_0 + N_{1} \tau^{1} + \cdots + N_m \tau^m, \quad N_i \in \Mat_{\ell}(\FZ{\FF_f}{Z}),\ 
\det N_0 \neq 0.
\end{equation}
We let
\[
\ker \eta \assign \{ \bx \in \cE(\FZ{\oFF_f}{Z}) \mid \eta(\bx) = 0 \}, \quad
\ker \eta^* \assign \{ \bx \in \cE^*(\FZ{\oFF_f}{Z}) \mid \eta^*(\bx) = 0 \},
\]
and we note from Proposition~\ref{P:kerdim}, because $\det N_0 \neq 0$, that these are finite dimensional $\FZ{\FF_q}{Z}$-vector spaces. We say that $\ker \eta$ and $\ker \eta^*$ have full dimension if they have full dimension in the sense of Definition~\ref{D:full} while taking $H = \FZ{\oFF_f}{Z}$. Furthermore, if $\eta = \cE_a$ for $a \in \cA$, we write $\cE[a] \assign (\ker \cE_a)(\FZ{\oFF_f}{Z})$ and $\cE^*[a] \assign (\ker \cE_a^*)(\FZ{\oFF_f}{Z})$.

We now follow the exposition in \citelist{\cite{Goss}*{\S 4.14} \cite{Poonen96}*{\S 9}} to construct a pairing between $\ker \eta$ and $\ker \eta^*$. Letting $\bx \in \ker \eta$ and $\by \in \ker \eta^*$, we see from \eqref{E:gamma1diag} and \eqref{E:delta1diag} that
\begin{align*}
\eta(\bx) = 0 \quad &\Rightarrow \quad \delta_1 \bigl(\bx^{\tr} \eta^* \bigr) = \eta\bigl( \delta_1(\bx^{\tr}) \bigr) = \eta(\bx) = 0,\\
\eta^*(\by) = 0 \quad &\Rightarrow \quad \gamma_1 \bigl(\by^{\tr} \eta \bigr) = \eta^*\bigl( \gamma_1(\by^{\tr}) \bigr) = \eta^*(\by) = 0.
\end{align*}
Therefore, \eqref{E:gamma1diag} and \eqref{E:delta1diag} also imply that there are unique $g_{\bx} \in \cN_{\cE}(\FZ{\oFF_f}{Z}) = \Mat_{1\times \ell}(\FZ{\oFF_f}{Z}[\sigma])$ and $h_{\by} \in \cM_{\cE}(\FZ{\oFF_f}{Z}) = \Mat_{1 \times \ell}(\FZ{\oFF_f}{Z}[\tau])$ so that
\begin{equation} \label{E:gxhy}
\bx^{\tr}\eta^* = (1 - \sigma)g_{\bx}, \quad \by^{\tr} \eta = (1 - \tau) h_{\by}.
\end{equation}
We then define
\begin{equation} \label{E:Poonendef}
\apairing{\bx}{\by}_{\eta} \assign g_{\bx}(\by) = (g_{\bx} \cdot \by)(1),
\end{equation}
where $g_{\bx}(\by)$ is the $\sigma$-analogue of the evaluation in \eqref{E:gentaueval} and $(g_{\bx} \cdot \by)(1)$ is the product of twisted polynomials in $\sigma$ evaluated at~$1$.
This defines a pairing on ${\ker \eta} \times {\ker \eta^*}$ with the following properties.

\begin{proposition}[cf.\ Poonen~\citelist{\cite{Goss}*{pp.~126--129} \cite{Poonen96}*{\S 8}}] \label{P:apairing}
The pairing defined above satisfies
\[
\apairing{\cdot}{\cdot}_{\eta} : {\ker \eta} \times {\ker \eta^*} \to \FZ{\FF_q}{Z},
\]
and it is $\FZ{\FF_q}{Z}$-bilinear. Furthermore, the following hold.
\begin{alphenumerate}
\item For $\bx \in \ker \eta$ and $\by \in \ker \eta^*$, we also have $\apairing{\bx}{\by}_{\eta} = h_{\by}(\bx)$.
\item $\apairing{\cdot}{\cdot}_{\eta}$ is $\Gal(\oFF_f/\FF_f)$-equivariant.
\item If $\ker \eta$ and $\ker \eta^*$ have full dimension, then $\apairing{\cdot}{\cdot}_{\eta}$ is non-degenerate.
\end{alphenumerate}
\end{proposition}

\begin{proof}
Taking $\bx^{\tr} \eta^* = (1-\sigma)g_{\bx}$ from~\eqref{E:gxhy} and evaluating at $\by$, we have
\[
\bigl( \bx^{\tr} \eta^* \bigr)(\by) = \bx^{\tr} (\eta^*(\by)) = 0 = \big( (1-\sigma) g_{\bx} \bigr)(\by) = (1-\sigma)(g_{\bx}(\by)),
\]
where we have used the associativity of~\eqref{E:gensigmaeval}. But $(1-\sigma)(g_{\bx}(\by)) = g_{\bx}(\by) - \sigma(g_{\bx}(\by))$, and so $g_{\bx}(\by) \in \FZ{\FF_q}{Z}$. It is straightforward to check that $\apairing{\cdot}{\cdot}_{\eta}$ is $\FZ{\FF_q}{Z}$-bilinear.

We next verify~(a) (cf.\ \citelist{\cite{Goss}*{Prop.~4.14.10} \cite{Poonen96}*{Prop.~14}}). By \eqref{E:Poonendef}, we need to show
\[
(g_{\bx}\cdot \by)(1) = (h_{\by}\cdot \bx)(1).
\]
From \eqref{E:gxhy} we have the equality of polynomials in $\FZ{\oFF_f}{Z}[\tau]$,
\begin{equation} \label{E:oneminustau}
\by^{\tr} \cdot g_{\bx}^* (1-\tau) = \by^{\tr} \cdot \eta \cdot \bx = (1-\tau) h_{\by} \cdot \bx.
\end{equation}
First, since $(\by^{\tr} g_{\bx}^* (1-\tau))(1) = 0$, it follows that $(1-\tau)(h_{\by}(\bx)(1)) = 0$ and so $(h_{\by}\cdot \bx)(1) \in \FZ{\FF_q}{Z}$. Second, this shows that $\by^{\tr}\cdot g_{\bx}^*$ and $h_{\by} \cdot \bx$ have the same degree in~$\tau$. If
\[
\by^{\tr} \cdot g_{\bx}^* = \sum_{i=0}^n b_i \tau^i, \quad
h_{\by}\cdot \bx = \sum_{i=0}^n c_i \tau^i, \quad b_i,\ c_i \in \FZ{\oFF_f}{Z},
\]
then we need to show
\[
(g_{\bx}\cdot \by)(1) = b_0 + b_1^{(-1)} + \dots + b_n^{(-n)} \mayeq c_0 + c_1 + \dots + c_n = (h_{\by}\cdot \bx)(1).
\]
Multiplying \eqref{E:oneminustau} on both sides by $(1-\tau)^{-1}$ in $\power{\FZ{\oFF_f}{Z}}{\tau}$, we have
\[
(1 + \tau + \tau^2 + \cdots)\cdot \by^{\tr} \cdot g_{\bx}^* = h_{\by}\cdot \bx \cdot (1 + \tau + \tau^2 + \cdots) \quad \in \power{\FZ{\oFF_f}{Z}}{\tau}.
\]
Comparing coefficients of $\tau^n$ on both sides, we obtain
\[
b_0^{(n)} + b_1^{(n-1)} + \cdots + b_n = c_0 + c_1 + \cdots + c_n = h_{\by}(\bx).
\]
Since $h_{\by}(\bx) \in \FZ{\FF_q}{Z}$, it follows that $h_{\by}(\bx) = \sigma^n(b_0^{(n)} + b_1^{(n-1)} + \cdots + b_n) =b_0 + b_1^{(-1)} + \cdots + b_n^{(-n)}$, and the desired result follows.

To verify Galois equivariance, since $\eta$ is defined over $\FZ{\FF_f}{Z}$, we find from \eqref{E:gxhy} that
\[
\tau^d \by^{\tr} \eta = \bigl( \by^{(d)} \bigr)^{\tr} \eta \tau^d = (1-\tau) \tau^d h_{\by} = (1-\tau)h_{\by}^{(d)} \tau^d \in \Mat_{1 \times \ell}(\FZ{\oFF_f}{Z}[\tau]).
\]
Thus if $\alpha \in \Gal(\oFF_f/\FF_f)$ is the $q^d$-th power Frobenius automorphism, $h_{\alpha(\by)} = h_{\by}^{(d)}$. Therefore, from (a) we have
\[
\apairing{\alpha(\bx)}{\alpha(\by)}_{\eta} = h_{\alpha(\by)}(\alpha(\bx)) = h_{\by}^{(d)}\bigl(\bx^{(d)} \bigr) = \alpha(h_{\by}(\bx)) =\alpha(\apairing{\bx}{\by}_{\eta}) = \apairing{\bx}{\by}_{\eta},
\]
which proves~(b).

For (c), we assume $(\ker \eta)(\FZ{\oFF_f}{Z})$ and $(\ker \eta^*)(\FZ{\oFF_f}{Z})$ have full dimension as in Definition~\ref{D:full}. After multiplying $\eta$ on the right and left by invertible matrices from $\GL_{\ell}(\FZ{\FF_f}{Z}[\tau])$ as in Proposition~\ref{P:AndUBV}, we can assume that $\eta= \diag(\eta_1, \dots, \eta_{\ell}) \in \Mat_{\ell}(\FZ{\FF_f}{Z}[\tau])$. By the definition of full dimension, Proposition~\ref{P:kerdim} implies that
\[
\dim_{\FZ{\FF_q}{Z}} \ker\eta = \dim_{\FZ{\FF_q}{Z}} \ker \eta^* = \deg_{\tau} \eta_1 + \cdots + \deg_{\tau} \eta_{\ell}.
\]
Furthermore, we have the canonical isomorphisms of $\FZ{\FF_q}{Z}$-vector spaces,
\[
\ker \eta \cong \ker \eta_1 \oplus \cdots \oplus \ker \eta_\ell, \quad
\ker \eta^* \cong \ker \eta_1^* \oplus \cdots \oplus \ker \eta_{\ell}^*,
\]
with $\dim_{\FZ{\FF_q}{Z}} \ker \eta_i = \dim_{\FZ{\FF_q}{Z}} \ker \eta_i^* = \deg_\tau \eta_i$ for each~$i$.

Now suppose that $\bx \in \ker \eta$ satisfies that $\apairing{\bx}{\by}_{\eta} = 0$ for all $\by \in \ker \eta^*$. Write $\bx = (x_1, \dots, x_\ell)^{\tr}$ for $x_1, \dots, x_\ell \in \FZ{\oFF_f}{Z}$. If we write $g_{\bx} = (g_1, \dots, g_\ell) \in \Mat_{1\times \ell}(\FZ{\oFF_f}{Z}[\sigma])$, then \eqref{E:gxhy} implies
\[
x_i \eta_i^* = (1-\sigma) g_i, \quad \forall i,\ 1 \leqslant i \leqslant \ell.
\]
Suppose we have $j$ with $x_j \neq 0$. Then $\deg_{\sigma} g_j = \deg_{\sigma} \eta_j^* - 1$. But for all $y_j \in \ker \eta_j^*$, if we let $\by = (0, \dots, 0, y_j, 0, \dots 0)^{\tr} \in \ker \eta^*$, where $y_j$ is in the $j$-th entry, then
\[
\apairing{\bx}{\by}_{\eta} = g_j(y_j) = 0.
\]
Since $\ker \eta_j^*$ has dimension $\deg_{\sigma} \eta_j^*$ over $\FZ{\FF_q}{Z}$, we see that $\ker g_j$ has larger $\FZ{\FF_q}{Z}$-dimension than its degree in $\sigma$ would allow by Lemma~\ref{L:kerneldim}. Thus it must be that $\bx = 0$. Similarly the kernel on the right of $\apairing{\cdot}{\cdot}_{\eta}$ is trivial.
\end{proof}

\subsubsection{The case $\eta = (1-\tau^d)\rI_{\ell}$} \label{SSS:1-Frob}
Let $\cE : \cA \to \Mat_{\ell}(\FZ{\FF_f}{Z}[\tau])$ be a $t$-module over $\FZ{\FF_f}{Z}$. Then $\eta = (1-\tau^d)\rI_{\ell}$ is an endomorphism of $\cE$, and we have $\ker \eta = \cE(\FZ{\FF_f}{Z}) = \FZ{\FF_f}{Z}^{\ell}$ and $\ker \eta^* = \cE^*(\FZ{\FF_f}{Z}) = \FZ{\FF_f}{Z}^{\ell}$. As $[\FF_f:\FF_q]=d$, we see that $\ker \eta$ and $\ker \eta^*$ have full dimension. For $\bx \in \FZ{\FF_f}{Z}^{\ell}$,
\[
\bx^{\tr} \eta^* = (1-\sigma^d)\bx^{\tr}
=(1-\sigma) \Bigl( \bx^{\tr} + \bigl( \bx^{(-1)} \bigr)^{\tr} \sigma + \cdots +
\bigl( \bx^{(-d+1)} \bigr)^{\tr} \sigma^{d-1} \Bigr).
\]
Thus by \eqref{E:gxhy}, we have $g_{\bx}(\by) = \bx^{\tr} \by + ( \bx^{(-1)} )^{\tr} \by^{(-1)} + \cdots + ( \bx^{(-d+1)} )^{\tr} \by^{(-d+1)}$ for $\by \in \FZ{\FF_f}{Z}^{\ell}$. From this we see that
\begin{equation} \label{E:tracepair}
\apairing{\bx}{\by}_{\eta} = \Tr^{\FF_f}_{\FF_q} \bigl( \bx^{\tr}  \by \bigr),
\end{equation}
and so in this case
\[
\apairing{\cdot}{\cdot}_{\eta} : \cE(\FZ{\FF_f}{Z}) \times \cE^*(\FZ{\FF_f}{Z}) \to \FZ{\FF_q}{Z}
\]
coincides with the trace pairing from $\FZ{\FF_f}{Z}$ to $\FZ{\FF_q}{Z}$. This could be predicted by \citelist{\cite{Goss}*{Prop.~4.14.11} \cite{Poonen96}*{Prop.~18}}, which was inspired by work of Elkies.

Returning to the general case of an endomorphism $\eta : \cE \to \cE$ with $\det(\pd\eta) \neq 0$, we establish an adjoint relationship between $\cE_a$ and $\apairing{\cdot}{\cdot}_{\eta}$.

\begin{proposition}[cf.\ \citelist{\cite{Goss}*{Prop.~4.14.13} \cite{Poonen96}*{Prop.~19}}] \label{P:adjointop}
Let $\cE : \cA \to \Mat_{\ell}(\FZ{\FF_f}{Z}[\tau])$ be a $t$-module over $\FZ{\FF_f}{Z}$, and let $\eta : \cE \to \cE$ be an endomorphism. Then for any $a \in \cA$,
\[
\apairing{\cE_a(\bx)}{\by}_{\eta} = \apairing{\bx}{\cE_a^*(\by)}_{\eta}, \quad \forall\, \bx \in \ker\eta,\ \by \in \ker \eta^*.
\]
\end{proposition}

\begin{proof}
We pick $g_{\bx}$ and $g_{\cE_a(\bx)}$ as in \eqref{E:gxhy}. Then
\[
\bx^{\tr} \eta^* \cE_a^* = (1-\sigma)g_{\bx} \cE_a^*, \quad
\cE_a(\bx)^{\tr} \eta^* = (1-\sigma) g_{\cE_a(\bx)},
\]
and after subtracting these equations and noting that $\eta^* \cE_a^* = \cE_a^* \eta^*$,  we obtain
\[
\bigl( \bx^{\tr} \cE_a^* - \cE_a(\bx)^{\tr} \bigr) \eta^* = (1-\sigma) \bigl( g_{\bx} \cE_a^* - g_{\cE_a(\bx)} \bigr).
\]
Since $\delta_1 (\bx^{\tr} \cE_a^* - \cE_a(\bx)^{\tr}) = 0$, \eqref{E:delta1diag} implies that $\bx^{\tr} \cE_a^* - \cE_a(\bx)^{\tr} = (1-\sigma)\bh$ for some $\bh \in \Mat_{1 \times \ell}(\FZ{\oFF_f}{Z}[\sigma])$. Therefore, $\bh \cdot \eta^* = g_{\bx} \cE_a^* - g_{\cE_a(\bx)}$.
Evaluating at $\by \in \ker \eta^*$,
\[
0 = \bh \bigl(\eta^*(\by)\bigr) = g_{\bx} \bigl(\cE_a^*(\by) \bigr) - g_{\cE_a(\bx)}(\by)
= \apairing{\bx}{\cE_a^*(\by)}_{\eta} - \apairing{\cE_a(\bx)}{\by}_{\eta}.
\qedhere
\]
\end{proof}

As in Poonen~\cite{Poonen96}*{\S 9}, we can use Proposition~\ref{P:adjointop} to define an $\cA$-bilinear pairing as follows. Let $\cA^{\wedge} \assign \Hom_{\FZ{\FF_q}{Z}}(\cA,\FZ{\FF_q}{Z})$, which carries an $\cA$-module structure in the usual fashion. Then for an endomorphism $\eta : \cE \to \cE$, we define
\begin{equation} \label{E:bpairing}
\bpairing{\cdot}{\cdot}_{\eta} : \ker \eta \times \ker \eta^* \to \cA^{\wedge}
\end{equation}
by
\[
\bpairing{\bx}{\by}_{\eta} \assign \bigl( a \mapsto \apairing{\cE_a(\bx)}{\by}_{\eta} \bigr).
\]
As shown earlier in the section, since $\det(\pd\eta) \neq 0$, the $\FZ{\FF_q}{Z}$-vector spaces $\ker \eta$ and $\ker \eta^*$ are finite dimensional, and thus they are finitely generated and torsion $\cA$-modules. Letting $h = \Aord{\ker \eta}{\cA}$, it follows that $\bpairing{\cdot}{\cdot}_{\eta}$ takes values in $(\cA/h\cA)^{\wedge} \assign \Hom_{\FZ{\FF_q}{Z}}(\cA/h\cA,\FZ{\FF_q}{Z})$ $\subseteq \cA^{\wedge}$. The former is a finitely generated torsion $\cA$-module isomorphic to $\cA/h\cA$.

\begin{proposition}[{cf.\ \citelist{\cite{Goss}*{Prop.~4.14.14} \cite{Poonen96}*{Prop.~19}}}] \label{P:bpairing}
Let $\eta : \cE \to \cE$ be an endomorphism over $\FZ{\FF_f}{Z}$ with $\det(\pd \eta) \neq 0$, and let $h = \Aord{\ker \eta}{\cA}$. Define the pairing
\[
\bpairing{\cdot}{\cdot}_{\eta} : \ker \eta \times \ker \eta^* \to (\cA/h\cA)^{\wedge}
\]
as in \eqref{E:bpairing}.
\begin{alphenumerate}
\item $\bpairing{\cdot}{\cdot}_{\eta}$ is $\cA$-bilinear and $\Gal(\oFF_f/\FF_f)$-equivariant.
\item If $\ker \eta$ and $\ker \eta^*$ have full dimension, then $\bpairing{\cdot}{\cdot}_{\eta}$ is non-degenerate.
\end{alphenumerate}
\end{proposition}

\begin{proof}
We will prove $\cA$-bilinearity momentarily, but given that, Galois equivariance (with trivial action on $(\cA/h\cA)^{\wedge}$) is straightforward from Proposition~\ref{P:apairing}(b). Also, in the situation of (b), if $\bx \in \ker \eta$ and $[\bx,\by]_{\eta} = 0$ for all $\by \in \ker \eta^*$, then for all $a \in \cA/h\cA$,
\[
0 = \bpairing{\bx}{\by}_{\eta}(a) = \apairing{\cE_a(\bx)}{\by}_{\eta}.
\]
The non-degeneracy of $\apairing{\cdot}{\cdot}_{\eta}$ from Proposition~\ref{P:apairing}(c) implies that $\cE_a(\bx) = 0$ for all $a \in \cA/h\cA$. In particular for $a=1$, which implies $\bx=0$. Likewise, for fixed $\by \in \ker \eta^*$, if $\bpairing{\bx}{\by}_{\eta} = 0$ for all $\bx \in \ker \eta$, then using $\apairing{\cE_a(\bx)}{\by}_{\eta} = \apairing{\bx}{\cE_a^*(\by)}_{\eta}$ from Proposition~\ref{P:adjointop}, a similar argument implies that $\by = 0$.

To prove the $\cA$-bilinearity of $\bpairing{\cdot}{\cdot}_{\eta}$, we note that it is $\FZ{\FF_q}{Z}$-bilinear, so we need only check that it respects multiplication by~$\cA$. For $\bx \in \ker \eta$, $\by \in \ker \eta^*$, and $b \in \cA$,
\[
b \cdot \bpairing{\bx}{\by}_{\eta} = \bigl( a \mapsto \apairing{\cE_{ba}(\bx)}{\by}_{\eta} \bigr)
\]
and
\[
\bpairing{\cE_b(\bx)}{\by}_{\eta} = \bigl( a \mapsto \apairing{\cE_{a}(\cE_{b}(\bx))}{\by}_{\eta} \bigr).
\]
Since $\cE_a\cE_b = \cE_{ba}$, we have $b \cdot \bpairing{\bx}{\by}_{\eta} = \bpairing{\cE_b(\bx)}{\by}_{\eta}$. Similarly, since Proposition~\ref{P:adjointop} implies $\apairing{\cE_a(\bx)}{\cE_b^*(\by)}_{\eta} = \apairing{\cE_b(\cE_a(\bx))}{\by}$, we have
\[
\bpairing{\bx}{\cE_b^*(\by)}_{\eta} = \bigl( a \mapsto \apairing{\cE_{b}(\cE_{a}(\bx))}{\by}_{\eta} \bigr),
\]
and so  $b \cdot \bpairing{\bx}{\by}_{\eta} = \bpairing{\bx}{\cE_b^*(\by)}_{\eta}$ as well.
\end{proof}

\begin{corollary} \label{C:kerandkerstar}
Let $\cE : \cA \to \Mat_{\ell}(\FZ{\FF_f}{Z}[\tau])$ be a $t$-module, and let $\eta : \cE \to \cE$ be an endomorphism over $\FZ{\FF_f}{Z}$ with $\det(\pd \eta) \neq 0$. If $\ker \eta$, $\ker \eta^*$ have full dimension, then
\[
\ker \eta \cong \ker \eta^*
\]
as $\cA$-modules and $\Gal(\oFF_f/\FF_f)$-modules. In particular as $\cA$-modules,
\[
\cE(\FZ{\FF_f}{Z}) \cong \cE^*(\FZ{\FF_f}{Z}).
\]
\end{corollary}

\begin{proof}
Let $h = \Aord{\ker \eta}{\cA}$. The non-degeneracy of $\bpairing{\cdot}{\cdot}_{\eta}$ implies that we have isomorphisms of $\cA$-modules, $\ker \eta^* \cong \Hom_{\cA}( \ker \eta, (\cA/h\cA)^{\wedge}) \cong \Hom_{\cA}(\ker \eta,\cA/h\cA)$.
This last module is non-canonically isomorphic to $\ker \eta$, and by Proposition~\ref{P:bpairing}(a), these maps respect the Galois action. The second part follows by taking $\eta = (1-\tau^d)\rI_{\ell}$ from \S\ref{SSS:1-Frob}.
\end{proof}

\subsubsection{Applications to Tate modules}
Let $\lambda \in \sA_+$ be irreducible, such that $\lambda(\theta) \neq f$. Then $\lambda$ is also irreducible in $\cA$, and $\det (\pd \cE_{\lambda}) \neq 0$. For $m \geqslant 1$, we write
\begin{align} \label{E:lambdapairings}
\apairing{\cdot}{\cdot}_m \assign \apairing{\cdot}{\cdot}_{\cE_{\lambda^m}} : \cE[\lambda^m] \times \cE^*[\lambda^m] &\to \FZ{\FF_q}{Z}, \\
\bpairing{\cdot}{\cdot}_m \assign \bpairing{\cdot}{\cdot}_{\cE_{\lambda^m}} : \cE[\lambda^m] \times \cE^*[\lambda^m] &\to (\cA/\lambda^m\cA)^{\wedge}. \notag
\end{align}
In a similar fashion to the proof of~Proposition~\ref{P:adjointop} (cf.\ \citelist{\cite{Goss}*{pp.~132--133} \cite{Poonen96}*{Prop.~21}}, we find for all $m \geqslant 1$, $\bx \in \cE[\lambda^{m+1}]$, and $\by \in \cE^*[\lambda^m]$,
\begin{equation} \label{E:compatibility}
\apairing{\cE_{\lambda}(\bx)}{\by}_{m} = \apairing{\bx}{\by}_{m+1}, \quad
\bpairing{\cE_{\lambda}(\bx)}{\by}_{m} = \bpairing{\bx}{\by}_{m+1}.
\end{equation}
These lead to the following result.

\begin{proposition}[{cf.\ \citelist{\cite{Goss}*{Cor.~4.14.17} \cite{Poonen96}*{Cor.~11}}}] \label{P:lambdacompatibility}
Let $\cE : \cA \to \Mat_{\ell}(\FZ{\FF_f}{Z}[\tau])$ be an Anderson $t$-module over $\FZ{\FF_f}{Z}$, and let $\lambda \in \sA_+$ be irreducible with $\lambda(\theta) \neq f$. For each $m \geqslant 1$, $\bx \in \cE[\lambda^{m+1}]$, and $\by \in \cE^*[\lambda^{m+1}]$, we have
\[
\bpairing{\cE_{\lambda}(\bx)}{\cE_{\lambda}^*(\by)}_{m} = \lambda \cdot \bpairing{\bx}{\by}_{m+1}.
\]
\end{proposition}

\begin{proof}
By \eqref{E:compatibility}, we have $\bpairing{\cE_{\lambda}(\bx)}{\cE_{\lambda}^*(\by)}_{m} = \bpairing{\bx}{\cE_{\lambda}^*(\by)}_{m+1}$, which then is the same as $\lambda \cdot \bpairing{\bx}{\by}_{m+1}$ by $\cA$-bilinearity.
\end{proof}

\begin{definition} \label{D:Tatemodules}
The $\lambda$-adic \emph{Tate modules},
\[
T_{\lambda}(\cE) \assign \varprojlim \cE[\lambda^m], \quad
T_{\lambda}(\cE^*) \assign \varprojlim \cE^*[\lambda^m],
\]
are modules over $\cA_{\lambda} \assign \varprojlim \cA/\lambda^m\cA$. When
\begin{alphenumerate}
\item $\cE$ is abelian with $\rank_{\FZ{\oFF_f}{Z}[t] }\cM_{\cE}(\FZ{\oFF_f}{Z}) = r$, 
\item $\cE$ is $\cA$-finite with $\rank_{\FZ{\oFF_f}{Z}[t]} \cN_{\cE}(\FZ{\oFF_f}{Z}) = r$,
\item $\cE[\lambda^m]$ and $\cE^*[\lambda^m]$ have full dimension for each $m \geqslant 1$,
\end{alphenumerate}
Corollary~\ref{C:Atorsion} implies that
\begin{equation}
T_{\lambda}(\cE) \cong \cA_{\lambda}^r, \quad T_{\lambda}(\cE^*) \cong \cA_{\lambda}^r.
\end{equation}
As usual $T_{\lambda}(\cE)$ and $T_{\lambda}(\cE^*)$ are $\Gal(\oFF_f/\FF_f)$-modules.
By Proposition~\ref{P:lambdacompatibility}, the pairings in~\eqref{E:lambdapairings} fit together to induce a pairing on Tate modules.
\end{definition}

\begin{theorem}[{cf.\ \citelist{\cite{Goss}*{Thm.~4.14.20} \cite{Poonen96}*{Prop.~21}}}] \label{T:Tatepairing}
Let $\cE : \cA \to \Mat_{\ell}(\FZ{\FF_f}{Z}[\tau])$ be an Anderson $t$-module over $\FZ{\FF_f}{Z}$, and let $\lambda \in \sA_+$ be irreducible with $\lambda(\theta) \neq f$.
We have a continuous $\cA_{\lambda}$-bilinear pairing,
\[
\bpairing{\cdot}{\cdot}_{\lambda} : T_{\lambda}(\cE) \times T_{\lambda}(\cE^*) \to \cA_{\lambda},
\]
with the following properties.
\begin{alphenumerate}
\item $\bpairing{\cdot}{\cdot}_{\lambda}$ is $\Gal(\oFF_f/\FF_f)$-equivariant, where the Galois action on $\cA_{\lambda}$ is trivial.
\item If conditions (a)--(c) from Definition~\ref{D:Tatemodules} hold, then $\bpairing{\cdot}{\cdot}_{\lambda}$ is non-degenerate.
\end{alphenumerate}
\end{theorem}

\begin{proof}
We first observe that $\varprojlim (\cA/\lambda^m\cA)^{\wedge} \cong \cA_{\lambda}$ (e.g., see \cite{Goss}*{Prop.~4.14.19}). The $\cA_{\lambda}$-bilinearity follows from Proposition~\ref{P:lambdacompatibility}, and Galois equivariance follows from Proposition~\ref{P:bpairing}(b). Under the conditions of Definition~\ref{D:Tatemodules}(a)--(c), Corollary~\ref{C:Atorsion} and Proposition~\ref{P:bpairing}(c) imply non-degeneracy.
\end{proof}

One consequence of this theorem is that under conditions (a)--(c) of Definition~\ref{D:Tatemodules},
\begin{equation} \label{E:Tatedual}
T_{\lambda}(\cE^*) \cong T_{\lambda}(\cE)^{\vee} \assign \Hom_{\cA_{\lambda}} ( T_{\lambda}(\cE),\cA_{\lambda}),
\end{equation}
as $\Gal(\oFF_f/\FF_f)$-modules. In particular the Galois representations on $T_{\lambda}(\cE)$ and $T_{\lambda}(\cE^*)$ are dual to each other.

\begin{remark}
Other pairings on $t$-module torsion over finite fields have been developed by Taguchi~\cite{Taguchi95} using a construction of duals for $t$-modules instead of adjoint $t$-modules. His pairings differ from $\bpairing{\cdot}{\cdot}_{\lambda}$ in that, like the Weil pairing, they take values in $T_{\lambda}(\sC)$ where~$\sC$ is the Carlitz module (see \cite{Taguchi95}*{Thm.~4.3}). For an adaptation of Taguchi's methods to the pairing $\bpairing{\cdot}{\cdot}_{\lambda}$ for Drinfeld modules, see Taguchi's appendix to Goss~\cite{Goss95}.
\end{remark}

\subsection{Characteristic polynomials and \texorpdfstring{$\mathcal{A}$}{A}-orders}
We now fix a $t$-module
\[
\cE : \cA \to \Mat_{\ell}(\FZ{\FF_f}{Z}[\tau]),
\]
and let $\lambda \in \sA_+$ be irreducible with $\lambda(\theta) \neq f$. We assume that $\cE$ satistifies Definition~\ref{D:Tatemodules}(a)--(c).
Then $\tau^d\cdot \rI_{\ell} : \cE \to \cE$ induces an automorphism of $T_\lambda(\cE)$, which coincides with the action of the $q^d$-th power Frobenius $\alpha \in \Gal(\oFF_f/\FF_f)$. Likewise, $\sigma^d \cdot \rI_{\ell} : \cE^* \to \cE^*$ induces an automorphism of $T_{\lambda}(\cE^*)$ that coincides with $\alpha^{-1}$. If $B \in \GL_r(\cA_{\lambda})$ represents the action of $\tau^d$ on $T_{\lambda}(\cE)$, then $(B^{-1})^{\tr}$ represents the action of $\tau^d$ on $T_{\lambda}(\cE)^{\vee}$. By \eqref{E:Tatedual}, we thus have equalities of characteristic polynomials,
\begin{align} \label{E:taudsigmad}
\Char(\tau^d,T_{\lambda}(\cE),X) &= \Char(\sigma^d,T_{\lambda}(\cE^*),X), \\
\Char(\sigma^d,T_{\lambda}(\cE),X) &= \Char(\tau^d,T_{\lambda}(\cE^*),X). \notag
\end{align}

Now choose $\bm \in \Mat_{r \times 1}(\cM_{\cE}(\FZ{\FF_f}{Z}))$ and $\bn \in \Mat_{r \times 1}(\cN_{\cE}(\FZ{\FF_f}{Z}))$ that are $\FZ{\FF_f}{Z}[t]$-bases as in \S\ref{SS:tmodsfinite}, together with $\Gamma$, $\Phi \in \Mat_r(\FZ{\FF_f}{Z}[t])$ so that $\tau \bm = \Gamma \bm$, $\sigma \bn = \Phi \bn$. Finally, fix $\rG$, $\rH \in \Mat_{r}(\FZ{\FF_f}{Z}[t])$ as in~\eqref{E:GHdef}.

For simplicity we temporarily abbreviate $\cM \assign \cM_{\cE}(\FZ{\oFF_f}{Z})$, and for an endomorphism $\eta : \cE \to \cE$ we write
\[
\biggl( \frac{\cM}{\cM\eta} \biggr)^{\tauid} \assign \ker \biggl(  \frac{\cM}{\cM\eta}  \xrightarrow{(\tau - 1)(\cdot)}  \frac{\cM}{\cM\eta} \biggr).
\]
We recall the $\FZ{\oFF_f}{Z}[t]$-linear isomorphism $\bi : \Mat_{1\times r}(\FZ{\oFF_f}{Z}[t]) \iso \cM$ from~\S\ref{SS:tmodsfinite}.

\begin{proposition} \label{P:twosquares}
Let $m \geqslant 1$. The following diagram of $\cA$-modules commutes:
\begin{displaymath}
\begin{tikzcd}
 \cE^*[\lambda^m] \arrow{r}{\sim} \arrow{d}{\sigma^d(\cdot)} &\biggl( \dfrac{\cM}{\lambda^m \cM} \biggr)^{\tauid} \arrow{d}{(\cdot)\tau^d} \arrow[leftarrow]{r}{\bi}[swap]{\sim} & \biggl( \dfrac{\Mat_{1 \times r}(\FZ{\oFF_f}{Z}[t])}{\lambda^m\Mat_{1 \times r}(\FZ{\oFF_f}{Z}[t])}\biggr)^{\tauid} \arrow{d}{(\cdot)\rG} \\
\cE^*[\lambda^m] \arrow{r}{\sim} &\biggl( \dfrac{\cM}{\lambda^m \cM} \biggr)^{\tauid} \arrow[leftarrow]{r}{\bi}[swap]{\sim} & \biggl( \dfrac{\Mat_{1 \times r}(\FZ{\oFF_f}{Z}[t])}{\lambda^m\Mat_{1 \times r}(\FZ{\oFF_f}{Z}[t])}\biggr)^{\tauid},
\end{tikzcd}
\end{displaymath}
where the first set of horizontal isomorphisms are from~\eqref{E:keretastar}.
\end{proposition}

\begin{proof}
From \eqref{E:gamma1diag} the map on $\cM$ induced by $\sigma^d \cdot \rI_{\ell} : \cE^* \to \cE^*$ is right-multiplication by~$\tau^d$. Noting that elements of $(\cM/\lambda^m\cM)^{\tauid}$ are classes fixed under \emph{left}-multiplication by $\tau-1$, it follows quickly that right-multiplication by~$\tau^d$ on $(\cM/\lambda^m\cM)^{\tauid}$ is a well-defined $\cA$-module map. Winding through the snake lemma applied to~\eqref{E:gamma1diag}, we see that the horizontal maps are defined as follows: for $\by \in \cE^*[\lambda^m]$, there exist $\beta$, $\beta' \in \cM$ so that $\gamma_1(\beta) = \by$ and $\lambda^m \cdot \beta = (\tau-1)\beta'$, and then
$\by \mapsto \beta' + \lambda^m \cM$.
But then $\gamma_1(\beta \tau^d) = \by^{(-d)} = \sigma^d(\by)$ and $\lambda^m \cdot \beta \tau^d = (\tau-1)\beta'\tau^d$, and so $\sigma^d(\by) = \beta' \tau^d + \lambda^m \cM$ as desired. This verifies the commutativity of the first square.

Lemma~\ref{L:GH} implies that the second square commutes without each term being fixed by left-multiplication by $\tau-1$. It suffices to show the right-most vertical map is well-defined. For $\bu \in \Mat_{1 \times r}(\FZ{\oFF_f}{Z}[t])$ such that $(\tau - 1)\bu \in \lambda^m \Mat_{1\times r}(\FZ{\oFF_f}{Z}[t])$, we have by definition of the $\tau$-action on $\Mat_{1\times r}(\FZ{\oFF_f}{Z}[t])$ through $\bi$ that for some $\bv \in \Mat_{1\times r}(\FZ{\oFF_f}{Z}[t])$,
\[
(\tau-1)\bu = \bu^{(1)}\Gamma - \bu = \lambda^m \bv.
\]
Recalling that $\rG = \Gamma^{(d-1)} \cdots \Gamma^{(1)}\Gamma$, we see that
\[
(\tau-1)\bu\rG = \bigl(\bu^{(1)}\rG^{(1)} \Gamma - \bu\rG \bigr) = \bigl( \bu^{(1)} \Gamma^{(d)} - \bu \bigr)\rG.
\]
Since $\Gamma \in \Mat_r(\FZ{\FF_f}{Z}[t])$, we have $\Gamma^{(d)} = \Gamma$, and so $(\tau-1)\bu\rG = \lambda^m \bv \rG$ as desired.
\end{proof}

\begin{corollary} \label{C:charpolys}
Let $\cE : \cA \to \Mat_{\ell}(\FZ{\FF_f}{Z}[\tau])$
be a $t$-module defined over $\FZ{\FF_f}{Z}$. Let $\lambda \in \sA_+$ be irreducible with $\lambda(\theta) \neq f$ such that Definitions~\ref{D:Tatemodules}(a)--(c) are satisfied.
For $\rG = \Gamma^{(d-1)} \cdots \Gamma^{(1)}\Gamma \in \Mat_r(\FZ{\FF_f}{Z}[t])$, where $\Gamma$ represents multiplication by~$\tau$ on $\cM_{\cE}(\FZ{\FF_f}{Z})$, we have
\[
\Char(\tau^d,T_{\lambda}(\cE),X) = \Char(\rG,X).
\]
Furthermore, this polynomial is in $\cA[X] = \FZ{\FF_q}{Z}[t,X]$ and is independent of the choice of $\lambda$ for which Definition~\ref{D:Tatemodules}(c) is satisfied.
\end{corollary}

\begin{proof}
For each $m \geqslant 1$, because $\cE[\lambda^m]$ and $\cE^*[\lambda^m]$ have full dimension, we see from Corollary~\ref{C:Atorsion} that each is isomorphic to $(\cA/\lambda^m\cA)^r$. By the proof of Proposition~\ref{P:kerdim}, there is an $\cA/\lambda^m\cA$-basis of $(\cM/\lambda^m\cM)^{\tauid}$ of rank $r$ that is also an $(\FZ{\oFF_f}{Z}[t]/\lambda^m\FZ{\oFF_f}{Z}[t])$-basis of $\cM/\lambda^m\cM$. Passing this basis to the right-hand column of Proposition~\ref{P:twosquares}, we obtain an $\FZ{\oFF_f}{Z}[t]$-basis $\bu_1, \dots, \bu_r \in \Mat_{1\times r}(\FZ{\oFF_f}{Z}[t])$, such that their images
\[
\bu_1, \dots, \bu_r \in \biggl( \dfrac{\Mat_{1 \times r}(\FZ{\oFF_f}{Z}[t])}{\lambda^m\Mat_{1 \times r}(\FZ{\oFF_f}{Z}[t])}\biggr)^{\tauid}
\]
form an $\cA/\lambda^m\cA$-basis. If we let $U = (\bu_1, \dots, \bu_r)^{\tr} \in \GL_r(\FZ{\FF_f}{Z}[t])$, then $U\rG U^{-1}$ represents the right-hand vertical map in Proposition~\ref{P:twosquares} with respect to this basis.
Therefore, Proposition~\ref{P:twosquares} implies that the characteristic polynomial of $\sigma^d$ acting on the $\cA/\lambda^m\cA$-module $\cE^*[\lambda^{m}]$ is
\[
\Char(\sigma^d,\cE^*[\lambda^m],X) = \Char(U\rG U^{-1},X) = \Char(\rG,X) \bmod \lambda^m.
\]
Passing the diagram in Proposition~\ref{P:twosquares} to the inverse limit, we obtain
\[
\Char(\tau^d,T_{\lambda}(\cE),X) = \Char(\sigma^d,T_{\lambda}(\cE^*),X) = \Char(\rG,X) \in \FZ{\FF_f}{Z}[t],
\]
where the first equality is~\eqref{E:taudsigmad}. Since $\rG$ is independent of the choice of $\lambda$, it remains to verify that $\Char(\rG,X) \in \cA[X]$. But using that $\Gamma^{(d)} = \Gamma$ and that $\Char(BC,X) = \Char(CB,X)$ for matrices $B$ and $C$,
\begin{multline*}
\Char(\rG,X)^{(1)} = \Char(\rG^{(1)},X) = \Char(\Gamma^{(d)} \cdots \Gamma^{(1)},X) \\
= \Char(\Gamma^{(d-1)}\cdots \Gamma^{(1)} \Gamma^{(d)},X)
= \Char(\rG,X),
\end{multline*}
and so $\Char(\rG,X) \in \cA[X]$.
\end{proof}

This corollary can be used to find characteristic polynomials of Frobenius acting on Tate modules for global $L$-functions, and moreover to calculate $\mu_{\phi}(a)$ and $\nu_{\phi}(a)$, as defined in \S\ref{SS:munu}. See Remark~\ref{R:Qfcalc}.

We can use Corollary~\ref{C:charpolys} to determine $\Aord{\cE(\FZ{\FF_f}{Z})}{\cA}$. For Drinfeld modules over finite fields, this was found by Gekeler~\cite{Gekeler91}*{Thm.~5.1(i)} using methods of Deuring involving reduced norms on quaternion endomorphism algebras, techniques which were not available in our current setting. Furthermore, the methods of Yu~\citelist{\cite{Goss}*{Prop.~4.12.21} \cite{YuJK95}} do not readily extend to Anderson $\cA$-modules, as we do not have as fine control over $\nu$-torsion when $\nu \in \cA \setminus A$.
However, the results of the previous sections yield the following. We note that this result extends Gekeler's result both to certain $t$-modules over finite fields (using $Z = \emptyset$) but also over general $\FZ{\FF_f}{Z}$.

\begin{theorem} \label{T:charpoly1}
Let $\cE : \cA \to \Mat_{\ell}(\FZ{\FF_f}{Z}[\tau])$ be a $t$-module defined over $\FZ{\FF_f}{Z}$. Let $\lambda \in \sA_+$ be irreducible with $\lambda(\theta) \neq f$ such that Definitions~\ref{D:Tatemodules}(a)--(c) are satisfied. Then
\[
\bigl[ \cE(\FZ{\FF_f}{Z}) \bigr]_{\cA} = \gamma \cdot \Char(\tau^d,T_{\lambda}(\cE),1),
\]
where $\gamma \in \FZ{\FF_q}{Z}^{\times}$ uniquely forces the right-hand expression to be monic in~$t$.
\end{theorem}

\begin{proof}
We let $M \assign \Mat_{1 \times r}(\FZ{\oFF_f}{Z}[t])$. By Lemma~\ref{L:Ffpoints}(a), we have
\begin{equation} \label{E:Estarorder}
\cE^*(\FZ{\FF_f}{Z}) \cong \biggl( \frac{M}{M(\rI - \rG)} \biggr)^{\tauid}.
\end{equation}
Because $\cE^*(\FZ{\FF_f}{Z})$ has full dimension, as in the proof of Proposition~\ref{P:kerdim} we can find an $\FZ{\FF_q}{Z}$-basis of $(M/M(\rI-\rG))^{\tauid}$ that is also an $\FZ{\oFF_f}{Z}$-basis of $M/M(\rI-\rG)$. Therefore,
\[
\FZ{\oFF_f}{Z} \otimes_{\FZ{\FF_q}{Z}} \biggl( \frac{M}{M(\rI - \rG)} \biggr)^{\tauid} \cong \frac{M}{M(\rI-\rG)}.
\]
A priori this is an isomorphism of $\FZ{\oFF_f}{Z}$-vector spaces, but it is also an isomorphism of $\FZ{\oFF_f}{Z}[t]$-modules. Since the left-hand side is the extension of scalars of a finitely generated torsion $\cA$-module, it follows that
\begin{equation} \label{E:detI-G}
\biggl[ \biggl( \frac{M}{M(\rI - \rG)} \biggr)^{\tauid}\, \biggr]_{\cA} = \bigAord{\frac{M}{M(\rI-\rG)}}{\FZ{\oFF_f}{Z}[t]} = \gamma \cdot \det(\rI - \rG), \quad \gamma \in \FZ{\oFF_f}{Z}^{\times},
\end{equation}
where $\gamma$ is uniquely chosen to make $\gamma\cdot \det(\rI-\rG)$ monic in $t$. On the other hand, Corollary~\ref{C:charpolys} shows that
\[
\Char(\tau^d,T_{\lambda}(\cE),1) = \det(\rI-\rG) \in \cA,
\]
so $\gamma \in \FZ{\FF_q}{Z}^{\times}$, and thus also by~\eqref{E:Estarorder},
\begin{equation}
\bigl[ \cE^*(\FZ{\FF_f}{Z}) \bigr]_{\cA} = \gamma \cdot \Char(\tau^d,T_{\lambda}(\cE),1).
\end{equation}
Since Corollary~\ref{C:kerandkerstar} implies $\Aord{\cE^*(\FZ{\FF_f}{Z})}{\cA} = \Aord{\cE(\FZ{\FF_f}{Z})}{\cA}$, we are done.
\end{proof}

We note that if $Z = \emptyset$ and $E$ is a $t$-module over $\FF_f$ that is abelian and $\sA$-finite, then Definition~\ref{D:Tatemodules}(a)--(c) are automatically satisfied. Indeed~(a) and~(b) are assured by Maurischat~\cite{Maurischat26}, and (c) follows from Anderson~\citelist{\cite{And86}*{Prop.~1.8.3} \cite{Goss}*{Cor.~5.6.4}}. We thus have the following corollary for $t$-modules over finite fields (cf.\ Taelman~\cite{Taelman09}*{Prop.~7}).

\begin{corollary} \label{C:charpoly1}
Let $E : \sA \to \Mat_{\ell}(\FF_f[\tau])$ be an abelian and $\sA$-finite $t$-module defined over $\FF_f$. Let $\lambda \in \sA_+$ be irreducible with $\lambda(\theta) \neq f$. Then
\[
\bigl[ E(\FF_f) \bigr]_{\sA} = \gamma \cdot \Char(\tau^d,T_{\lambda}(E),1),
\]
where $\gamma \in \FF_q^{\times}$ forces the right-hand expression to be monic in~$t$.
\end{corollary}

\section{Rigid analytic twists of Drinfeld modules} \label{S:RAtwists}

The main objects of study of the present paper are Anderson $t$-modules over $\LLhat_z$ that are obtained by conjugating a Drinfeld module over $\C$ by the rigid analytic trivialization of another. In particular we focus on the case that both Drinfeld modules are defined over~$A$ and have everywhere good reduction. These constructions were inspired by Angl\`{e}s, Pellarin, and Tavares Ribeiro~\cite{APT16}, who investigated the case of conjugating the Carlitz module by the Anderson-Thakur function~$\omega_z$, and by Angl\`{e}s and Tavares Ribeiro~\cite{AT17} and Gezmi\c{s}~\cite{Gezmis19}, who studied the case of conjugating a Drinfeld module by~$\omega_z$.

\subsection{Properties of \texorpdfstring{$\EE(\phi \times \psi)$}{E(phi x psi)}} \label{SS:EEprops}
We start with two Drinfeld modules $\phi$, $\psi : \sA \to A[\tau]$, defined over~$A$ with everywhere good reduction, such that
\begin{alignat}{2} \label{E:phidef}
\phi_t &= \theta + \kappa_1 \tau + \cdots + \kappa_r \tau^r, \quad &\kappa_i \in A,\ \kappa_r \in \FF_q^{\times}, \\
\label{E:psidef}
\psi_t &= \theta + \eta_1 \tau + \cdots + \eta_{\ell} \tau^{\ell}, \quad &\eta_i \in A,\ \eta_{\ell} \in \FF_q^{\times}.
\end{alignat}
We let $\pi_1, \dots, \pi_r \in \Lambda_{\phi}$ and $\lambda_1, \dots, \lambda_{\ell} \in \Lambda_{\psi}$ be $A$-bases of their respective period lattices. As in \eqref{E:Gammadef}, \eqref{E:Thetadef}, and \eqref{E:Upsilondef}, we construct
\begin{alignat*}{3}
\Gamma_{\phi} &\in \Mat_{r}(A[t]), \quad &\Theta_{\phi} = \Gamma_{\phi}^{\tr} &\in \Mat_{r}(A[t]), \quad &\Upsilon_{\phi} \in \GL_r(\TT_t), \\
\Gamma_{\psi} &\in \Mat_{\ell}(A[t]), \quad &\Theta_{\psi} = \Gamma_{\psi}^{\tr} &\in \Mat_{\ell}(A[t]), \quad &\Upsilon_{\psi} \in \GL_{\ell}(\TT_t).
\end{alignat*}
We set
\begin{gather} \label{E:zmatrices}
\Gamma_{\psi,z} \assign \Gamma_{\psi}|_{t=z} \in \Mat_{\ell}(A[z]), \quad \Theta_{\psi,z} \assign \Theta_{\psi}|_{t=z} \in \Mat_{\ell}(A[z]), \\
\Upsilon_{\psi,z} \assign \Upsilon_{\psi}|_{t=z} \in \GL_{\ell}(\TT_z). \notag
\end{gather}
We note that the identity $\Upsilon_{\psi,z}^{(1)} = \Upsilon_{\psi,z} \Theta_{\psi,z}$ implies
\begin{equation} \label{E:UpThetaz}
\Upsilon_{\psi,z}^{-1} \Upsilon_{\psi,z}^{(j)} = \Theta_{\psi,z} \Theta_{\psi,z}^{(1)} \cdots \Theta_{\psi,z}^{(j-1)} \in \Mat_{\ell}(A[z]), \quad j \geqslant 0,
\end{equation}
where we use the convention that the empty product is the identity matrix.
We define $\phi^{\oplus \ell} : \bA \to \Mat_{\ell}(\AAA[\tau])$ to be the $\ell$-fold direct sum of $\phi$, which we consider to be a constant $t$-module defined over $\AAA$. We then conjugate $\phi^{\oplus \ell}$ by $\Upsilon_{\psi,z}$ to form the $t$-module
\begin{equation}
\EE \assign \EE(\phi \times \psi) : \bA \to \Mat_{\ell}(\AAA[\tau])
\end{equation}
such that
\begin{equation} \label{E:EEtdef}
\EE_t \assign \Upsilon_{\psi,z}^{-1} \cdot \phi^{\oplus \ell}_t \cdot \Upsilon_{\psi,z}
= \theta \rI_{\ell} + \kappa_1 \Theta_{\psi,z} \tau + \cdots + \kappa_r \Theta_{\psi,z} \Theta_{\psi,z}^{(1)} \cdots \Theta_{\psi,z}^{(r-1)} \tau^r.
\end{equation}
We consider $\EE$ to be $t$-module over $\LLhat_z$ as in \S\ref{SS:tmodTatealg}, and moreover, $\Upsilon_{\psi,z} : \EE \to \phi^{\oplus \ell}$ is a $t$-module isomorphism. We recount some fundamental properties below.

\begin{remark}
We note that $\EE(\phi \times \psi)$ is different from $\EE(\psi \times \phi)$. Indeed the former has dimension~$\ell$ and the latter~$r$, and generally they are not isomorphic when $\phi \not\cong \psi$.
\end{remark}

\begin{remark}
As pointed out by the referee, if one were to take $\ell$ Drinfeld modules $\phi_1, \dots, \phi_{\ell}$ defined over $A$ (or some Anderson $t$-modules over $A$ whose dimensions added to~$\ell$), we could consider the $t$-module $\EE : \bA \to \Mat_{\ell}(\TT_z[\tau])$ defined by
\[
\EE_t = \Upsilon_{\psi,z}^{-1} \cdot \bigl( (\phi_1)_t \oplus \dots \oplus (\phi_\ell)_t\bigr) \cdot \Upsilon_{\psi,z}.
\]
Properties of such $t$-modules and their periods would be interesting to investigate (cf.~\cites{APT16, Demeslay22, GezmisP19}), and especially in situations where $\EE_t \in \Mat_{\ell}(\AAA[\tau])$, one would expect their $L$-series to have appealing special value identities.
\end{remark}

\subsubsection{Abelianness and $\bA$-finiteness} \label{SSS:AbAfin}
The Drinfeld module $\phi$ is abelian and $\sA$-finite, and naturally its extension to a Drinfeld module over $\LLhat_z$ is also abelian and $\bA$-finite. By taking direct sums $\phi^{\oplus \ell}$ is also abelian and $\bA$-finite. Since $\EE_t$ is isomorphic to $\phi^{\oplus \ell}$ over~$\LLhat_z$, it follows that $\EE_t$ is also abelian and $\bA$-finite.

Alternatively, since $\Theta_{\psi,z} \in \GL_r(\LLhat_z)$, it follows that $\EE_t$ is strictly pure in the sense of \cite{NamoijamP24}*{Ex.~3.38, Ex.~4.129} and is abelian and $\bA$-finite through the discussions there (see also \cite{HartlJuschka20}*{\S 2.5.2, p.~112}). There we see that if $\bs_1, \dots, \bs_{\ell}$ denote the standard basis vectors of $\Mat_{1 \times \ell}(\LLhat_z)$, then
\begin{gather*}
\{ \tau^j \bs_{i} : 1 \leqslant i \leqslant \ell,\ 0 \leqslant j \leqslant r-1 \} \subseteq \Mat_{1\times \ell}(\LLhat_z[\tau]) = \cM_{\EE}(\LLhat_z), \\
\{ \sigma^j \bs_{i} : 1 \leqslant i \leqslant \ell,\ 0 \leqslant j \leqslant r-1 \} \subseteq \Mat_{1\times \ell}(\LLhat_z[\sigma]) = \cN_{\EE}(\LLhat_z),
\end{gather*}
are $\LLhat_z[t]$-bases of $\cM_{\EE}(\LLhat_z)$ and $\cN_{\EE}(\LLhat_z)$.
We note that $\EE$ has dimension~$\ell$ and rank~$r \ell$.

\subsubsection{Exponentials, logarithms, and period lattices} \label{SSS:explogper}
Suppose that $\Exp_{\phi} = \sum_{i \geqslant 0} B_i \tau^i$, $\Log_{\phi} = \sum_{i \geqslant 0} C_i \tau^i \in \power{K}{\tau}$ are the exponential and logarithm series of $\phi$. Then we have
\begin{align} \label{E:ExpLogEE}
\Exp_{\EE} &= \Upsilon_{\psi,z}^{-1} \cdot \Exp_{\phi^{\oplus \ell}} \cdot \Upsilon_{\psi,z}
= \sum_{i=0}^{\infty} B_i \Theta_{\psi,z} \Theta_{\psi,z}^{(1)} \cdots \Theta_{\psi,z}^{(i-1)} \tau^i, \\
\Log_{\EE} &= \Upsilon_{\psi,z}^{-1} \cdot \Log_{\phi^{\oplus \ell}} \cdot \Upsilon_{\psi,z}
= \sum_{i=0}^{\infty} C_i \Theta_{\psi,z} \Theta_{\psi,z}^{(1)} \cdots \Theta_{\psi,z}^{(i-1)} \tau^i, \notag
\end{align}
both of which are in $\power{\Mat_{\ell}(K[z])}{\tau}$.

We let $\Lambda_{\phi} \subseteq \LLhat_z$ be the kernel of $\Exp_{\phi} : \LLhat_z \to \LLhat_z$, and we fix generators $\pi_1, \dots, \pi_r \in \C$ so that  $\Lambda_\phi = \AAA \pi_1 + \cdots + \AAA \pi_r$ as in Theorem~\ref{T:constant}(c). Because $\phi^{\oplus \ell}$ is simply a direct sum, the following properties then follow from Theorem~\ref{T:constant}.
\begin{itemize}
\item $\Exp_{\phi^{\oplus \ell}} : \LLhat_z^{\ell} \to \LLhat_z^{\ell}$ is surjective.
\item $\Lambda_{\phi^{\oplus \ell}} \assign \ker \Exp_{\phi^{\oplus \ell}} \subseteq \LLhat_z^{\ell}$ is free of rank $r\ell$ over $\AAA$, and $\Lambda_{\phi^{\oplus \ell}} = \Lambda_{\phi}^{\oplus \ell}$.
\end{itemize}
From \eqref{E:ExpLogEE} the following also hold.
\begin{itemize}
\item $\Exp_{\EE} : \LLhat_z^{\ell} \to \LLhat_z^{\ell}$ is surjective.
\item $\Lambda_{\EE} \assign \ker \Exp_{\EE} \subseteq \LLhat_z^{\ell}$ satisfies
\begin{equation}
\Lambda_{\EE} = \Upsilon_{\psi,z}^{-1}\, \Lambda_{\phi}^{\oplus \ell}.
\end{equation}
\end{itemize}
Notably we see that $\EE$ is uniformizable, and for $\nu \in \sA$, $\nu \neq 0$, we have isomorphisms of $\bA$-modules,
\begin{equation} \label{E:EEnu1}
\EE[\nu] \cong \Lambda_{\EE}/\nu \Lambda_{\EE} \cong \bigl( \bA / \nu \bA \bigr)^{r \ell},
\end{equation}
where $\bA$ operates on $\Lambda_{\EE} \subseteq \Lie(\EE)(\LLhat_z)$ through scalar multiplication by $\AAA$. In particular, by Corollary~\ref{C:Atorsion}, $\EE[\nu]$ has full dimension. Furthermore,
\begin{equation} \label{E:EEnu2}
\EE[\nu] = \Upsilon_{\psi,z}^{-1} \cdot \phi^{\oplus \ell}[\nu],
\end{equation}
which also makes these isomorphisms explicit.

\subsubsection{Adjoint of $\EE$}
The adjoint $\EE^* \assign \EE^*(\phi \times \psi) : \bA \to \Mat_{\ell}(\HH[\tau])$ of $\EE$ is defined over $\HH = K^{\perf}(z)$, where $K^{\perf}$ is the perfection of $K$. By~\eqref{E:EEtdef}, 
\begin{equation}
\EE^*_t = \Upsilon_{\psi,z}^{\tr} \cdot \smash{\bigl( \phi_t^* \bigr)}^{\oplus \ell} \cdot \bigl( \Upsilon_{\psi,z}^{\tr} \bigr)^{-1}.
\end{equation}
Now for $\nu \in \sA$, $\nu \neq 0$, we have that $\phi^*[\nu] \cong (\sA/\nu\sA)^r$ (see \cite{Goss}*{\S 4.14}), and so it follows from the considerations in \S\ref{SSS:explogper} that over $\LLhat_z$ we have $(\phi^{\oplus \ell})^*[\nu] \cong (\bA/\nu\bA)^{r \ell}$. Thus by Corollary~\ref{C:Atorsion}, $(\phi^{\oplus \ell})^*[\nu]$ has full dimension. Through the isomorphism $\Upsilon_{\psi,z} : \EE \to \phi^{\oplus \ell}$, it follows that $\EE^*[\nu]$ has full dimension. Furthermore, to make things explicit,
\[
\EE^*[\nu] = \Upsilon_{\psi,z}^{\tr} \cdot \bigl(\phi^{\oplus \ell} \bigr){}^*[\nu].
\]

\subsubsection{Characteristic polynomials of Frobenius}
Let $f \in A_+$ be irreducible of degree $d$, and let $\lambda \in \sA_+$ be irreducible so that $\lambda(\theta) \neq f$. Let
\[
\rho_{\phi,\lambda} : \Gal(K^{\sep}/K) \to \Aut (T_{\lambda}(\phi)) \cong \GL_r(\sA_{\lambda})
\]
be the Galois representation associated to $T_{\lambda}(\phi)$, and similarly define $\rho_{\psi,\lambda} : \Gal(K^{\sep}/K) \to \Aut(T_{\lambda}(\psi))$ for $\psi$. As outlined in \S\ref{SS:munu}, if $\alpha_f \in \Gal(K^{\sep}/K)$ is a Frobenius element for~$f$, then because $\phi$ and $\psi$ have good reduction at~$f$,
\[
\Char(\alpha_f,T_{\lambda}(\phi),X)|_{t=\theta} = P_{\phi,f}(X), \quad
\Char(\alpha_f,T_{\lambda}(\psi),X)|_{t=\theta} = P_{\psi,f}(X),
\]
both of which lie in $A[X]$. Now $\lambda$ is also irreducible in $\bA = \FF_q(z)[t]$, and from \eqref{E:EEnu1},
\[
T_{\lambda}(\EE) = \varprojlim \EE[\lambda^m] \cong \bA_{\lambda}^{r\ell},
\]
which induces another Galois representation
\[
\rho_{\EE,\lambda} : \Gal(K^{\sep}/K) \to \Aut(T_{\lambda}(\EE)) \cong \GL_{r\ell}(\bA_{\lambda}).
\]
We have the following result that will be fundamental for analyzing $L(\EE^{\vee},0)$ in \S\ref{S:Lseries}--\S\ref{S:convolutions}. Recall the notation $(P\otimes Q)(X)$ from Definition~\ref{D:polytensor}.

\begin{proposition} \label{P:EEcharpolys}
Let $f \in A_+$ be irreducible, $f \neq \theta$, and let $\lambda \in \sA_+$ be irreducible so that $\lambda(\theta) \neq f$. For the $t$-module $\EE = \EE(\phi \times \psi)$, let
\begin{align*}
P_{\phi,f}(X) &= \Char(\tau^d,T_{\lambda}(\ophi),X)|_{t=\theta} \in A[X], \\
P_{\psi,f(z)}^{\vee}(X) &= \Char(\tau^d,T_{\lambda}(\opsi)^{\vee},X)|_{t=z} \in \FF_q(z)[X],
\end{align*}
as in \eqref{E:PfandPfvee}. Then
\[
\Char(\alpha_f,T_{\lambda}(\EE(\phi\times \psi)),X)|_{t=\theta} = \bigl( P_{\phi,f} \otimes P_{\psi,f(z)}^{\vee} \bigr)(X) \in \AAA[X].
\]
\end{proposition}

\begin{proof}
By \eqref{E:EEnu2}, for $\alpha \in \Gal(K^{\sep}/K)$ and $(\bv_1, \dots, \bv_\ell)^{\tr} \in T_{\lambda}(\phi^{\oplus \ell})$, we have
\begin{equation} \label{E:alphaaction}
\alpha \left( \Upsilon_{\psi,z}^{-1} \cdot \begin{pmatrix} \bv_1 \\ \vdots\\ \bv_\ell \end{pmatrix} \right) = \alpha(\Upsilon_{\psi,z})^{-1} \cdot \begin{pmatrix} \alpha(\bv_1) \\ \vdots\\ \alpha(\bv_\ell) \end{pmatrix}.
\end{equation}
As in \S\ref{SS:munu}, the characteristic polynomial of $\alpha_f$ acting on $T_{\lambda}(\phi)$ coincides with $P_{\phi,f}(X)$ and is independent of the choice of $\lambda$, so we can take $\lambda = t$. If we write $\rho_{\psi,z}(\alpha) \assign \rho_{\psi,t}(\alpha)|_{t=z} \in \GL_{\ell}(\power{\FF_q}{z})$, then we see from \cite{CP12}*{Cor.~3.2.4} that
\[
\alpha(\Upsilon_{\psi,z}) = \rho_{\psi,z}(\alpha) \cdot \Upsilon_{\psi,z},
\]
having chosen a basis of $T_{t}(\psi)$ appropriately as in \cite{CP12}*{\S 3.2}. Continuing with \eqref{E:alphaaction},
\[
\rho_{\EE,t}(\alpha) = \left( \begin{array}{c|c|c}
&& \\[-10pt]
(\rho_{\psi,z}(\alpha)^{-1})_{11} \cdot \rho_{\phi,t}(\alpha) & \phantom{\cdots} \cdots \phantom{\cdots} & (\rho_{\psi,z}(\alpha)^{-1})_{1\ell} \cdot \rho_{\phi,t}(\alpha) \\[5pt] \hline && \\[-10pt]
\vdots &  & \vdots \\[5pt] \hline && \\[-5pt]
(\rho_{\psi,z}(\alpha)^{-1})_{\ell 1} \cdot \rho_{\phi,t}(\alpha) &  \cdots & (\rho_{\psi,z}(\alpha)^{-1})_{\ell\ell} \cdot \rho_{\phi,t}(\alpha) 
\end{array}
\right),
\]
which we take in $\GL_{r\ell}(\power{\FF_q}{t,z})$. Thus we have an isomorphism,
\[
\rho_{\EE,t} \cong \rho_{\psi,z}^{-1} \otimes \rho_{\phi,t},
\]
of representations over $\power{\FF_q}{t,z}$. From this we see that 
\begin{align*}
\Char(\alpha_f,T_{t}(\EE),X) &= \Char(\alpha_f,T_{t}(\psi)^{\vee},X)|_{t=z} \otimes \Char(\alpha_f,T_{t}(\phi),X) \\
&= P_{\psi,f(z)}^{\vee} (X) \otimes P_{\phi,f}(X). \qedhere
\end{align*}
\end{proof}

\begin{definition} \label{D:EEcharpolys}
For $f \in A_+$ irreducible (including $f=\theta$) and $\EE = \EE(\phi \times \psi)$ as above, Proposition~\ref{P:EEcharpolys} prompts us to define
\[
\bP_f(X) \assign \bigl( P_{\phi,f} \otimes P_{\psi,f(z)}^{\vee} \bigr)(X) \in \AAA[X],
\quad
\bP_f^{\vee}(X) \assign \bigl( P_{\phi,f}^{\vee} \otimes P_{\psi,f(z)}\bigr)(X) \in K[z][X].
\]
If we need to emphasize the dependence on $\phi$ and $\psi$, we will write $\bP_{\phi \times \psi,f}(X) = \bP_{f}(X)$.
We will see in \S\ref{SS:EEmodf} that, for $f \neq \theta$, $\bP_f(X)$ coincides with the characteristic polynomial of $\tau^d$ acting on the $\lambda$-adic Tate module of $\EE$ modulo~$f$, and $\bP_f^{\vee}(X)$ arises from $\tau^d$ acting on the dual. When $f=\theta$, its role in determining the $\AAA$-order of $\EE(\FF_\theta(z))$ is checked directly (see Theorem~\ref{T:EEAord}). We do gain a connection between $\EE(\phi \times \psi)$ and $\EE(\psi \times \phi)$, in that
\begin{equation} \label{E:swap}
\bP_{\psi \times \phi,f}(X) = \bP_{\phi \times \psi,f}^{\vee}(X)\big|_{z \leftrightarrow \theta}\, , \quad
\bP_{\psi \times \phi,f}^{\vee}(X) = \bP_{\phi \times \psi,f}(X)\big|_{z \leftrightarrow \theta}\, ,
\end{equation}
where ``$z\leftrightarrow \theta$'' indicates that the roles of $z$ and $\theta$ have been swapped.
\end{definition}

\begin{example} \label{Ex:Carlitz3}
\emph{Twisting a Drinfeld module by the Carlitz module}. Let $\EE = \EE(\sC \times \sC)$, where $\sC$ is the Carlitz module. Then $\EE$ has dimension~$1$ and rank~$1$, and by \eqref{E:EEtdef},
\[
\EE_t = \omega_z^{-1} \cdot \sC_t \cdot \omega_z = \theta + (z-\theta)\tau.
\]
Thus $\EE$ was studied extensively by Angl\`{e}s, Pellarin, and Tavares Ribeiro~\cites{AnglesPellarin15,APT16}, as well as its multivariable versions. In particular, $\EE[\nu] = \omega_z^{-1} \sC[\nu]$ for each $\nu \in \sA$. We know that $P_{\sC,f}(X) = X - f$ (e.g., see \cite{Gekeler91}*{Ex.~5.11}), and so by Proposition~\ref{P:EEcharpolys},
\begin{align} \label{E:PCxC}
\bP_{\sC \times \sC,f}(X) &=  (X - f) \otimes \biggl( X - \frac{1}{f(z)} \biggr) = X - \frac{f}{f(z)}, \\
\bP_{\sC \times \sC,f}^{\vee}(X) &= \biggl( X - \frac{1}{f} \biggr) \otimes (X - f(z))  = X - \frac{f(z)}{f}. \notag
\end{align}

For $\phi$ a Drinfeld module as in \eqref{E:phidef}, we can also form $\EE = \EE(\phi \times \sC)$, given by
\[
\EE_t = \omega_z^{-1} \cdot \phi_t \cdot \omega_z = \theta + \kappa_1(z-\theta) \tau + \cdots + \kappa_r (z-\theta) \cdots (z-\theta^{q^{r-1}}) \tau^r.
\]
Then $\EE$ is a deformation of $\phi$ defined by Angl\`{e}s and Tavares Ribeiro~\cite{AT17}*{\S 3} and also investigated by Gezmi\c{s}~\cite{Gezmis19}. A short calculation shows
\begin{align} \label{E:PphixC}
\bP_{\phi \times \sC,f}(X) &= P_{\phi,f}(X) \otimes \biggl(X - \frac{1}{f(z)} \biggr) = \frac{P_{\phi,f}(f(z)X)}{f(z)^r}, \\
\bP_{\phi \times \sC,f}^{\vee}(X) &= P_{\phi,f}^{\vee}(X) \otimes (X - f(z)) = f(z)^{r} P_{\phi,f}^{\vee}\biggl( \frac{X}{f(z)} \biggr). \notag
\end{align}
\end{example}

\begin{example}
\emph{Twisting the Carlitz module by a Drinfeld module}. Letting $\phi : \sA \to A[\tau]$ be a Drinfeld module of rank~$r$ defined as in \eqref{E:phidef}, we take $\EE = \EE(\sC \times \phi)$, given by
\[
\EE_t = \Upsilon_{\phi,z}^{-1} \cdot \sC_t^{\oplus r} \cdot \Upsilon_{\phi,z} = \theta \rI_r + \Theta_{\phi,z} \tau^r \in \Mat_{r}(\AAA[\tau]).
\]
It follows from~\eqref{E:swap} and~\eqref{E:PphixC} that
\begin{equation} \label{E:PCxphi}
\bP_{\sC \times \phi,f}(X) = f^{r} P_{\phi,f(z)}^{\vee}\biggl( \frac{X}{f} \biggr), \quad
\bP_{\sC \times \phi,f}^{\vee}(X) = \frac{P_{\phi,f(z)}(fX)}{f^r}.
\end{equation}
\end{example}

\subsection{Reduction modulo \texorpdfstring{$f$}{f} and \texorpdfstring{$\mathbb{A}$}{A}-orders} \label{SS:EEmodf}
We continue with the notation of the previous section, and in particular have Drinfeld modules $\phi$ and $\psi$ defined as in~\eqref{E:phidef} and~\eqref{E:psidef} and $\EE = \EE(\phi \times \psi)$ as in \eqref{E:EEtdef}. We fix $f \in A_+$ irreducible of degree~$d$. We let
$\ophi : \sA \to \FF_f[\tau]$, $\opsi : \sA \to \FF_f[\tau]$,
denote the reductions modulo $f$. As both have everywhere good reduction, we see that $\ophi$ has rank~$r$ and $\opsi$ rank~$\ell$. Likewise, the entries of coefficients of $\EE_t$ are all in $\AAA$, so we can form the reduction with entries in $\AAA/f\AAA \cong \FF_f(z)$,
\[
\oEE : \bA \to \Mat_{\ell}(\FF_f(z)[\tau]).
\]
In general ``$\overline{\beta}$'' will denote reduction of an element or object $\beta$ modulo~$f$, and so,
\begin{equation} \label{E:oEEtdef}
\oEE_t = \otheta \rI_{\ell} + \okappa_1 \oTheta_{\psi,z} \tau + \cdots + \okappa_r
\oTheta_{\psi,z} \oTheta_{\psi,z}^{(1)} \cdots \oTheta_{\psi,z}^{(r-1)} \tau^r,
\end{equation}
where $\okappa_r = \kappa_r \in \FF_q^{\times}$. We note from \eqref{E:Gammadef} that $\det \oTheta_{\psi,z} = (-1)^{\ell}(z-\otheta)/\eta_{\ell} \neq 0$, and so $\oTheta_{\psi,z} \in \GL_{\ell}(\FF_f(z))$. Thus as in~\S\ref{SSS:AbAfin} we see that $\cM_{\oEE}(\FF_f(z))$ and $\cN_{\oEE}(\FF_f(z))$ both have rank $r\ell$ as $\FF_f(z)[t]$-modules. In this sense, $\EE$ also has everywhere good reduction.

Our main result in this section is the following theorem for determining $\Aord{\oEE(\FF_f(z))}{\AAA} = \Aord{\oEE(\FF_f(z))}{\bA} \big|_{t=\theta}$ in terms of the value
$\bP_{f}(1)$, where $\bP_f(X) \in \AAA[X]$ is taken from Definition~\ref{D:EEcharpolys}. Define completely multiplicative functions $\chi_{\phi}$, $\chi_{\psi} : \sA_+ \to \FF_q^{\times}$ as in~\eqref{E:chidef}.

\begin{theorem} \label{T:EEAord}
Let $f \in \sA_+$ be irreducible. For $\oEE : \bA \to \Mat_{\ell}(\FF_f(z)[\tau])$ defined above,
\[
\bigAord{\soEE(\FF_f(z))}{\AAA} = (-1)^{r\ell} \chi_{\phi}(f)^{\ell}\, \ochi_{\psi}(f)^r\cdot f(z)^r \cdot \bP_f(1) = \frac{\bP_f(1)}{\bP_f(0)} \cdot f^{\ell}.
\]
\end{theorem}

We recall that $\ochi_{\psi}$ is the multiplicative inverse of $\chi_{\psi}$ and \emph{not} the reduction modulo~$f$. This conflict of notation is isolated to the characters $\chi_{\phi}$ and $\chi_{\psi}$ and should not cause much confusion. The proof of this theorem takes the rest of the section and is split into the two cases where $f \neq \theta$ and $f = \theta$.

Assume for the time being that $f \neq \theta$. Let $K_f^{\nr}$ be the maximal unramified and separable extension of $K_f$, and let $K^{\nr}_f \supseteq O^{\nr}_f \supseteq M^{\nr}_f$ be its subring of $f$-integral elements and its maximal ideal. Because $\phi$ and $\psi$ have everywhere good reduction, we see from \cite{Takahashi82}*{Thm.~1} (see also \cite{Goss}*{Thm.~4.10.5}) that
\[
\phi[t^m], \ \psi[t^m] \subseteq O^{\nr}_f, \quad \forall\, m \geqslant 1.
\]
Moreover, the natural reduction maps
\begin{equation} \label{E:phipsitorsred}
\phi(O^{\nr}_f)[t^m] = \phi[t^m] \iso \ophi[t^m], \quad \psi(O^{\nr}_f)[t^m] = \psi[t^m] \iso \opsi[t^m],
\quad \forall\, m \geqslant 1,
\end{equation}
are $\sA$-module isomorphisms (e.g., see \cite{Takahashi82}*{\S 2}).
Thus for each $i=1, \dots, \ell$, the Anderson generating functions $g_1, \dots, g_{\ell}$ associated to $\psi$ as in Example~\ref{Ex:Drinfeld2} each satisfy $g_i \in \power{O^{\nr}_f}{t}$. Furthermore, since $\det \Upsilon_{\psi}^{(1)} = c(t-\theta) \det \Upsilon_{\psi}$, it follows that $\det \Upsilon_{\psi} = c' \omega$ for some $c' \in \oFF_q$ (e.g., see \citelist{ \cite{EP14}*{\S 7} \cite{GezmisP19}*{Eq.~(6.3.2)} \cite{NamoijamP24}*{Prop.~4.48}}). By the formulation $\omega = \sum_{m\geqslant 0} \exp_{\sC}(\tpi/\theta^{m+1}) t^m$ in \eqref{E:omegadef}, it follows that $\omega \in \power{O^{\nr}_f}{t}^{\times}$, since $\exp_{\sC}(\tpi/\theta) = (-\theta)^{1/(q-1)} \in (O^{\nr}_f)^{\times}$. It follows that
\[
\Upsilon_{\psi,z} \in \GL_{\ell} \bigl( \power{O^{\nr}_f}{z} \bigr),
\]
and therefore from~\eqref{E:EEnu2},
\begin{equation} \label{E:EEtm}
\EE[t^m] = \Upsilon_{\psi,z}^{-1} \cdot \phi^{\oplus \ell}[t^m] \subseteq \EE\bigl( \laurent{O^{\nr}_f}{z} \bigr).
\end{equation}
By \eqref{E:phipsitorsred}, for each $m \geqslant 1$ we have an isomorphism $\phi^{\oplus \ell}(\laurent{O^{\nr}_f}{z})[t^m] \iso \smash{\ophi}^{\oplus \ell}(\laurent{\oFF_f}{z})[t^m]$ of $\laurent{\FF_q}{z}[t]$-modules, and the following diagram of $\laurent{\FF_q}{z}[t]$-modules commutes:
\begin{equation} \label{E:Ofreduction}
\begin{tikzcd}[column sep=large]
 \EE\bigl( \laurent{O^{\nr}_f}{z} \bigr)[t^m] \arrow{r} \arrow{d}{\Upsilon_{\psi,z}} & \oEE \bigl( \laurent{\oFF_f}{z} \bigr)[t^m] \arrow{d}{\oUpsilon_{\psi,z}} \\
\phi^{\oplus \ell}\bigl( \laurent{O^{\nr}_f}{z} \bigr)[t^m] \arrow{r}{\sim} & \smash{\ophi}^{\oplus \ell} \bigl( \laurent{\oFF_f}{z} \bigr)[t^m].
\end{tikzcd}
\end{equation}
The left-hand column is an isomorphism from~\eqref{E:EEnu2}. The map $\oUpsilon_{\psi,z}: \oEE(\laurent{\oFF_f}{z}) \to \smash{\ophi}^{\oplus \ell}(\laurent{\oFF_f}{z})$ is an isomorphism since $\Upsilon_{\psi,z} \in \GL_{\ell}(\power{O^{\nr}_f}{z})$, and so the inverse restricted to $t^m$-torsion, $\smash{\oUpsilon}_{\psi,z}^{-1}: \smash{\ophi}^{\oplus \ell}(\laurent{\oFF_f}{z})[t^m] \to \oEE(\laurent{\oFF_f}{z})[t^m]$, is also an isomorphism. Taking $\cA = \laurent{\FF_q}{z}[t]$ in Corollary~\ref{C:Atorsion} we see that
\[
\dim_{\laurent{\FF_q}{z}} \smash{\ophi}^{\oplus \ell}(\laurent{\oFF_f}{z})[t^m] = r \ell m,
\]
and so also $\dim_{\laurent{\FF_q}{z}} \oEE(\laurent{\oFF_f}{z})[t^m] = r \ell m$. Thus $\oEE(\laurent{\oFF_f}{z})[t^m]$ has full dimension. Moreover, Corollary~\ref{C:Atorsion} implies
\[
\oEE(\laurent{\oFF_f}{z})[t^m] \cong \biggl( \frac{\laurent{\FF_q}{z}[t]}{t^m \laurent{\FF_q}{z}[t]} \biggr)^{\oplus \ell},
\]
and each of the maps in \eqref{E:Ofreduction} is an isomorphism of $\laurent{\FF_q}{z}[t]$-modules. Since $\phi$ and $\psi$ are defined over $A$ and $\EE$ is defined over $\AAA$, it follows that these maps also commute with the $\Gal(\oFF_f/\FF_f)$-action.

\begin{proof}[Proof of Theorem~\ref{T:EEAord}. Case $f \neq \theta$]
By the discussion above, $\oEE(\laurent{\oFF}{z})[t^m]$ has full dimension for each $m \geqslant 1$. Therefore, Theorem~\ref{T:charpoly1} implies that
\[
\bigAord{\soEE(\laurent{\FF_f}{z})}{\laurent{\FF_q}{z}[t]} = \gamma \cdot
\Char(\tau^d, T_t(\oEE),1),
\]
where $\gamma \in \laurent{\FF_q}{z}^{\times}$ is chosen to make the expression monic in~$t$. On the other hand, because the maps in~\eqref{E:Ofreduction} commute with the Galois action, if we let $\alpha_f \in \Gal(K^{\sep}/K)$ be a Frobenius element, we have
\[
\Char(\tau^d, T_{t}(\oEE),X) = \Char(\alpha_f, T_{t}(\EE), X).
\]
Combining these findings with Proposition~\ref{P:EEcharpolys}, we see that
\begin{equation} \label{E:oEElaurent1}
\bigAord{\soEE(\laurent{\FF_f}{z})}{\laurent{\FF_q}{z}[\theta]} = \gamma \cdot \bP_f(1).
\end{equation}
Now as $\bP_f(X) = (P_{\phi,f} \otimes P_{\psi,f(z)}^{\vee})(X)$, it follows from \eqref{E:PfandPfvee} and Definition~\ref{D:polytensor} that the constant term of $\bP_{f}(X)$ is
\begin{equation} \label{E:bPf0}
\bP_f(0) = (-1)^{r \ell} \bigl( \ochi_{\phi}(f) f \bigr)^{\ell} \cdot \biggl( \frac{\chi_{\psi}(f)}{f(z)} \biggr)^r.
\end{equation}
Writing $\bP_f(X) = \sum_{i=0}^{r\ell} b_i X^i$, $b_i \in \bA[X]$ and
letting $c_0, \dots, c_{r-1} \in A$ be given as in~\eqref{E:Pfdef}, Definition~\ref{D:polytensor} implies that, for $0 \leqslant m \leqslant r\ell -1$, each $b_m$ is a polynomial in $c_0, \dots, c_{r-1}$ with coefficients in $\FF_q(z)$.
Assigning the weight $r-i$ to each $c_i$, then as formal expressions, each monomial in $c_0, \dots, c_{r-1}$ in $b_m$ has the same total weight $r\ell - m$. That is, if $c_0^{n_0} \dots c_{r-1}^{n_{r-1}}$ is a monomial in $b_m$, then
$\sum_{i=0}^{r-1} (r-i)n_i = r\ell - m$, and so by~\S\ref{SSS:Pf},
\[
\deg_{\theta} \bigl(c_0^{n_0} \cdots c_{r-1}^{n_{r-1}} \bigr) \leqslant \sum_{i=0}^{r-1} \frac{d}{r} \cdot n_i(r-i) = \frac{d}{r} (r \ell - m) = d \ell - \frac{dm}{r}.
\]
From~\eqref{E:bPf0}, this is an equality if $m=0$. On the other hand, this inequality implies,
\[
 0 < m \leqslant r\ell -1 \quad \Rightarrow \quad \deg_{\theta} b_m < d \ell.
\]
Therefore from \eqref{E:bPf0}, $\gamma = (-1)^{r \ell} \chi_{\phi}(f)^{\ell} \ochi_{\psi}(f)^r \cdot f(z)^r$. By \eqref{E:oEElaurent1} it remains to verify that $\Aord{\soEE(\laurent{\FF_f}{z}}{\laurent{\FF_q}{z}[t]} = \Aord{\soEE(\FF_f(z))}{\bA}$. However, suppose $\oEE(\FF_f(z)) \cong \bA/h_1\bA \oplus \cdots \oplus \bA/h_s \bA$ for monic $h_1, \dots, h_s \in \bA$. Since $\laurent{\FF_f}{z} \cong \laurent{\FF_q}{z} \otimes_{\FF_q} \FF_f \cong \laurent{\FF_q}{z} \otimes_{\FF_q(z)} \FF_f(z)$ as $\laurent{\FF_q}{z}$-vector spaces, we have an isomorphism of $\laurent{\FF_q}{z}[t]$-modules,
\[
\oEE(\laurent{\FF_f}{z}) \cong \laurent{\FF_q}{z} \otimes_{\FF_q(z)} \oEE(\FF_f(z)).
\]
It follows that
\[
\oEE(\laurent{\FF_f}{z}) \cong \frac{\laurent{\FF_q}{z}[t]}{h_1 \laurent{\FF_q}{z}[t]} \oplus \cdots \oplus \frac{\laurent{\FF_q}{z}[t]}{h_s \laurent{\FF_q}{z}[t]},
\]
as desired.
\end{proof}

In the case that $f=\theta$, the previous argument does not work because among other issues, (a) the coefficients of the Anderson generating functions $g_1, \dots, g_r$ in Example~\ref{Ex:Drinfeld2} may have some coefficients that are not $\theta$-integral and (b) furthermore $\omega_z \notin \power{O^{\nr}_{\theta}}{z}^{\times}$. However, this case can be checked directly.

\begin{proof}[Proof of Theorem~\ref{T:EEAord}. Case $f = \theta$]
In this case, $\otheta = 0$ and so
\[
\oGamma_{\phi} =\begin{pmatrix}
0 & 1 & \cdots & 0 \\
\vdots & \vdots & \ddots & \vdots \\
0 & 0 & \cdots & 1 \\
\okappa_r^{-1} t & -\okappa_1\okappa_r^{-1} & \cdots & -\kappa_{r-1} \okappa_r^{-1}
\end{pmatrix},
\quad
\oTheta_{\psi,z} = \begin{pmatrix}
0 & \cdots & 0 & z\oeta_{\ell}^{-1} \\
1 & \cdots & 0 & -\oeta_1 \oeta_\ell^{-1} \\
\vdots & \ddots & \vdots & \vdots \\
0 & \cdots & 1 & -\oeta_{\ell-1} \oeta_\ell^{-1}
\end{pmatrix},
\]
which are in $\Mat_{\ell}(\FF_q[t])$ and $\Mat_{\ell}(\FF_q[z])$ respectively. By Corollary~\ref{C:charpolys},
\[
P_{\phi,t}(X) = \Char(\oGamma_{\phi},X), \quad P_{\psi,z} = \Char(\oTheta_{\psi,z},X),
\]
where $P_{\phi,t}(X) = P_{\phi,\theta}(X)|_{\theta=t}$. Since $\oTheta_{\psi,z}^{(1)} = \oTheta_{\psi,z}$,
we have $\oTheta_{\psi,z} \smash{\oTheta}_{\psi,z}^{(1)} \cdots \smash{\oTheta}_{\psi,z}^{(i-1)} = \smash{\oTheta}_{\psi,z}^{i}$ for each $i \geq 1$, and so by \cite{NamoijamP24}*{Eq.~(4.130)},
if we let
\[
\Gamma \assign \begin{pmatrix}
0 & \rI_{\ell} & \cdots & 0 \\
\vdots & \vdots & \ddots & \vdots \\
0 & 0 & \cdots & \rI_{\ell} \\
\okappa_r^{-1} t \oTheta_{\psi,z}^{-r} & -\okappa_1 \okappa_r^{-1} \oTheta_{\psi,z}^{-r + 1} & \cdots & -\okappa_{r-1} \okappa_r^{-1} \oTheta_{\psi,z}^{-1}
\end{pmatrix},
\]
then $\Gamma$ represents multiplication by $\tau$ on $\cM_{\oEE}(\FF_q(z))$ as in \S\ref{SSS:AbAfin}. Furthermore, since $d=1$, we have $\rG = \Gamma$, and so by Corollary~\ref{C:kerandkerstar} and equations \eqref{E:Estarorder} and \eqref{E:detI-G},
\[
\bigAord{\soEE(\FF_q(z))}{\bA} = \gamma \cdot \Char(\Gamma,1),
\]
where $\gamma \in \FF_q(z)^{\times}$ is chosen to make this expression monic in~$t$. If we take the block diagonal matrix $B = \diag(\rI_{\ell}, \oTheta_{\psi,z}, \dots, \oTheta_{\psi,z}^{r-1}) \in \GL_{r\ell}(\FF_q(z))$, then one checks
\[
B \Gamma B^{-1} = \begin{pmatrix}
0 & \oTheta_{\psi,z}^{-1} & \cdots & 0 \\
\vdots & \vdots & \ddots & \vdots \\
0 & 0 & \cdots & \oTheta_{\psi,z}^{-1} \\
\okappa_r^{-1} t \oTheta_{\psi,z}^{-1} & -\okappa_1 \okappa_r^{-1} \oTheta_{\psi,z}^{-1} & \cdots & -\okappa_{r-1} \okappa_r^{-1} \oTheta_{\psi,z}^{-1}
\end{pmatrix}
= \oGamma_{\phi} \otimes \oTheta_{\psi,z}^{-1}.
\]
Thus, $\Char(\Gamma,X) = \Char( \oGamma_{\phi} \otimes \oTheta_{\psi,z}^{-1}, X)$,
and so $\Aord{\soEE(\FF_q(z))}{\bA} = \gamma \cdot (P_{\phi,t} \otimes P_{\psi,z}^{\vee})(1)$. Verifying that $\gamma = (-1)^{r\ell} \chi_{\phi}(f)^{\ell} \ochi(f)^{r}$ is exactly the same as in the $f \neq \theta$ case.
\end{proof}

\section{Convolutions of Goss and Pellarin \texorpdfstring{$L$}{L}-series} \label{S:Lseries}

In a series of articles \citelist{\cite{Goss79} \cite{Goss83} \cite{Goss92} \cite{Goss}*{Ch.~8}}, Goss defined and investigated function field valued $L$-series attached to Drinfeld modules and $t$-modules defined over finite extensions of~$K$. These $L$-functions possess a rich structure of special values, initiated by Carlitz~\cite{Carlitz35}*{Thm.~9.3} for the eponymous Carlitz zeta function and continued by Goss~\citelist{\cite{Goss92} \cite{Goss}*{Ch.~8}}. Anderson and Thakur~\cite{AndThak90} further revealed the connection between Carlitz zeta values and coordinates of logarithms on tensor powers of the Carlitz module.

Taelman \cites{Taelman09, Taelman10, Taelman12} discovered a breakthrough on special $L$-values for Drinfeld modules that related them to the product of an analytic regulator and the $A$-order of a class module. These results have been extended in several directions, including to $t$-modules defined over $\oK$ and more refined special value identities~\cites{ANT20, ANT22, AnglesTaelman15, Beaumont23, FGHP22, CEP18, Fang15, Gezmis21, GezmisNamoijam26}.

In~\cite{Pellarin12}, Pellarin introduced a new class of $L$-functions that are deformations of the Carlitz zeta function in additional variables and take values in Tate algebras. Results on Pellarin $L$-series and their special values have been investigated extensively~\cites{ANT17a, ANT17b, AnglesPellarin14, AnglesPellarin15, APT16, APT18, AT17, Gezmis19, Gezmis20, GezmisPellarin22, Goss13, GreenP18, PellarinPerkins16, PellarinPerkins22, Perkins14, Tavares21}. Important for the present paper is the work of Demeslay~\cites{APT16, DemeslayPhD, Demeslay22}, who extended Taelman's special value formulas to $L$-series of $t$-modules over Tate algebras. See \cites{AT17, APT16, Beaumont23, Gezmis19, Gezmis20} for additional applications of Demeslay's work.

\subsection{Goss \texorpdfstring{$L$}{L}-series} \label{SS:GossL}
Let $\phi : \sA \to A[\tau]$ be a Drinfeld module over $A$ with everywhere good reduction as in~\eqref{E:phidef}. Goss~\citelist{\cite{Goss92}*{\S 3} \cite{Goss}*{\S 8.6}} associated the Dirichlet series
\begin{equation}
L(\phi^{\vee},s) = \prod_{f \in A_+,\ \textup{irred.}} Q_f^{\vee} \bigl( f^{-s} \bigr)^{-1}, \quad
L(\phi,s) = \prod_{f \in A_+, \ \textup{irred.}} Q_f \bigl( f^{-s} \bigr)^{-1}.
\end{equation}
where as in~\S\ref{SS:munu}, $Q_f(X) \in A[X]$ is the reciprocal polynomial of the characteristic polynomial $P_f(X)$ of Frobenius acting on the Tate module $T_{\lambda}(\ophi)$ and $Q_f^{\vee}(X) \in K[X]$ is the reciprocal polynomial of $P_f^{\vee}(X)$ arising from $T_{\lambda}(\ophi)^{\vee}$.
In the future we will write simply ``$\prod_{f}$'' to indicate that a product is over all irreducible $f \in A_+$.

\begin{remark} \label{R:Qfcalc}
Corollary~\ref{C:charpolys} makes the calculation of $P_f(X)$ and $P_f^{\vee}(X)$ reasonable (and hence $Q_f(X)$ and $Q_f^{\vee}(X)$ also). If we take $\Gamma \in \Mat_{r}(A[t])$ as in~\eqref{E:Gammadef}, then for $f \in A_+$ of degree~$d$, we take the reduction $\oGamma \in \Mat_{r}(\FF_f[t])$. Corollary~\ref{C:charpolys} implies
\begin{align} \label{E:charpolyscalc}
P_f(X) &= \Char\bigl( \smash{\oGamma}^{(d-1)} \cdots \smash{\oGamma}^{(1)} \oGamma, X \bigr)|_{t=\theta} \in A[X], \\
P_f^{\vee}(X) &= \Char\bigl( \oGamma^{-1} (\smash{\oGamma}^{(1)})^{-1} \cdots (\smash{\oGamma}^{(d-1)})^{-1}, X \bigr)|_{t=\theta} \in K[X]. \notag
\end{align}
\end{remark}

Since we have assumed $\phi$ has everywhere good reduction, we have avoided Euler factors at primes of bad reduction. On the other hand, bad primes also greatly complicate our convolution problem, so for the present paper we do not consider them (see Remark~\ref{R:badred}).

The bounds on the coefficients of $P_f(X)$ from \S\ref{SSS:Pf} imply that $L(\phi,s)$ converges in $K_{\infty}$ for $s \in \ZZ_+$ and that $L(\phi^{\vee},s)$ converges for $s \in \ZZ_{\geqslant 0}$ (e.g., see \cite{CEP18}*{\S 3}). Goss extended the definition of these $L$-series to $s$ in a non-archimedean analytic space, but we will not pursue these extensions here (see Remark~\ref{R:convergence}). We will henceforth assume $s \in \ZZ$.

By \eqref{E:munugen}, we find that
\begin{equation}
L(\phi^{\vee},s) = \sum_{a \in A_+} \frac{\mu_{\phi}(f)}{a^{s+1}}
\end{equation}
(see \cite{CEP18}*{Eqs.~(12)--(14)}). In particular, for the Carlitz module $P_{\sC,f}^{\vee}(X) = X - 1/f$, so
\[
L(\sC^{\vee},s) = \sum_{a \in A_+} \frac{1}{a^{s+1}} = \zeta_{\sC}(s+1)
\]
is a shift of the Carlitz zeta function.

Taelman~\cite{Taelman12}*{Thm.~1} proved a special value identity for $L(\phi^{\vee},0)$ as follows. First,
\begin{equation} \label{E:Qfvee1}
Q_f^{\vee}(1)^{-1} = \frac{f}{(-1)^r\, \ochi(f) \cdot P_f(1)} = \frac{\Aord{\FF_f}{A}}{\Aord{\ophi(\FF_f)}{A}},
\end{equation}
where the first equality follows from~\eqref{E:QfandQfvee} and the second from Gekeler~\cite{Gekeler91}*{Thm.~5.1} (and also from Corollary~\ref{C:charpoly1} combined with the definition of $P_f(X)$). We then have
\begin{equation} \label{E:Taelman}
L(\phi^{\vee},0) = \prod_f \frac{\bigAord{\FF_f}{A}}{\bigAord{\ophi(\FF_f)}{A}} = \Reg_{\phi} \cdot \rH(\phi),
\end{equation}
where the first equality follows from \eqref{E:Qfvee1} and the second is Taelman's identity.
The formula on the right contains the regulator $\Reg_{\phi} \in K_{\infty}$ and the order of the class module $\rH(\phi)\in A$ (see \cite{Taelman12} for details).
We will use Demeslay's generalization of Taelman's formula to $t$-modules over Tate algebras. See Theorem~\ref{T:Demeslay} and Remark~\ref{R:TD}.

\begin{remark}
In this paper the $L$-function $L(\phi,s)$ is defined using the Galois action on $T_{\lambda}(\ophi)$. This is consistent with previous descriptions in \cites{CEP18, Gekeler91, GezmisNamoijam26, Taelman12}, but Goss's original definition \citelist{\cite{Goss92}*{\S 3} \cite{Goss}*{\S 8.6}} expressed $L(\phi,s)$ in terms of the geometric Frobenius acting on $H^1(T_{\lambda}(\ophi),\bk_{\lambda})$. Ultimately these lead to the same $L$-functions. However, our alignment,
\[
L(\phi,s) \longleftrightarrow T_{\lambda}(\phi), \qquad
L(\phi^{\vee},s) \longleftrightarrow T_{\lambda}(\phi)^{\vee},
\]
leads to a possible notational incongruity in Taelman's formula~\eqref{E:Taelman} and Demeslay's generalization in Theorem~\ref{T:Demeslay}, where arithmetic invariants of $\phi$ are expressed in terms of $L(\phi^{\vee},0)$.
\end{remark}

\subsection{Pellarin \texorpdfstring{$L$}{L}-series} \label{SS:Pellarin}
In~\cite{Pellarin12}, Pellarin defined the series,
\begin{equation}
L(\AAA,s) \assign \sum_{a \in A_+} \frac{a(z)}{a^s} = \prod_f \biggl( 1 - \frac{f(z)}{f^s} \biggr)^{-1} \in \TT_z(K_{\infty}),
\end{equation}
which converges in~$\TT_z(K_{\infty})$ for $s \in \ZZ_+$ and is entire as a function of~$z$. Among other properties, Pellarin proved the following special value formula.

\begin{theorem}[{Pellarin~\cite{Pellarin12}*{Thm.~1}}] \label{T:Pellarin}
We have
\[
L(\AAA,1) = - \frac{\tpi}{(z-\theta) \omega_z}.
\]
\end{theorem}

We note further that if we take $\bP_f^{\vee}(X) = \bP_{\sC \times \sC,f}^{\vee}(X) = X - f(z)/f$ in~\eqref{E:PCxC}, then
\[
\bP_f^{\vee}(1)^{-1} = \frac{f}{f-f(z)} = \frac{\bigAord{\FF_f(z)}{\AAA}}{\bigAord{\soEE(\sC \times \sC)(\FF_f(z))}{\AAA}},
\]
where the calculation that $\Aord{\soEE(\sC \times \sC)(\FF_f(z))}{\AAA} = f- f(z)$ follows from~\cite{APT16}*{Lem.~5.8}. It also follows from Theorem~\ref{T:EEAord}, since $\chi_{\sC}(f) = 1$ for all $f$ and $\bP_f(1) = 1-f/f(z)$ from \eqref{E:PCxC}. Letting $\bQ_f^{\vee}(X)$ be the reciprocal polynomial of $\bP_f^{\vee}(X)$, we also find
\begin{equation} \label{E:LCxCs-1}
L(\EE(\sC \times \sC)^{\vee},s) \assign \prod_f \bQ_f^{\vee} \bigl( f^{-s} \bigr)^{-1} = L(\AAA,s+1).
\end{equation}
Thus the value  $L(\EE(\sC \times \sC)^{\vee},0) = L(\AAA,1)$ can be obtained through Theorem~\ref{T:Pellarin}. It was this type of calculation that led us to the $L$-series $L(\bsmu_{\phi,\theta} \times \bsnu_{\psi,z},s)$ in \S\ref{SS:Lphixpsi}.

\subsection{Demeslay's class module formula} \label{SS:Demeslayformula}
In \cites{DemeslayPhD, Demeslay22}, Demeslay proved an extension of Taelman's class module formula to Anderson $t$-modules defined over $\AAA$. In fact Demeslay's formula~\cite{Demeslay22}*{Thm.~2.9} applies over much more general base rings, but we will only require his identity over $\AAA$.

Let $\bE : \bA \to \Mat_{\ell}(\AAA[\tau])$ be an abelian and $\bA$-finite Anderson $t$-module defined over~$\AAA$. 
The exponential series $\Exp_{\bE} \in \power{\KK}{\tau}$ of $\bE$ induces an $\FF_q(z)$-linear function
\begin{equation}
\Exp_{\bE,\KK_{\infty}} \colon \Lie(\bE)(\KK_{\infty}) \to \bE(\KK_{\infty}) \quad \Leftrightarrow \quad
\Exp_{\bE,\KK_{\infty}} \colon \KK_{\infty}^{\ell} \to \KK_{\infty}^{\ell}.
\end{equation}
Now $\Lie(\bE)(\KK_{\infty})$ has a canonical $\KK_{\infty}$-vector space structure, but Demeslay~\cite{Demeslay22}*{\S 2.3} pointed out that it has another structure of a vector space over $\laurent{\FF_q(z)}{t^{-1}}$. Namely we extend $\pd : \bA \to \Mat_{\ell}(\KK_{\infty})$ to an $\FF_q(z)$-algebra homomorphism,
\begin{align}
\laurent{\FF_q(z)}{t^{-1}} \stackrel{\pd}{\longrightarrow} \Mat_{\ell}(\KK_\infty) \; :\; \sum_{j \geqslant j_0} c_j t^{-j} \longmapsto \sum_{j \geqslant j_0} c_j \cdot \pd\bE_t^{-j}.
\end{align}
Notably the series on the right converges by~\cite{Demeslay22}*{Lem.~2.6}. As Demeslay continued, $\Lie(\bE)(\KK_\infty)$ obtains an $\laurent{\FF_q(z)}{t^{-1}}$-vector space structure via~$\pd$. For any $g \in \laurent{\FF_q(z)}{t^{-q^{\ell}}}$, we have $\pd g = g\cdot \rI_{\ell}$, and so $\Lie(\bE)(\KK_{\infty})$ has dimension $\ell q^{\ell}$ as over $\laurent{\FF_q(z)}{t^{-q^{\ell}}}$, which implies it has dimension~$\ell$ over $\laurent{\FF_q(z)}{t^{-1}}$.

Since $\KK_{\infty} = \laurent{\FF_q(z)}{\theta^{-1}} \cong \laurent{\FF_q(z)}{t^{-1}}$, we will abuse notation and use the map~$\pd$ to define new $\KK_\infty$-vector space and $\AAA$-module structures on $\Lie(\bE)(\KK_{\infty})$ that are possibly different from scalar multiplication.
With respect to this $\KK_{\infty}$-structure, Demeslay showed~\cite{Demeslay22}*{Prop.~2.7} that $\Lie(\bE)(\AAA) \subseteq \Lie(\bE)(\KK_{\infty})$ is an $\AAA$-lattice and that the standard basis $\{ \bs_1, \dots, \bs_{\ell} \}$ of $\Lie(\bE)(\KK_{\infty})$, when equated with $\KK_{\infty}^{\ell}$, is an $\AAA$-basis of $\Lie(\bE)(\AAA)$ via~$\pd$.
Demeslay further proved~\cite{Demeslay22}*{Prop.~2.8} that
\begin{equation} \label{E:bElattice}
\Exp_{\bE,\KK_{\infty}}^{-1} \bigl( \bE(\AAA) \bigr) \subseteq \Lie(\bE)(\KK_{\infty})
\end{equation}
is an $\AAA$-lattice as in Definition~\ref{D:lattice}. In particular it has rank~$\ell$ as an $\AAA$-module via~$\pd$.

\begin{remark}
For applications in the present paper, all $t$-modules will satisfy $\pd \bE_t = \theta \cdot \rI_{\ell}$, and so the $\KK_{\infty}$-vector space structure on $\Lie(\bE)(\KK_{\infty})$ and the $\AAA$-module structure on $\Exp_{\bE,\KK_{\infty}}^{-1}(\bE(\AAA))$ will be induced by the usual scalar multiplication.
\end{remark}

Choose an $\AAA$-basis $\{ \bslambda_1, \dots, \bslambda_{\ell} \}$ of $\Exp_{\bE,\KK_{\infty}}^{-1}(\AAA)$ via $\pd$, and let $V \in \GL_{\ell}(\KK_{\infty})$ be chosen so that its columns are the coordinates of $\bslambda_1, \dots, \bslambda_{\ell}$ with respect to $\bs_1, \dots, \bs_{\ell}$ (via $\pd$). Following Taelman~\cites{Taelman10, Taelman12}, Demeslay defined the regulator of $\bE$ as
\begin{equation}
\Reg_{\bE} \assign \gamma \cdot \det(V) \in \KK_{\infty},\quad \gamma \in \FF_q(z)^{\times},
\end{equation}
where $\gamma$ is chosen so that $\Reg_{\bE}$ has sign~$1$ (leading coefficient as an element of $\laurent{\FF_q(z)}{\theta^{-1}}$ is~$1$).
This value is independent of the choice of $\AAA$-basis.

\begin{remark} \label{R:RegLog1}
If $\pd \bE_t = \theta \rI_{\ell}$ and the standard basis vectors $\bs_1, \dots, \bs_{\ell} \in \bE(\KK_{\infty}) = \KK_{\infty}^{\ell}$ fall within the domain of convergence of $\Log_{\bE}(\bz)$, then there is $\gamma \in \FF_q(z)^{\times}$ so that
\begin{equation} \label{E:RegLog1}
\Reg_{\bE} = \gamma \cdot \det \Bigl( \Log_{\bE}(\bs_1), \dots, \Log_{\bE}(\bs_\ell) \Bigr) \rassign \gamma \cdot \det\bigl( \Log_{\bE}(\rI_\ell) \bigr).
\end{equation}
If  $\bE = \EE = \EE(\phi \times \psi)$ as in \S\ref{SS:EEprops}, then under these conditions we have
\begin{equation} \label{E:EERegLog1}
\Reg_{\EE} = \det\bigl( \Log_{\EE}(\rI_\ell) \bigr) = \det \Bigl( \Upsilon_{\psi,z}^{-1} \Log_{\phi} \bigl(\Upsilon_{\psi,z}\bigr) \Bigr).
\end{equation}
Indeed, we have that $\Reg_{\EE} = \gamma \cdot \det\bigl(\Log_{\EE}(\rI_{\ell})\bigr)$ for some $\gamma \in \FF_q(z)^{\times}$ that forces the expression to be monic. However, in this case $\Log_{\EE}(\rI_{\ell}) = \Upsilon_{\psi,z}^{-1} \Log_{\phi} (\Upsilon_{\psi,z})$, and then we must have that the entries of $\Upsilon_{\psi,z}$ all fall within the radius of convergence $R_{\phi}$ of $\Log_{\phi}$ as in~\eqref{E:Rphi}. It is shown in \citelist{\cite{EP13}*{Cor.~4.2} \cite{EP14}*{Prop.~6.10} \cite{KhaochimP23}*{Cor.~4.5}} that in this case the term of $\Log_{\phi}(\Upsilon_{\psi,z})$ of greatest $\dnorm{\,\cdot\,}$-norm is uniquely the first term $\Upsilon_{\psi,z}$. Therefore, the term of $\Log_{\EE}(\rI_{\ell})$ of greatest $\dnorm{\,\cdot\,}$-norm is uniquely the first term $\rI_{\ell}$, forcing $\gamma=1$.
On the other hand, when $\Upsilon_{\psi,z}$ is not within the radius of convergence of $\Log_{\phi}(z)$, the determination of $\Reg_{\EE}$ can be subtle (see \S\ref{SSS:bigUps}).
\end{remark}

Also following Taelman, Demeslay~\cite{Demeslay22}*{Prop.~2.8} defined the class module of $\bE$ as
\begin{equation} \label{E:bEclass}
\rH(\bE) \assign \frac{\bE(\KK_{\infty})}{\Exp_{\bE,\KK_{\infty}}(\Lie(\bE)(\KK_{\infty})) + \bE(\AAA)},
\end{equation}
and he proved that $\rH(\bE)$ is a finitely generated and torsion $\AAA$-module. Demeslay's class module formula is the following.

\begin{theorem}[{Demeslay~\citelist{\cite{DemeslayPhD}*{Thm.~2.1.9} \cite{Demeslay22}*{Thm.~2.9}}}] \label{T:Demeslay}
Let $\bE : \bA \to \Mat_{\ell}(\AAA[\tau])$ be an Anderson $t$-module. Then
\[
\prod_f \frac{\Aord{\Lie(\sobE)(\FF_f(z))}{\AAA}}{\Aord{\sobE(\FF_f(z))}{\AAA}} = \Reg_{\bE} \cdot \bigAord{\rH(\bE)}{\AAA},
\]
where the left-hand side converges in $\KK_\infty$.
\end{theorem}

\begin{remark} \label{R:TD}
Suppose $\phi : \sA \to A[\tau]$ is a Drinfeld module, which we consider to be a constant Drinfeld module over $\AAA$ as in Theorem~\ref{T:constant}. Then $\Reg_{\phi}$ and $\Aord{\rH(\phi)}{\AAA}$ agree with the regulator and class module order of Taelman, and Demeslay's result provides the same identity as Taelman's class module formula~\cite{Taelman12}*{Thm.~1} in~\eqref{E:Taelman}.
\end{remark}

\subsection{The \texorpdfstring{$L$}{L}-function of \texorpdfstring{$\mathbb{E}(\phi \times \psi)$}{E(phi x psi)}} \label{SS:Lphixpsi}
Let $\phi$, $\psi : \sA \to A[\tau]$ be Drinfeld modules defined over $A$ of ranks $r$ and $\ell$ respectively as in \eqref{E:phidef} and~\eqref{E:psidef}. We form the $t$-module $\EE = \EE(\phi \times \psi) : \bA \to \Mat_{\ell}(\AAA[\tau])$ defined over $\AAA$ as in \eqref{E:EEtdef}. For each $f \in A_+$, we let $\bP_f(X) = \bP_{\phi \times \psi,f}(X)$ and $\bP_{f}^{\vee}(X) = \bP_{\phi \times \psi,f}^{\vee}(X)$ as in Definition~\ref{D:EEcharpolys}. We further set $\bQ_f(X)$ and $\bQ_{f}^{\vee}(X)$ to be their reciprocal polynomials. We now consider the $L$-function
\begin{equation} \label{E:LEEdef}
L(\EE^{\vee},s) = L(\EE(\phi \times \psi)^{\vee},s) \assign \prod_f \bQ_{f}^{\vee} \bigl( f^{-s} \bigr)^{-1}, \quad s \geqslant 0.
\end{equation}
The following lemma addresses convergence of $L(\EE^{\vee},s)$.

\begin{lemma} \label{L:LEEconverge}
Let $\EE = \EE(\phi \times \psi) : \bA \to \Mat_{\ell}(\AAA[\tau])$ be defined as in~\eqref{E:EEtdef}. For a fixed integer $s \geqslant 0$, the value $L(\EE^{\vee},s)$ converges in $\TT_z(K_{\infty})$.
\end{lemma}

\begin{proof}
(cf.~\cite{CEP18}*{Cor.~3.6})
We first note by Definition~\ref{D:EEcharpolys} that $\bQ_f^{\vee}(X) \in K[z,X]$. Moreover, if we let $\alpha_1, \dots, \alpha_r \in \oK$ be the reciprocals of the roots of $P_{\phi,f}^{\vee}(X)$ and $\beta_1, \dots, \beta_{\ell} \in \overline{\FF_q(z)}$ be the roots of $P_{\psi,f(z)}(X)$, then by definition
\[
\bQ_f^{\vee}(X) = \prod_{\substack{1 \leqslant i \leqslant r \\ 1 \leqslant j \leqslant \ell}}
\bigl( 1 - \alpha_i \beta_j X \bigr) = \prod_{1 \leqslant j \leqslant \ell} Q_{\phi,f}^{\vee}(\beta_j X).
\]
We note that $\beta_j$ is integral over $\FF_q[z]$ for each $j$. If we let $Q_{\phi,f}^{\vee}(\beta_j X) = 1 + \sum_{i=1}^{r} b_{ij} X^i$, then by combining the degree estimates in \S\ref{SSS:Pf} with \eqref{E:PfandPfvee}, we see that $\deg b_{ij} \leqslant -id/r$, where $d = \deg f$. From this we see that
\[
\deg \Bigl( 1 - \bQ_f^{\vee} \bigl(f^{-s} \bigr) \Bigr) \leqslant -\frac{d}{r} - ds = -\biggl(\frac{1}{r} +s \biggr)d.
\]
Since there are only finitely many polynomials of any given degree, this implies that the product defining $L(\EE^{\vee},s)$ converges when $s \geqslant 0$ in the completion of $K[z]$ with respect to $\dnorm{\,\cdot\,}$, i.e., $\TT_z(K_\infty)$ as discussed in \S\ref{SSS:fields}.
\end{proof}

For each irreducible $f \in \sA_+$, Definition~\ref{D:EEcharpolys} implies that $\bQ_f^{\vee}(1) = \bP_f(1)/\bP_f(0)$. By combining Theorem~\ref{T:EEAord} and~\eqref{E:LEEdef}, we obtain the following identity for $L(\EE^{\vee},0)$, which shows that Demeslay's class module formula (Theorem~\ref{T:Demeslay}) applies to the special values we are considering.

\begin{proposition} \label{P:LEE0}
Let $\EE = \EE(\phi \times \psi) : \bA \to \Mat_{\ell}(\AAA[\tau])$ be defined as in~\eqref{E:EEtdef}. Then in $\TT_z(K_{\infty})$ we have
\[
L(\EE^{\vee},0) = \prod_f \frac{\bigAord{\FF_f(z)^{\ell}}{\AAA}}{\bigAord{\soEE(\FF_f(z))}{\AAA}}.
\]
\end{proposition}

\begin{example} \label{Ex:LphixC}
\emph{Twisting a Drinfeld module by Carlitz}. Using \eqref{E:LCxCs-1} as a guide, we consider $L(\EE(\phi \times \sC)^{\vee},s)$ for a Drinfeld module $\phi : \sA \to A[\tau]$ of rank~$r$ defined as in~\eqref{E:phidef}. We let $\bQ_f^{\vee}(X) = \bQ_{\phi \times \sC,f}^{\vee}(X) \in K[z][X]$ be defined as above. We note that \eqref{E:PphixC} implies $\bQ_f^{\vee}(X) = Q_{\phi,f}^{\vee}(f(z) X)$, and so
\[
L(\EE(\phi \times \sC)^{\vee},s) = \prod_f Q_{\phi,f}^{\vee} \bigl( f(z) f^{-s} \bigr)^{-1}.
\]
It follows from~\eqref{E:munugen} that
\begin{equation} \label{E:LphixCs-1}
L(\EE(\phi \times \sC)^{\vee},s) = \sum_{a \in A_+} \frac{\mu_{\phi}(a) a(z)}{a^{s+1}}, \quad s \geqslant 0.
\end{equation}
Gezmi\c{s}~\cite{Gezmis19}*{Thm.~1.1, Cor.~1.3} investigated the value when $s=0$, finding that when $\deg \omega_z = 1/(q-1) < \log_q(R_{\phi})$, so that $\Log_{\phi}(\omega_z)$ is well-defined, 
then as in~\eqref{E:EERegLog1},
\begin{equation} \label{E:LphixC1}
L(\EE(\phi\times \sC)^{\vee},0) = \sum_{a \in A_+} \frac{\mu_{\phi}(a) a(z)}{a} = \frac{\Log_{\phi}(\omega_z)}{\omega_z}.
\end{equation}
By \eqref{E:Rphi}, $\deg \omega_z < \log_q(R_\phi) \Leftrightarrow \deg \kappa_i < q^i - (q^i-1)/(q-1)$, for all $i$, $1 \leqslant i \leqslant r$.
\end{example}

\begin{example} \label{Ex:LCxphi}
\emph{Twisting Carlitz by a Drinfeld module}.
We also consider $L(\EE(\sC \times \phi)^{\vee},s)$. By \eqref{E:PCxphi}, we find that $\bQ_{\sC \times \phi,f}^{\vee}(X) = Q_{\phi,f(z)}(X/f)$, and thus by \eqref{E:LEEdef}
\[
L(\EE(\sC \times \phi)^{\vee},s) = \prod_f Q_{\phi,f(z)}\bigl( f^{-s-1} \bigr)^{-1}.
\]
Letting $\nu_{\phi,z} : A_+ \to \FF_q[z]$ be defined by $\nu_{\phi,z}(a) \assign \nu_{\phi}(a)|_{\theta=z}$, we also find from~\eqref{E:munugen},
\begin{equation} \label{E:LCxphis-1}
L(\EE(\sC \times \phi)^{\vee},s) = \sum_{a \in A_+} \frac{\nu_{\phi,z}(a)}{a^{s+1}}.
\end{equation}
\end{example}

\begin{proposition} \label{P:LCxphi1}
For $\phi : \sA \to A[\tau]$ of rank $r$ as in \eqref{E:phidef}, if $\deg \kappa_i \leqslant q$ for each $i$, $1 \leqslant i \leqslant r-1$, then
\[
L(\EE(\sC \times \phi)^{\vee},0) = \sum_{a \in A_+} \frac{\nu_{\phi,z}(a)}{a}
= \det \Bigl ( \Upsilon_{\phi,z}^{-1} \Log_{\sC} \bigl(\Upsilon_{\phi,z} \bigr) \Bigr).
\]
\end{proposition}

\begin{proof}
As in \eqref{E:EERegLog1}, we show that $\deg \Upsilon_{\phi,z} < \log_q(R_{\sC}) = q/(q-1)$.
By \cite{KhaochimP23}*{Thm.~4.4}, if we pick $\xi \in \phi[t]$ of maximum degree, then $\deg \Upsilon_{\phi,z} = q^{r-1} \deg \xi > 0$.
Initially \cite{KhaochimP23}*{Thm.~4.4} requires a particular choice of $A$-basis for $\Lambda_{\phi}$, but changing to a different basis does not affect the calculation.
By the Newton polygon for $\phi_t(X)$ (e.g., see~\cite{KhaochimP23}*{Fig.~1}), since $\kappa_r \in \FF_q^{\times}$ it follows that $\deg \xi = \max \{ \deg(\kappa_i)/(q^r-q^i)\}_{i=0}^{r-1}$, where we set $\kappa_0=\theta$. It follows that
\[
\deg \Upsilon_{\phi,z} < \frac{q}{q-1} \quad \Leftrightarrow \quad \deg \kappa_i < q \biggl( 1 + \frac{1}{q} + \cdots + \frac{1}{q^{r-1-i}} \biggr), \quad \forall\,i,\, 0 \leqslant i \leqslant r-1.
\]
Because $q^{-1} + \cdots + q^{-r+i} < 1/(q-1)$ for each $i$, the result follows.
\end{proof}

\begin{remark}
If $\phi =\psi= \sC$, then Gezmi\c{s}'s result \eqref{E:LphixC1} and Proposition~\ref{P:LCxphi1} provide the same identity, which is an alternative forumulation of Pellarin's Theorem~\ref{T:Pellarin}, originally proved by Angl\`{e}s, Pellarin, and Tavares Ribeiro~\cite{APT16}*{Lem.~7.1}.
\end{remark}

The $L$-series $L(\EE(\phi \times \sC)^{\vee},s)$ and $L(\EE(\sC \times \phi)^{\vee},s)$ in \eqref{E:LphixCs-1} and \eqref{E:LCxphis-1} both provide possible extensions, for the Drinfeld module~$\phi$, of Pellarin's $L$-series $L(\AAA,s)$ as in \eqref{E:LCxCs-1}. As such we refer to them as \emph{Pellarin $L$-series for $\phi$}. In the next section we explore more general convolutions between two Drinfeld modules.

\section{Convolutions and special \texorpdfstring{$L$}{L}-values} \label{S:convolutions}

Throughout this section we fix two Drinfeld modules $\phi$, $\psi : \sA \to A[\tau]$ with everywhere good reduction as in \eqref{E:phidef} and~\eqref{E:psidef}, together with the convolution $t$-module $\EE(\phi \times \psi)$ as in~\eqref{E:EEtdef}. Since we have already covered the case that either Drinfeld module is the Carlitz module in \S\ref{SS:Lphixpsi}, we will assume $r$, $\ell \geqslant 2$. Our first task is to express the Dirichlet series for $L(\EE(\phi \times \psi)^{\vee},s)$ in terms of Schur polynomials from \S\ref{SS:Schur}, following the ideas of Bump~\cite{Bump89} and Goldfeld~\cite{Goldfeld}*{Ch.~7, 12}.

\subsection{The functions \texorpdfstring{$\bsmu_{\phi,\theta}$}{mu\_\{phi,theta\}} and \texorpdfstring{$\bsnu_{\phi,\theta}$}{nu\_\{phi,theta\}}} \label{SS:bsmubsnu}
Let $f \in A_+$ be irreducible, and let $P_{\phi,f}(X)$ and $P_{\phi,f}^{\vee}(X)$ be defined as in~\eqref{E:PfandPfvee}. We let $\alpha_1, \dots, \alpha_r \in \oK$ be the roots of $P_{\phi,f}^{\vee}(X)$. For $k_1, \dots, k_{r-1} \geqslant 0$, we define
\begin{align} \label{E:bsmudef}
\bsmu_{\phi,\theta}\bigl( f^{k_1}, \dots, f^{k_{r-1}} \bigr) &\assign
S_{k_1, \dots, k_{r-1}} ( \alpha_1, \dots, \alpha_r ) \cdot f^{k_1 + \cdots + k_{r-1}}, \\
\label{E:bsnudef}
\bsnu_{\phi,\theta} \bigl( f^{k_1}, \dots, f^{k_{r-1}} \bigr) &\assign
S_{k_1, \dots, k_{r-1}} \bigl( \alpha_1^{-1}, \dots, \alpha_r^{-1}
 \bigr),
\end{align}
where $S_{k_1, \dots, k_{r-1}}$ is the Schur polynomial of \eqref{E:Sk}. We note that by~\eqref{E:QfandQfvee} and~\eqref{E:Sei},
\begin{align}
Q_{\phi,f}^{\vee}(fX) &= \begin{aligned}[t]
1 &{}- \bsmu_{\phi,\theta}(f,1,\ldots,1) X + \bsmu_{\phi,\theta}(1,f,1,\ldots,1) f X^2 \\
&{}+ \cdots + (-1)^{r-1}\bsmu_{\phi,\theta}(1,\ldots,1,f) f^{r-2} X^{r-1} + (-1)^r \chi_{\phi}(f) f^{r-1}X^r,
\end{aligned}
\\
Q_{\phi,f}(X) &= \begin{aligned}[t]
1 &{}- \bsnu_{\phi,\theta}(f,1,\ldots,1) X + \bsnu_{\phi,\theta}(1,f,1,\ldots,1) X^2 \\
&{}+ \cdots + (-1)^{r-1}\bsnu_{\phi,\theta}(1,\ldots,1,f) X^{r-1} + (-1)^r \ochi_{\phi}(f) f X^r.
\end{aligned}
\end{align}

We then extend $\bsmu_{\phi,\theta}$ and $\bsnu_{\phi,\theta}$ uniquely to functions on $(A_+)^{r-1}$, by requiring that if $a_1, \dots, a_{r-1}$, $b_1, \dots, b_{r-1} \in A_+$ satisfy $\gcd(a_1 \cdots a_{r-1}, b_1\cdots b_{r-1}) = 1$, then
\begin{align*}
\bsmu_{\phi,\theta}(a_1b_1, \dots, a_{r-1}b_{r-1}) &= \bsmu_{\phi,\theta}(a_1, \dots, a_{r-1}) \bsmu_{\phi,\theta}(b_1, \dots, b_{r-1}), \\
\bsnu_{\phi,\theta}(a_1b_1, \dots, a_{r-1}b_{r-1}) &= \bsnu_{\phi,\theta}(a_1, \dots, a_{r-1}) \bsnu_{\phi,\theta}(b_1, \dots, b_{r-1}).
\end{align*}

\begin{proposition} \label{P:bsmunuprops}
For $a$, $a_1, \dots, a_{r-1} \in A_+$, the following hold.
\begin{alphenumerate}
\item $\bsmu_{\phi,\theta}(a_1, \dots, a_{r-1}) \in A$ and $\bsnu_{\phi,\theta}(a_1, \dots, a_{r-1}) \in A$.
\item $\bsmu_{\phi,\theta}(a,1, \dots, 1) = \mu_{\phi,\theta}(a)$ and $\bsnu_{\phi,\theta}(a,1,\dots, 1) = \nu_{\phi,\theta}(a)$.
\item $\bsmu_{\phi,\theta}(a_1, \dots, a_{r-1}) = \chi_{\phi}(a_1 \cdots a_{r-1}) \cdot \bsnu_{\phi,\theta}(a_{r-1}, \dots, a_1)$.
\item We have
\begin{align*}
\deg_{\theta} \bsmu_{\phi,\theta}(a_1, \dots, a_{r-1}) &\leqslant \frac{1}{r} \bigl( (r-1)\deg_{\theta} a_1 + (r-2) \deg_{\theta} a_2 + \cdots + \deg_{\theta} a_{r-1} \bigr), \\
\deg_{\theta} \bsnu_{\phi,\theta}(a_1, \dots, a_r) &\leqslant \frac{1}{r} \bigl( \deg_{\theta} a_1 + 2\deg_{\theta} a_2 + \cdots + (r-1) \deg_{\theta} a_{r-1} \bigr).
\end{align*}
\end{alphenumerate}
\end{proposition}

\begin{proof}
By the multiplicativity of $\bsmu_{\phi,\theta}$ and $\bsnu_{\phi,\theta}$ it suffices to check these identities on powers of an irreducible $f \in A_+$. We first verify (c). Let $k_1, \dots, k_{r-1} \geqslant 0$. By \eqref{E:PfandPfvee} we have $\alpha_1 \cdots \alpha_r = \chi_{\phi}(f) f^{-1}$. Substituting $x_1 \leftarrow \alpha_1^{-1}, \dots, x_r \leftarrow \alpha_r^{-1}$ into Lemma~\ref{L:Skreorder}, we find from \eqref{E:bsmudef} and~\eqref{E:bsnudef} that
\[
\ochi_{\phi}(f)^{k_1 + \cdots + k_{r-1}} \bsmu_{\phi,\theta}\bigl( f^{k_1}, \dots, f^{k_{r-1}} \bigr) = \bsnu_{\phi,\theta} \bigl( f^{k_{r-1}}, \dots, f^{k_{1}} \bigr).
\]
This implies the desired result since $\chi_{\phi} : A_+ \to \FF_q^{\times}$ is completely multiplicative.

As $\alpha_1^{-1}, \dots, \alpha_{r}^{-1}$ are the roots of $P_{\phi,f}(X)$, which is a monic polynomial in $A[X]$, it follows that they are integral over $A$. Therefore, since Schur polynomials are symmetric with integer coefficients, \eqref{E:bsnudef} implies that $\bsnu_{\phi,\theta}(f^{k_1}, \dots, f^{k_{r-1}}) \in A$. By part (c), $\bsmu_{\phi,\theta}(f^{k_1}, \dots, f^{k_{r-1}}) \in A$, and thus (a) is proved.

To prove (b) we combine \eqref{E:munugen} and \eqref{E:homsymm} and note that for $k \geqslant 0$,
\begin{equation} \label{E:munufkhk}
\mu_{\phi,\theta}\bigl(f^k\bigr) = h_k(f \alpha_1, \dots, f \alpha_r), \quad
\nu_{\phi,\theta}\bigl(f^k\bigr) =h_k \bigl( \alpha_1^{-1}, \dots, \alpha_r^{-1}
\bigr).
\end{equation}
Then \eqref{E:Shi} implies $\mu_{\phi,\theta}(f^k) = S_{k,0, \ldots, 0}(\alpha_1, \dots, \alpha_r)f^k = \bsmu_{\phi,\theta}(f^k,1, \dots, 1)$, and likewise $\nu_{\phi,\theta}(f^k) = \bsnu_{\phi,\theta}(f^k,1, \dots, 1)$. Finally for (d), we note that since $\deg_{\theta} \alpha_i^{-1} = d/r$ for each~$i$ from~\S\ref{SSS:Pf}, it follows from~\eqref{E:bsnudef} that the degree of $\bsnu_{\phi,\theta}(f^{k_1}, \dots, f^{k_{r-1}})$ in~$\theta$ is at most $(\deg S_{k_1, \dots, k_{r-1}}) \cdot d/r = (k_1 + 2k_2 + \cdots + (r-1)k_{r-1}) \cdot d/r$. The desired inequality for $\deg_{\theta} \bsnu_{\phi,\theta}(a_1, \dots, a_{r-1})$ then follows from the multiplicativity of $\bsnu_{\phi,\theta}$. The inequality for $\deg_{\theta} \bsmu_{\phi,\theta}(a_1, \dots, a_{r-1})$ then follows from~(c).
\end{proof}

The functions $\bsmu_{\phi,\theta}$ and $\bsnu_{\phi,\theta}$ satisfy a number of recursive relations induced by relations on Schur polynomials (cf.~\cite{Goldfeld}*{p.~278}). Fix $f \in A_+$ irreducible. Then Pieri's rule~\eqref{E:Pieri} implies that for $k$, $k_1, \dots, k_{r-1} \geqslant 0$,
\begin{align}
\bsmu_{\phi,\theta} \bigl(f^k, &\,1, \dots, 1\bigr) \bsmu_{\phi,\theta} \bigl( f^{k_1}, \dots, f^{k_{r-1}} \bigr) \\
&= \begin{aligned}[t]
\smash{\sum_{\substack{\substack{m_0 + \cdots + m_{r-1}=k \\ m_1 \leqslant k_1,\, \ldots,\, m_{r-1} \leqslant k_{r-1}}}}} \bsmu_{\phi,\theta} \bigl( f^{k_1 + m_0 - m_1}, f^{k_2 + m_1-m_2}, \dots, &\,f^{k_{r-1}+m_{r-2}-m_{r-1}} \bigr) \\
&{} \cdot \chi_{\phi}(f)^{m_{r-1}} f^{k-m_0},
\end{aligned} \notag \\[10pt]
\bsnu_{\phi,\theta} \bigl(f^k, &\,1, \dots, 1\bigr) \bsnu_{\phi,\theta} \bigl( f^{k_1}, \dots, f^{k_{r-1}} \bigr) \\
&= \begin{aligned}[t]
\smash{\sum_{\substack{\substack{m_0 + \cdots + m_{r-1}=k \\ m_1 \leqslant k_1,\, \ldots,\, m_{r-1} \leqslant k_{r-1}}}}} \bsnu_{\phi,\theta} \bigl( f^{k_1 + m_0 - m_1}, f^{k_2 + m_1-m_2}, \dots, &\, f^{k_{r-1}+m_{r-2}-m_{r-1}} \bigr) \\
&{} \cdot \ochi_{\phi}(f)^{m_{r-1}} f^{m_{r-1}}.
\end{aligned} \notag
\end{align}
The dual Pieri rule~\eqref{E:dualPieri} implies that for $0 \leqslant k \leqslant r-1$,
\begin{align}
\bsmu_{\phi,\theta} &\bigl(\underbrace{1,\ldots,1,f,1,\ldots,1}_{\textup{$k$-th place}}) \bsmu_{\phi,\theta} \bigl( f^{k_1}, \dots, f^{k_{r-1}} \bigr) \\
&= \begin{aligned}[t]
\smash{\sum_{\substack{m_0 + \cdots + m_{r-1} = k \\ (m_0, \ldots, m_{r-1})\, \in\, \cI_{k_1, \dots, k_{r-1}}}}} \bsmu_{\phi,\theta} \bigl( f^{k_1 + m_0 - m_1}, f^{k_2 + m_1-m_2}, \dots, &\,f^{k_{r-1}+m_{r-2}-m_{r-1}} \bigr) \\
&{} \cdot \chi_{\phi}(f)^{m_{r-1}} f^{1-m_0},
\end{aligned} \notag \\[10pt]
\bsnu_{\phi,\theta} &\bigl(\underbrace{1,\ldots,1,f,1,\ldots,1}_{\textup{$k$-th place}}) \bsnu_{\phi,\theta} \bigl( f^{k_1}, \dots, f^{k_{r-1}} \bigr) \\
&= \begin{aligned}[t]
\smash{\sum_{\substack{m_0 + \cdots + m_{r-1} = k \\ (m_0, \ldots, m_{r-1})\, \in\, \cI_{k_1, \dots, k_{r-1}}}}} \bsnu_{\phi,\theta} \bigl( f^{k_1 + m_0 - m_1}, f^{k_2 + m_1-m_2}, \dots, &\, f^{k_{r-1}+m_{r-2}-m_{r-1}} \bigr) \\
&{} \cdot \ochi_{\phi}(f)^{m_{r-1}} f^{m_{r-1}}.
\end{aligned} \notag
\end{align}
In particular for $k \geqslant 1$ (cf.\ \cite{Goldfeld}*{p.~278}),
\begin{align}
\bsmu_{\phi,\theta}( f^k, 1, \dots, 1) \bsmu_{\phi,\theta} (f,1,\ldots, 1)
&= \begin{aligned}[t]
\bsmu_{\phi,\theta}(f^{k+1},&\,1,\ldots, 1) \\
&{}+ \bsmu_{\phi,\theta}(f^{k-1},f,1, \ldots, 1) \cdot f,
\end{aligned}
\\
\bsmu_{\phi,\theta}( f^k, 1, \dots, 1) \bsmu_{\phi,\theta} (1,f,1,\ldots, 1)
&= \begin{aligned}[t]
\bsmu_{\phi,\theta}(f^k,&\,f,1,\ldots,1) \\
&{}+ \bsmu_{\phi,\theta}(f^{k-1},1,f,1, \ldots,1) \cdot f,
\end{aligned}
\notag \\
\bsmu_{\phi,\theta}( f^k, 1, \dots, 1) \bsmu_{\phi,\theta} (1,1,f,1,\ldots, 1)
&= \begin{aligned}[t]
\bsmu_{\phi,\theta}(f^k,&\,1,f,1,\ldots,1) \\
&{}+ \bsmu_{\phi,\theta}(f^{k-1},1,1,f,1, \ldots,1) \cdot f,
\end{aligned}
\notag \\
\bsmu_{\phi,\theta}( f^k, 1, \dots, 1) \bsmu_{\phi,\theta} (1,\ldots, 1,f)
&= \begin{aligned}[t]
\bsmu_{\phi,\theta}(f^k,&\,1,\ldots,1,f) \\
&{}+ \bsmu_{\phi,\theta}(f^{k-1},1, \ldots,1) \cdot \chi_{\phi}(f) f,
\end{aligned}
\notag
\end{align}
and likewise,
\begin{align}
\bsnu_{\phi,\theta}( f^k, 1, \dots, 1) \bsnu_{\phi,\theta} (f,1,\ldots, 1)
&= \begin{aligned}[t]
\bsnu_{\phi,\theta}(f^{k+1},&\,1,\ldots, 1) \\
&{}+ \bsnu_{\phi,\theta}(f^{k-1},f,1, \ldots, 1),
\end{aligned}
\\
\bsnu_{\phi,\theta}( f^k, 1, \dots, 1) \bsnu_{\phi,\theta} (1,f,1,\ldots, 1)
&= \begin{aligned}[t]
\bsnu_{\phi,\theta}(f^k,&\,f,1,\ldots,1) \\
&{}+ \bsnu_{\phi,\theta}(f^{k-1},1,f,1, \ldots,1),
\end{aligned}
\notag \\
\bsnu_{\phi,\theta}( f^k, 1, \dots, 1) \bsnu_{\phi,\theta} (1,1,f,1,\ldots, 1)
&= \begin{aligned}[t]
\bsnu_{\phi,\theta}(f^k,&\,1,f,1,\ldots,1) \\
&{}+ \bsnu_{\phi,\theta}(f^{k-1},1,1,f,1, \ldots,1),
\end{aligned}
\notag \\
\bsnu_{\phi,\theta}( f^k, 1, \dots, 1) \bsnu_{\phi,\theta} (1,\ldots, 1,f)
&= \begin{aligned}[t]
\bsnu_{\phi,\theta}(f^k,&\,1,\ldots,1,f) \\
&{}+ \bsnu_{\phi,\theta}(f^{k-1},1, \ldots,1) \cdot \ochi_{\phi}(f) f.
\end{aligned}
\notag
\end{align}

\subsubsection{Calculating $\bsmu_{\phi,\theta}$ and $\bsnu_{\phi,\theta}$} \label{SSS:munucalc}
To calculate $\bsmu_{\phi,\theta}$ and $\bsnu_{\phi,\theta}$ one can also use the Jacobi-Trudi identity from~\S\ref{SS:JT}. Substituting \eqref{E:munufkhk} into~\eqref{E:JT} and using Proposition \ref{P:bsmunuprops}(b), for $k_1, \dots, k_{r-1} \geqslant 0$ we have
\begin{equation} \label{E:JTnu}
\bsnu_{\phi,\theta}\bigl( f^{k_1}, \dots, f^{k_{r-1}} \bigr) = \det \Bigl(
\nu_{\phi,\theta} \bigl( f^{k_i + \cdots + k_{r-1}-i+j} \bigr) \Bigr)_{i,j=1}^{r-1}.
\end{equation}
We can use~\eqref{E:charpolyscalc} to determine $P_{\phi,f}(X)$ and then use~\eqref{E:nurec} to obtain the entries of this determinant. Then $\bsmu_{\phi,\theta}(f^{k_1}, \dots, f^{k_{r-1}})$ can be found from Proposition~\ref{P:bsmunuprops}(c).

\subsection{Convolution \texorpdfstring{$L$}{L}-series for \texorpdfstring{$r=\ell$}{r=l}} \label{SS:Lconvrxr}
In this section we assume that $\phi$ and $\psi$ from \eqref{E:phidef} and \eqref{E:psidef} have the same rank $r=\ell \geqslant 2$, and we investigate
\[
L(\EE(\phi\times \psi)^{\vee},s) = \prod_f \bQ_f^{\vee}\bigl( f^{-s} \bigr)^{-1}
\]
from~\eqref{E:LEEdef}. We at first fix $f \in A_+$ irreducible, and as previously we let $\alpha_1, \dots, \alpha_r \in \oK$ be the roots of $P_{\phi,f}^{\vee}(X)$, and we let $\beta_1, \dots, \beta_{r} \in \overline{\FF_q(z)}$ be the roots of $P_{\psi,f(z)}(X)$.

As $\bQ_f^{\vee}(X)$ is the reciprocal polynomial of $\bP_{f}^{\vee}(X) = P_{\phi,f}^{\vee}(X) \otimes P_{\psi,f(z)}(X)$ from Definition~\ref{D:EEcharpolys}, we can expand $\bQ_f^{\vee}(f^{-s})^{-1}$ using Cauchy's identity (Theorem~\ref{T:Cauchy}). We note from~\eqref{E:PfandPfvee} that $\alpha_1 \cdots \alpha_r = \chi_{\phi}(f) f^{-1}$ and $\beta_1 \cdots \beta_r = \ochi_{\psi}(f) f(z)$. By the definitions of $\bsmu_{\phi,\theta}$ and $\bsnu_{\psi,z}$ from \eqref{E:bsmudef} and~\eqref{E:bsnudef}, Theorem~\ref{T:Cauchy} implies
\begin{align*}
\bQ_f^{\vee}\bigl( &f^{-s} \bigr)^{-1} \\
&= \biggl( 1 - \frac{\chi_{\phi}(f) \ochi_{\psi}(f) f(z)}{f^{rs+1}} \biggr)^{-1} 
\underset{k=(k_1, \dots, k_{r-1})}{\sum_{k_1=0}^{\infty} \cdots \sum_{k_{r-1}=0}^{\infty}} S_{k}(\bsalpha) S_{k}(\bsbeta) f^{-s(k_1 + 2k_2 + \cdots + (r-1)k_{r-1})} \\
&= \begin{aligned}[t]
\biggl( 1 - \frac{\chi_{\phi}(f) \ochi_{\psi}(f) f(z)}{f^{rs+1}} \biggr)^{-1} \smash{\sum_{k_1, \ldots, k_{r-1} \geqslant 0}} \bsmu_{\phi,\theta} &\bigl(f^{k_1}, \dots, f^{k_{r-1}} \bigr) \bsnu_{\psi,z} \bigl( f^{k_1}, \dots, f^{k_{r-1}} \bigr) \\
&{}\cdot f^{-k_1-k_2 - \cdots - k_{r-1}-s(k_1 + 2k_2 + \cdots (r-1)k_{r-1})},
\end{aligned}
\end{align*}
where $\bsalpha = (\alpha_1, \dots, \alpha_r)$ and $\bsbeta = (\beta_1, \dots, \beta_r)$.
We define the twisted Pellarin $L$-series
\begin{equation}
L(\AAA,\chi_{\phi}\ochi_{\psi},s) \assign \sum_{a \in A_+} \frac{\chi_{\phi}(a)\ochi_{\psi}(a) a(z)}{a^s},
\end{equation}
and finally we define the $L$-series, for $\bsnu_{\phi,z}(a_1, \dots, a_{r-1})\assign \bsnu_{\phi,\theta}(a_1, \dots, a_{r-1})|_{\theta=z} \in \FF_q[z]$,
\begin{equation} \label{E:Lmunurxr}
L(\bsmu_{\phi,\theta} \times \bsnu_{\psi,z},s) \assign \sum_{a_1 \in A_+} \cdots \sum_{a_{r-1} \in A_+} \frac{\bsmu_{\phi,\theta}(a_1, \dots, a_{r-1}) \bsnu_{\psi,z}(a_1, \dots, a_{r-1})}{a_1 \cdots a_{r-1} (a_1 a_2^2 \cdots a_{r-1}^{r-1})^s}.
\end{equation}
The convergence of this series in $\TT_z(K_{\infty})$ can be deduced from Proposition~\ref{P:bsmunuprops}(d) for $s \geqslant 0$. More specifically, if $(a_1, \dots, a_r) \neq (1,\dots, 1)$, then Proposition~\ref{P:bsmunuprops}(d) implies
\[
\deg_{\theta} \biggl( \frac{\bsmu_{\theta}(a_1, \dots, a_{r-1})}{a_1 \cdots a_{r-1}(a_1 a_2^2 \cdots a_{r-1}^{r-1})^s} \biggr) \leqslant -\frac{1}{r},
\]
which implies that $\dnorm{L(\bsmu_{\phi,\theta} \times \bsnu_{\psi,z},s) - 1} < 0$ and so $L(\bsmu_{\phi,\theta} \times \bsnu_{\psi,z},s) \in \TT_z(K_{\infty})^{\times}$. Similarly, $L(\AAA,\chi_{\phi}\ochi_{\psi},s+1) \in \TT_z(K_{\infty})^{\times}$.
After some straightforward simplification we arrive at the following result.

\begin{theorem} \label{T:LEErxr}
Let $\phi$, $\psi : \sA \to A[\tau]$ be Drinfeld modules both of rank $r \geqslant 2$ with everywhere good reduction, as defined in~\eqref{E:phidef} and~\eqref{E:psidef}. Then for $s \geqslant 0$,
\[
L(\EE(\phi \times \psi)^{\vee},s) = L(\AAA,\chi_{\phi}\ochi_{\psi},rs+1) \cdot
L(\bsmu_{\phi,\theta} \times \bsnu_{\psi,z},s),
\]
and $L(\EE(\phi \times \psi)^{\vee},s) \in \TT_z(K_{\infty})^{\times}$.
\end{theorem}

\begin{remark}
As mentioned in Remark~\ref{R:convergence}, one might expect that for fixed $s$, the series $L(\EE(\phi \times \psi)^{\vee},s)$ is in fact an entire function in $\TT_z(K_{\infty})$, as one finds in \cites{AnglesPellarin14, APT16, Pellarin12, Gezmis19, Goss13}. This important question requires further study.
\end{remark}

We note that in the case that the leading coefficients $\kappa_r$ and $\eta_r$ of $\phi_t$ and $\psi_t$ are equal, then by \eqref{E:chidef} we have $L(\AAA,\chi_{\phi}\ochi_{\psi},s) = L(\AAA,s)$. We can substitute $s=0$ into Theorem~\ref{T:LEErxr} and obtain the following special value identities.

\begin{corollary} \label{C:Lmunurxr}
Let $\phi$, $\psi : \sA \to A[\tau]$ be Drinfeld modules both of rank $r \geqslant 2$ with everywhere good reduction, as defined in~\eqref{E:phidef} and~\eqref{E:psidef}.
\begin{alphenumerate}
\item If $\kappa_r = \eta_r$, then
\begin{align*}
L(\bsmu_{\phi,\theta} \times \bsnu_{\psi,z},0)
&= \sum_{a_1 \in A_+} \cdots \sum_{a_{r-1} \in A_+} \frac{\bsmu_{\phi,\theta}(a_1, \dots, a_{r-1}) \bsnu_{\psi,z}(a_1, \dots, a_{r-1})}{a_1 \cdots a_{r-1}} \\[10pt]
&= (\theta-z) \cdot \frac{\omega_z}{\tpi} \cdot \Reg_{\EE} \cdot \Aord{\rH(\EE)}{\AAA}.
\end{align*}
\item If $\kappa_r=\eta_r$ and $\dnorm{\Upsilon_{\psi,z}} < R_{\phi}$, where $R_{\phi}$ is the radius of convergence of $\Log_{\phi}(z)$ in~\eqref{E:Rphi}, then
\[
L(\bsmu_{\phi,\theta} \times \bsnu_{\psi,z},0)
=(\theta - z) \cdot \frac{\omega_z}{\tpi} \cdot \det \Bigl( \Upsilon_{\psi,z}^{-1} \Log_{\phi} \bigl(\Upsilon_{\psi,z} \bigr) \Bigr).
\]
\item More generally, if we let $\gamma = (\kappa_r/\eta_r)^{1/(q-1)} \in \oFF_q^{\times}$, then
\[
L(\bsmu_{\phi,\theta} \times \bsnu_{\psi,z},0) = \frac{1}{\gamma \omega_z} \Log_{\sC} \bigl( \gamma \omega_z \bigr) \cdot \Reg_{\EE} \cdot \Aord{\rH(\EE)}{\AAA},
\]
and if additionally $\dnorm{\Upsilon_{\psi,z}} < R_{\phi}$,
\[
L(\bsmu_{\phi,\theta} \times \bsnu_{\psi,z},0) = \frac{1}{\gamma \omega_z} \Log_{\sC} \bigl( \gamma \omega_z \bigr) \cdot \det \Bigl( \Upsilon_{\psi,z}^{-1} \Log_{\phi} \bigl(\Upsilon_{\psi,z} \bigr) \Bigr).
\]
\end{alphenumerate}
\end{corollary}

\begin{proof}
By Theorem~\ref{T:LEErxr} we have $L(\bsmu_{\phi,\theta} \times \bsnu_{\psi,z},0) = L(\EE(\phi \times \psi)^{\vee},0)/L(\AAA,1)$. Then
using Proposition~\ref{P:LEE0} together with Pellarin's identity for $L(\AAA,1)$ (Theorem~\ref{T:Pellarin}) and Demeslay's special value formula (Theorem~\ref{T:Demeslay}), we obtain~(a).

For~(b), the discussion in Remark~\ref{R:RegLog1} implies that $\Reg_{\EE} = \det(\Upsilon_{\psi,z}^{-1} \Log_{\phi} (\Upsilon_{\psi,z}))$, and so by~(a) it remains to verify that $\Aord{\rH(\EE)}{\AAA} = 1$ in this case. Now by construction the product in Demeslay's Theorem~\ref{T:Demeslay} has $\dnorm{\,\cdot\,}$-norm~$1$ (moreover, this can be verified directly from~\eqref{E:LEEdef}), and so by Proposition~\ref{P:LEE0}, $\dnorm{L(\EE(\phi \times \psi)^{\vee},0)} = 1$. One checks easily that $\dnorm{L(\AAA,1)} = 1$, and so $\dnorm{L(\bsmu_{\phi,\theta} \times \bsnu_{\psi,z},0)} = 1$. Furthermore, we have $\dnorm{\det(\Upsilon_{\psi,z}^{-1} \Log_{\phi} (\Upsilon_{\psi,z}))}=1$ by Remark~\ref{R:RegLog1}, so both sides of the proposed identity in~(b) have $\dnorm{\,\cdot\,}$-norm~$1$. Since $\Aord{\rH(\EE)}{\AAA}$ is monic in $\AAA = \FF_q(z)[\theta]$, this cannot be the case if $\Aord{\rH(\EE)}{\AAA} \neq 1$.
Part (c) follows in exactly the same way, using the identity
\[
L(\AAA,\chi_{\phi}\ochi_{\psi},1) = \frac{1}{\gamma \omega_z} \Log_{\sC} \bigl( \gamma \omega_z \bigr),
\]
which is due to Angl\`es, Pellarin, and Tavares Ribeiro~\cite{APT16}*{Eqs.~(18), (25)}.
\end{proof}

\subsection{Convolution \texorpdfstring{$L$}{L}-series for \texorpdfstring{$r\neq\ell$}{r<>l}} \label{SS:Lconvrxl}
We continue with the same situation as in~\S\ref{SS:Lconvrxr}, but here we assume that the ranks of $\phi$ and $\psi$ are not equal. This splits into two cases, $r < \ell$ and $r > \ell$. As we have addressed the cases where either $\phi$ or $\psi$ is the Carlitz module in Examples~\ref{Ex:LphixC} and~\ref{Ex:LCxphi}, we will assume that $r$, $\ell \geqslant 2$.

\subsubsection{The case $r < \ell$}
Let $f \in A_+$ be irreducible, and let $\alpha_1, \dots, \alpha_r \in \oK$ be the roots of $P_{\phi,f}^{\vee}(X)$ and $\beta_1, \dots, \beta_{\ell} \in \overline{\FF_q(z)}$ the roots of $P_{\psi,f(z)}(X)$. Using that $\alpha_1 \cdots \alpha_r = \chi_{\phi}(f) f^{-1}$, we apply Bump's specialization of Cauchy's identity (Corollary~\ref{C:Cauchynl}), and similar to calculations in \S\ref{SS:Lconvrxr}, we find
\begin{align*}
\bQ_f^{\vee}\bigl( f^{-s} \bigr)^{-1} &= \underset{\substack{k=(k_1, \dots, k_{r-1}) \\ k'=(k_1, \dots, k_{r},0 \dots, 0)}}{\sum_{k_1=0}^{\infty} \cdots \sum_{k_{r}=0}^{\infty}} S_{k}(\bsalpha) S_{k'}(\bsbeta) \bigl( \chi_{\phi}(f) f^{-1} \bigr)^{k_r} f^{-s(k_1 + 2k_2 + \cdots + rk_r)} \\
&= \begin{aligned}[t]
 \smash{\sum_{k_1, \ldots, k_{r} \geqslant 0}} \bsmu_{\phi,\theta} \bigl(f^{k_1}, \dots, f^{k_{r-1}} \bigr) &\bsnu_{\psi,z} \bigl( f^{k_1}, \dots, f^{k_{r}},1, \ldots, 1 \bigr) \\
&{}\cdot \chi_{\phi} \bigl(f^{k_r} \bigr) f^{-k_1-k_2 - \cdots - k_{r}-s(k_1 + 2k_2 + \cdots + rk_{r})}.
\end{aligned}
\end{align*}
The expression $\bsnu_{\psi,z}( f^{k_1}, \dots, f^{k_{r}},1, \ldots, 1)$ generically has $1$'s in exactly the last $\ell-1-r$ places. We thus define the $L$-series when $r < \ell$,
\begin{equation} \label{E:Lmunurxl1}
L(\bsmu_{\phi,\theta} \times \bsnu_{\psi,z},s) \assign
\sum_{a_1, \ldots, a_r \in A_+}
\frac{\chi_{\phi}(a_r)\bsmu_{\phi,\theta}(a_1, \dots, a_{r-1}) \bsnu_{\psi,z}(a_1, \dots, a_r,1, \ldots, 1)}{a_1 \cdots a_r (a_1 a_2^2 \cdots a_r^r)^s}.
\end{equation}

\subsubsection{The case $r > \ell$}
The case that $r > \ell$ is similar to the $r < \ell$ case, making use of Corollary~\ref{C:Cauchynl} again, with only minor modifications. We skip the details, but when $r > \ell$ we define the $L$-series
\begin{multline} \label{E:Lmunurxl2}
L(\bsmu_{\phi,\theta} \times \bsnu_{\psi,z},s)  \assign \\
\sum_{a_1, \ldots, a_{\ell} \in A_+}
\frac{\ochi_{\psi}(a_{\ell}) a_{\ell}(z) \bsmu_{\phi,\theta}(a_1, \dots, a_{\ell},1, \dots, 1) \bsnu_{\psi,z}(a_1, \dots, a_{\ell-1})}{a_1 \cdots a_{\ell} (a_1 a_2^2 \cdots a_{\ell}^{\ell})^s}.
\end{multline}
After some reasonably straightforward calculation that we omit, we obtain the following theorem that covers both cases.

\begin{theorem} \label{T:LEErxl}
Let $\phi$, $\psi : \sA \to A[\tau]$ be Drinfeld modules of ranks $r$ and $\ell$ respectively with everywhere good reduction, as defined in~\eqref{E:phidef} and~\eqref{E:psidef}. Assume that $r$, $\ell \geqslant 2$ and that $r \neq \ell$. Then for $s \geqslant 0$,
\[
L(\EE(\phi \times \psi)^{\vee},s) = L(\bsmu_{\phi,\theta} \times \bsnu_{\psi,z},s),
\]
and $L(\EE(\phi \times \psi)^{\vee},s) \in \TT_z(K_{\infty})^{\times}$.
\end{theorem}

If we degenerate the expressions in \eqref{E:Lmunurxl1} and~\eqref{E:Lmunurxl2} to the $r=1$ or $\ell=1$ cases of the Carlitz module, then we recover the identities in~\eqref{E:LphixCs-1} and~\eqref{E:LCxphis-1}.
Just as for Corollary~\ref{C:Lmunurxr}, we obtain the following result with the same proof, though there is no longer a factor coming from $L(\AAA,\chi_{\phi}\ochi_{\psi},1)$.

\begin{corollary} \label{C:Lmunurxl}
Let $\phi$, $\psi : \sA \to A[\tau]$ be Drinfeld modules of ranks $r$ and $\ell$ respectively with everywhere good reduction, as defined in~\eqref{E:phidef} and~\eqref{E:psidef}. Assume that $r$, $\ell \geqslant 2$ and that $r \neq \ell$.
\begin{alphenumerate}
\item If $r < \ell$, then
\begin{align*}
L(\bsmu_{\phi,\theta} \times \bsnu_{\psi,z},0)
&= \sum_{a_1, \ldots, a_{r} \in A_+} \frac{\chi_{\phi}(a_r) \bsmu_{\phi,\theta}(a_1, \dots, a_{r-1}) \bsnu_{\psi,z}(a_1, \dots, a_{r},1, \dots, 1)}{a_1 \cdots a_r} \\[10pt]
&= \Reg_{\EE} \cdot \Aord{\rH(\EE)}{\AAA}.
\end{align*}
\item If $r > \ell$, then
\begin{align*}
L(\bsmu_{\phi,\theta} \times \bsnu_{\psi,z},0)
&= \sum_{a_1, \ldots, a_{\ell} \in A_+} \frac{\ochi_{\psi}(a_\ell) a_{\ell}(z) \bsmu_{\phi,\theta}(a_1, \dots, a_{\ell},1, \ldots, 1) \bsnu_{\psi,z}(a_1, \dots, a_{\ell-1})}{a_1 \cdots a_{\ell}} \\[10pt]
&= \Reg_{\EE} \cdot \Aord{\rH(\EE)}{\AAA}.
\end{align*}
\item If $\dnorm{\Upsilon_{\psi,z}} < R_{\phi}$, where $R_{\phi}$ is the radius of convergence of $\Log_{\phi}(z)$ in~\eqref{E:Rphi}, then
\[
L(\bsmu_{\phi,\theta} \times \bsnu_{\psi,z},0)
= \det \Bigl( \Upsilon_{\psi,z}^{-1} \Log_{\phi} \bigl(\Upsilon_{\psi,z} \bigr) \Bigr).
\]
\end{alphenumerate}
\end{corollary}

\subsection{Calculations and examples}
In this section we consider computations of the special value formulas in Corollaries~\ref{C:Lmunurxr} and~\ref{C:Lmunurxl}. In practice, $\bsmu_{\phi,\theta}$ and $\bsnu_{\psi,z}$ can be evaluated using the Jacobi-Trudi identity as described in \S\ref{SSS:munucalc}. Unfortunately the convergence of $L(\bsmu_{\phi,\theta} \times \bsnu_{\psi,z},0)$ is slow and yields few digits of accuracy in $\TT_z(K_{\infty})$. In the first three examples we have $\dnorm{\Upsilon_{\psi,z}} < R_{\phi}$, and we consider the complementary case in~\S\ref{SSS:bigUps} and Example~\ref{Ex:four}.

\begin{example} \label{Ex:one}
($r=\ell=2$).
We let $q=3$ and let $\phi$, $\psi : \FF_3[t] \to \FF_3[\theta][\tau]$ be defined by
\[
\phi_t = \theta  + \theta^2 \tau - \tau^2, \quad
\psi_t = \theta + \tau - \tau^2.
\]
Then for $\EE = \EE(\phi \times \psi)$,
\[
\EE_t =
\begin{pmatrix}
\theta & 0 \\ 0 & \theta
\end{pmatrix}
+
\begin{pmatrix}
0 & -z\theta^2 + \theta^3 \\
\theta^2 & \theta^2
\end{pmatrix}
\tau + 
\begin{pmatrix}
z-\theta & z-\theta \\
-1 & z-1 - \theta^3
\end{pmatrix}
\tau^2.
\]
For example calculations, when $f = \theta^2+1$ and $g = \theta^3 - \theta - 1$, we find that
\begin{align*}
P_{\phi,f}(X) &= X^2 - (\theta+1)X + \theta^2+1, & P_{\psi,f}(X) &= X^2 - (\theta+1) + \theta^2+1,\\
P_{\phi,g}(X) &= X^2 - (\theta-1) X + \theta^3 -\theta -1, & P_{\psi,g}(X) &= X^2 - X + \theta^3 -\theta - 1,
\end{align*}
and thus by~\eqref{E:munuf}, $\mu_{\phi,\theta}(\theta^2+1) = \theta + 1$, $\mu_{\phi,\theta}(\theta^3-\theta-1) = \theta-1$, $\nu_{\psi,z}(\theta^2+1) = z + 1$, and $\nu_{\psi,z}(\theta^3-\theta-1) = 1$. Moreover, Definition~\ref{D:EEcharpolys} implies
\begin{align*}
\bP_f(X) &=
\begin{aligned}[t]
X^4 - \frac{(\theta+1)(z+1)}{z^2+1} X^3 &{}- \frac{(z+\theta)(\theta z + 1)}{(z^2+1)^2} X^2 \\
&{}- \frac{(\theta+1)(\theta^2+1)(z+1)}{(z^2+1)^2}X + \frac{(\theta^2+1)^2}{(z^2+1)^2},
\end{aligned}
\\[10pt]
\bP_g(X) &=
\begin{aligned}[t]
X^4 - \frac{\theta-1}{z^3-z-1} X^3 &{}+ \frac{(z^3-z)\theta^3 + (z^3-z-1)\theta^2 - \theta -1}{(z^3-z-1)^2}X^2 \\
&{}- \frac{(\theta-1)(\theta^3 -\theta-1)}{(z^3-z-1)^2}X + \frac{(\theta^3-\theta-1)^2}{(z^3-z-1)^2}.
\end{aligned}
\end{align*}
Theorem~\ref{T:EEAord} then yields
\begin{align*}
\Aord{\soEE(\FF_f(z))}{\AAA} &= (z^2+1)^2\cdot \bP_f(1)\\
&= \theta^4 - (z+1)\theta^3 + (z+1)\theta^2 - (z^3-z^2-z)\theta + z^4 - z^3 + z^2, \\
\Aord{\soEE(\FF_g(z))}{\AAA} &= (z^3-z-1)^2 \cdot \bP_g(1) \\
&=
\begin{aligned}[t]
\theta^6 + (z^3-z-1)\theta^3 &{}+ (z^3-z+1)\theta^2 \\
&{}- (z^3-z+1)\theta + z^6 + z^4 - z^3+z^2+z-1.
\end{aligned}
\end{align*}
Now to examine $L(\bsmu_{\phi,\theta} \times \bsnu_{\psi,z},0)$, we note from \cite{KhaochimP23}*{Eq.~(2.11)} that $ R_{\phi} = \norm{\theta}^{1/2}$, and from \citelist{\cite{KhaochimP23}*{Thm.~4.4} \cite{Maurischat19}*{Thm.~3.1}} it follows that $\dnorm{\Upsilon_{\psi,z}} = \norm{\theta}^{3/8}$. Thus we are in the situation of Corollary~\ref{C:Lmunurxr}(b) with $r=\ell=2$. We find through direct calculation that
\begin{align*}
L(\bsmu_{\phi,\theta} \times \bsnu_{\psi,z},0)
= \sum_{a \in A_+} \frac{\mu_{\phi,\theta}(a) \nu_{\psi,z}(a)}{a}
= &\:1 - \theta^{-1} + z \theta^{-2} - \theta^{-3} + \theta^{-4} + (1-z)\theta^{-5} \\
&{}+ z\theta^{-6}
+ (z+1)^2 \theta^{-8} - (z+1)\theta^{-9} + O(\theta^{-10}).
\end{align*}
The quantity $\det (\Upsilon_{\psi,z}^{-1} \Log_{\phi}(\Upsilon_{\psi,z})) = \det \Log_{\EE} (\rI_2)$ can be calculated to high precision via \eqref{E:ExpLogEE}, and as expected from Corollary~\ref{C:Lmunurxr}(b), the product $(\theta-z)\cdot \omega_z/\tpi \cdot \det (\Upsilon_{\psi,z}^{-1} \Log_{\phi}(\Upsilon_{\psi,z}))$ agrees numerically with this value.
\end{example}

\begin{example} \label{Ex:two}
($r=3$, $\ell=2$)
Let $\phi$, $\psi : \FF_3[t] \to \FF_3[\theta][\tau]$ be defined by
\[
\phi_t = \theta + \theta^2 \tau + \theta \tau^2 + \tau^3, \quad
\psi_t = \theta - \tau + \tau^2.
\]
We find that for $\EE = \EE(\phi \times \psi)$,
\begin{align*}
\EE_t =
\begin{pmatrix}
\theta & 0 \\ 0 & \theta
\end{pmatrix}
+ \theta^2
\begin{pmatrix}
0 & z-\theta \\ 1 & 1
\end{pmatrix}
\tau &{}+ \theta
\begin{pmatrix}
z-\theta & z - \theta \\
1 & z+1 - \theta^3
\end{pmatrix}
\tau^2 \\
&{}+
\begin{pmatrix}
z-\theta & (z+1-\theta^9)(z-\theta) \\
z+1 - \theta^3 & -z+1 - \theta^3 - \theta^9
\end{pmatrix}
\tau^3.
\end{align*}
Again \cite{KhaochimP23}*{Eq.~(2.11)} implies that $R_{\phi} = \norm{\theta}^{1/2}$, and we find from \citelist{\cite{KhaochimP23}*{Thm.~4.4} \cite{Maurischat19}*{Thm.~3.1}}  that $\dnorm{\Upsilon_{\psi,z}} = \norm{\theta}^{3/8}$. We calculate
\begin{align*}
L(\bsmu_{\phi,\theta} \times \bsnu_{\psi,z},0) &= \sum_{a_1, a_2 \in A_+} \frac{\ochi_{\psi}(a_2) a_2(z) \bsmu_{\phi,\theta}(a_1,a_2) \nu_{\psi,z}(a_1)}{a_1 a_2} \\
&= 1- \theta^{-1} - z\theta^{-2} - \theta^{-3} + O(\theta^{-4}).
\end{align*}
This agrees numerically with $\det( \Upsilon_{\psi,z}^{-1} \Log_{\phi}( \Upsilon_{\psi,z} ))$, as expected from Corollary~\ref{C:Lmunurxl}(bc).
\end{example}

\begin{example} \label{Ex:three}
($r=\ell=3$). Let $\phi$, $\psi : \FF_3[t] \to \FF_3[\theta][\tau]$ be given by
\[
\phi_t = \theta + \theta^2 \tau + \tau^2 + \tau^3, \quad
\psi_t = \theta + \tau + \tau^2 + \tau^3,
\]
so that for $\EE = \EE(\phi \times \psi)$ we have
\begin{align*}
\EE_t =
\begin{pmatrix}
\theta & 0 & 0 \\
0 & \theta & 0 \\
0 & 0 & \theta
\end{pmatrix}
+ \theta^2
\begin{pmatrix}
0 & 0 & z - \theta \\
1 & 0 & -1 \\
0 & 1 & -1
\end{pmatrix}
\tau &{}+
\begin{pmatrix}
0 & z-\theta & -z + \theta \\
0 & -1 & z+1-\theta^3 \\
1 & -1 & 0
\end{pmatrix}
\tau^2 \\
&{}+
\begin{pmatrix}
z-\theta & -z + \theta & 0 \\
-1 & z+1 - \theta^3 & -z + \theta^3 \\
-1 & 0 & z+1 - \theta^9
\end{pmatrix}
\tau^3.
\end{align*}
Here \cite{KhaochimP23}*{Eq.~(2.11)} implies that $R_{\phi} = \norm{\theta}^{1/2}$, and we find from \cite{KhaochimP23}*{Thm.~4.4} that $\dnorm{\Upsilon_{\psi,z}} = \norm{\theta}^{9/26} < R_{\phi}$. We calculate
\begin{align*}
L(\bsmu_{\phi,\theta} \times \bsnu_{\psi,z},0) &= \sum_{a_1, a_2 \in A_+} \frac{\bsmu_{\phi,\theta}(a_1,a_2) \bsnu_{\psi,z}(a_1,a_2)}{a_1 a_2} \\
&= 1 + \theta^{-1} + \theta^{-2} + (1-z)\theta^{-3} + O(\theta^{-4}),
\end{align*}
which agrees numerically with the expected value $(\theta-z) \cdot \omega_z / \tpi \cdot \det (\Upsilon_{\psi,z}^{-1} \Log_{\phi}(\Upsilon_{\psi,z}))$ from Corollary~\ref{C:Lmunurxr}(b).
\end{example}

\subsubsection{The case $\dnorm{\Upsilon_{\psi,z}} \geqslant R_{\phi}$} \label{SSS:bigUps}
Suppose that $\EE = \EE(\phi \times \psi)$ has been chosen with $\phi$, $\psi : \sA \to A[\tau]$ of ranks $r$ and $\ell$ respectively. Further suppose that $\dnorm{\Upsilon_{\psi,z}} \geqslant R_{\phi}$. Then the identities in Corollaries~\ref{C:Lmunurxr}(b) and~\ref{C:Lmunurxl}(c) do not hold, but one can approach determining $\Reg_{\EE}$ by using ideas of Demeslay~\cite{Demeslay22}*{Prop.~2.8} and Taelman~\cite{Taelman10}*{Thm.~1}.

Defining the polydisc $\rD(\KK_\infty) \assign \{ \bz \in \KK_{\infty}^{\ell} \mid \dnorm{\bz} < \dnorm{\Upsilon_{\psi,z}}^{-1} \cdot R_{\phi} \}$, we find that $\rD(\KK_\infty)$ is contained within the domain of convergence of $\Log_{\EE} = \Upsilon_{\psi,z}^{-1} \cdot \Log_{\phi} \cdot \Upsilon_{\psi,z}$ and that~\eqref{E:Ddef} applies. The exponential induces a surjective map of $\FF_q(z)$-vector spaces,
\begin{equation} \label{E:Expquot}
\KK_{\infty}^{\ell} \xrightarrow{\Exp_{\EE,\KK_{\infty}}}
\frac{\EE(\KK_{\infty})}{\EE(\AAA) + \rD(\KK_{\infty})}.
\end{equation}
As $\KK_{\infty} = \laurent{\FF_q(z)}{\theta^{-1}}$, it follows that the right-hand quotient in~\eqref{E:Expquot} is finite dimensional over $\FF_q(z)$: indeed, if we let $j_0 \assign \log_q(\dnorm{\Upsilon_{\psi,z}}^{-1} \cdot R_{\phi}) \in \RR$, then every vector in the right-hand space has entries in the $\FF_q(z)$-linear span of $\{ \theta^{j} \mid j \in \ZZ,\ j_0 \leqslant j \leqslant -1 \}$ modulo $\EE(\AAA) + \rD(\KK_{\infty})$. (We note that $\dnorm{\Upsilon_{\psi,z}} < R_{\phi}$ implies that $j_0 > 0$, in which case $\EE(\AAA) + \rD(\KK_{\infty}) = \EE(\KK_{\infty})$, and moreover the standard basis vectors are in $\EE(\AAA) \cap \rD(\KK_{\infty})$.)

Since $\KK_{\infty}$ is infinite dimensional over $\FF_q(z)$, we can find $\bx_1, \dots, \bx_{\ell} \in \KK_{\infty}^{\ell}$ so that 
\[
 \Exp_{\EE} ( \bx_1), \ldots,
 \Exp_{\EE} ( \bx_{\ell}) \in \EE(\AAA) + \rD(\KK_{\infty}).
\]
For each $\bx_i$, write $\Exp_{\EE}(\bx_i) = \bb_i + \by_i$ with $\bb_i \in \EE(\AAA)$ and $\by_i \in \rD(\KK_{\infty})$.
Then forming $\ell \times \ell$ matrices $\rX = (\bx_i)$, $\rB = (\bb_i)$, and $\rY = (\by_i)$, we have $\Exp_{\EE,\KK_{\infty}}(\rX) = \rB + \rY$, and
\[
\rX - \Log_{\EE}(\rY) \in \Mat_{1 \times \ell} \bigl( \Exp_{\EE,\KK_{\infty}}^{-1} \bigl( \EE(\AAA) \bigr) \bigr).
\]
When chosen appropriately, the columns of $\rX - \Log_{\EE}(\rY)$ are $\AAA$-linearly independent (or equivalently $\KK_{\infty}$-linearly independent since $\Exp_{\EE,\KK_{\infty}}^{-1}(\EE(\AAA))$ is an $\AAA$-lattice) and their $\AAA$-span comprises a submodule of $\smash{\Exp_{\EE,\KK_{\infty}}^{-1}}(\EE(\AAA))$ of finite $\AAA$-index, from which it is possible to determine $\smash{\Exp_{\EE,\KK_{\infty}}^{-1}}(\EE(\AAA))$ in full on a case by case basis.

\begin{example} \label{Ex:four}
Let $\phi : \FF_3[t] \to \FF_3[\theta][\tau]$ be defined by
$\phi_t = \theta + \theta^3 \tau + \tau^2$,
and let $\EE = \EE(\phi \times \phi)$, so
\[
\EE_t =
\begin{pmatrix}
\theta & 0 \\ 0 & \theta
\end{pmatrix}
+ \theta^3
\begin{pmatrix}
0 & -\theta + z \\
1 & -\theta^3
\end{pmatrix}
\tau +
\begin{pmatrix}
z-\theta & -z \theta^9 + \theta^{10}\\
-\theta^3 & z - \theta^3 + \theta^{12}
\end{pmatrix}
\tau^2.
\]
Similar to the previous examples, we compute that $R_{\phi} = 1$ and that $\dnorm{\Upsilon_{\psi,z}} = \norm{\theta}^{3/2} = 3^{3/2} > 1$. In this case $\rD(\KK_{\infty}) = \{ O(\theta^{-2}) \} \subseteq \KK_{\infty}^{2}$, and so $\EE(\KK_{\infty})/(E(\AAA) + \rD(\KK_{\infty})) \cong \FF_q(z) \cdot \theta^{-1}$.
By direct calculation we find
\begin{align*}
\Exp_{\EE} \begin{pmatrix} 1 \\ 0 \end{pmatrix} &=
\begin{pmatrix} 1 \\ 1 \end{pmatrix} +
\begin{pmatrix}
-\theta^{-5} + z\theta^{-6} - \theta^{-8} + z\theta^{-9} + \cdots \\
\theta^{-2}- \theta^{-3} + \theta^{-4} + \theta^{-8} - \theta^{-9} + \cdots
\end{pmatrix}, \\
\Exp_{\EE} \begin{pmatrix} \theta^{-1} \\ 0 \end{pmatrix} &=
\begin{pmatrix} 1 \\ 0 \end{pmatrix} \theta^{-1} +
\begin{pmatrix}
-\theta^{-14} + z \theta^{-15} + \cdots \\
\theta^{-3} + \theta^{-5} + \theta^{-7} + \cdots
\end{pmatrix}, \\
\Exp_{\EE} \begin{pmatrix} \theta \\ 0 \end{pmatrix} &=
\begin{aligned}[t]
\begin{pmatrix}
\theta^{10} - z\theta^{9} - \theta^{4} + z \theta^3 + z \\
\theta^{12} - \theta^6 - \theta^3 + \theta + z-1
\end{pmatrix}
&{}+ \begin{pmatrix} 0 \\ 1 \end{pmatrix} \theta^{-1}\\
&{}+
\begin{pmatrix}
-\theta^{-2} + z\theta^{-3} - \theta^{-4} + \cdots \\
-\theta^{-2}+  - \theta^{-3} + \theta^{-7} + \cdots
\end{pmatrix},
\end{aligned}
\\
\Exp_{\EE} \begin{pmatrix} 0 \\ 1 \end{pmatrix} &=
\begin{aligned}[t]
\begin{pmatrix}
-\theta^{10} + z\theta^{9} + \theta^{4} - z \theta^3 \\
-\theta^{12} + \theta^6 + \theta^3 - \theta - z-1
\end{pmatrix}
&{}+ \begin{pmatrix} -1 \\ -1 \end{pmatrix} \theta^{-1}\\
&\hspace*{-90pt}{}+
\begin{pmatrix}
(z+1)\theta^{-2} - (z+1)\theta^{-3} +(z+1) \theta^{-4} + \cdots \\
\theta^{-2} + (z-1)\theta^{-6} - \theta^{-7} + \cdots
\end{pmatrix}.
\end{aligned}
\end{align*}
Picking $\bx_1 = \bigl( \begin{smallmatrix} 1 \\ 0 \end{smallmatrix} \bigr)$ and $\bx_2 = \bigl( \begin{smallmatrix} \theta + \theta^{-1} \\ 1 \end{smallmatrix} \bigr)$, we find that $\Exp_{\EE}(\bx_1)$, $\Exp_{\EE}(\bx_2) \in \EE(\AAA) + \rD(\KK_{\infty})$. Moreover, defining $\by_1$, $\by_2 \in \rD(\KK_{\infty})$, as in the previous paragraph, we obtain
\begin{align*}
\bslambda_1 &\assign \begin{pmatrix} 1 \\ 0 \end{pmatrix} - \Log_{\EE}(\by_1) =
\begin{pmatrix}
1 - \theta^{-7} + (z+1)\theta^{-8} - (z+1)\theta^{-9} + \cdots \\
-\theta^{-2} - \theta^{-4} - \theta^{-5} - \theta^{-7} - \theta^{-8} - \theta^{-9} + \cdots
\end{pmatrix}, \\
\bslambda_2 &\assign \begin{pmatrix} \theta + \theta^{-1} \\ 1 \end{pmatrix} - \Log_{\EE}(\by_2) \\
&=
\begin{pmatrix}
\theta + \theta^{-1} - z\theta^{-2} + \theta^{-3} - z\theta^{-4} + \theta^{-5} -z\theta^{-6} + \theta^{-7}-z\theta^{-8} + \theta^{-9} + \cdots \\
1 - \theta^{-5} + \theta^{-6} - \theta^{-7} + z\theta^{-8} - (z+1) \theta^{-9} + \cdots
\end{pmatrix},
\end{align*}
and $\bslambda_1$, $\bslambda_2 \in \Exp_{\EE,\KK_{\infty}}^{-1}(\EE(\AAA))$. Taking a determinant we obtain,
\begin{multline} \label{E:detU}
\det(\bslambda_1,\bslambda_2) = 1 + \theta^{-1} - \theta^{-3} + (-z+1)\theta^{-4} \\ {}+ \theta^{-5} + z\theta^{-6} + (-z+1)\theta^{-7} + \theta^{-8} - (z+1)\theta^{-9} + O(\theta^{-10}).
\end{multline}
Since the infinite product in Demeslay's formula in Theorem~\ref{T:Demeslay} is a $1$-unit in $\KK_{\infty}$, the fact that $\det(\bslambda_1,\bslambda_2)$ is also a $1$-unit in $\KK_{\infty}$ implies that
\[
\Exp_{\EE,\KK_{\infty}}^{-1}\bigl( \EE(\AAA) \bigr) = \AAA \bslambda_1 + \AAA \bslambda_2 \quad \textup{and} \quad
\Aord{\rH(\EE)}{\AAA} = 1.
\]
Finally, as in previous examples we can compute directly that
\begin{multline*}
L(\bsmu_{\phi,\theta} \times \bsnu_{\phi,z},0) = \sum_{a \in A_+} \frac{\mu_{\phi,\theta}(a) \nu_{\phi,z}(a)}{a} = 1 + \theta^{-1} - \theta^{-2} + (z+1)\theta^{-3} + \theta^{-4} \\{}- (z+1)\theta^{-5} + (z^2-1)\theta^{-6} + (-z^2+z)\theta^{-8} + z^3 \theta^{-9} + O(\theta^{-10}),
\end{multline*}
and after multiplying by $\tpi/((\theta-z)\omega_z)$ we find that it indeed agrees with $\Reg_{\EE} = \det(\bslambda_1,\bslambda_2)$ in \eqref{E:detU} up to $O(\theta^{-10})$, as expected by Corollary~\ref{C:Lmunurxr}(a).
\end{example}

\section{Modules of Stark units} \label{S:Stark}

As suggested by one of the referees, we expand our investigations to modules of Stark units for $\EE = \EE(\phi \times \psi)$ in order to derive an explicit description of $L(\EE(\phi \times \psi)^{\vee},0)$ in terms of logarithms of special points. See Theorem~\ref{T:StarkLvalue} and the discussion on log-algebraicity in~\S\ref{SS:logalg}. Our main goal is to prove identities for $\EE$ similar to~\cite{AT17}*{Thm.~1} of Angl\`es and Tavares Ribeiro for Drinfeld modules, which was later extended to the setting of Anderson $t$-modules in~\cites{ANT20, ANT22, APT18}. See also~\cites{ANT17b, APT16, GreenNgoDac23, Tavares21}. We continue with the notation from \S\ref{S:RAtwists}--\ref{S:convolutions}.

\subsection{Deformations of \texorpdfstring{$\EE$}{E}} \label{SS:deformations}
Let $\bszeta$ be a new variable, and set $\tbA \assign \FF_q(z,\bszeta)[t]$ and $\tAA \assign \FF_q(z,\bszeta)[\theta]$. Following \cite{AT17}*{\S 2.4} we define the $t$-module
\[
\DD \assign \DD(\phi \times \psi) : \tbA \to \Mat_{\ell}(\tAA[\tau])
\]
by setting
\[
\DD_t \assign \theta \rI_{\ell} + \kappa_1 \Theta_{\psi,z} \bszeta \tau + \cdots + \kappa_r \Theta_{\psi,z} \Theta_{\psi,z}^{(1)} \cdots \Theta_{\psi,z}^{(r-1)} \bszeta^r \tau^r \in \Mat_{\ell}\bigl( \FF_q[z,\bszeta,\theta][\tau] \bigr).
\]
Then $\DD$ is a $t$-module over $\FF_q(z,\bszeta,\theta)$ in the sense of \S\ref{SS:cAtmodules} with $Z = \{z,\bszeta\}$, and we also consider $\DD$ over the complete field $\tKK_{\infty} \assign \laurent{\FF_q(z,\bszeta)}{\theta^{-1}}$. We refer to $\DD$ as the $\bszeta$-deformation of $\EE$. We consider the infinite product,
\begin{equation} \label{E:LDtA}
\cL(\DD/\tAA) \assign \prod_{\substack{f \in A_+ \\ \textup{irred.} }}  \frac{\bigl[ \Lie(\oDD)(\FF_f(z,\bszeta))]_{\tAA}}{\bigl[ \,\oDD (\FF_f(z,\bszeta)) \bigr]_{\tAA}} \quad \in \tKK_{\infty},
\end{equation}
which converges by Demeslay~\cite{Demeslay22}*{Thm.~2.9}. Demeslay's theorem (Theorem~\ref{T:Demeslay}) also applies so that
\begin{equation} \label{E:LDtAA}
\cL(\DD/\tAA) = \Reg_{\DD} \cdot [\rH(\DD)]_{\tAA}.
\end{equation}

In order to relate $\cL(\DD/\tAA)$ to $L(\EE(\phi \times \psi)^{\vee},0)$ we need to calculate $[\Lie(\oDD)(\FF_f(z,\bszeta))]_{\tAA}$ and $[\oDD(\FF_f(z,\bszeta))]_{\tAA}$ for each $f \in A_+$. For the former, we have $[\Lie(\oDD)(\FF_f(z,\bszeta))]_{\tAA} = f^{\ell}$ as before, whereas for the latter one would \emph{hope} that Theorem~\ref{T:charpoly1} would apply. However, this is not the case, as one of the hypotheses for Theorem~\ref{T:charpoly1} is that $\oDD$ would need to satisfy Definition~\ref{D:Tatemodules}(c), particularly that torsion modules of $\oDD$ have full dimension over $\FF_f(z,\bszeta)$. Instead we analyze $\oDD$ more carefully so that the techniques of Theorem~\ref{T:charpoly1} can be brought to bear. Inspired by results of Gezmi\c{s}~\cite{Gezmis19}*{Prop.~2.4} we obtain the following (cf.\ Theorem~\ref{T:EEAord}).

\begin{proposition} \label{P:deformationorder}
For $f \in A_+$ irreducible of degree $d$, fix $\bP_f(X) \in \AAA[X]$ as in Definition~\ref{D:EEcharpolys}. Then
\[
\bigl[\, \oDD(\FF_f(z,\bszeta)) \bigr]_{\tAA} = \frac{f^{\ell} \cdot \bP_f(\bszeta^d)}{\bP_f(0)}.
\]
In particular,
\[
\bigl[\, \oEE(\FF_f(z)) \bigr]_{\AAA} = \bigl[\, \oDD(\FF_f(z,\bszeta)) \bigr]_{\tAA}\, \big|_{\bszeta=1}.
\]
\end{proposition}

\begin{proof}
We write
\[
\EE_t = \theta \rI_{\ell} + M_1 \tau + \cdots + M_{r} \tau^{r}, \quad M_i \in \Mat_{\ell}(\AAA)
\]
so that
\[
\oEE_t = \otheta \rI_{\ell} + \oM_1 \tau + \cdots + \oM_{r} \tau^{\ell}, \quad
\oDD_t = \otheta \rI_{\ell} + \oM_1 \bszeta \tau + \cdots + \oM_{r} \bszeta^{r} \tau^r.
\]
Let $\cM(\oEE) \assign \Mat_{1 \times \ell}(\FF_f(z)[\tau])$ be the $t$-motive of $\oEE$, and let $\cM(\oDD) \assign \Mat_{1 \times \ell}(\FF_f(z,\bszeta)[\tau])$ be the $t$-motive of $\oDD$. Then as in \S\ref{SSS:AbAfin}, $\{ \tau^j \bs_i : 1 \leqslant i \leqslant \ell,\, 0 \leqslant j \leqslant r-1 \}$ forms an $\FF_f(z)[t]$-basis of $\cM(\oEE)$, as well as an $\FF_f(z,\bszeta)[t]$-basis of $\cM(\oDD)$. By \cite{NamoijamP24}*{Ex.~3.38, Ex.~4.129}, multiplication by $\tau$ on $\cM(\oEE)$ is represented with respect to this basis by
\[
\Gamma_{\oEE} = \begin{pmatrix}
0 & \rI_{\ell} & \cdots & 0 \\
\vdots & \vdots & \ddots & \vdots \\
0 & 0 & \cdots & \rI_{\ell} \\
\oM_r^{-1}(t-\otheta)\rI_{\ell} & -\oM_{r}^{-1} \oM_1 & \cdots & -\oM_{r}^{-1} \oM_{r-1}
\end{pmatrix},
\]
and likewise multiplication by $\tau$ on $\cM(\oDD)$ is represented by
\[
\Gamma_{\oDD} = \begin{pmatrix}
0 & \rI_{\ell} & \cdots & 0 \\
\vdots & \vdots & \ddots & \vdots \\
0 & 0 & \cdots & \rI_{\ell} \\
\bszeta^{-r}\oM_r^{-1}(t-\otheta)\rI_{\ell} & -\bszeta^{-r+1}\oM_{r}^{-1} \oM_1 & \cdots & -\bszeta^{-1} \oM_{r}^{-1} \oM_{r-1}
\end{pmatrix}.
\]
Letting $\Delta \assign \diag(\rI_{\ell}, \bszeta \rI_{\ell}, \dots, \bszeta^{r-1} \rI_{\ell})$, one checks that
\[
\Gamma_{\oDD} = \bszeta^{-1} \Delta^{-1} \Gamma_{\oEE} \Delta,
\]
which implies as in~\eqref{E:GHdef} (and using that $\Delta^{(1)} = \Delta$),
\begin{equation}
\rG_{\oDD} = \Gamma_{\oDD}^{(d-1)} \cdots \Gamma_{\oDD}^{(1)} \Gamma_{\oDD} = \bszeta^{-d} \Delta^{-1} \Gamma_{\oEE}^{(d-1)} \cdots \Gamma_{\oEE}^{(1)} \Gamma_{\oEE} \Delta = \bszeta^{-d} \Delta^{-1} \rG_{\oEE} \Delta.
\end{equation}
Let
\[
\tbP_f(X) \assign \Char( \rG_{\oDD},X) |_{t=\theta}, \quad \bigl( \textup{a priori $\in \FF_q(z,\bszeta)[\theta][X]$} \bigr).
\]
Then
\begin{multline*}
\tbP_f(\bszeta^{-d}X) = \det \bigl( \bszeta^{-d}X \cdot \rI - \rG_{\oDD} \bigr)\big|_{t=\theta}
= \det \bigl( \bszeta^{-d} X \cdot \rI - \bszeta^{-d} \rG_{\oEE} \bigr)\big|_{t=\theta} \\
= \bszeta^{-d\ell r} \det(X \cdot \rI - \rG_{\oEE} \bigr)\big|_{t=\theta}
= \bszeta^{-d \ell r} \bP_f(X),
\end{multline*}
where we have used Corollary~\ref{C:charpolys} in the last step. These expressions lie in $\bszeta^{-d \ell r} \cdot \FF_q(z)[\theta][X]$, and we further obtain
\begin{equation} \label{E:tbP1}
\tbP_f(1) = \bszeta^{-d \ell r}\bP_f(T^d).
\end{equation}

Now we can also consider the dual $t$-motive $\cN(\oDD)$ for $\oDD$, and one finds by~\cite{NamoijamP24}*{Ex.~3.38, Ex.~4.129} that for $C = ( \smash{\oM}_{r}^{(-r)})^{\tr}$,
\[
\Phi_{\oDD} = \begin{pmatrix}
0 & \rI_{\ell} & \cdots & 0 \\
\vdots & \vdots & \ddots & \vdots \\
0 & 0 & \cdots & \rI_{\ell} \\
\bszeta^{-r} C^{-1} (t-\otheta)\rI_{\ell} & -\bszeta^{-r+1}C^{-1} \smash{\oM}_1^{(-1)} & \cdots & -\bszeta^{-1} C^{-1} \smash{\oM}_{r-1}^{(-r+1)}
\end{pmatrix}
\]
represents multiplication by $\sigma$ on $\cN(\oDD)$. Moreover, if we let
\[
V = \begin{pmatrix}
\bszeta \smash{\oM}_1^{\tr} & \bszeta^2 (\smash{\oM}_2^{(-1)})^{\tr} & \cdots & \bszeta^{r-1} (\smash{\oM}_{r-1}^{(-r+2)})^{\tr} & \bszeta^r (\smash{\oM}_r^{(-r+1)})^{\tr} \\
\bszeta^2 \smash{\oM}_2^{\tr} & \bszeta^3 (\smash{\oM}_3^{(-1)})^{\tr} & \cdots & \bszeta^r (\smash{\oM}_{r}^{(-r+2)})^{\tr} & \\
\vdots & \vdots & \reflectbox{$\ddots$} \\
\bszeta^{r-1} \smash{\oM}_{r-1}^{\tr} & \bszeta^r (\smash{\oM}_r^{(-1)})^{\tr} \\
\bszeta^r \smash{\oM}_r^{\tr}
\end{pmatrix} \in \GL_r(\FF_f(z,\bszeta)),
\]
then by \cite{NamoijamP24}*{Ex.~4.129},
\begin{equation} \label{E:Vconj}
V^{(-1)} \Phi_{\oDD} = \Gamma_{\oDD}^{\tr} V.
\end{equation}
By a straightforward modification of \eqref{E:Estarorder}--\eqref{E:detI-G} in the proof of Theorem~\ref{T:charpoly1}, using Lemma~\ref{L:Ffpoints}(b) instead of Lemma~\ref{L:Ffpoints}(a), we find that
\begin{equation} \label{E:oDDorder}
\bigl[\, \oDD(\FF_f(z,\bszeta)) \bigr]_{\tAA} = \gamma \cdot \det \bigl( \rI - \rH_{\oDD} \bigr) \big|_{t=\theta}, \quad \gamma \in \FF_q(z,\bszeta)^{\times},
\end{equation}
where as in~\eqref{E:GHdef},
\[
\rH_{\oDD} = \Phi_{\oDD}^{(-d+1)} \cdots \Phi_{\oDD}^{(-1)} \Phi_{\oDD},
\]
and where $\gamma$ is chosen so that the expression has sign~$1$. Using~\eqref{E:Vconj},
\[
\rH_{\oDD} = \bigl( V^{-1} \big)^{(-d)} \bigl( \Gamma_{\oDD}^{(-d+1)} \bigr)^{\tr} \cdots \bigl( \Gamma_{\oDD}^{(-1)} \bigr)^{\tr} \Gamma_{\oDD}^{\tr} V = V^{-1} \bigl( \Gamma_{\oDD} \Gamma_{\oDD}^{(-1)} \cdots \Gamma_{\oDD}^{(-d+1)} \bigr)^{\tr} V.
\]
In the last step we have used that $V^{(-d)} = V$ since the entries of $V$ are in $\FF_f(z,\bszeta)$. Thus
\[
\rH_{\oDD} = V^{-1} \bigl( \rG_{\oDD}^{(-d+1)} \bigr)^{\tr} V,
\]
and
\[
\Char(\rH_{\oDD},X) = \Char (\rG_{\oDD}^{(-d+1)},X) = \Char(\rG_{\oDD},X)^{(-d+1)} = \Char(\rG_{\oDD},X),
\]
where in the last equality we have used that the coefficients of $\Char(\rG_{\oDD},X)$ lie in $\FF_q(z,\bszeta)[t]$ and are fixed by Frobenius twisting. Returning to~\eqref{E:oDDorder} and using~\eqref{E:tbP1},
\[
\bigl[ \oDD(\FF_f(z,\bszeta)) \bigr]_{\tAA} = \gamma \cdot \det \bigl( \rI - \rG_{\oDD} \bigr)\big|_{t=\theta} = \gamma \cdot \tbP_f(1) = \gamma \cdot \bszeta^{-d \ell r} \bP_f(\bszeta^d).
\]
Similar to the proof of Theorem~\ref{T:EEAord}, we find that $\gamma \cdot \bszeta^{-d\ell r} = f^{\ell}/\bP_f(0)$, which completes the first identity. The second identity then follows from Theorem~\ref{T:EEAord}.
\end{proof}

Recalling that $\tKK_{\infty} = \laurent{\FF_q(z,\bszeta)}{\theta^{-1}}$ is the completion of $K_{\infty}(z,\bszeta)$ with respect to the Gauss norm, we let
\[
\TT_{z,\bszeta}, \quad \TT_{z,\bszeta}(K_{\infty}) \assign \TT_{z,\bszeta} \cap \power{K_{\infty}}{z,\bszeta} = \laurent{\FF_q[z,\bszeta]}{\theta^{-1}},
\]
denote Tate algebras in the variables $z$ and $\bszeta$. Then also $\tKK_{\infty}$ is the completion of the fraction field of $\TT_{z,\bszeta}(K_{\infty})$.

\subsubsection{Deformations $\cL(\tAA,\chi_{\phi} \ochi_{\psi})$ and $\cL(\bsmu_{\phi,\theta} \times \bsnu_{\psi,z})$}
Similar to in \S\ref{SS:Lconvrxr}--\ref{SS:Lconvrxl}, we define
\[
\cL(\tAA,\chi_{\phi} \ochi_{\psi}) \assign \sum_{a \in A_+} \frac{\chi_{\phi}(a) \ochi_{\psi}(a) a(z)}{a} \bszeta^{r \deg a} \quad
\in \TT_{z,\bszeta}(K_{\infty})^{\times},
\]
which we note is a unit in $\TT_{z,\bszeta}(K_{\infty})$ since except for when $a=1$, the Gauss norm of each term is at most $\inorm{\theta}^{-1}< 1$. We also define $\cL(\bsmu_{\phi,\theta} \times \bsnu_{\psi,z}) \in \TT_{z,\bszeta}(K_{\infty})^{\times}$ in the following way. For $(a_1, \dots, a_{n}) \in A_+^{n}$, set
\[
\delta(a_1, \dots, a_{n}) \assign \deg a_1 + 2 \deg a_2 + \dots + n \deg a_{n}.
\]
When $r=\ell$, we define
\[
\cL(\bsmu_{\phi,\theta} \times \bsnu_{\psi,z}) \assign \sum_{a_1, \ldots, a_{r-1} \in A_+} \frac{\bsmu_{\phi,\theta}(a_1, \dots, a_{r-1}) \bsnu_{\psi,z}(a_1, \dots, a_{r-1})}{a_1 \cdots a_{r-1}} \bszeta^{\delta(a_1, \dots, a_{r-1})}. 
\]
When $r < \ell$, set
\[
\cL(\bsmu_{\phi,\theta} \times \bsnu_{\psi,z}) \assign \sum_{a_1, \ldots, a_{r} \in A_+} \frac{\chi_{\phi}(a_r)\bsmu_{\phi,\theta}(a_1, \dots, a_{r-1}) \bsnu_{\psi,z}(a_1, \dots, a_{r},1, \dots, 1)}{a_1 \cdots a_{r}} \bszeta^{\delta(a_1, \dots, a_{r})}, 
\]
and when $r > \ell$, set
\begin{multline*}
\cL(\bsmu_{\phi,\theta} \times \bsnu_{\psi,z}) \assign \\
\sum_{a_1, \ldots, a_{\ell} \in A_+} \frac{\ochi_{\psi}(a_\ell) a_{\ell}(z) \bsmu_{\phi,\theta}(a_1, \dots, a_{\ell},1, \dots, 1) \bsnu_{\psi,z}(a_1, \dots, a_{\ell-1})}{a_1 \cdots a_{\ell}} \bszeta^{\delta(a_1, \dots, a_{\ell})}.
\end{multline*}
That $\cL(\bsmu_{\phi,\theta} \times \bsnu_{\psi,z}) \in \TT_{z,\bszeta}(K_{\infty})^{\times}$ follows from the same arguments as in \S\ref{SS:Lconvrxr}--\S\ref{SS:Lconvrxl}.

\begin{corollary} \label{C:deformationconv}
For $\DD : \tbA \to \Mat_{\ell}(\tAA[\tau])$ defined as above. The following hold.
\begin{alphenumerate}
\item If $r=\ell$, then
\[
\cL(\DD/\tAA) = \cL(\tAA, \chi_{\phi} \ochi_{\psi}) \cdot
\cL(\bsmu_{\phi,\theta} \times \bsnu_{\psi,z}) \quad \in \TT_{z,\bszeta}(K_{\infty})^{\times}.
\]
\item If $r \neq \ell$, then
\[
\cL(\DD/\tAA) = \cL(\bsmu_{\phi,\theta} \times \bsnu_{\psi,z}) \quad \in \TT_{z,\bszeta}(K_{\infty})^{\times}.
\]
\end{alphenumerate}
\end{corollary}

\begin{proof}
For $f \in A_+$ irreducible of degree~$d$, Proposition~\ref{P:deformationorder}, we see that
\[
\frac{\bigl[ \Lie(\oDD)(\FF_f(z,\bszeta))]_{\tAA}}{\bigl[\, \oDD(\FF_f(z,\bszeta)) \bigr]_{\tAA}} = \frac{\bP_f(0)}{\bP_f(\bszeta^d)} = \bQ_f^{\vee} \bigl( \bszeta^d \bigr)^{-1}.
\]
If we first consider the case $r=\ell$, then just as in \S\ref{SS:Lconvrxr}, Cauchy's identity (Theorem~\ref{T:Cauchy}) implies that
\begin{multline*}
\bQ_f^{\vee} \bigl(\bszeta^d \bigr)^{-1} \\
= \begin{aligned}[t]
\biggl( 1 - \frac{\chi_{\phi}(f) \ochi_{\psi}(f) f(z)}{f} \bszeta^{rd} \biggr)^{-1} \smash{\sum_{k_1, \ldots, k_{r-1} \geqslant 0}} \bsmu_{\phi,\theta} &\bigl(f^{k_1}, \dots, f^{k_{r-1}} \bigr) \bsnu_{\psi,z} \bigl( f^{k_1}, \dots, f^{k_{r-1}} \bigr) \\
&{}\cdot f^{-k_1-k_2 - \cdots - k_{r-1}} \bszeta^{d(k_1 + 2k_2 + \cdots (r-1)k_{r-1})}.
\end{aligned}
\end{multline*}
Multiplying over all irreducible $f \in A_+$, we obtain the product in~(a). Likewise, as in \S\ref{SS:Lconvrxl}, when $r < \ell$ we apply Bump's specialization of Cauchy's identity (Corollary~\ref{C:Cauchynl}) to obtain
\[
\bQ_f^{\vee}\bigl( \bszeta^d \bigr)^{-1}
= \begin{aligned}[t]
 \smash{\sum_{k_1, \ldots, k_{r} \geqslant 0}} \bsmu_{\phi,\theta} \bigl(f^{k_1}, \dots, f^{k_{r-1}} \bigr) &\bsnu_{\psi,z} \bigl( f^{k_1}, \dots, f^{k_{r}},1, \ldots, 1 \bigr) \\
&{}\cdot \chi_{\phi} \bigl(f^{k_r} \bigr) f^{-k_1-k_2 - \cdots - k_{r}} \bszeta^{d(k_1 + 2k_2 + \cdots + rk_{r})}.
\end{aligned}
\]
Again multiplying over all irreducible $f \in A_+$ yields the product in (b) for $r < \ell$. The case for $r > \ell$ is similar.
\end{proof}

\begin{remark} \label{R:LEELDDeval}
From Theorems~\ref{T:LEErxr} and~\ref{T:LEErxl} and Corollary~\ref{C:deformationconv}, we have
\begin{equation} \label{E:DAeval}
L(\EE(\phi \times \psi)^{\vee},0) = \cL(\DD/\tAA) |_{\bszeta=1}.
\end{equation}
\end{remark}

\subsection{Stark units} \label{SS:Stark}
Combining various aspects of \cites{ANT20, ANT22, APT18, AT17} on regulators, class modules, and Stark units, we investigate these objects for the $t$-modules 
\[
\EE : \bA \to \Mat_{\ell}(A[z][\tau]), \quad \DD : \tbA \to \Mat_{\ell}(A[z,\bszeta][\tau]).
\]
We adapt the definitions and methods of \citelist{\cite{ANT22}*{\S 2--4} \cite{AT17}*{\S 2--3}}, which cover the cases of $t$-modules over algebraic extensions of $K$ and Drinfeld modules over Tate algebras. However, the reader should be aware that there is an interchange in notation, in that $z$ and~$\bszeta$ here correspond with $t$ and $z$ respectively in~\cite{AT17}*{\S 3}.

Using the expression for $\Exp_{\EE}$ from~\eqref{E:ExpLogEE}, we find that
\[
\Exp_{\DD} = \sum_{i=0}^{\infty} B_i \Theta_{\psi,z} \Theta_{\psi,z}^{(1)} \cdots \Theta_{\psi,z}^{(i-1)} \bszeta^i \tau^i \quad \in \power{\Mat_{\ell}(K[z,\bszeta])}{\tau}.
\]
We define unit modules (cf.~\eqref{E:bElattice}),
\begin{align} \label{E:unitmods}
\rU(\EE/\AAA) &\assign \bigl\{ \bz \in \Lie(\EE)(\KK_{\infty}) \bigm| \Exp_{\EE}(\bz) \in \EE(\AAA) \bigr\} = \Exp_{\EE,\KK_{\infty}}^{-1} \bigl( \EE(\AAA) \bigr), \\
\rU(\EE/A[z]) &\assign \bigl\{ \bz \in \Lie(\EE)(\TT_z(K_{\infty})) \bigm| \Exp_{\EE}(\bz) \in \EE(A[z]) \bigr\}, \notag \\
\rU(\DD/\tAA) &\assign \bigl\{ \bz \in \Lie(\DD)(\tKK_{\infty}) \bigm| \Exp_{\DD}(\bz) \in \DD(\tAA) \bigr\}, \notag \\
\rU(\DD/A[z,T]) &\assign \bigl\{ \bz \in \Lie(\DD)(\TT_{z,\bszeta}(K_{\infty})) \bigm| \Exp_{\DD}(\bz) \in \DD(A[z,\bszeta]) \bigr\}. \notag
\end{align}
We see that $\rU(\EE/\AAA)$ is an $\bA$-module and $\rU(\EE/A[z])$ is an $\sA[z]$-module, and likewise $\rU(\DD/\tAA)$ is an $\tbA$-module and $\rU(\DD/A[z,\bszeta])$ is an $\sA[z,\bszeta]$-module. Demeslay~\cite{Demeslay22}*{Prop.~2.8} proved that $\rU(\DD/\tAA)$ is an $\tAA$-lattice in $\Lie(\DD)(\tKK_{\infty})$. In addition to the class module $\rH(\EE) = \rH(\EE/\AAA)$ of~\eqref{E:bEclass}, we also have 
\begin{align}
\rH(\EE/A[z]) &\assign \frac{\EE(\TT_z(K_{\infty}))}{\Exp_{\EE}(\Lie(\EE)(\TT_z(K_{\infty}))) + \EE(A[z])}, \\
\rH(\DD/\tAA) &\assign \frac{\DD(\tKK_{\infty})}{\Exp_{\DD}(\Lie(\DD)(\tKK_{\infty})) + \DD(\tAA)}, \notag \\
\rH(\DD/A[z,\bszeta]) &\assign \frac{\DD(\TT_{z,\bszeta}(K_{\infty}))}{\Exp_{\DD}(\Lie(\DD)(\TT_{z,\bszeta}(K_{\infty}))) + \DD(A[z,\bszeta])}. \notag
\end{align}
Then $\rH(\EE/\AAA)$ is an $\bA$-module and $\rH(\EE/A[z])$ is an $\sA[z]$-module, and likewise $\rH(\DD/\tAA)$ is an $\tbA$-module and $\rH(\DD/A[z,\bszeta])$ is an $\sA[z,\bszeta]$-module. Demeslay~\cite{Demeslay22}*{Prop.~2.8} proved that $\rH(\DD/\tAA)$ is a finite dimensional $\FF_q(z,\bszeta)$-vector space and thus a finitely generated torsion $\tbA$-module.

In \cite{AT17}*{\S 3}, Angl\`es and Tavares Ribeiro prove various properties and relations among these modules in the setting of $\bszeta$-deformations of Drinfeld modules defined over~$A[z]$. As~$\EE$ is defined over $A[z]$, their results transfer readily to this higher dimensional setting.

\begin{proposition}[{cf.~Angl\`es, Tavares Ribeiro \cite{AT17}*{Prop.~1, Lem.~6}}] \label{P:integrality1}
The following hold.
\begin{alphenumerate}
\item $\rU(\EE/\AAA)$ is the $\FF_q(z)$-span of $\rU(\EE/A[z])$ in $\Lie(\EE)(\KK_{\infty})$.
\item $\rU(\EE/A[z])$ is finitely generated as an $\sA[z]$-module.
\item There exists an $\bA$-basis $\bslambda_1, \dots, \bslambda_{\ell}$ of $\rU(\EE/\AAA)$ such that $\bslambda_i \in \rU(\EE/A[z])$ and
\[
\Reg_{\EE} = \det_{1 \leqslant i \leqslant \ell}(\bslambda_i) \in \TT_z(K_{\infty})^{\times}.
\]
\item $[\rH(\EE/\AAA)]_{\AAA} \in A[z] \cap \TT_z(K_{\infty})^{\times}$.
\end{alphenumerate}
\end{proposition}

\begin{proof}
Part~(a) is essentially the same as~\cite{AT17}*{Prop.~1(1)}. Part~(b) is similar to~\cite{AT17}*{Prop.~1(2)}. Indeed by~(a), we note that for $\bslambda \in \rU(\EE/\AAA)$, there is $\delta \in \FF_q[z]$ so that $\delta \bslambda \in \rU(\EE/A[z])$. Thus we can pick an $\bA$-basis $\bslambda_1, \dots, \bslambda_{\ell}$ of $\rU(\EE/\AAA)$ such that $\bslambda_i \in \rU(\EE/A[z])$.  Identifying $\Lie(\EE)(\KK_{\infty})$ with $\KK_{\infty}^{\ell}$, we have $\bslambda_i \in \TT_z(K_{\infty})^{\ell}$. Since $\rU(\EE/\AAA)$ is an $\AAA$-lattice in $\Lie(\EE)(\KK_{\infty})$, any $\bA$-basis of $\rU(\EE/\AAA)$ will also be a $\KK_{\infty}$-basis of $\Lie(\EE)(\KK_{\infty})$. Thus
\[
\rU(\EE/\AAA) = \bigoplus_{i=1}^{\ell} \AAA \bslambda_i, \quad
\KK_{\infty}^{\ell} = \bigoplus_{i=1}^{\ell} \KK_{\infty} \bslambda_i.
\]
Set
\[
N = \bigoplus_{i=1}^{\ell} \sA[z] \cdot \bslambda_i = \bigoplus_{i=1}^{\ell} A[z] \bslambda_i.
\]
Letting $V$ be the $\TT_{z}(K_{\infty})$-span of $N$ and $W$ be the $\TT_{z}(K_{\infty})$-span of $\rU(\EE/A[z])$, we have $V \subseteq W \subseteq \TT_{z}(K_{\infty})^{\ell}$. By part (a), $\rU(\EE/A[z]) \subseteq \FF_q(z) \cdot N$, which implies there exists $\delta \in \FF_q[z] \setminus \{0\}$ so that $\delta W \subseteq V$. Therefore,
\[
\delta \rU(\EE/A[z]) \subseteq V \cap \FF_q(z)N = N,
\]
and so $\rU(\EE/A[z])$ is finitely generated over $\sA[z]$.

For (c)--(d), we proceed as in~\cite{AT17}*{Lem.~6}. Choosing $\bslambda_1, \dots, \bslambda_\ell \in \rU(\EE/A[z]) \subseteq \TT_z(K_{\infty})^{\ell}$ as above, we let
\[
\varepsilon = \det_{1 \leqslant i \leqslant \ell}(\bslambda_i) \in \TT_z(K_{\infty}).
\]
In this way as in \S\ref{SS:Demeslayformula},
\[
\Reg_{\EE} = [ \Lie(\EE)(\KK_{\infty}) : \rU(\EE/\AAA) ]_{\AAA} = \gamma \cdot \varepsilon, \quad \gamma \in \FF_q(z)^{\times},
\]
where $\gamma$ is chosen so that $\Reg_{\EE}$ has sign~$1$ (leading coefficient~$1$ with respect to~$\theta$ as an element of $\laurent{\FF_q(z)}{\theta^{-1}}$) and the covolume $[-:-]_{\AAA}$ is defined as in~\cite{Demeslay22}*{\S 2}. We note that for any $\delta \in \FF_q[z] \setminus \{0\}$,
\[
A[z] \cap \delta \TT_z(K_{\infty}) = \delta A[z].
\]
In particular, since $\TT_z(K_{\infty})$ is a principal ideal domain, we can use the elementary divisors theorem to adjust $\bslambda_1, \dots, \bslambda_\ell$ if necessary to assume that $\varepsilon$ is not divisible in $\TT_z(K_{\infty})$ by any element of $\FF_q[z]$ (i.e., that $\varepsilon$ is primitive). Thus for $\delta \in \FF_q[z]$ irreducible, if we let $\ord_{\delta}$ denote is valuation on $\LL_z(K_{\infty})$, then $\ord_{\delta}(\varepsilon) = 0$.

By Theorem~\ref{T:Demeslay} and Proposition~\ref{P:LEE0},
\[
L(\EE^{\vee},0) = \Reg_{\EE} \cdot [ \rH(\EE/\AAA)]_{\AAA} = \gamma \cdot \varepsilon \cdot [ \rH(\EE/\AAA)]_{\AAA},
\]
and by Theorems~\ref{T:LEErxr} and~\ref{T:LEErxl}, this value is in $\TT_z(K_{\infty})^{\times}$.  Thus for $\delta \in \FF_q[z]$ irreducible,
\[
\ord_{\delta}(\gamma) + \ord_{\delta}([ \rH(\EE/\AAA)]_{\AAA}) = 0.
\]
Since $[ \rH(\EE/\AAA)]_{\AAA}$ has sign~$1$, we have $\ord_{\delta} ([ \rH(\EE/\AAA)]_{\AAA}) \leqslant 0$, and so it must be that $\gamma \in \FF_q[z]$. Since $L(\EE^{\vee},0)$ and $[\rH(\EE/\AAA)]_{\AAA}$ both have sign~$1$, we also have
\[
\sgn(\gamma \cdot \varepsilon) = \gamma \cdot \sgn(\varepsilon) = 1.
\]
Since $\gamma$, $\sgn(\varepsilon) \in \FF_q[z]$, we conclude that both are in $\FF_q^{\times}$. Adjusting the $\bslambda_i$'s by $\FF_q^{\times}$-multiples if necessary, we can arrange that $\gamma = \sgn(\varepsilon) = 1$. It then follows that for each $\delta \in \FF_q[z]$ irreducible, $\ord_{\delta}([ \rH(\EE/\AAA)]_{\AAA}) = 0$, and so $[\rH(\EE/\AAA)]_{\AAA} \in A[z] \subseteq \TT_z(K_{\infty})$. Since $L(\EE^{\vee},0) \in \TT_z(K_{\infty})^{\times}$, we conclude that $\varepsilon$ and $[\rH(\EE/\AAA)]_{\AAA}$ are both in $\TT_z(K_{\infty})^{\times}$.
\end{proof}
\begin{remark}
Choosing $\bslambda_1, \dots, \bslambda_\ell \in \rU(\EE/A[z])$ as in this proposition, we can set for $1 \leqslant i \leqslant \ell$,
\[
\Exp_{\EE}(\bslambda_i) \rassign \bsalpha_i \in \EE(A[z]).
\]
As in \eqref{E:EERegLog1}, if the entries of $\Upsilon_{\psi,z}$ are within the radius of convergence of $\Log_{\phi}$, then we can choose the $\bslambda_i$'s so that $\bsalpha_i = \bs_i \in \EE(\KK_{\infty})$ are the standard basis vectors. Otherwise, the precise polynomials $\bsalpha_i$ are difficult to determine, similar for example to Example~\ref{Ex:four}.
\end{remark}

We now extend the considerations of Proposition~\ref{P:integrality1} to the $\bszeta$-deformation $\DD$ of $\EE$, continuing to follow results in~\citelist{\cite{ANT22}*{\S 2} \cite{AT17}*{\S 3}}.

\begin{proposition}[{cf.~Angl\`es, Ngo Dac, Tavares Ribeiro \citelist{\cite{ANT22}*{\S 2.2} \cite{AT17}*{Prop.~5}}}] \label{P:integrality2}
The following hold.
\begin{alphenumerate}
\item $\rU(\DD/\tAA)$ is the $\FF_q(z,\bszeta)$-span of $\rU(\DD/A[z,\bszeta])$ in $\Lie(\DD)(\tKK_{\infty})$.
\item $\rU(\DD/A[z,\bszeta])$ is finitely generated as an $\sA[z,\bszeta]$-module.
\item There exists an $\tbA$-basis $\bseta_1, \dots, \bseta_{\ell}$ of $\rU(\DD/\tAA)$ such that each $\bseta_i \in \rU(\DD/A[z,\bszeta])$ and
\[
\Reg_{\DD} = \det_{1 \leqslant i \leqslant \ell}(\bseta_i) \in \TT_{z,\bszeta}(K_{\infty})^{\times}.
\]
\item $\rH(\DD/A[z,\bszeta])$ is a finitely generated $\FF_q[z,\bszeta]$-module, and $\rH(\DD/\tAA) = \{0\}$.
\end{alphenumerate}
In particular,
\[
\cL(\DD/\tAA) = \Reg_{\DD}.
\]
\end{proposition}

\begin{proof}
As in the proof of Proposition~\ref{P:integrality1}, part~(a) is similar to~\cite{AT17}*{Prop.~1(1)}. Part~(b) is similar to the proof of Proposition~\ref{P:integrality1}(b).
Part (c) is similar to the proof of Proposition~\ref{P:integrality1}(c), using that $\cL(\DD/\tAA) \in \TT_{z,\bszeta}(K_{\infty})^{\times}$ as in Corollary~\ref{C:deformationconv} and that $\TT_{z,\bszeta}(K_{\infty})$ is a unique factorization domain.

For (d), we proceed as in \citelist{\cite{ANT22}*{Prop.~2.2} \cite{AT17}*{Prop.~2, Prop.~5}}. We have
\[
\TT_{z,\bszeta}(K_{\infty}) = A[z,\bszeta] \oplus \rD,
\]
where $\rD = \{ \alpha \in \TT_{z,\bszeta}(K_{\infty}) \mid \dnorm{\alpha} < 1 \}$. As in \S\ref{SSS:Expfnpers}, $\Exp_{\EE}$ induces an isomorphism of $\FF_q[z,\bszeta]$-modules on $\theta^{-N} \rD^{\ell}$ for $N$ sufficiently large. It follows as in~\cite{Demeslay22}*{Prop.~2.8} that $\rH(\DD/A[z,\bszeta])$ is finitely generated over $\FF_q[z,\bszeta]$, and is thus a finitely generated and torsion $\sA[z,\bszeta]$-module. Furthermore, the $\FF_q(z,\bszeta)$-module generated by $\TT_{z,\bszeta}(K_{\infty})$ is dense in $\tKK_{\infty}$, so the inclusion $\TT_{z,\bszeta}(K_{\infty}) \hookrightarrow \tKK_{\infty}$ implies that the induced map
\begin{equation} \label{E:HDAiso}
\rH(\DD/A[z,\bszeta]) \otimes_{\FF_q[z,\bszeta]} \FF_q(z,\bszeta) \iso \rH(\DD/\tAA)
\end{equation}
is an isomorphism of $\tbA$-modules.

Now as $\TT_{z,\bszeta}(K_{\infty}) = \TT_{z}(K_{\infty}) \oplus \bszeta \TT_{z,\bszeta}(K_{\infty})$, it follows that for $\bx \in \TT_{z,\bszeta}(K_{\infty})^{\ell}$,
\[
\Exp_{\DD}(\bx) \equiv \bx|_{\bszeta=0} \pmod{\bszeta \TT_{z,\bszeta}(K_{\infty})^{\ell}},
\]
and so
\[
\TT_{z,\bszeta}(K_{\infty})^{\ell} = \bszeta \TT_{z,\bszeta}(K_{\infty})^{\ell} + \Exp_{\DD} \bigl( \TT_{z,\bszeta}(K_{\infty})^{\ell} \bigr).
\]
Therefore, multiplication by $\bszeta$ on $\rH(\DD/A[z,\bszeta])$ is surjective, and we have an exact sequence of finitely generated $\FF_q[z,\bszeta]$-modules,
\[
0 \to \rH(\DD/A[z,\bszeta])[\bszeta] \to \rH(\DD/A[z,\bszeta]) \xrightarrow{\bszeta(\cdot)} \rH(\DD/A[z,\bszeta]) \to 0.
\]
Tensoring with $\FF_q(z)$ over $\FF_q[z]$, we have that $\FF_q(z) \otimes_{\FF_q[z]} \rH(\DD/A[z,\bszeta])$ is a finitely generated $\FF_q(z)[\bszeta]$-module on which multiplication by $\bszeta$ is surjective. By the structure theorem of finitely generated $\FF_q(z)[\bszeta]$-modules, it must be a torsion $\FF_q(z)[\bszeta]$-module with no $\bszeta$-torsion. Tensoring then with $\FF_q(z,\bszeta)$ over $\FF_q(z)$ and using the isomorphism in~\eqref{E:HDAiso}, we obtain that $\rH(\DD/\tAA) = \{ 0 \}$.
\end{proof}

\begin{remark} \label{R:EEclass}
Similar considerations show that $\rH(\EE/A[z])$ is a finitely generated $\FF_q[z]$-module and that, as in \eqref{E:HDAiso}, $\FF_q(z) \otimes_{\FF_q[z]} \rH(\EE/A[z]) \cong \rH(\EE/\AAA)$ as $\bA$-modules.
\end{remark}

\subsubsection{Stark units for $\EE=\EE(\phi \times \psi)$}
With Propositions~\ref{P:integrality1} and~\ref{P:integrality2} in hand, we can define the module of Stark units, as in~\citelist{\cite{ANT22}*{\S 2.2} \cite{AT17}*{\S 3}}. See also \cites{ANT17b, ANT20, APT16, APT18, Tavares21}. We especially follow the situation of \cite{AT17}*{\S 3}, where as in our case the coefficient rings include additional variables.

For $n \geqslant 1$, define the map $\ev : \TT_{z,\bszeta}(K_{\infty})^n \to \TT_z(K_{\infty})^n$ by
\[
\ev(\bx) \assign \bx|_{\bszeta = 1}, \quad \bx \in \TT_{z,\bszeta}(K_{\infty})^n,
\]
which is a continuous map of $\FF_q[z]$-algebras. Because $1$ is fixed by Frobenius, for $\bx \in \TT_{z,\bszeta}(K_{\infty})^{\ell}$ we have
\[
\ev \bigl( \Exp_{\DD}(\bx) \bigr) = \Exp_{\EE}\bigl( \ev(\bx) \bigr), \quad
\ev \bigl( \DD_t(\bx) \bigr) = \EE_t \bigl( \ev(\bx) \bigr).
\]
As in \cite{AT17}*{\S 3}, we conclude that we have an induced isomorphism of $\sA[z]$-modules,
\begin{equation}
\frac{\rH(\DD/A[z,\bszeta])}{(\bszeta - 1) \rH(\DD/A[z,\bszeta])} \iso \rH(\EE/A[z]).
\end{equation}
On the level of unit modules we use the definitions in~\eqref{E:unitmods} to obtain an inclusion of $\sA[z]$-modules
\[
\rU_{\St}(\EE/A[z]) \assign \ev \bigl( \rU(\DD/A[z,\bszeta]) \bigr) \subseteq \rU(\EE/A[z]),
\]
which is the module of \emph{Stark units} for $\EE/A[z]$.

\begin{proposition}[{cf.\ Angl\`es, Tavares Ribeiro~\cite{ANT22}*{Prop.~3}}] \label{P:oalphaiso}
Define $\alpha : \TT_{z,\bszeta}(K_{\infty})^{\ell} \to \TT_{z,\bszeta}(K_{\infty})^{\ell}$ by
\[
\alpha(\bx) \assign \frac{1}{\bszeta - 1} \Bigl( \Exp_{\DD}(\bx) - \Exp_{\EE}(\bx) \Bigr).
\]
Then $\alpha$ induces an isomorphism of $\sA[z]$-modules,
\[
\oalpha : \frac{\rU(\EE/A[z])}{\rU_{\St}(\EE/A[z])} \iso \rH(\DD/A[z,\bszeta])[\bszeta - 1].
\]
\end{proposition}

\begin{proof}
The proof is almost the same as \cite{AT17}*{Prop.~3}, but we include the details for completeness. Because $\alpha$ is $\FF_q[z]$-linear, it induces an $\FF_q[z]$-linear map
\[
\oalpha : \rU(\EE/A[z]) \to \rH(\DD/A[z,\bszeta]).
\]
Write $\EE_t = \theta \rI_{\ell} + M_1 \tau + \cdots + M_r \tau^r$ and $\DD_t = \theta \rI_{\ell} + M_1 \bszeta \tau + \cdots + M_r \bszeta^r \tau^r$. For $\bx \in \rU(\EE/A[z])$,
\[
\alpha(\theta \bx) = \DD_t(\alpha(\bx)) + \underbrace{\Biggl( \sum_{j=1}^r \biggl( \frac{\bszeta^j-1}{\bszeta-1}\biggl) M_j \tau^j \Biggr) \bigl( \Exp_{\EE}(\bx) \bigr)}_{\in \DD(A[z,\bszeta])},
\]
where the second term is in $\DD(A[z,\bszeta])$ since $\Exp_{\EE}(\bx) \in A[z]^{\ell}$ by the definition of $\rU(\EE/A[z])$ and each $M_j \in \Mat_{\ell}(A[z])$. Thus $\oalpha$ is a morphism of $\sA[z]$-modules. Furthermore, for $\bx \in \rU(\EE/A[z])$,
\[
(\bszeta - 1)\alpha(\bx) = \Exp_{\DD}(\bx) - \Exp_{\EE}(\bx) \in \Exp_{\DD}\bigl( \TT_{z,\bszeta}(K_{\infty})^{\ell} \bigr) + \DD(A[z,\bszeta]),
\]
and so we have
\begin{equation} \label{E:oalpha}
\oalpha : \rU(\EE/A[z]) \to \rH(\DD/A[z,\bszeta])[\bszeta -1].
\end{equation}

We first show $\im \oalpha = \rH(\DD/A[z,\bszeta])[\bszeta -1]$. Suppose $\bx \in \TT_{z,\bszeta}(K_{\infty})^{\ell}$ represents a class in $\rH(\DD/A[z,\bszeta])[\bszeta -1]$. Then for some $\by \in \TT_{z,\bszeta}(K_{\infty})^{\ell}$ and $\bsalpha \in A[z,\bszeta]^{\ell}$,
\[
(\bszeta - 1)\bx = \Exp_{\DD}(\by) + \bsalpha.
\]
We write $\by = \bu + (z-1)\bv$ with $\bu \in \TT_z(K_{\infty})^{\ell}$, $\bv \in \TT_{z,\bszeta}(K_{\infty})^{\ell}$, and $\bsalpha = \bsbeta + (\bszeta-1)\bsgamma$ with $\bsbeta \in A[z]^{\ell}$, $\bsgamma \in A[z,\bszeta]$. Substituting into the above expression yields
\[
(\bszeta-1) \bigl( \bx - \bsgamma - \Exp_{\DD}(\bv) \bigr) = \Exp_{\DD}(\bu) + \bsbeta,
\]
and evaluating $\bszeta =1$ yields $\ev\bigl( \Exp_{\DD}(\bu) \bigr) + \bsbeta = \Exp_{\EE}(\bu) + \bsbeta = 0$, and so
\[
\bsbeta = -\Exp_{\EE}(\bu) \in A[z]^{\ell}.
\]
By \eqref{E:unitmods} this implies that $\bu \in \rU(\EE/A[z])$, and therefore, in $\TT_{z,\bszeta}(K_{\infty})^{\ell}$,
\[
(\bszeta - 1)\bigl( \bx - \Exp_{\DD}(\bv) - \bsgamma \bigr) = \Exp_{\DD}(\bu) - \Exp_{\EE}(\bu) = (\bszeta - 1) \alpha(\bu).
\]
But then canceling $\bszeta-1$ yields $\bx - \Exp_{\DD}(\bv) - \bsgamma = \alpha(\bu)$, and so $\bx$ represents the same class as $\alpha(\bu)$ in $\rH(\DD/A[z,\bszeta])[\bszeta - 1]$. Since $\bu \in \rU(\EE/A[z])$, the map $\oalpha$ in \eqref{E:oalpha} is surjective.

Next suppose that $\bx \in \rU_{\St}(\EE/A[z])$. By the definition of $\rU_{\St}(\EE/A[z])$, there exist $\bu \in \rU(\DD/A[z,\bszeta])$ and $\bv \in \TT_{z,\bszeta}(K_{\infty})^{\ell}$ such that $\bx = \bu + (\bszeta - 1)\bv$. Therefore,
\[
\Exp_{\DD}(\bx) = \Exp_{\DD}(\bu) + (\bszeta - 1)\Exp_{\DD}(\bv).
\]
But $\Exp_{\DD}(\bu) \in A[z,\bszeta]^{\ell}$ and $\ev(\Exp_{\DD}(\bu)) = \Exp_{\EE}(\bx) \in A[z]^{\ell}$, and so
\[
\Exp_{\DD}(\bx) - \Exp_{\EE}(\bx) = \bigl( \Exp_{\DD}(\bu) - \Exp_{\EE}(\bx) \bigr) + (\bszeta - 1)\Exp_{\DD}(\bv)
\]
is an element of $(\bszeta - 1) \bigl( \Exp_{\DD}(\TT_{z,\bszeta}(K_{\infty})^{\ell}) + \DD(A[z,\bszeta]) \bigr)$. This implies that $\oalpha(\bx) = 0$.

Finally, suppose $\bx \in \rU(\EE/A[z])$ represents an element of $\ker \oalpha$. Then
\[
\alpha(\bx) = \frac{1}{\bszeta - 1} \Bigl( \Exp_{\DD}(\bx) - \Exp_{\EE}(\bx) \Bigr) \in \Exp_{\DD}(\TT_{z,\bszeta}\bigl( K_{\infty})^{\ell} \bigr) + \DD(A[z,\bszeta]).
\]
Since $\Exp_{\EE}(\bx) \in A[z]^{\ell}$, it follows that
\[
\Exp_{\DD}(\bx) \in (\bszeta - 1) \Exp_{\DD}\bigl( \TT_{z,\bszeta}(K_{\infty})^{\ell} \bigr) + \DD(A[z,\bszeta]),
\]
from which it follows that $\bx \in (\bszeta - 1) \TT_{z,\bszeta}(K_{\infty})^{\ell} + \rU(\DD/A[z,\bszeta])$. Since $\bx = \ev(\bx)$, we have $\bx \in \ev(\rU(\DD/A[z,\bszeta])) = \rU_{\St}(\EE/A[z])$, and thus $\ker \oalpha \subseteq \rU_{\St}(\EE/A[z])$.
\end{proof}

We now set
\[
\rU_{\St}(\EE/\AAA) \assign \Span_{\FF_q(z)} \bigl( \rU_{\St}(\EE/A[z]) \bigr) \subseteq \rU(\EE/\AAA),
\]
which is an inclusion of $\bA$-modules. We recall also the covolume $[-:-]_{\AAA}$ defined in~\cite{Demeslay22}*{\S 2}. We obtain the following theorem, which is a direct analogue of \cite{AT17}*{Thm.~1} in our setting (see also \cite{ANT22}*{Thm.~4.7}).

\begin{theorem}[{cf.\ Angl\`es, Tavares Ribeiro~\cite{AT17}*{Thm.~1}}] \label{T:StarkLvalue}
For $\EE = \EE(\phi \times \psi) : \bA \to \Mat_{\ell}(A[z][\tau])$, let $\DD = \DD(\phi \times \psi) : \tbA \to \Mat_{\ell}(A[z,\bszeta][\tau])$ be its $\bszeta$-deformation.
\begin{alphenumerate}
\item $\rU(\EE/\AAA)/\rU_{\St}(\EE/\AAA)$ is a finitely generated torsion $\bA$-module, and
\[
\biggl[ \frac{\rU(\EE/\AAA)}{\rU_{\St}(\EE/\AAA)} \biggr]_{\AAA}
= \bigl[ \rH(\EE/\AAA) \bigr]_{\AAA}.
\]
\item Moreover,
\[
L(\EE(\phi \times \psi)^{\vee},0) = \bigl[ \Lie(\EE)(\AAA) : \rU_{\St}(\EE/\AAA) \bigr]_{\AAA} = \Reg_{\DD}|_{\bszeta = 1} = \cL(\DD/\tAA)|_{\bszeta=1}.
\]
\end{alphenumerate}
\end{theorem}

\begin{proof}
The finite generation of $\rU(\EE/\AAA)$ over $\bA$ follows from Proposition~\ref{P:integrality1}(a)--(b). We also have isomorphisms of $\bA$-modules,
\begin{equation} \label{E:UmodUStisos}
\frac{\rU(\EE/\AAA)}{\rU_{\St}(\EE/\AAA)}
\cong \FF_q(z) \otimes_{\FF_q[z]} \frac{\rU(\EE/A[z])}{\rU_{\St}(\EE/A[z])}
\cong \FF_q(z) \otimes_{\FF_q[z]} \rH(\DD/A[z,\bszeta])[\bszeta - 1],
\end{equation}
where the second isomorphism follows from Proposition~\ref{P:oalphaiso}. In the proof of Proposition~\ref{P:integrality2}, it is shown that $\FF_q(z) \otimes_{\FF_q[z]} \rH(\DD/A[z,\bszeta])$ is a finitely generated torsion $\FF_q(z)[\bszeta]$-module, and hence is finite dimensional over $\FF_q(z)$. Thus, $\FF_q(z) \otimes_{\FF_q[z]} \rH(\DD/A[z,\bszeta])[\bszeta - 1]$ is a finitely generated torsion $\bA$-module, whence so is $\rU(\EE/\AAA)/\rU_{\St}(\EE/\AAA)$.

Evaluation yields an exact sequence of $\bA$-modules,
\begin{multline*}
0 \to (\bszeta - 1) \cdot \FF_q(z) \otimes_{\FF_q[z]} \rH(\DD/A[z,\bszeta]) \to \FF_q(z) \otimes_{\FF_q[z]} \rH(\DD/A[z,\bszeta]) \\
\xrightarrow{\ev} \FF_q(z) \otimes_{\FF_q[z]} \rH(\EE/A[z]) \to 0.
\end{multline*}
Just as in Remark~\ref{R:EEclass}, we have an isomorphism of $\bA$-modules,
\[
\FF_q(z) \otimes_{\FF_q[z]} \rH(\EE/A[z]) \cong \rH(\EE/\AAA),
\]
and thus we obtain an exact sequence of $\bA$-modules,
\begin{multline*}
0 \to \FF_q(z) \otimes_{\FF_q[z]} \rH(\DD/A[z,\bszeta])[\bszeta - 1] \to \FF_q(z) \otimes_{\FF_q[z]} \rH(\DD/A[z,\bszeta]) \\
\xrightarrow{(\bszeta - 1)(\cdot)} \FF_q(z) \otimes_{\FF_q[z]} \rH(\DD/A[z,\bszeta])
\to \rH(\EE/\AAA) \to 0.
\end{multline*}
As the taking of Fitting ideals alternates in exact sequences \cite{AT17}*{\S 2.1}, we conclude that
\[
\bigl[ \rH(\EE/\AAA) \bigr]_{\AAA}
= \bigl[ \FF_q(z) \otimes_{\FF_q[z]} \rH(\DD/A[z,\bszeta])[\bszeta - 1] \bigr]_{\AAA},
\]
Then Proposition~\ref{P:oalphaiso} and~\eqref{E:UmodUStisos} complete the equality in~(a).

For (b), the final equality follows from Proposition~\ref{P:integrality2} upon evaluation of $\bszeta=1$. We already know that $L(\EE(\phi \times \psi)^{\vee},0) = \cL(\DD/\tAA)|_{\bszeta=1}$ from~\eqref{E:DAeval}. Finally, from Theorem~\ref{T:Demeslay} and Proposition~\ref{P:LEE0} we have
\begin{align*}
L(\EE(\phi \times \psi)^{\vee},0)
&= \bigl[ \Lie(\EE)(\AAA) : \rU(\EE/\AAA) \bigr]_{\AAA} \cdot \bigl[ \rH(\EE/\AAA) \bigr]_{\AAA} \\
&= \bigl[ \Lie(\EE)(\AAA) : \rU(\EE/\AAA)  \bigr]_{\AAA} \cdot \biggl[ \frac{\rU(\EE/\AAA)}{\rU_{\St}(\EE/\AAA)} \biggr]_{\AAA} \\
&= \bigl[ \Lie(\EE)(\AAA) : \rU(\EE/\AAA) \bigr]_{\AAA} \cdot \bigl[ \rU(\EE/\AAA) : \rU_{\St}(\EE/\AAA) \bigr]_{\AAA} \\
&= \bigl[ \Lie(\EE)(\AAA) : \rU_{\St}(\EE/\AAA) \bigr]_{\AAA},
\end{align*}
where the first equality follows from~(a), and the remaining ones utilize basic properties of $\AAA$-orders (Fitting ideals) and covolumes (see~\cite{AT17}*{\S 2.1}).
\end{proof}

\subsection{Log-algebraicity considerations} \label{SS:logalg}
If we choose $\bseta_{1}, \dots, \bseta_{\ell} \in \rU(\DD/A[z,\bszeta])$ as in Proposition~\ref{P:integrality2}, then
\[
\Exp_{\DD}(\bseta_i) \rassign \bsbeta_i \in A[z,\bszeta], \quad 1 \leqslant i \leqslant \ell,
\]
generate an $\tbA$-submodule of $\DD(\tAA)$. In this way $\bseta_1, \dots, \bseta_\ell$ are log-algebraic in the sense of \cites{And94, And96}. Determining the precise values of the $\bsbeta_i$'s is challenging, as one sees for log-algebraicity of Drinfeld modules and Anderson $t$-modules over finite extensions of~$K$, as e.g., in~\cites{And96, ANT20, APT18, AT17, AnglesTaelman15, P22, Thakur}. Moreover,
\begin{equation}
\bseta_i|_{\bszeta=1} \in \rU_{\St}(\EE/\AAA), \quad 1 \leqslant i \leqslant \ell,
\end{equation}
form an $\bA$-basis of $\rU_{\St}(\EE/\AAA)$. Applying $\Exp_{\EE}$, the elements
\[
\Exp_{\EE}\bigl( \bseta_i|_{\bszeta=1} \bigr) = \bsbeta_i|_{\bszeta=1} \in \EE(A[z]), \quad 1 \leqslant i \leqslant \ell,
\]
generate a (finitely generated) $\bA$-submodule of $\EE(\AAA)$, namely
\begin{equation}
\cS_{\EE} \assign \Exp_{\EE} \bigl( \rU_{\St}(\EE/\AAA) \bigr),
\end{equation}
which is a candidate for the module of \emph{special points} for $\EE$ in the sense of Anderson~\cite{And96}. Indeed, as pointed out in \cite{Tavares21}*{\S 7.4.3}, Angl\`es and Taelman~\cite{AnglesTaelman15}*{\S 7, Thm.~7.5} proved in the case of the Carlitz module that the image of the module of Stark units is precisely Anderson's module of special points. Moreover, the identity $[\rU(\EE/\AAA)/\rU_{\St}(\EE/\AAA)]_{\AAA} = [\rH(\EE/\AAA)]_{\AAA}$ of Theorem~\ref{T:StarkLvalue} betokens $\cS_{\EE}$ playing the role of the group of circular units for $\EE$. It would be interesting to fully unravel the log-algebraicity theory for $\EE$.

Log-algebraicity of special points on $t$-modules are richly intertwined with special $L$-values, going back to work of Anderson and Thakur~\cite{AndThak90} and Anderson~\cites{And94, And96}. For example, the reader is directed to \cites{ANT17b, ANT20, AnglesPellarin14, APT18, AnglesTaelman15, AT17, CEP18, GreenNgoDac23, GreenP18, Tavares21, Thakur} for applications of log-algebraicity to $L$-values in different contexts. Corollary~\ref{C:deformationconv} and Proposition~\ref{P:integrality2} imply that
\begin{equation}
\begin{cases}
\cL(\tAA, \chi_{\phi} \ochi_{\psi}) \cdot \cL(\bsmu_{\phi,\theta} \times \bsnu_{\psi,z}), &\textup{if $r = \ell$,} \\
\cL(\bsmu_{\phi,\theta} \times \bsnu_{\psi,z}), &\textup{if $r \neq \ell$,}
\end{cases}
\end{equation}
are determinants of logarithms of elements of $\DD(A[z,\bszeta])$, namely $\bseta_1, \dots, \bseta_{\ell}$. Moreover, Theorem~\ref{T:StarkLvalue}, together with Theorems~\ref{T:LEErxr} and~\ref{T:LEErxl} and Remark~\ref{R:LEELDDeval}, implies that
\begin{equation}
\begin{cases}
L(\AAA,\chi_{\phi},\ochi_{\psi},1) \cdot L(\bsmu_{\phi,\theta} \times \bsnu_{\psi,z},0), &\textup{if $r = \ell$,} \\
L(\bsmu_{\phi,\theta} \times \bsnu_{\psi,z},0), &\textup{if $r \neq \ell$,}
\end{cases}
\end{equation}
can then be interpreted as determinants of logarithms of special points in~$\cS_{\EE}$. The exact way that this unfolds for specific Drinfeld modules would be an interesting undertaking.

\end{document}